\documentclass[3p]{elsarticle}

\usepackage{amssymb}
\usepackage{amsmath}
\usepackage{amsthm}

\usepackage{natbib}

\usepackage{mathtools}
\usepackage{float}
\usepackage{tikz}
\usepackage{dutchcal}
\usepackage[font=footnotesize]{subcaption}
\usepackage{accents}
\usepackage[inline,shortlabels]{enumitem}
\setlist{itemsep=0\baselineskip} 
\usepackage{ifthen}
\usepackage[utf8]{inputenc} 
\usepackage{xcolor}
\usepackage{mathrsfs}
\usepackage[bookmarks=true, colorlinks=true, linkcolor=blue!50!black,
citecolor=orange!50!black, urlcolor=orange!50!black, pdfencoding=unicode]{hyperref}
\usepackage[capitalize]{cleveref}

\theoremstyle{definition}
\newtheorem{definition}{Definition}[section]
\newtheorem{example}[definition]{Example}
\newtheorem{remark}[definition]{Remark}
\newtheorem{notation}[definition]{Notation}

\theoremstyle{plain}

\newtheorem{theorem}[definition]{Theorem}
\newtheorem{proposition}[definition]{Proposition}
\newtheorem{corollary}[definition]{Corollary}
\newtheorem{lemma}[definition]{Lemma}

\newtheorem{assumption}{Assumption}
\newtheorem*{theorem*}{Theorem}
\newtheorem*{proposition*}{Proposition}
\newtheorem*{corollary*}{Corollary}
\newtheorem*{lemma*}{Lemma}
\newtheorem*{warning*}{Warning}

\crefname{assumption}{Assumption}{Assumptions}

\crefformat{enumi}{\##2#1#3}
\crefalias{chapter}{section}

\usetikzlibrary{ 
  cd,
  math,
  decorations.markings,
  decorations.pathreplacing,
  positioning,
  arrows,
  arrows.meta,
  shapes,
  shadows,
  shadings,
  calc,
  fit,
  quotes,
  intersections,
  circuits,
  circuits.ee.IEC,
  calligraphy
}

\tikzset{
  biml/.tip={Glyph[glyph math command=triangleleft, glyph length=1.1ex]},
  bimr/.tip={Glyph[glyph math command=triangleright, glyph length=1.1ex]},
}

\tikzset{
  tick/.style={postaction={
      decorate,
      decoration={markings, mark=at position 0.5 with
    	{\draw[-] (0,.4ex) -- (0,-.4ex);}}}
  }
} 
\tikzset{
  slash/.style={postaction={
      decorate,
      decoration={markings, mark=at position 0.5 with
    	{\draw[-] (.3ex,.3ex) -- (-.3ex,-.3ex);}}}
  }
}
\tikzset{
  cotor/.style={postaction={
      decorate,
      decoration={markings, mark=at position 0.5 with
    	{\draw[->,dashed,line width=0.05mm] (1ex,.6ex) -- (-1.8ex,.6ex);}}}
  }
}
\tikzset{
  cotol/.style={postaction={
      decorate,
      decoration={markings, mark=at position 0.5 with
    	{\draw[->,dashed,line width=0.05mm] (1ex,-.6ex) -- (-1.8ex,-.6ex);}}}
  }
}

\definecolor{rcol}{RGB}{200,0,0}
\definecolor{gcol}{RGB}{0,120,0}
\definecolor{bcol}{RGB}{0,0,210}
\definecolor{ycol}{RGB}{210,130,0}
\definecolor{ocol}{RGB}{210,100,0}

\tikzset{corollas/.style={scale=.9,baseline}}
\tikzset{smallcorollas/.style={scale=.5,baseline}}

\tikzset{vertex/.style={circle,draw,inner sep=1pt,minimum size=6,text=black}}
\tikzset{edge/.style={-stealth}}
\tikzset{label/.style={rectangle,rounded corners,fill=white,inner
    sep=1pt,text=black}}

\tikzset{annot/.style={font=\scriptsize}}
\tikzset{dasht/.style={dashed,draw opacity=0.5}}
\tikzset{transparent/.style={draw opacity=0.3,fill opacity=0.3}}
\tikzset{small/.style={minimum size=3.5}}
\tikzset{large/.style={minimum size=10}}
\tikzset{larger/.style={minimum size=12}}
\tikzset{huge/.style={minimum size=17}}
\tikzset{huger/.style={minimum size=27}}

\tikzset{stringd/.style={baseline}}
\tikzset{container/.style={line width=0.25mm, rounded corners}}
\tikzset{hub/.style={container}}
\tikzset{wire/.style={}}
\tikzset{obj/.style={font=\scriptsize}}

\newcommand{\bito}[1][]{\mathbin{\begin{tikzcd}[ampersand replacement=\&, cramped]\ar[r, biml-bimr, "{#1}"]\&{}\end{tikzcd}}}

\DeclareSymbolFont{symbolsC}{U}{pxsyc}{m}{n}
\DeclareMathSymbol{\multimapdotbothA}{\mathrel}{symbolsC}{23}

\newcommand{\adj}[5][30pt]{
  \begin{tikzcd}[ampersand replacement=\&, column sep=#1]
    #2\ar[r, shift left=7pt, "#3"]
    \ar[r, phantom, "\scriptstyle\Rightarrow"]\&
    #5\ar[l, shift left=7pt, "#4"]
  \end{tikzcd}
}

\newcommand{\xtickar}[1]{\begin{tikzcd}[baseline=-0.5ex,cramped,sep=small,ampersand 
    replacement=\&]{}\ar[r,tick, "{#1}"]\&{}\end{tikzcd}}

\DeclareSymbolFont{stmry}{U}{stmry}{m}{n}
\DeclareMathSymbol\fatsemi\mathop{stmry}{"23}

\DeclareFontFamily{U}{mathx}{\hyphenchar\font45}
\DeclareFontShape{U}{mathx}{m}{n}{
  <5> <6> <7> <8> <9> <10>
  <10.95> <12> <14.4> <17.28> <20.74> <24.88>
  mathx10
}{}
\DeclareSymbolFont{mathx}{U}{mathx}{m}{n}
\DeclareFontSubstitution{U}{mathx}{m}{n}
\DeclareMathAccent{\widecheck}{0}{mathx}{"71}

\renewcommand{\ss}{\subseteq}

\DeclareMathAlphabet{\mathbf}{OT1}{cmr}{b}{n}
\DeclareFontSeriesDefault[rm]{bf}{b}

\DeclarePairedDelimiter{\present}{\langle}{\rangle}

\DeclarePairedDelimiter{\corners}{\ulcorner}{\urcorner}

\newcommand{\res}[2]{\left.{#1}\right|_{#2}}

\DeclareMathOperator{\dom}{dom}
\DeclareMathOperator{\cod}{cod}

\DeclareMathOperator*{\colim}{colim}

\DeclareMathOperator{\ob}{Ob}

\DeclareMathOperator{\el}{El}

\newcommand{\ord}[1]{\underaccent{\bar}{#1}}
\newcommand{\cat}[1]{#1}
\newcommand{\Cat}[1]{\mathbf{#1}}

\newcommand{\call}[1]{\mathcal{#1}}

\newcommand{\id}{\mathrm{id}}
\newcommand{\then}{\mathbin{\fatsemi}}

\newcommand{\too}{\longrightarrow}
\newcommand{\tto}{\rightrightarrows}
\newcommand{\To}[2][]{\xrightarrow[#1]{#2}}

\newcommand{\Too}[1]{\xrightarrow{\;\;#1\;\;}}
\newcommand{\from}{\leftarrow}

\newcommand{\From}[1]{\xleftarrow{#1}}
\newcommand{\Fromm}[1]{\xleftarrow{\;\;#1\;\;}}

\newcommand{\imp}{\Rightarrow}
\renewcommand{\iff}{\Leftrightarrow}

\newcommand{\tickar}{\xtickar{}}

\newcommand{\profunctor}{\tickar}

\newcommand{\inv}{^{-1}}
\newcommand{\op}{^\tn{op}}
\newcommand{\co}{^\tn{co}}

\newcommand{\tn}[1]{\textnormal{#1}}

\newcommand{\codisc}[1]{c_{#1}}

\newcommand{\dual}[1]{{#1}^\vee}
\newcommand{\ddual}[1]{{#1}^{\vee\vee}}

\newcommand{\yoncool}[2]{{}_{#1} I_{#2}}
\newcommand{\pnt}[1]{\dot{#1}}

\newcommand{\catfun}[2]{{#2}^{#1}}

\newcommand{\nn}{\mathbb{N}}

\newcommand{\zz}{\mathbb{Z}}
\newcommand{\rr}{\mathbb{R}}

\newcommand{\finset}{\Cat{Fin}}
\newcommand{\smset}{\Cat{Set}}
\newcommand{\smcat}{\Cat{Cat}}
\newcommand{\catsharp}{\Cat{Cat}^{\sharp}}

\newcommand{\ppolyfun}{\mathbb{P}\Cat{olyFun}}
\newcommand{\ccat}{\mathbb{C}\Cat{at}}
\newcommand{\ccatsharp}{\mathbb{C}\Cat{at}^{\sharp}}
\newcommand{\ccatsharpdisc}{\ccatsharp_{\tn{disc}}}
\newcommand{\ccatsharplin}{\ccatsharp_{\tn{lin}}}
\newcommand{\ccatsharpcon}{\ccatsharp_{\tn{con}}}
\newcommand{\ccatsharpdisccon}{\ccatsharp_{\tn{disc,con}}}
\newcommand{\sspan}{\mathbb{S}\Cat{pan}}
\newcommand{\spancat}{\sspan}
\newcommand{\catsharpspan}{\spancat_{\mathrm{retro}}(\smcat)}
\newcommand{\polyfunb}{\ccatsharp}
\newcommand{\mon}{\Cat{Mon}}
\newcommand{\comon}{\Cat{Comon}}
\newcommand{\ccomod}{\mathbb{C}\Cat{omod}}
\newcommand{\comod}{\ccomod}
\newcommand{\mmod}{\mathbb{M}\Cat{od}}
\newcommand{\coco}{\Cat{Fam}}

\newcommand{\prov}{\Cat{E}_\tn{ver}}
\newcommand{\proh}{\mathbb{E}_\tn{hor}}
\newcommand{\vcomp}{K_v}
\newcommand{\hcomp}{K_h}

\newcommand{\Praf}{\Cat{PRA}}

\newcommand{\set}{\tn{-}\Cat{Set}}

\newcommand{\yon}{\mathcal{y}}
\newcommand{\poly}{\Cat{Poly}}

\newcommand{\tri}{\mathbin{\triangleleft}}

\newcommand{\agg}{\circledast}
\newcommand{\rdag}{^{\rotatebox{0}{$\dagger$}}}
\newcommand{\ldag}{^{\rotatebox{180}{$\dagger$}}}
\newcommand{\ot}[2]{\,\mbox{${}_{#1}\otimes_{#2}$}\,}
\newcommand{\ih}[4]{{}_{#1}[#2,#3]_{#4}}
\newcommand{\thg}{{\mathord{\text{--}}}}

\newcommand{\alttensor}{\odot}
\newcommand{\altbito}{\tickar}
\newcommand{\lpush}{_*}

\newcommand{\biglens}[2]{
  \begin{bmatrix}{\vphantom{f}#2} \\ {\vphantom{f}#1} \end{bmatrix}
}
\newcommand{\littlelens}[2]{
  [\begin{smallmatrix}{\vphantom{f}#2} \\ {\vphantom{f}#1} \end{smallmatrix}]
}
\newcommand{\lens}[2]{
  \mathchoice{\biglens{#1}{#2}}{\littlelens{#1}{#2}}{\littlelens{#1}{#2}}{\littlelens{#1}{#2}}
}

\newcommand{\biglenskan}[4]{
  \,\mbox{$\prescript{}{#3}{\begin{bmatrix}{\vphantom{f_f^f}#2} \\ {\vphantom{f_f^f}#1} \end{bmatrix}_{#4}}$}\,
}
\newcommand{\littlelenskan}[4]{
  \,\mbox{$\prescript{}{#3}{\begin{bsmallmatrix}{\vphantom{f}#2} \\ {\vphantom{f}#1} \end{bsmallmatrix}_{#4}}$}\,
}
\newcommand{\lenskan}[4]{
  \relax\if@display
    \biglenskan{#1}{#2}{#3}{#4}
  \else
    \littlelenskan{#1}{#2}{#3}{#4}
  \fi
}

\newcommand{\qqand}{\qquad\text{and}\qquad}

\newcommand{\coto}{\overset{\raisebox{-2pt}{\smash{$\scriptstyle\!\!\dashleftarrow$}}}{\longrightarrow}}

\journal{Journal of Pure and Applied Algebra}

\begin{document}

\begin{frontmatter}



  \title{Functorial Aggregation}


  \author[add1]{David I. Spivak\corref{cor1}} 
  \cortext[cor1]{Corresponding author.}
  \address[add1]{Topos Institute, Berkeley, CA 94704, United States}
  \ead{david@topos.institute}

  \author[add2]{Richard Garner} 
  \address[add2]{School of Mathematical and Physical
    Sciences, Macquarie University, NSW 2109, Australia}
  \ead{richard.garner@mq.edu.au}

  \author[add3]{Aaron David Fairbanks} 
  \address[add3]{Department of Mathematics and Statistics, Dalhousie
    University, Halifax, NS B3H 4R2, Canada}
  \ead{adavidfairbanks@gmail.com}

  \begin{abstract}
    We study polynomial comonads and polynomial
    bicomodules. Polynomial comonads amount to categories. Polynomial
    bicomodules between categories amount to parametric right adjoint
    functors between corresponding copresheaf categories. These may
    themselves be understood as generalized polynomial functors. They
    are also called data migration functors because of applications in
    categorical database theory. We investigate several universal
    constructions in the framed bicategory of categories,
    retrofunctors, and parametric right adjoints. We then use the
    theory we develop to model database aggregation alongside
    querying, all within this rich ecosystem.
  \end{abstract}



  \begin{keyword}
    aggregation \sep categories \sep polynomial functors \sep
    retrofunctors \sep parametric right adjoints \sep comonads \sep bicomodules


  \end{keyword}

\end{frontmatter}






\section{Introduction}

In this paper, we investigate multi-variable polynomials, but instead
of sets of variables, our polynomials have categories of
variables. This needs explanation. Let us start with what
\emph{polynomial} means. Usually, a polynomial is a function
\[x \quad \mapsto \quad \sum_{i \in I} x^{n_i}.\]
We can plug numbers in and get numbers out. A first
generalization is to replace the numbers with sets
\[X \quad \mapsto \quad \!\:\sum_{i \in I} X^{S_i}\mathrlap{\,\: \coloneqq\: \sum_{i \in I}
    \smset(S_i, X)}\] so that we can plug sets in and get sets
out. Instead of a function, this is a \emph{functor}
$\smset \to \smset$. Hence we see polynomials as the objects of a
category $\poly$: a full subcategory of $\catfun{\smset}{\smset}$. As
we would expect of polynomials, we find they are closed under
composition (\cref{prop.comp_monoidal}), i.e., $\poly$ inherits
monoidal structure from $\catfun{\smset}{\smset}$. There is much more
to say about $\poly$, and we will return to it soon.

A second generalization is to use not just one input variable, but several
of them; and while we are about it, we may as well also allow
multiple output variables. Now a polynomial is of the form
\begin{equation}
  \label{eqn.multivarpoly}
  (X_a)_{a \in A} \quad \mapsto \quad (Y_b)_{b \in B} \qquad \text{where} \qquad
  Y_b \coloneqq \sum_{i \in I_b} \prod_{j \in J_i} X_{s(j)}
\end{equation}
so we plug $A$-many sets in, and get $B$-many sets out. Such a
multi-variable polynomial is equivalently expressed by a diagram of functions:
\begin{equation}
  \label{eqn.bridgediag}
  \begin{tikzcd}
    &J\ar[dl, "s"']\ar[r, "\pi"]&I\ar[dr, "t"]\\[-10pt]
    A&&&B
  \end{tikzcd}
\end{equation}
where the fibers of $t$ are the sets $I_b$, and the fibers of $\pi$
are the sets $J_i$.
This \emph{bridge diagram}
%
acts as a blueprint of the corresponding functor
$\smset^A \to \smset^B$. Indeed, calculating the value of a
polynomial functor like~\eqref{eqn.multivarpoly} can be done in three steps:
\begin{itemize}
\item Step $\Delta$: copy the $A$-indexed family of inputs into a
  $J$-indexed family of variable instances;
\item Step $\Pi$: multiply variable instances to obtain an $I$-indexed family of monomials;
\item Step $\Sigma$: sum the monomials appropriately to obtain a $B$-indexed
  family of outputs.
\end{itemize}
These three steps are themselves functors, controlled by the maps $s$,
$\pi$ and $t$ respectively. Indeed:
\begin{itemize}
\item $\Delta_s\colon \smset^A \to \smset^J$ is reindexing along $s$;
\item $\Pi_\pi\colon \smset^J \to \smset^I$ is the right adjoint to
  reindexing along $\pi$; and
\item $\Sigma_t\colon \smset^I \to \smset^B$ is the left adjoint to
reindexing along $t$,
\end{itemize}
and the whole multi-variable polynomial functor corresponding
to~\eqref{eqn.bridgediag} is the composite
$\Sigma_t\circ\Pi_\pi\circ\Delta_s$. Multivariable polynomials
were first described in this way
in~\cite{abbott2003categories,gambino2003wellfounded}; for a
comprehensive account, see~\cite{kock2012polynomial}. 

In fact, polynomial functors between powers of $\smset$ are closed
under composition, and so we obtain a $2$-category whose objects are
sets $A,B,C,\dots$ and where the hom-category from $A$ to $B$ is the
category of polynomial functors $\smset^A \rightarrow \smset^B$. As
shown by Gambino and Kock in \cite{kock2012polynomial}, this
$2$-category can be enhanced to a framed bicategory
$\ppolyfun_\smset$, whose structure can be described either in terms
of the polynomial functors themselves, or else in terms of their
representing bridge diagrams. A useful equivalent characterization of
such polynomial functors is that they are \emph{parametric right
  adjoint} functors (\cref{def.prafunctor})---or \emph{pra-functors}
for short.

We now come to our third generalization of polynomials which, as
promised, is to replace the sets of variables $A,B \in \smset$ with
categories $\cat{c}, \cat{d}\in \smcat$.
As opposed to a mere $A$-indexed family of sets, a \emph{$\cat{c}$-set} (a
functor $\cat{c} \to \smset)$ is able to model richer data, in that it
records relationships between elements (for each arrow
$f\colon a_1 \to a_2$ in $\cat{c}$, each element of type $a_1$ points via $f$ to
a particular element of type $a_2$).
It is the basic form of data that would be stored in a database (as in
\cite{spivak2012functorial}).

Indeed, it is straightforward to model the organizational form (the
\emph{schema}) of a database as a category $\cat{c}$ and the data in
the database as a $\cat{c}$-set. (A $c$-set is simply an extremely
general form of organized data, not tied to any particular
implementation of databases.) We may then use the left and right Kan
extensions
\[
  \begin{tikzcd}[column sep=50pt]
    \cat{c}\set\ar[r, shift left=7pt, "\Sigma_F"]\ar[r,shift right=7pt, "\Pi_F"']&
    \cat{d}\set\ar[l, "\Delta_F" description]
  \end{tikzcd}
\]
induced by a functor $F\colon\cat{c}\to\cat{d}$ to \emph{migrate data}
between databases. (We will look more concretely at what these
functors do later.)

The reader might expect that the appropriate generalization of
polynomial in this context would be a functor of the form
\begin{equation}\label{eqn.prafunctor}
  (\Sigma_T\circ\Pi_\pi\circ\Delta_S)\colon\cat{c}\set\too\cat{d}\set
\end{equation}
induced by a ``bridge diagram'' of categories:
\begin{equation}
  \label{eqn.catbridge}
    \begin{tikzcd}
    &\cat{e}\ar[dl, "S"']\ar[r, "\pi"]&\cat{b}\ar[dr, "T"]\\[-10pt]
    \cat{c}&&&\cat{d}\rlap{ .}
  \end{tikzcd}
\end{equation}

However, it turns out this is slightly \emph{too} general, and not
quite what we want in practice. Indeed, for databases, general
functors of the form $\Sigma_T$ are not as useful as one might think,
because they can identify data. For example, identifying two users in
a database would collaterally identify the contents of what they point
to as well---it could easily equate different strings in such a way
that, \emph{across the entire database}, we have an equality of strings
$\text{``a''} = \text{``b''}$! As one can imagine, freely
interchanging arbitrary instbnces of ``a'' with ``b'' cbuses proalems.

Still, there is a special case where $\Sigma_T$ does describe a
useful operation for migrating data, namely summing: forming the union
of data from various sources without making identifications (as seen
in polynomials). Formally, we capture this by demanding $T$
is a discrete opfibration (what we call \emph{\'etale}), in which case $\Sigma_T$ can be
computed using coproducts rather than general colimits. The
mathematics works out more nicely too: given a bridge
diagram~\eqref{eqn.catbridge} in which $T\colon\cat{b}\to\cat{d}$ is
\'etale, the composite functor~\eqref{eqn.prafunctor} is a prafunctor,
and in fact, every prafunctor from $c$-sets to $d$-sets can be
obtained in this way (\cref{prop.bridge_diagram}).

Just as with polynomial functors between powers of $\smset$,
prafunctors between categories of the form
$c\set \coloneqq \smset^{\cat{c}}$ are closed under composition. So we
obtain a $2$-category of prafunctors which, as we will see shortly,
can be enhanced to a framed bicategory which we term
$\ccatsharp$. This generalises Gambino and Kock's framed bicategory
$\ppolyfun_\smset$ in the sense that, when we restrict from objects of
the form $c\set$ for $c$ a category to ones of the form $\smset^C$ for
$C$ a set, we obtain exactly $\ppolyfun_\smset$ (which in this article
we denote by $\ccatsharpdisc$).

In database terminology, prafunctors model \emph{disjoint unions of
  conjunctive queries}. (In the language of categories, a conjunctive
query is a limit operation: it returns the set of tuples satisfying
certain relationships, i.e., linked by certain arrows. So a disjoint
union of conjunctive queries refers to a coproduct of limits.) For
example, in the US, every city is in a state and every county is also
in a state, i.e., we have a cospan of sets
\begin{equation}\label{eqn.cities}
  \text{city}\to\text{state}\from\text{county}
\end{equation}
and one can perform a disjoint union of conjunctive queries on such
data, e.g., asking for the set
\begin{equation}\label{eqn.duc_query_example}
  (\text{city}\times_{\text{state}}\text{city})+
  (\text{city}\times_{\text{state}}\text{county})+
  (\text{county}\times_{\text{state}}\text{county}),
\end{equation}
and this is modelled by a prafunctor $\cat{c}\set \rightarrow
\smset$ where $c$ is the category $\bullet \rightarrow \bullet
\leftarrow \bullet$ indexing diagrams of the form~\eqref{eqn.cities};
see~\cref{ex.variable_is_duc} below.

We sometimes refer to prafunctors as \emph{data migration functors},
because a prafunctor as in \eqref{eqn.prafunctor} migrates data from
$\cat{c}$ to $\cat{d}$ by assigning to each object of $\cat{d}$ a
disjoint union of conjunctive queries on $\cat{c}\set$.

\subsection{From $\poly$ to $\ccatsharp$}
Having learned a bit about fancy polynomials, let us return to
consider the monoidal category of single-variable polynomials
$\poly \subset \catfun{\smset}{\smset}$.
What are the monoids and comonoids for the composition
monoidal structure? Naturally, they are polynomial monads and comonads
on $\smset$. The monoids are interesting, but comonoids will be more
relevant to us, as, in fact:

\begin{center}
  \emph{Polynomial comonads are categories.}
\end{center}

This was noticed by Ahman and Uustalu in 2016
\cite{ahman2016directed}, and is already intriguing. Who would have
thought that the concept of category is inherent in the notion of
single-variable polynomial? In fact, the situation is even more
interesting than this. We will show that:

\begin{center}
  \emph{Polynomial bicomodules between categories $c$ and $d$ are
    parametric right adjoint functors $c\set \to d\set$.}
\end{center}

Thus, the entire generalization from single-variable polynomials, to
set-valued multi-variable polynomials, to category-valued
multi-variable polynomials, can be found sitting right there in the
category of single-variable polynomials! What is more, we can obtain
not just categories and prafunctors, but also the whole apparatus
which surrounds them, namely, the framed bicategory $\ccatsharp$
mentioned above. Indeed, from any monoidal category $\mathcal{V}$
whose tensor product satisfies mild limit-preservation properties we
can build a framed bicategory of comonoids, comonoid homomorphisms and
comonoid bicomodules~\cite[Example~11.6]{shulman2008framed}. This
construction (suitably dualized) is precisely the one that produces
rings, ring homomorphisms, and bimodules, as arising from abelian groups, as well as
categories, functors, and profunctors, as arising from spans.

The first main result of this paper says that, when applied to the
monoidal category $\poly$, this construction produces exactly
$\ccatsharp$. Indeed:
\begin{itemize}
\item The objects, i.e.\ the comonoids in $\poly$, are categories;
\item The loose maps, the bicomodules between comonoids, are
 the prafunctors $\cat{c}\set \rightarrow \cat{d}\set$;
\item The tight maps, i.e.\ the comonoid homomorphisms, are
  not as one might expect functors between categories, but rather
  \emph{retrofunctors} as in \cref{def.retrofunctor} below.
\end{itemize}
We may choose to see this as the deity of category theory telling us
what the notion of multi-variable polynomial should be when the
variable types form a category rather than a mere set: it is a data
migration~functor.

Thus we turn to $\poly$ (and $\ccatsharp$ begotten by it)
to understand data.


\subsection{Aggregation}
Suppose given a function $\pi\colon E\to D$, say $E$ is the set of
employees in a company, $D$ is the set of departments, and every
employee $e\in E$ works in a department $\pi(e)\in D$.%
\footnote{We should also assume that the fibers are \emph{compact} in
  the sense that for every $d\in D$, the preimage $\pi\inv(d)\ss E$ is
  a finite set.} 
Now if there is a function $s\colon E\to \rr$, assigning a salary
$s(e)$ to each employee, 
we can \emph{aggregate} the salaries of the employees in each
department to get a total department salary. That is, there should be
an induced function $(\mathtt{sum}\ s)_\pi\colon D\to\rr$ as in the
diagram 
\begin{equation}\label{eqn.aggregate}
  \begin{tikzcd}
    E\ar[r, "s"]\ar[d, "\pi"']&\rr\\
    D\ar[ur, dashed, "(\mathtt{sum}\ s)_\pi"']
  \end{tikzcd}
\end{equation}
where \eqref{eqn.aggregate} is not intended to commute. The aggregate
function is given by 
\begin{equation}\label{eqn.sum_example}
  (\mathtt{sum}\ s)_\pi(d)\coloneqq\sum_{\pi(e)=d}s(e).
\end{equation}
In general, we can replace $\rr$ with any commutative monoid, and call
the induced map aggregation.  

Since the free
commutative monoid on a set $E$ is the monoid $M(E)$ of multisets in
$E$, perhaps the most important aggregation function is the
``group-by'' operation
\[
  \begin{tikzcd}
    E\ar[r]\ar[d, "\pi"']&M(E)\\
    D\ar[ur, dashed, "\mathtt{groupBy}_\pi"']
  \end{tikzcd}
\]
sending each $d\in D$ to the set
$\mathtt{groupBy}_\pi(d)\coloneqq\{e\in E\mid\pi(e)=d\}$. With
$\mathtt{groupBy}_\pi\colon D\to M(E)$ in hand, one can use maps out
of $M(E)$ to do things that range from very simple, like counting the rows in a
table, to more fancy operations like plotting graphs of data. As such, aggregation is arguably the most valuable practical operation one performs on databases.

Whereas data migration (querying) is functorial, aggregation is not. A
map of instances as $c$-sets is a natural transformation; for example
inserting a new row into a table is natural. If we query a database
before and after someone inserts a new employee $e$ into $E$, we'll
get a map between the results. But if we aggregate before and after an
insertion, there will be no particular relation between the results:
the new sum of real numbers may be different from---greater or less than---the original. How might
we rectify this? And since aggregation is not functorial in the same
way that querying is, in what way could the querying and aggregation
stories possibly interoperate?

Indeed, modeling aggregation in the same framework as data migration
is a challenge. Someone might say ``Why don't you just do it? If you
want to add all the salaries up, just use the fact that $M$ is a
commutative monoid and add them up!'' In some sense this is right. If
there are no rules of the game, then nothing prevents us from
proceeding ad hoc. But $\ccatsharp$ is a context in which data
migration naturally lives. The universal constructions in $\ccatsharp$
are certain basic moves we can perform: e.g., bicomodule
composition (\cref{def.bicomodule}), taking Dirichlet products
(\cref{chap.dirichlet}), taking weak duals
(\cref{thm.linear_conjunctive_dual}). These are all constructions that
evidently naturally occur in this context, amounting to what makes
sense in the world of data migration. In \cref{chap.aggregation} we
describe aggregation in this context, playing by the rules of the
polynomial ecosystem.

\subsection{Plan of the paper}

We begin in \cref{chap.singlevariable} by introducing the category
$\poly$ of polynomial functors in one variable. In
\cref{chap.catsharp} we generalize to multi-variable polynomials,
which comprise a framed bicategory $\ccatsharp$; this is what we refer
to as the \emph{polynomial ecosystem}.

By Ahman-Uustalu's result, the objects in $\ccatsharp$ (categories)
can be identified with comonoids in $\poly$, and by Garner's result,
the horizontal maps in $\ccatsharp$ from $c$ to $d$ (parametric right
adjoint functors between $c\set$ and $d\set$) can be identified with
bicomodules in $\poly$. We first show this correspondence using an
abstract argument in \cref{chap.abstract}, and then we unpack the
details of the correspondence concretely in \cref{chap.concrete}. The
reader could pick either one of these sections to read first,
depending on taste.

We then in \cref{chap.structures} discuss two important subcategories
of $\ccatsharp$. The first is the full sub framed bicategory spanned
by the discrete categories. We observe this is equivalent to
Gambino-Kock's framed bicategory $\ppolyfun_\smset$ from
\cite{kock2012polynomial} (the set-variable polynomials described
above). Inside of this we find $\sspan\ss\ccatsharp$, the usual framed
bicategory of spans in $\smset$. Since categories are monoids in
$\sspan$, we see the ordinary framed bicategory of categories,
functors, and profunctors as living in the polynomial ecosystem.

In \cref{chap.dirichlet}, we construct a local monoidal closed
structure on $\ccatsharp$. We show how this structure lets us perform
the operation of transposing a span as a composite of two other, more
primitive operations.

In \cref{chap.aggregation}, we put the pieces together to explain how
aggregation fits into $\ccatsharp$. The explanation will use
everything discussed up to this point.

\subsection{Notation}

\begin{notation}\label{notation.main}
  We use upper-case letters to denote sets, e.g., $A,B\in\smset$ and
  occasionally to denote functors between categories, e.g., $F,G$. We
  use lower-case letters for polynomials, e.g., $p,q\in\poly$, for
  elements of sets, e.g., $i,j\in I$, and even for small categories,
  e.g., $c,d \in \smcat$. We denote large categories using bold, e.g.,
  $\smset$, $\poly$, and $\smcat$. We denote framed bicategories using
  blackboard bold font on the first letter, e.g., $\sspan$, $\ccat$
  and $\ccatsharp$.

  Note that framed bicategories may be viewed as either bicategories
  or double categories; our convention is that the blackboard bold
  font always denotes the bicategory of horizontal/loose 1-cells,
  while the non-blackboard bold font denotes the (strict 2-)category
  of vertical/tight 1-cells. For example, $\ccat$ denotes the the
  (framed) bicategory of categories and profunctors (vertical 1-cells
  are functors) while $\smcat$ denotes the category, or strict
  2-category, of categories and functors.

  Given a category $c$, we refer to functors $c\to\smset$
  as \emph{$c$-sets} (a.k.a.\ $c$-copresheaves) and denote the category
  $\catfun{c}{\smset}$ by $c\set$. Given a $c$-set
  $P\colon c\to\smset$, we denote its category of elements by
  $\el_c(P)\in\smcat$. We refer to objects
  $(i,x)\in\ob(\el_c(P))$ as \emph{elements} of $P$, i.e., an
  element of $P$ consists of an object $i\in\ob(c)$ and an
  element $x\in P(i)$.

  Given $n\in\nn$ we abuse notation and also write $n\in\smset$ to
  denote the set of $n$ elements $n \coloneqq \{\mathsf{1},\ldots, \mathsf{n}\}$.
  
  We sometimes write $=$ to denote the presence of a canonical
  isomorphism, so we can write $3^0=1$ and $3^1=3$, but we should
  \emph{not} write $3^2=9$.

  We often use the name of an object to denote the identity morphism
  on it, $c=\id_c$. We use $\then$ to denote diagrammatic composition
  and $\circ$ to denote Leibniz-ordered composition, e.g., given maps
  $c\To{f}d\To{g}e$, we have $f\then g=g\circ f$ and may write
  either.

  It is standard mathematical practice to denote a group $(G,e,*)$
  simply by $G$ (its \emph{carrier}). We call this \emph{carrier
    notation} and follow that convention throughout the article. Thus
  when a polynomial $p$ is endowed with extra structure, we will
  generally denote it simply by $p$.

  We denote an adjunction $F\dashv G$ (that is, $F$ is left adjoint to
  $G$) by
  \[\adj{c}{F}{G}{d}.\] 
  The 2-arrow always points along the left adjoint, because it denotes
  the direction of both the unit and the counit of the adjunction
  \[
    \eta\colon c\to F\then G\qqand 
    \epsilon\colon G\then F\to d.\]
  We will in later sections of the paper denote $F$'s left adjoint (if
  it has one) by $F\ldag$
  and denote $F$'s right adjoint (if it has one) by $F\rdag$.%
  \footnote{One remembers this notation in two ways: $\ldag$ looks
    more like an L and $\rdag$ looks more like an r, and in
    prafunctors, $\ldag$ will tend to push things down from the
    exponent to the base, whereas $\rdag$ will tend to push things up
    from the base to the exponent. But we'll also usually remind the
    reader which is which.}

  We use $\alttensor$ to denote the monoidal product in a generic
  monoidal category, reserving $\otimes$ for the Dirichlet product on
  polynomials introduced in \cref{chap.dirichlet}.
  
  We denote a profunctor $c\op \times d \to \smset$ between categories
  $c$ an $d$ by $c \profunctor d$; we also use this notation more
  generally for a bi(co)module in a generic monoidal category.
  
  For convenience of reference, \cref{tab.catnot} and \cref{tab.not}
  summarize several names of categories and symbols used in this
  article.
  \begin{table}[htbp]
    \begin{center}
      \begin{tabular}{c c p{10cm} }
        Notation & Introduced & Meaning\\[.1cm]
        $\smset$ & & Category of sets and functions \\[.1cm]
        $\finset$ & & Category of finite sets and functions \\[.1cm]
        $\smcat$ & & Category of categories and functors \\[.1cm]
        $\ccat$ & & Framed bicategory of categories, profunctors, and functors \\[.1cm]
        $\sspan$ & & Framed bicategory of sets, spans, and functions \\[.1cm]
        $\mon(\mathscr{V})$ & & Category of monoids in monoidal category $\mathscr{V}$\\[.1cm]
        $\comon(\mathscr{V})$ & & Category of comonoids in monoidal category $\mathscr{V}$ \\[.1cm] 
        $\mmod(\mathscr{V})$ & & Framed bicategory of monoids and bimodules in monoidal category (or bicategory) $\mathscr{V}$\\[.1cm]
        $\ccomod(\mathscr{V})$ & & Framed bicategory of comonoids and bicomodules in monoidal category (or bicategory) $\mathscr{V}$ \\[.1cm] 
        $\coco(c)$ & & Free coproduct completion of category $c$ \\[.1cm] 
        $c\set$ & \cref{notation.main} & Category of $c$-sets, i.e., $\catfun{c}{\smset}$ \\[.1cm]
        $\poly$ & \cref{def.polynomial_functor} & Category of (single-variable) polynomial functors \\[.1cm]
        $\ccatsharp$ & \cref{chap.catsharp} & Framed bicategory of categories, category-variable polynomials, and retrofunctors\\[.1cm] 
        $\smset[c]$ & \cref{def.duc_query} & Category of $c$-variable polynomial functors \\[.1cm]
        $d\set[c]$ & \cref{cor.TFAE_indexed_duc} & Category of $c$-variable $d$-valued polynomial functors \\[.1cm]
        $\catsharp$ & \cref{def.retrofunctor} & Category of categories and retrofunctors \\[.1cm]
        $\ccatsharpdisc$ & \cref{sec.catsharpdisc} & Full sub framed bicategory of $\ccatsharp$ spanned by discrete categories \\[.1cm]
        $\ccatsharplin$ & \cref{sec.sspan} & Locally full sub framed bicategory of $\ccatsharpdisc$ spanned by 1-cells with linear carrier \\[.1cm]
        $\ccatsharpcon$ & \cref{prop.profunctor} & Locally full sub bicategory of $\ccatsharp$ spanned by conjunctive 1-cells \\[.1cm]
        $\ccatsharpdisccon$ & \cref{cor.span_left_adj} & Intersection of $\ccatsharpdisc$ and $\ccatsharpcon$ (as sub bicategories of $\ccatsharp$)\\[.1cm]
        $\Cat{Dir}$ & \cref{def.dir} & Category of Dirichlet polynomial functors\\[.1cm]
        $d\text-\Cat{Dir}[c]$ & \cref{def.dir_map} & Category of $c$-variable $d$-valued polynomials and Dirichlet maps
      \end{tabular}\caption{Table of notation for large categories.}\label{tab.catnot}
    \end{center}
  \end{table}
  \begin{table}
    \begin{center}
      \begin{tabular}{c c p{10cm} }
        $g \circ f$ & \cref{notation.main} & Composition ($g$ after $f$) \\[.1cm]
        $f \then g$ & \cref{notation.main} & Composition ($f$ then $g$) \\[.1cm]
        $\el_c(P)$ & \cref{notation.main} & Category of elements of $c$-set $P$\\[.1cm]
        $\adj{c}{F}{G}{d}$ & \cref{notation.main} & Adjunction with $F$ left adjoint and $G$ right adjoint \\[.1cm]
        $G\ldag$ & \cref{notation.main} & Left adjoint of $G$\\[.1cm]
        $F\rdag$ & \cref{notation.main} & Right adjoint of $F$\\[.1cm]
        $a \alttensor b$ & \cref{notation.main} & Generic monoidal product \\[.1cm]
        $c \profunctor d$ & \cref{notation.main} & Profunctor $c\op \times d \to \smset$, or generic bi(co)module \\[.1cm]
        $\yon^S$ & \cref{not.yon} & Representable functor $\smset(S,\thg)$ \\[.1cm]
        $p(1)$ & \cref{def.polynomial_functor} & Positions of polynomial $p$ \\[.1cm]
        $p[i]$ & \cref{def.polynomial_functor} & Directions of polynomial $p$ at position $i$\\[.1cm]
        $\pnt{p}$ & \cref{rem.poly_bundle} & Derivative of polynomial $p$ \\[.1cm]
        $\varphi_1$ & \cref{prop.comb_description} & Map $p(1) \to q(1)$ induced by map of polynomials $\varphi \colon p \to q$\\[.1cm]
        $\varphi^\sharp_i$ & \cref{prop.comb_description} & Map $q[\varphi_1(i)] \to p[i]$ induced by map of polynomials $\varphi \colon p \to q$\\[.1cm]
        $q \tri p$ & \cref{prop.comp_monoidal} & Polynomial functor composition ($q$ after $p$) \\[.1cm]
        $\lens{p}{q}$ & \cref{prop.JoshMeyers} & Coclosure for $\tri$ (left Kan extension of $p$ along $q$) \\[.1cm]
        $c \coto d$ & \cref{def.retrofunctor} & Retrofunctor between categories $c$ and $d$\\[.1cm]
        $\overline F$ & \cref{ex.\'etale_retrofunctor} & Retrofunctor $c\coto d$ from \'etale functor $F \colon c \to d$\\[.1cm]
        $\tilde F$ & \cref{ex.boo_retrofunctor} & Retrofunctor $d \coto c$ from bijective on objects functor $F\colon c \to d$\\[.1cm]
        $c \bito[m] d$ & \cref{sec.bicomodules} & Polynomial bicomodule $m$ from $c$ to $d$  \\[.1cm]
        $m_a$ & \cref{not.indexed_component} & $a$-indexed component of polynomial bicomodule $m$ for $a\in m(1)$ \\[.1cm]
        $\lenskan{p}{q}{d}{e}$ & \cref{prop.generalcoclosure} & Left Kan extension in $\ccatsharp$ of $c\bito[p] e$ along $c \bito[q] d$ \\[.1cm]
        $p \otimes q$ & \cref{prop.mon_closed} & Dirichlet product of polynomials $p$ and $q$ \\[.1cm]
        $[p, q]$ & \cref{prop.mon_closed} & Internal hom (usually with respect to $\otimes$)\\[.1cm]
        $p\ot{c}{d}q$ & \cref{sec.multivar_tensor} & Dirichlet product of bicomodules $c \bito[p, q] d$\\[.1cm]
        $\yoncool{c}{d}$ & \cref{sec.multivar_tensor} & Monoidal unit for Dirichlet product of bicomodules $c \bito d$\\[.1cm]
        $\ih{c}{p}{q}{d}$ & \cref{prop.local_closure_general} & Internal hom with respect to $\ot{c}{d}$\\[.1cm]
        $\dual{p}$ & \cref{ex.duality_in_poly} & Weak dual of $p$ \\[.1cm]
        $\agg$ & \cref{chap.aggregation} & Aggregation functor
      \end{tabular}\caption{Table of miscellaneous notation.}\label{tab.not}
      \end{center}
    \end{table}
  \end{notation}

  \begin{notation}[String diagrams]
    \emph{String diagrams} (\cref{fig.string_diagrams}) as in
    \cite{joyal1991geometry} offer a convenient way to visualize
    morphisms in a monoidal category $(\mathscr{V}, \alttensor, I)$. A
    morphism
    $\varphi\colon a_1 \alttensor \ldots \alttensor a_n \to b_1
    \alttensor \ldots \alttensor b_m$ is depicted as a box with $n$ input
    wires and $m$ output wires. Composition of morphisms is depicted as
    plugging inputs into outputs, and tensor of morphisms is depicted as
    placing boxes in parallel. Identity morphisms appear as plain empty
    wires, as they have no effect when composed with anything.

    \begin{figure}[h]
      \centering
      \begin{subfigure}[c]{0.33\textwidth}
        \centering
        \begin{tikzpicture}[stringd]
          \draw [wire] (-.5, 1) node[obj,above] {$a_1$} -- (-.5, .3);
          \draw [wire] (0, 1) node[obj,above] {$a_2$} -- (0, .3);
          \draw [wire] (.5, 1) node[obj,above] {$a_3$} -- (.5, .3);
          \draw [wire] (-.35, -1) node[obj,below] {$b_1$} -- (-.35, -.3);
          \draw [wire] (.35, -1) node[obj,below] {$b_2$} -- (.35, -.3);
          \draw[hub] (-.8,-.3) rectangle (.8, .3) node[pos=.5] {$\varphi$};
        \end{tikzpicture}
        \caption{A morphism $a_1 \alttensor a_2 \alttensor a_3 \to b_1 \alttensor b_2$.}
      \end{subfigure}%
      \begin{subfigure}[c]{0.33\textwidth}
        \centering
        \begin{tikzpicture}[stringd]
          \draw [wire] (-.35, -.35) -- node[obj,left] {$b_1$} (-.35, .35);
          \draw [wire] (.35, -.35) -- node[obj,right] {$b_2$} (.35, .35);
          \begin{scope}[shift={(0,.65)}]
            \draw [wire] (-.5, 1) node[obj,above] {$a_1$} -- (-.5, .3);
            \draw [wire] (0, 1) node[obj,above] {$a_2$} -- (0, .3);
            \draw [wire] (.5, 1) node[obj,above] {$a_3$} -- (.5, .3);
            \draw[hub] (-.8,-.3) rectangle (.8, .3) node[pos=.5] {$\varphi$};
          \end{scope}
          \begin{scope}[shift={(0,-.65)}]
            \draw [wire] (-.35, -1) node[obj,below] {$c_1$} -- (-.35, -.3);
            \draw [wire] (.35, -1) node[obj,below] {$c_2$} -- (.35, -.3);
            \draw[hub] (-.6,-.3) rectangle (.6, .3) node[pos=.5] {$\psi$};
          \end{scope}
        \end{tikzpicture}
        \caption{A composite of morphisms $\varphi \then \psi$.}
      \end{subfigure}%
      \begin{subfigure}[c]{0.33\textwidth}
        \centering
        \begin{tikzpicture}[stringd]
          \begin{scope}[shift={(-.8, 0)}]
            \draw [wire] (-.5, 1) node[obj,above] {$a_1$} -- (-.5, .3);
            \draw [wire] (0, 1) node[obj,above] {$a_2$} -- (0, .3);
            \draw [wire] (.5, 1) node[obj,above] {$a_3$} -- (.5, .3);
            \draw [wire] (-.35, -1) node[obj,below] {$b_1$} -- (-.35, -.3);
            \draw [wire] (.35, -1) node[obj,below] {$b_2$} -- (.35, -.3);
            \draw[hub] (-.8,-.3) rectangle (.8, .3) node[pos=.5] {$\varphi$};
          \end{scope}
          \begin{scope}[shift={(1, 0)}]
            \draw [wire] (-.35, 1) node[obj,above] {$a_4$} -- (-.35, .3);
            \draw [wire] (.35, 1) node[obj,above] {$a_5$} -- (.35, .3);
            \draw [wire] (-.35, -1) node[obj,below] {$b_3$} -- (-.35, -.3);
            \draw [wire] (.35, -1) node[obj,below] {$b_4$} -- (.35, -.3);
            \draw[hub] (-.6,-.3) rectangle (.6, .3) node[pos=.5] {$\psi$};
          \end{scope}
        \end{tikzpicture}
        \caption{A tensor of morphisms $\varphi \alttensor \psi$.}
      \end{subfigure}
      \caption{String diagrams.}\label{fig.string_diagrams}
    \end{figure}

    For readers who are not comfortable with string diagrams, we include commutative diagrams as well. 
  \end{notation}

\subsection{Acknowledgements}
David would like to thank Nick Smith for insisting on something he
already knew but had lost courage to pursue: that aggregation is
really worth understanding.

We also thank Simon Henry, Josh Meyers, David Jaz Myers,
and Nelson Niu for contributions that directly helped this research
effort.
Thanks to
Kevin Carlson for catching an error. Conversations with Nathanael Arkor,
Steve Awodey, Spencer Breiner, Brendan Fong, Joachim Kock, Sophie
Libkind, Joe Moeller, Juan Orendain, Evan Patterson, Todd Trimble, and
Ryan Wisnesky were also beneficial. 

We also very much appreciate the careful review by Spencer Breiner, who
refereed this article for JPAA.

This material is based upon work supported by the Air Force Office of
Scientific Research under award number FA9550-20-1-0348.
\section{Single-variable polynomial functors}
\label{chap.singlevariable}
In this section, we introduce the category $\poly$ of polynomial
endofunctors of $\smset$, describe various perspectives on its
objects and morphisms, and explain the composition product of these
single-variable polynomial functors.

\subsection{The category of polynomial functors}
\begin{notation}[Representable functor]\label{not.yon}
  Given a set $S$, we denote the corresponding representable functor
  $\smset\to\smset$ by
  \[\yon^S\coloneqq\smset(S,\thg).\]
\end{notation}

In particular, $\yon=\yon^1$ is the
identity, and $1 = \yon^0$ is the constant singleton functor.

\begin{definition}[Polynomial functor]\label{def.polynomial_functor}

  A \emph{polynomial functor} (of sets) is a functor $p\colon\smset\to\smset$
  that is isomorphic to a sum of representables. This means there exists
  a set $I$, a set $S_i$ for each $i\in I$, and an isomorphism
  \[
    p\cong\sum_{i\in I}\yon^{S_i}.
  \]

  To simplify notation, we identify $I$ with the set $p(1)$, and we
  refer to each set $S_i$ as $p[i]$. Hence we generally write
  \begin{equation}\label{eqn.poly_notation}
    p=\sum_{i\in p(1)}\yon^{p[i]}.
  \end{equation}
  
  A \emph{morphism} $p\to q$ of polynomial functors is simply a natural transformation between them. We denote the category of polynomial functors by $\poly$.
\end{definition}

\begin{definition}[Positions and directions]
  When $p$ is a polynomial functor, we call elements of $p(1)$ the \emph{positions} of $p$, and for each $i\in p(1)$ we call elements of $p[i]$ the \emph{directions at $i$}.

  The terminology suggests a visualization strategy: we draw positions as
  nodes and their directions as emanating arrows (\cref{fig.forest}).
  We refer to such a node with emanating arrows
  (symbolizing a representable summand of a polynomial functor) as a
  \emph{corolla}. The usefulness of this visualization will be
  more apparent after we introduce the composition product in
  \cref{sec.subst}.
\end{definition}

\begin{figure}[h]
  \centering
  \begin{tikzpicture}[smallcorollas]
    \draw[rounded corners] (-5.9,-1.8) rectangle (5.4,.8);
    \begin{scope}[shift={(-4.5,0)}]
      \node [rcol,vertex] (s1) at (0, -1.2) {};
      \coordinate (t11) at (-.8,.3) {};
      \coordinate (t12) at (0,.3) {};
      \coordinate (t13) at (.8,.3) {};
      \draw [rcol,edge] (s1) -- (t11);
      \draw [rcol,edge] (s1) -- (t12);
      \draw [rcol,edge] (s1) -- (t13);
    \end{scope}
    \begin{scope}[shift={(-1.5,0)}]
      \node [rcol,vertex] (s2) at (0, -1.2) {};
      \coordinate (t21) at (-.6,.3) {};
      \coordinate (t22) at (.6,.3) {};
      \draw [rcol,edge] (s2) -- (t21);
      \draw [rcol,edge] (s2) -- (t22);
    \end{scope}
    \begin{scope}[shift={(1.5,0)}]
      \node [rcol,vertex] (s3) at (0, -1.2) {};
      \coordinate (t31) at (-.6,.3) {};
      \coordinate (t32) at (.6,.3) {};
      \draw [rcol,edge] (s3) -- (t31);
      \draw [rcol,edge] (s3) -- (t32);
    \end{scope}
    \begin{scope}[shift={(4.5,0)}]
      \node [rcol,vertex] (s1) at (0, -1.2) {};
    \end{scope}
  \end{tikzpicture}

  \caption{The polynomial $ p \coloneqq \yon^3 +2\yon^2 + 1$ as
    a corolla forest.}\label{fig.forest}
\end{figure}

\begin{remark}[Polynomials as bundles]\label{rem.poly_bundle}
  Specifying a polynomial amounts to specifying a \emph{bundle} (i.e.,
  function between sets): the map sending the total set of directions
  to positions.
  \begin{equation}\label{eqn.poly_asbundle}
    \begin{tikzcd}
      \displaystyle\sum_{i\in p(1)}p[i]\ar[d, "\pi_p"]\\
      p(1)
    \end{tikzcd}
  \end{equation}

  If $p$ is a polynomial, let $\pnt{p}$ denote its derivative in the
  usual calculus sense. Then the positions $\pnt{p}(1)$ of the
  derivative are the directions of the original $p$. Thus we may write
  the bundle corresponding to $p$ as
  \begin{equation}\label{eqn.poly_bundleform}
    \begin{tikzcd}
      \pnt{p}(1)\ar[d, "\pi_p"]\\
      p(1)
    \end{tikzcd}
  \end{equation}
  That said, we will often simply denote such a bundle as $E\To{\pi} B$.
\end{remark}
The following standard characterization of polynomial functors will
allow us to quickly prove several results in the coming sections. It
will be generalized in \cref{cor.TFAE_indexed_duc}.

\begin{proposition}\label{prop.connected_limits}
  A functor $p\colon\smset\to\smset$ is a polynomial iff it preserves connected limits.
\end{proposition}
\begin{proof}
  Representable functors certainly preserve connected limits, as do
  their (pointwise) coproducts since coproducts commute with
  connected limits in $\smset$.  In the other direction, suppose $F$
  preserves connected limits. Let
  $F'\colon\smset\cong\smset/1\to\smset/F(1)$ be the induced map to
  the slice over $F(1)$, so that $F$ is the composite of $F'$ and the
  forgetful functor $\smset/F(1)\to\smset$, the latter of which is
  given by a coproduct of $F(1)$-many representable functors. Note
  that $F'$ preserves limits, since it preserves both the terminal
  object (by construction) and connected limits (as connected
  limits in a slice category are computed in the underlying
  category). Hence to show that $F$ is a coproduct of representables,
  it suffices to observe that any limit-preserving functor
  $\smset\to\smset$ is representable. Indeed, such a functor has a
  left adjoint $L$, and the representing set is $L(1)$.
\end{proof}

We say that a polynomial $p$ is \emph{constant} if it is of the form
$p\cong A=\sum_{a\in A}\yon^0$ and \emph{linear} if it is of the form $p\cong
A\yon=\sum_{a\in A}\yon^1$ for some set $A\in\smset$. These correspond to bundles of the
form $0\to A$ and $A\to A$, respectively.

\begin{remark}\label{rem.set_copies}
  $\poly$ includes two embedded copies of $\smset$: the full
  subcategory spanned by constants $A$, and the full subcategory
  spanned by linears $A\yon$. $\poly$ also includes an embedded copy
  of $\smset\op$: the full subcategory spanned by representables
  $\yon^A$.
\end{remark}

The previous facts become evident in light of the following
combinatorial description of polynomial morphisms. This will be
crucial; many of the proofs rely on the reader's ability to verify
calculations on morphisms of polynomials, which are made much easier
by understanding morphisms combinatorially.

\begin{proposition}[Combinatorial description of morphisms in
  $\poly$]\label{prop.comb_description}
  We have the formula
  \begin{equation}\label{eqn.poly_map_combinatorics}
    \poly(p,q)=\prod_{i\in p(1)}\sum_{j\in q(1)}\smset(q[j],p[i]).
  \end{equation}
\end{proposition}
In other words, to specify a morphism $\varphi\colon p\to q$ is to
specify a way to take a $p$-position $i\in p(1)$ and provide both a
$q$-position $j\in q(1)$ and a function $q[j]\to p[i]$ from
$q$-directions at $j$ backwards to $p$-directions at $i$. Note that, using the distributive law of products over sums, we may also
write this formula as
\begin{equation}\label{eqn.mapsharp_formal}
  \poly(p,q)=\sum_{\varphi_1\! \colon p(1) \rightarrow q(1)}\;\prod_{i\in p(1)}\smset\big(q[\varphi_1(i)],p[i]\big)\,.
\end{equation}

\begin{figure}[h]
  \centering
  \begin{subfigure}[t]{0.5\textwidth}
    \centering
    \begin{tikzpicture}[corollas]
      \path (-2.1,-1.7) -- (2.1,1.7);
      \begin{scope}[shift={(-1.3,0)}]
        \node [rcol,vertex,larger] (as) at (0, -.75) {$i$};
        \coordinate (at1) at (-.75,.75) {};
        \coordinate (at2) at (0,.75) {};
        \coordinate (at3) at (.75,.75) {};
        \draw [rcol,edge] (as) -- (at1);
        \draw [rcol,edge] (as) -- (at2);
        \draw [rcol,edge] (as) -- (at3);
      \end{scope}

      \begin{scope}[shift={(1.3,0)}]
        \node [bcol,vertex,larger] (bs) at (0, -.75) {$j$};
        \coordinate (bt1) at (-.75,.75) {};
        \coordinate (bt2) at (-.25,.75) {};
        \coordinate (bt3) at (.25,.75) {};
        \coordinate (bt4) at (.75,.75) {};
        \draw [bcol,edge] (bs) -- (bt1);
        \draw [bcol,edge] (bs) -- (bt2);
        \draw [bcol,edge] (bs) -- (bt3);
        \draw [bcol,edge] (bs) -- (bt4);
      \end{scope}
      \draw[|->,shorten <=2,shorten >=2,bend right=30] (as) to (bs);
      \node at (0,-1.5) {$\varphi_1$};
      \draw[->,dotted,shorten <=2,shorten >=2,bend right=30] (bt1) to (at2);
      \draw[->,dotted,shorten <=2,shorten >=2,bend right=30] (bt2) to (at2);
      \draw[->,dotted,shorten <=2,shorten >=2,bend right=30] (bt3) to (at3);
      \draw[->,dotted,shorten <=2,shorten >=2,bend right=30] (bt4) to (at3);
      \node at (0, 1.5) {$\varphi^\sharp$};
    \end{tikzpicture}
    \caption{For each $i \in p(1)$, we pick a $j \in q(1)$ and map its
      directions backwards to directions at $i$.}
  \end{subfigure}%
  \begin{subfigure}[t]{0.5\textwidth}
    \centering
    \begin{tikzpicture}[corollas]
      \path (-2.1,-1.7) -- (2.1,1.7);
      \begin{scope}[shift={(-1.3,0)}]
        \node [rcol,vertex,larger] (s) at (0, -.75) {$i$};
        \coordinate (t1) at (-.75,.75) {};
        \coordinate (t2) at (0,.75) {};
        \coordinate (t3) at (.75,.75) {};
        \draw [rcol,edge] (s) -- (t1);
        \draw [rcol,edge] (s) -- (t2);
        \draw [rcol,edge] (s) -- (t3);
      \end{scope}

      \node at (0, 0) {$\underset{\varphi}{\mapsto}$};
      \begin{scope}[shift={(1.3,0)}]
        \node [bcol,vertex,larger] (s) at (0, -.75) {$j$};
        \coordinate (t1) at (-.75,.75) {};
        \coordinate (t2) at (0,.75) {};
        \coordinate (t3) at (.75,.75) {};
        \draw [dasht] (s) -- (t1);
        \draw [bcol,edge,bend left=10] (s) to (t2);
        \draw [bcol,edge,bend right=10] (s) to (t2);
        \draw [bcol,edge,bend left=10] (s) to (t3);
        \draw [bcol,edge,bend right=10] (s) to (t3);
      \end{scope}
    \end{tikzpicture}
    \caption{We illustrate
      this in ``one step'' replacing each $p$-corolla
      with a $q$-corolla, $q$-directions assuming places of
      $p$-directions.}
  \end{subfigure}
  \caption{Morphism of polynomials.}\label{fig.corollamap1}
\end{figure}
\begin{proof}
  The formula in \eqref{eqn.poly_map_combinatorics} comes from the
  definition of coproduct, the Yoneda lemma, and the fact that $\poly$
  is a full subcategory of all functors $\smset\to\smset$:
  \begin{equation*}
    \begin{aligned}
      \poly(p,q)&=
                  \poly\Big(\sum_{i\in p(1)}\yon^{p[i]},\sum_{j\in q(1)}\yon^{q[j]}\Big)\\&=
      \prod_{i\in p(1)}\poly\Big(\yon^{p[i]},\sum_{j\in q(1)}\yon^{q[j]}\Big)\\&=
      \prod_{i\in p(1)}\Big(\sum_{j\in q(1)}\yon^{q[j]}\Big)(p[i])\\&=
      \prod_{i\in p(1)}\sum_{j\in q(1)}\smset(q[j],p[i]).\qedhere
    \end{aligned}
  \end{equation*}
\end{proof}
As indicated in~\eqref{eqn.mapsharp_formal}, we typically write the
map on positions as $\varphi_1$, as below-left; for indeed, it is just
the component of the natural transformation $\varphi$ at the set $1$.
We then denote the map backwards on directions using a
$\sharp$-symbol,%
\footnote{For readers familiar with basic algebraic geometry, the
  $(\varphi_1,\varphi^\sharp)$ notation is meant to evoke the standard
  notation for maps of ringed spaces $(X,\cat{O}_X)\to(Y,\cat{O}_Y)$:
  a map $f\colon X\to Y$ and a map of sheaves
  $f^\sharp\colon f^*\cat{O}_Y\to\cat{O}_X$ \cite[Section
  II.2]{hartshorne1977algebraic}. This is quite analogous; indeed one
  can view $p\in\poly$ as a sheaf of sets on the discrete space
  $p(1)$, assigning to each $I\ss p(1)$ the set of sections
  $\prod_{i\in I}p[i]$. Then a map of polynomials is just as for
  ringed spaces: forward on points, backwards on sections.} as
below-right:
\begin{equation}\label{eqn.mapsharp}
  \varphi_1\colon p(1)\to q(1)
  \qqand
  \varphi^\sharp(i,\thg)\colon q[j]\to p[i]
\end{equation}
where $\varphi_1(i)=j$. Sometimes we will prefer to write
$\varphi^\sharp(i, \thg)$ as $\varphi^\sharp_i(\thg)$. 

\begin{example}\label{ex.lenses}
  In functional programming, a \emph{lens} is a sort of morphism between pairs of sets, often denoted $\lens{S}{T}\to\lens{A}{B}$. It consists of two functions:
  \begin{equation}\label{eqn.lensmaps}
    S\to A
    \qqand
    S\times B\to T.
  \end{equation}
  Rather than saying how they compose, we can give the whole story at
  once: The category of lenses is isomorphic to the full subcategory
  of $\poly$ spanned by the monomials. Indeed, one can read off from
  \eqref{eqn.poly_map_combinatorics} that a map $S\yon^T\to A\yon^B$
  consists of the two functions displayed in \eqref{eqn.lensmaps}.
  We will generalize this in \cref{ex.lens_coclosure}, seeing that these objects arise as a special case of a coclosure operation in $\poly$.
\end{example}

Here is yet another way to understand maps between polynomials.

\begin{proposition}\label{prop.map_bundle_view}
  Let $p,p'$ be polynomials and $E\To{\pi} B$ and $E'\To{\pi'} B'$ the
  corresponding bundles. A morphism $\varphi\colon p\to p'$
  corresponds to a diagram of the form
  \begin{equation}\label{eqn.map_bundle_view}
    \begin{tikzcd}
      E\ar[d, "\pi"']&\bullet\ar[d]\ar[l, "\varphi^\sharp"']\ar[r]\ar[dr, phantom, very near start, "\lrcorner"]&E'\ar[d, "\pi'"]\\
      B\ar[r, equal]&B\ar[r, "\varphi_1"']&B'
    \end{tikzcd}
  \end{equation}
\end{proposition}
\begin{proof}
  After verifying that the pullback object (denoted $\bullet\coloneqq
  B\times_{B'}E'$) can be identified with $\sum_{i\in
    B}p'[\varphi_1(i)]$, the result is just a rephrasing of
  \cref{eqn.mapsharp}. 
\end{proof}

We conclude this section with some observations on limits and colimits
in $\poly$.
\begin{proposition}\label{cor.limits}
  Small limits and coproducts in $\poly$ exist and are computed in $\catfun{\smset}{\smset}$.
\end{proposition}
\begin{proof}
  In light of \cref{prop.connected_limits}, it suffices to observe
  small limits and coproducts of connected limit preserving functors
  $\smset \to \smset$ preserve connected limits. Indeed, we compute
  them ``pointwise'', since both limits and coproducts commute with
  connected limits in $\smset$.
\end{proof}

The following is quite useful; we invite the reader to
check it explicitly using \cref{eqn.poly_map_combinatorics}.
\begin{corollary}\label{cor.sum_product}
  The categorical sum and product of polynomials $p+q$ and $pq=p\times q$
  correspond to the usual algebraic sum and product of polynomials.
\end{corollary}
\begin{remark}
  In fact, $\poly$ has all small colimits; however, coequalizers are not
  generally preserved by the forgetful functor
  $\poly\to\catfun{\smset}{\smset}$, and we will not need them in this
  paper.
\end{remark}

%
%

\subsection{Composition of single-variable polynomials}
\label{sec.subst}
In this section, we describe the monoidal structure on the category
$\poly$ induced by functor composition---which corresponds to
\emph{substitution} of polynomials in the usual sense.

\begin{proposition}[Composition of polynomial functors]\label{prop.comp_monoidal}
  The composite $p \circ q$ of two polynomial functors $p,q \colon
  \smset \rightarrow \smset$ is again polynomial; we denote it $p\tri
  q \colon \smset \rightarrow \smset$.%
  \footnote{We are thinking of polynomials as objects rather than morphisms in a category, namely $\poly$. We use $\tri$ for the composition (also known as \emph{substitution}) monoidal structure on $\poly$, to reserve the symbol $\circ$ for composition of morphisms.} 
  It can be described by the formula
  \[
    p\tri q\,\cong\sum_{i\in p(1)}\,\sum_{j\colon\! p[i]\to q(1)}\;\yon^{\;\sum\limits_{d\in p[i]}q[j(d)]}.
  \]
\end{proposition}
In other words, the corollas of $p \tri q$ are two-layer trees: a position consists of a $p$-position $i$ of $p$ and $q$-positions for each direction at $i$, and its directions are those of these $q$-positions.

\begin{figure}[h]
  \centering
  \begin{tikzpicture}[smallcorollas]
    \begin{scope}[shift={(-4,0)}]
      \node [vertex] (s) at (0, -1.2) {};
      \coordinate (t11) at (-2.25,1.55) {};
      \coordinate (t12) at (-1.35,1.55) {};
      \coordinate (t13) at (-.45,1.55) {};
      \coordinate (t21) at (.45,1.55) {};
      \coordinate (t31) at (1.35,1.55) {};
      \coordinate (t32) at (2.25,1.55) {};
      \draw [edge] (s) -- (t11);
      \draw [edge] (s) -- (t12);
      \draw [edge] (s) -- (t13);
      \draw [edge] (s) -- (t21);
      \draw [edge] (s) -- (t31);
      \draw [edge] (s) -- (t32);

      \node [annot] at (.6, -1.2) {\rlap{$\in (p \tri q)(1)$}};
    \end{scope}
    \node at (0,0) {$\coloneqq$};
    \begin{scope}[shift={(4,0)}]
      \node [rcol,vertex] (s) at (0, -1.2) {};
      \node [bcol,vertex] (m1) at (-1.35,.3) {};
      \node [bcol,vertex] (m2) at (0,.3) {};
      \node [bcol,vertex] (m3) at (1.35,.3) {};
      \coordinate (t11) at (-2.225,1.55) {};
      \coordinate (t12) at (-1.35,1.55) {};
      \coordinate (t13) at (-.475,1.55) {};
      \coordinate (t21) at (0,1.55) {};
      \coordinate (t31) at (.725,1.55) {};
      \coordinate (t32) at (1.975,1.55) {};
      \draw [rcol,edge] (s) -- (m1);
      \draw [rcol,edge] (s) -- (m2);
      \draw [rcol,edge] (s) -- (m3);
      \draw [bcol,edge] (m1) -- (t11);
      \draw [bcol,edge] (m1) -- (t12);
      \draw [bcol,edge] (m1) -- (t13);
      \draw [bcol,edge] (m2) -- (t21);
      \draw [bcol,edge] (m3) -- (t31);
      \draw [bcol,edge] (m3) -- (t32);

      \node [annot] at (.45, -1.2) {\rlap{$\in p(1)$}};
      \node [annot] at (1.8,  .3) {\rlap{$\in q(1)$}};
    \end{scope}
  \end{tikzpicture}
  \caption{A corolla of $p \tri q$.}
\end{figure}
\begin{proof}
  The composite of two connected-limit-preserving functors is also
  connected-limit-preserving, so the result follows by
  \cref{prop.connected_limits}. The formula is simply a use of the
  distributivity law:
  \begin{equation*}
    p\tri q
    =
    \sum_{i\in p(1)}\left(\sum_{j\in q(1)}\yon^{q[j]}\right)^{p[i]}
    \cong
    \sum_{i\in p(1)}\;\sum_{j\colon p[i]\to q(1)}\;\prod_{d\in p[i]}\yon^{q[j(d)]}
    \cong
    \sum_{i\in p(1)}\;\sum_{j\colon p[i]\to q(1)}\yon^{\;\sum\limits_{d\in p[i]}q[j(d)]}\text{ .} \qedhere
  \end{equation*}
\end{proof}

In particular, the evaluation of a polynomial $p$ at a set $X$ can be
denoted $p\tri X$, e.g., $p(1)=p\tri 1$.

\begin{definition}\label{def.monoidal_category_poly}
  The monoidal category $(\poly, \yon, \tri)$ is the full monoidal
  subcategory of $(\catfun{\smset}{\smset}, \id_{\smset}, \circ)$
  spanned by the polynomial functors.
\end{definition}

\begin{remark}[Composition product of polynomial morphisms]\label{rem.comp_comb}
  The composition product of polynomial morphisms is their horizontal
  composition as natural transformations, but we can also describe
  this in terms of positions and directions.
  Let
  $\varphi\colon p \to p'$ and $\psi\colon q \to q'$ be morphisms of
  polynomials. Then:
  \begin{itemize}
  \item $(\varphi \tri \psi)_1$ sends the position $(i, j)
  \in (p \tri q)(1)$, where $i \in p(1)$ and $j \colon p[i] \to q(1)$, to the
  position $(i', j') \in (p' \tri q')(1)$ where $i' \coloneqq
  \varphi_1(i) \in p'$ and $j' \coloneqq \psi_1 \circ j \circ
  \varphi^\sharp_i\colon p'[i'] \to q'(1)$. 
\item The map $\smash{(\varphi
  \tri \psi)^\sharp_{(i, j)}}$ sends a direction $(d', e') \in (p' \tri
  q')[(i', j')]$ where $d' \in p'[i']$ and $e' \in q'[j'(d')]$ to the
  direction $(d, e) \in (p \tri q)[(i, j)]$ where $d \coloneqq
  \varphi^\sharp_i(d') \in p[i]$ and $e \coloneqq \varphi^\sharp_{j
    (d)}(e') \in q[j(d)]$.
  \end{itemize}
  In brief, $(\varphi\tri\psi)^\sharp$ sends a direction written $(d',
  e')$ to a direction written $(\varphi^\sharp(d'),
  \psi^\sharp(e'))$.
\end{remark}

\begin{figure}[h]
  \centering
  \begin{subfigure}[c]{0.5\textwidth}
    \centering
  \begin{tikzpicture}[smallcorollas]
    \begin{scope}[shift={(-3.5,0)}]
      \node [rcol,vertex] (s) at (0, -1.2) {};
      \node [bcol,vertex] (m1) at (-1.35,.3) {};
      \node [bcol,vertex] (m2) at (0,.3) {};
      \node [bcol,vertex] (m3) at (1.35,.3) {};
      \coordinate (t11) at (-2.225,1.55) {};
      \coordinate (t12) at (-1.35,1.55) {};
      \coordinate (t13) at (-.475,1.55) {};
      \coordinate (t21) at (0,1.55) {};
      \coordinate (t31) at (.725,1.55) {};
      \coordinate (t32) at (1.975,1.55) {};
      \draw [rcol,edge] (s) -- (m1);
      \draw [rcol,edge] (s) -- (m2);
      \draw [rcol,edge] (s) -- (m3);
      \draw [bcol,edge] (m1) -- (t11);
      \draw [bcol,edge] (m1) -- (t12);
      \draw [bcol,edge] (m1) -- (t13);
      \draw [bcol,edge] (m2) -- (t21);
      \draw [bcol,edge] (m3) -- (t31);
      \draw [bcol,edge] (m3) -- (t32);

      \node [annot] at (.45, -1.2) {\rlap{$\in p(1)$}};
      \node [annot] at (1.8,  .3) {\rlap{$\in q(1)$}};
    \end{scope}
    \node at (.5,-1) {$\underset{\varphi}{\mapsto}$};
    \node at (.5,.6) {$\underset{\psi}{\mapsto}$};
    \begin{scope}[shift={(3.5,0)}]
      \node [ocol,vertex] (s) at (0, -1.2) {};
      \node [gcol,vertex, transparent] (m1) at (-1.35,.3) {};
      \node [gcol,vertex] (m2) at (0,.3) {};
      \node [gcol,vertex] (m3) at (1.35,.3) {};
      \coordinate (t11) at (-2.225,1.55) {};
      \coordinate (t12) at (-1.35,1.55) {};
      \coordinate (t13) at (-.475,1.55) {};
      \coordinate (t21) at (0,1.55) {};
      \coordinate (t31) at (.725,1.55) {};
      \coordinate (t32) at (1.975,1.55) {};
      \draw [dasht] (s) -- (m1);
      \draw [ocol,edge,bend left=17] (s) to (m2);
      \draw [ocol,edge,bend right=17] (s) to (m2);
      \draw [ocol,edge,bend left=17] (s) to (m3);
      \draw [ocol,edge,bend right=17] (s) to (m3);
      \draw [gcol,edge,transparent] (m1) -- (t11);
      \draw [gcol,edge,transparent,bend left=17] (m1) to (t12);
      \draw [gcol,edge,transparent,bend right=17] (m1) to (t12);
      \draw [gcol,edge,transparent] (m1) -- (t13);
      \draw [gcol,edge] (m2) -- (t21);
      \draw [dasht] (m3) -- (t31);
      \draw [gcol,edge,bend left=17] (m3) to (t32);
      \draw [gcol,edge,bend right=17] (m3) to (t32);

      \node [annot] at (.45, -1.2) {\rlap{$\in p'(1)$}};
      \node [annot] at (1.8,  .3) {\rlap{$\in q'(1)$}};
    \end{scope}
    \end{tikzpicture}
    \end{subfigure}%
    \begin{subfigure}[c]{0.5\textwidth}
      \centering
      \begin{tikzpicture}[smallcorollas]
        \begin{scope}[shift={(-3.25,0)}]
          \node [rcol,vertex] (s) at (0, -1.2) {};
          \node [bcol,vertex] (m1) at (-1.35,.3) {};
          \node [bcol,vertex] (m2) at (0,.3) {};
          \node [bcol,vertex] (m3) at (1.35,.3) {};
          \coordinate (t11) at (-2.225,1.55) {};
          \coordinate (t12) at (-1.35,1.55) {};
          \coordinate (t13) at (-.475,1.55) {};
          \coordinate (t21) at (0,1.55) {};
          \coordinate (t31) at (.725,1.55) {};
          \coordinate (t32) at (1.975,1.55) {};
          \draw [rcol,edge] (s) -- (m1);
          \draw [rcol,edge] (s) -- (m2);
          \draw [rcol,edge] (s) -- (m3);
          \draw [bcol,edge] (m1) -- (t11);
          \draw [bcol,edge] (m1) -- (t12);
          \draw [bcol,edge] (m1) -- (t13);
          \draw [bcol,edge] (m2) -- (t21);
          \draw [bcol,edge] (m3) -- (t31);
          \draw [bcol,edge] (m3) -- (t32);

          
        \end{scope}
        \node at (0,0) {$\underset{\varphi \tri \psi}{\mapsto}$};
    \begin{scope}[shift={(3.25,0)}]
      \node [ocol,vertex] (s) at (0, -1.2) {};
      \coordinate (t11) at (-2.225,1.55) {};
      \coordinate (t12) at (-1.35,1.55) {};
      \coordinate (t13) at (-.475,1.55) {};
      \coordinate (t21) at (0,1.55) {};
      \coordinate (t31) at (.725,1.55) {};
      \coordinate (t32) at (1.975,1.55) {};
      \draw [dasht] (s) -- (t11);
      \draw [dasht] (s) -- (t12);
      \draw [dasht] (s) -- (t13);
      \draw [dasht] (s) -- (t31);
      \node [gcol,vertex,fill=white] (m2) at (-.26,.3) {};
      \node [gcol,vertex,fill=white] (m2_) at (.26,.3) {};
      \node [gcol,vertex,fill=white] (m3) at (1.05,.3) {};
      \node [gcol,vertex,fill=white] (m3_) at (1.65,.3) {};
      \draw [ocol,edge,bend left=5] (s) to (m2);
      \draw [ocol,edge,bend right=5] (s) to (m2_);
      \draw [ocol,edge,bend left=5] (s) to (m3);
      \draw [ocol,edge,bend right=5] (s) to (m3_);
      \draw [gcol,edge, bend left=5] (m2) to (t21);
      \draw [gcol,edge,bend right=5] (m2_) to (t21);
      \draw [gcol,edge,bend left=14] (m3) to (t32);
      \draw [gcol,edge,bend right=14] (m3) to (t32);
      \draw [gcol,edge,bend left=14] (m3_) to (t32);
      \draw [gcol,edge,bend right=14] (m3_) to (t32);


    \end{scope}
  \end{tikzpicture}
  \end{subfigure}
  \caption{The horizontal composite of polynomial morphisms
    $\varphi \tri \psi$ replaces each $(p \tri q)$-corolla with a
    $(p' \tri q')$-corolla: simply replace
    the constituent $p$-corolla with a $p'$-corolla via $\varphi$ and
    $q$-corollas with $q'$-corollas via $\psi$.}
\end{figure}

We learned the following from Josh Meyers (personal communication); it will be generalized in \cref{prop.generalcoclosure}.

\begin{proposition}[Coclosure for $\tri$]\label{prop.JoshMeyers}
  The composition operation $\tri$ has a (right) co-closure. That is, for every
  $p,q\in\poly$ we can define a polynomial $\lens{p}{q}\in\poly$ and
  bijections, natural in $p'$, of the form
  \begin{equation}\label{eqn.adjunction_coclosure}
    \poly\left(\lens{p}{q},p'\right) \cong \poly(p,p'\tri q)\rlap{ .}
  \end{equation}
  Explicitly, the polynomial $\lens p q$ is given by
  \begin{equation}\label{eqn.formula_coclosure}
    \lens{p}{q}\coloneqq\sum_{i\in p(1)}\yon^{q(p[i])}.
  \end{equation}
\end{proposition}
\begin{proof}
  The result follows from the combinatorial description
  \eqref{eqn.poly_map_combinatorics} in the sense that on both sides
  of the adjunction isomorphism \eqref{eqn.adjunction_coclosure}, the
  required set-theoretic maps are exactly the same: start with $i\in
  p(1)$, assign some $i'\in p'(1)$, and then, for any $d'\in p'[i']$
  assign some $j\in q(1)$, and then, for any $e\in q[j]$ assign some
  $d\in p[i]$.
\end{proof}

\begin{remark}[Coclosure is left Kan extension]
  When $p, q\colon \smset \to \smset$ are both
  polynomial, the left Kan extension of $p$ along $q$ is the
  polynomial $\lens{p}{q}$.
  In general, note that \emph{coclosure (resp. closure)} and
  \emph{left (resp. right) Kan extension} are merely different
  terms describing the same concept in different contexts. Left Kan
  extension along $q$ refers to a left adjoint of $(\thg) \circ q$;
  coclosure at $q$ refers to a left adjoint of $(\thg) \tri q$.
\end{remark}

\begin{example}[Lens notation as coclosure]\label{ex.lens_coclosure}
  Suppose $P,Q$ are sets. Then one reads off from \eqref{eqn.formula_coclosure} that 
  \[
    \lens{P}{Q}=P\yon^Q,
  \]
  recovering our notation from the example on lenses, \cref{ex.lenses}.
\end{example}

Finally in this section, we note that the composition monoidal
structure interacts well with limits in $\poly$; this will allow us to
define composition of bicomodules in \cref{sec.bicomodules}.
\begin{lemma}[Composition preserves connected limits]\label{lemma.preservation_of_equalizers}
  Each functor $(\thg) \tri p$ and $p \tri (\thg) \colon \poly
  \rightarrow \poly$ preserves connected limits.
\end{lemma}
\begin{proof}
  Coclosure (\cref{prop.JoshMeyers}) implies $(\thg) \tri p$ preserves
  all limits. As for $p \tri (\thg)$, note that $p$ preserves
  connected limits by \cref{prop.connected_limits}, and that
  limits in $\poly$ are pointwise by \cref{cor.limits}.
\end{proof}

\section{The polynomial ecosystem}
\label{chap.catsharp}
In this section, we study the generalization of $\poly$ to
multi-variable polynomials. We could, as in the introduction, consider
first of all set-indexed collections of variables, but instead we jump
straight to category-indexed collections of variables. Rather than a
monoidal category, we will obtain a $2$-category, whose objects are
categories $c,d,e,\dots$, and whose hom-category from $c$ to $d$ is
the category of polynomials with inputs indexed by $c$ and outputs
indexed by $d$---which is, equivalently, the category of prafunctors
$c\set \rightarrow d\set$.

In fact, we get more than this: we will exhibit a \emph{framed
  bicategory} $\ccatsharp$, which incorporates not just categories and
prafunctors, but also \emph{retrofunctors} between categories
(\cref{def.retrofunctor}). For the purposes of this paper, we refer to this framed
bicategory $\ccatsharp$ as the \emph{polynomial ecosystem}.%
\footnote{In other contexts, it may be useful to think of a larger framed bicategory as the \emph{full} polynomial ecosystem, e.g.\ $\Cat{EM}\text{-}\ccatsharp$, the Eilenberg-Moore completion of $\ccatsharp$; see \cite[Section 5]{shapiro2024polynomial}.}

\subsection{Category-indexed multi-variable polynomials}
We start by considering polynomials whose inputs
indexed by a category, but with only a single output.

\begin{definition}[$c$-variable polynomial]\label{def.duc_query}
  Let $c$ be a category. A \emph{$c$-variable polynomial} is a
  functor $p\colon c\set\to\smset$ that is isomorphic to a sum of
  representables. This means there exists a set $I$, a $c$-set
  $P_i\colon c\to\smset$ for each $i \in I$, and an isomorphism
  \begin{equation}\label{eqn.duc_query}
    p \cong \sum_{i\in I}c\set(P_i,\thg).
  \end{equation}
  Like before, we simplify notation by identifying $I$ with the set $p(1)$, and 
  referring to each $\cat{c}$-set $P_i$ as $p[i]$. Hence we may write
  \begin{equation*}
    p=\sum_{i\in p(1)}c\set(p[i], \thg).
  \end{equation*}
  A \emph{morphism of $c$-variable polynomials} is a natural
  transformation between them. We denote the category of $c$-variable
  polynomials by $\smset[c]$.
\end{definition}

The notation $\smset[c]$ is supposed to invoke (a categorification of)
the polynomial ring $\zz[y_1,\ldots,y_C]$, where we have replaced the
coefficient ring $(\zz,0,+,1,*)$ by the rig category
$(\smset,0,+,1,\times)$.

\begin{example}[$c$-variable polynomials for discrete
  $c$]\label{ex.discrete_multivariable}
  \begin{enumerate}[(i)]
  \item When $c$ is the empty category $0$, there is a canonical
  isomorphism $0\set=1$. (There is one functor from the empty category
  to $\smset$.) A $c$-variable polynomial is a functor
  $0\set\to\smset$, which is just a set, and there are no further
  restrictions. So polynomials in no variables are sets:
  $\smset[0]\cong\smset$.
\item 
  When $c$ is the one-object one-morphism category, which we'll later refer to by its polynomial carrier $\yon$, a
  $c$-variable polynomial is a functor $\smset\to\smset$ that can be
  written as a coproduct of representables. In other words
  $\smset[\yon]\cong\poly$.
\item   Suppose $c$ is the discrete category on a two-element set
  $\{a, b\}$. Then $X\in c\set$ can be identified with a pair of sets
  $(X(a),X(b))$, so a polynomial in $X$ can be rewritten
  \[\sum_{i\in I}c\set(P_i,X)\cong\sum_{i\in I}X(a)^{P_i(a)}\times X(b)^{P_i(b)}\]
  for arbitrary sets $P_i(a),P_i(b)$, and thus $\smset[\{a,b\}]$ is just the category of polynomials in two variables. Of course, the situation is
  analogous on replacing $c$ with an arbitrary discrete category.
  \end{enumerate}
\end{example}

\begin{example}[$c$-variable polynomials as duc-queries]\label{ex.variable_is_duc}
  Consider the cospan category:
  \[c\coloneqq\fbox{$\text{city}\to\text{state}\from\text{county}$}\]
  A $c$-set $X$ then describes sets of cities, states, and counties,
  where each city and each county belongs to a certain state. Suppose
  we are interested in extracting pairs of cities in the same state,
  $\text{city}\times_{\text{state}}\text{city}$. That is, we are
  looking for appearances of the shape:
  \[
    \begin{tikzpicture}
      \node (x) at (-.5,1) {$x$};
      \node (y) at (-.5,-1) {$y$};
      \node (z) at (.5,0) {$z$};
      \node at (-.25,1) {\rlap{\footnotesize$\in$ \text{city}}};
      \node at (-.25,-1) {\rlap{\footnotesize$\in$ \text{city}}};
      \node at (.75,0) {\rlap{\footnotesize$\in$ \text{state}}};
      \draw[->] (x) -- (z);
      \draw[->] (y) -- (z);
    \end{tikzpicture}
  \]
  This diagram itself describes a $c$-set $P$, and what we are
  looking for are natural transformations from $P$ to $X$. So the set
  we desire is $c\set(P, X)$.

  In general, a query for tuples of elements $(x_i)$ satisfying
  various equations the form $f(x_i) = x_j$ will correspond to a
  certain $c$-set $P$ (the corresponding colimit of representable
  $c$-sets), where the functor $c\set(P, \thg)$ extracts the set of such
  tuples as a functor $c\set\to\smset$. In database terminology such a query is called
  a \emph{conjunctive query}.

  Hence a $c$-variable polynomial specifies a \emph{disjoint union of
    conjunctive queries}, or \emph{duc-query} for short. For example,
  taking the same category $c$ from above, there is a $c$-variable
  polynomial
  \[
    (\text{city}\times_{\text{state}}\text{city})+(\text{city}\times_{\text{state}}\text{county})+\text{state}
  \]
  that extracts the disjoint union of pairs of cities in the same
  state, pairs of a city and county in the same state, and
  states.
  %
\end{example}

Note that the morphisms of $c$-variable polynomials admit a
combinatorial description, generalizing~\cref{prop.comb_description} above, with exactly the same
proof.
\begin{proposition}
We have the formula\label{prop.d-poly_map_combinatorics}
  \begin{equation}\label{eqn.d-poly_map_combinatorics}
    \smset[c](p,q)=\prod_{i\in p(1)}\sum_{j\in q(1)}c\set(q[j],p[i]) =
    \sum_{\varphi_1\! \colon p(1) \rightarrow q(1)}\;\prod_{i\in p(1)}c\set\big(q[\varphi_1(i)],p[i]\big)\,.
  \end{equation}
\end{proposition}

Our next result provides a number of alternative formulations of the
notion of $c$-variable polynomial. Before stating it, we recall the
definition of parametric right adjoint functor.

\begin{definition}[Slice factorization, prafunctor]
  \label{def.prafunctor}
  Let $\Cat{C}$ and $\Cat{D}$ be categories and suppose $\Cat{C}$ has
  a terminal object. The \emph{slice factorization} of a functor $F
  \colon \Cat{C} \rightarrow \Cat{D}$ is given by:
  \begin{equation}
    \label{eqn.slicefactorization}
    \Cat{C} \xrightarrow{F_{/1}} \Cat{D}/F1 \xrightarrow{U} \Cat{D}
  \end{equation}
  where $U$ is the forgetful functor from the slice, and $F_{/1}$ is
  the functor sending $X \in \Cat{C}$ to
  $F(! \colon X \rightarrow 1)$. A functor $F$ as above is called a
  \emph{parametric right adjoint}, or \emph{prafunctor}, if $F_{/1}$
  is a right adjoint.
\end{definition}

\begin{proposition}\label{prop.TFAE_duc}
  For any category $c$, the following categories are equivalent:
  \begin{enumerate}[(1)]
  \item the category $\smset[c]$ of $c$-variable polynomials, as in \cref{def.duc_query};
  \item the free coproduct completion $\coco((c\set)\op)$ of $(c\set)\op$;
  \item the category $\Cat{CLP}(c\set,\smset)$ of connected-limit-preserving functors $c\set\to\smset$; and
  \item the category $\Praf(c\set,\smset)$ of parametric right
    adjoint functors $c\set\to\smset$.
  \end{enumerate}
\end{proposition}
Later \cref{thm.garner} will add a fifth element, \emph{the category
  $\comod(\poly)(c, \yon)$ of right $c$-comodules}, to the list.

\begin{proof}
  We have (1) $\iff$ (2) since the free coproduct completion of a category
  $\Cat{C}$ can be found as the full subcategory of functors
  $\Cat{C}^\mathrm{op} \rightarrow \smset$ which are coproducts of
  representables. Now, since each of $\smset[c]$,
  $\Cat{CLP}(c\set,\smset)$, and $\Praf(c\set,\smset)$ is a full
  subcategory of the category of all functors $c\set\to\smset$, it
  suffices to show that all three have the same essential image.

  First of all, if $F \colon c\set\to\smset$ is a parametric
  right adjoint, then it admits a
  factorization~\eqref{eqn.slicefactorization} whose first part is a
  right adjoint. Since right adjoints and projections from slice
  categories preserve connected limits, we have $(4)\Rightarrow (3)$.
  Conversely, if $F\colon c\set\to\smset$ preserves connected limits,
  then $F_{/1}$ preserves the terminal object by construction, and
  preserves connected limits since $F$ does and $U$ reflects them. So
  $F_{/1} \colon \cat{c}\set \rightarrow \smset$ preserves limits, and
  any such functor is a right adjoint; which gives $(3) \Rightarrow (4)$.

  We also have (1) $\Rightarrow$ (3): representable functors preserve all
  limits, and so $c$-variable polynomials, which are coproducts of
  representables, preserve connected limits, since coproducts commute
  with connected limits in $\smset$. Thus, to complete the proof, it
  suffices to show (4) $\Rightarrow$ (2).
  So suppose that $F \colon \cat{c}\set \rightarrow \smset$ is a
  prafunctor. The factorization in~\eqref{eqn.slicefactorization} has
  as interposing object the slice category $\smset / F1$, which is
  equivalent to the power $\smset^{F1}$, and so can be rewritten as
  \begin{equation*}
    \cat{c}\set \xrightarrow{\ F_{/1}\ } \smset^{F1} \xrightarrow{\ \ \Sigma\ \ } \smset\rlap{ .}
  \end{equation*}
  Here, the second factor is the coproduct functor, and it thus
  suffices to show that the composite of $F_{/1}$ with any product
  projection $\pi_i \colon \smset^{F1} \rightarrow \smset$ is
  representable. But $F_{/1}$ and $\pi_i$ are both right adjoints; so
  $\pi_i \circ F_{/1} \colon \cat{c}\set \rightarrow \smset$ is as
  well, and so represented by the value of the corresponding left
  adjoint at $1$.
\end{proof}

We now generalize from polynomials with single outputs, to polynomials
with multiple outputs.
\begin{definition}[$c$-variable $d$-valued polynomial]\label{def.cvalued_duc_query}
  Let $c,d$ be categories. A \emph{$c$-variable $d$-valued polynomial}
  is a functor $p\colon c\set\to d\set$ that is pointwise isomorphic
  to a sum of representables. This means that, for each object $a$ of
  $d$, there exists a set $p_a(1)$ and $c$-sets
  $p_a[i] \colon c\to\smset$ for each $i \in p_a(1)$, and isomorphisms
  \begin{equation}
    \label{eqn.cvalued_duc_query}
    p(\thg)(a) \cong \sum_{i\in p_a(1)}c\set(p_a[i], \thg).
  \end{equation}
  A \emph{morphism of $c$-variable $d$-valued polynomials} is a natural
  transformation between them.
\end{definition}

The situation here is slightly different to before, since knowledge of
the sets $p_a(1)$ and $c$-sets $p_a[i]$
in~\eqref{eqn.cvalued_duc_query} does not completely determine $p$.
The missing aspect is the action by morphisms
$f \colon a \rightarrow b$ in $d$; these induce
transformations between $c$-valued polynomials
$p(\thg)(a) \Rightarrow p(\thg)(b)$, functorially in $f$.
Equivalently, by~\eqref{eqn.d-poly_map_combinatorics}, we have
functions $p_f(1) \colon p_a(1) \rightarrow p_b(1)$, together with
$c$-set maps $p_f^\sharp(i, \thg) \colon p_b[j] \rightarrow p_a[i]$
where $p_f(1)(i) = j$. Bearing in mind the functoriality in $f$, we
thus obtain:
\begin{proposition}
  \label{prop.cvalued_characterization}To give a $c$-variable
  $d$-valued polynomial is equivalently to give a $d$-set $p(1)$
  (with values $p_a(1)$ for $a \in d$) together with a functor
  $p[\thg]\colon(\el_d p(1))\op \to c\set$ (with values $p_a[i]$ for
  $(a,i) \in \el_d p(1)$).
\end{proposition}

Now \cref{prop.TFAE_duc} easily generalizes to the following.

\begin{proposition}\label{cor.TFAE_indexed_duc}
  For categories $c$ and $d$, the following categories are equivalent:
  \begin{enumerate}[(1)]
  \item the category of $c$-variable, $d$-valued polynomials;
  \item the category of functors $d\to\smset[c]$;
  \item the category of functors $d\to\coco((c\set)\op)$;
  \item the category $\Cat{CLP}(c\set,d\set)$ of
    connected-limit-preserving functors $c\set\to d\set$; and
  \item the category $\Praf(c\set,d\set)$ of parametric right adjoint
    functors $c\set\to d\set$.
  \end{enumerate}
\end{proposition}
Later \cref{thm.garner} will add a sixth element, \emph{the category
  $\comod(\poly)(c, d)$ of $(c, d)$-bicomodules}, to the list.
\begin{proof}
  (2) is simply a restatement of (1) under exponential transpose, so
  that (1) $\iff$ (2). Next, since limits in $\cat{d}\set$ are
  computed pointwise, a functor $\cat{c}\set \rightarrow \cat{d}\set$
  will preserve connected limits just when its transpose to a functor
  $d \rightarrow \catfun{\cat{c}\set}{\smset}$ factors through
  $\Cat{CLP}(c\set,\smset)$. So we may identify the category of (4)
  with the category of functors
  $d \rightarrow \Cat{CLP}(c\set,\smset)$, and so the equivalence of
  (2), (3) and (4) follows by applying~\cref{prop.TFAE_duc}
  componentwise. Finally, (4) $\iff$ (5) follows by exactly the same
  argument as in~\cref{prop.TFAE_duc}.
\end{proof}

We will use the janus notation $d\set[c]$ to denote the category of
functors $d\to\smset[c]$; it generalizes both $d\set[0]\cong d\set$
and $1\set[c]\cong\smset[c]$ from \cref{def.duc_query}.
Generalizing \cref{cor.limits}, we have:
\begin{proposition}\label{prop.multivar_limits}
  Small limits and coproducts exist in each category $d\set[c]$, and
  are computed as in $\catfun{c\set}{d\set}$.
\end{proposition}
\begin{proof}
  Just as in \cref{cor.limits}, this follows from the fact that limits
  and coproducts commute with connected limits in $\smset$, so that pointwise
  limits and coproducts of connected-limit preserving functors in
  $\catfun{c\set}{d\set}$ are again connected-limit preserving.
\end{proof}

To conclude this section, we
generalize~\cref{prop.d-poly_map_combinatorics} to give a
combinatorial description of morphisms between multi-variable
polynomials.

\begin{proposition}[Combinatorial description of natural
  transformations]\label{rem.comb_bicomodule_maps}
  Let $p,q \in d\set[c]$ be specified by $d$-sets $p(1), q(1)$ and
  functors $p[\thg]\colon(\el_d p(1))\op \to c\set$ and  $q[\thg]\colon(\el_d q(1))\op \to c\set$.
  We have the formula
  \begin{align}
    d\set[c](p,q)
    &=\sum_{\varphi_1\! \colon p(1) \rightarrow q(1)}\catfun{(\el_d p(1))\op}{c\set}(q[\varphi_1(-)], p[\thg]) \nonumber\\
    &=\sum_{\varphi_1\! \colon p(1) \rightarrow q(1)}\;\int\limits_{(b,i) \in \el_d p(1)}c\set\big(q_b[\varphi_1(i)],p_b[i]\big)\label{eqn.bicomod_map_combinatorics}
  \end{align}
  Here $\int$ denotes an end, taken over the category of elements of
  $p(1)$.
\end{proposition}
One can see that this formula reduces to~\eqref{eqn.mapsharp_formal}
in the case that $c=d=\yon$, and
to~\eqref{eqn.d-poly_map_combinatorics} when $d=\yon$.

\subsection{Composition of multi-variable polynomials}

In this section, we describe the composition of multi-variable
polynomials, and the $2$-category of multi-variable polynomials this
induces. In preparation for this, we give a lemma which gives a
convenient description of morphisms \emph{into} one of the values of a
polynomial functor.

\begin{lemma}
  \label{lem.mapsintopoly}
  Let $p \colon c\set \rightarrow d\set$ be a multi-variable
  polynomial functor. For any $X \in c\set$ and $Y \in d\set$ we have
  natural isomorphisms
  \begin{equation}
    \label{eqn.mapsintopoly1}
    d\set(Y, p(X)) \cong \sum_{\alpha \colon Y \rightarrow p(1)} c\set(\tilde \alpha, X)
  \end{equation}
  where $\tilde \alpha$ is the colimit of the functor
  \begin{equation}
    \label{eqn.mapsintopoly2}
    (\el_d Y)^\mathrm{op} \xrightarrow{(\el_d \alpha)^\mathrm{op}} (\el_d p(1))^\mathrm{op} \xrightarrow{p[\thg]} c\set\rlap{ .}
  \end{equation}
\end{lemma}
\begin{proof}
  Note that $d\set(Y, p(\thg)) \colon c\set \rightarrow \smset$ is the
  composite of the connected-limit preserving $p$ and the
  limit-preserving representable $d\set(Y, \thg)$. So
  by~\cref{prop.TFAE_duc}, it is a $c$-variable polynomial functor,
  and so can be written in the form~\eqref{eqn.mapsintopoly1} for
  \emph{some} family of $c$-sets $\tilde\alpha$. To see which ones,
  let us fix $\alpha$ and note $\tilde \alpha$ is a representing
  object for the functor $c\set \rightarrow \smset$ given by:
  \begin{equation}
    \label{eqn.mapsintopoly3}
    X \qquad \mapsto \qquad \{ f \colon Y \rightarrow p(X) \mid Y \xrightarrow{f} p(X) \xrightarrow{p(!)} p(1) = \alpha \}\rlap{ .}
  \end{equation}
  The side-condition on the right-hand side specifies just those $f$ whose
  components have the form:
  \begin{align*}
    Y(a) & \rightarrow p(X)(a) = \sum_{i \in p_a(1)} c\set(p_a[i], X) \\
    x & \mapsto (\alpha_d(x),\, \theta \colon p_a[\alpha_d(x)] \rightarrow X)\rlap{ ;}
  \end{align*}
  but the data of these components are precisely the data of a cocone
  under the diagram~\eqref{eqn.mapsintopoly2} with vertex $X$. It
  follows that the representing object $\tilde \alpha$
  of~\eqref{eqn.mapsintopoly3} must be the colimit of this diagram, as
  desired.
\end{proof}

\begin{proposition}[Composition of multi-variable polynomial functors]\label{prop.comp_multi_monoidal}
  If $q \colon
  c\set \rightarrow d\set$ and $p \colon d\set \rightarrow e\set$ are
  multi-variable polynomial functors, then so is their composite $p \circ
  q \colon c\set \rightarrow e\set$. If $p$ and $q$ are specified as
  in~\cref{prop.cvalued_characterization} by:
  \begin{equation*}
    p(1) \in e\set, \qquad p[\thg]\colon(\el_e p(1))\op \to d\set \qquad \text{and} \qquad 
    q(1) \in d\set, \qquad q[\thg]\colon(\el_d q(1))\op \to c\set
  \end{equation*}
  then their composite is specified by
  $r(1) \in e\set$ and $r[\thg]\colon(\el_e p(1))\op \to c\set$
  where:
  \begin{enumerate}[(1)]
  \item The value of $r(1)$ at $a \in e$ is $r_a(1) = \{\,i \in
    p_a(1),\, \alpha \colon p_a[i] \rightarrow q(1)
    \text{ in }d\set\,\}$;
  \item The value of $r(1)$ at $f \colon a \rightarrow b$ in $e$ is
    the map $r_f(1) \colon r_a(1) \rightarrow r_b(1)$ sending $(i, \alpha)$
    to $(j, \beta)$ where $j = p_f(1)(i)$ and
    \begin{equation*}
      \beta \coloneqq p_b[j] \xrightarrow{p^\sharp_f(i, \thg)} p[a,i] \xrightarrow{\ \ \alpha\ \ } q(1)\rlap{ .}
    \end{equation*}
  \item For any $(i, \alpha) \in r_a(1)$, we have $r_a[i, \alpha]
    \in c\set$
    given by the colimit of the functor
    \begin{equation*}
      (\el_d p_a[i])^\mathrm{op} \xrightarrow{\el_d \alpha} (\el_d q(1))^\mathrm{op} \xrightarrow{q[\thg]} c\set\rlap{ .}
    \end{equation*}
  \item If $(i, \alpha) \in r_a(1)$ has image $(j, \beta)$ under
    $r_f(1) \colon r_a(1) \rightarrow r_{b}(1)$, then the map
    $r_f^\sharp(i, \alpha, \thg) \colon r_b[j,
    \beta] \rightarrow r_a[i, \alpha]$ is the unique map on colimits
    induced by the commuting triangle:
    \begin{equation*}
      \begin{tikzcd}[column sep={5em,between origins},row sep={1em}]
        {(\el_d p_b[j])\op} \ar{rrrr}{(\el_d p_f^\sharp(i,\thg)\op} \ar[dr, pos=.4, "\el_d \beta"'] & & & & 
        {(\el_d p_a[i])\op} \ar[dl, pos=.4, "\el_d \alpha"]  \\
        & {(\el_d q(1))\op} \ar[dr, "{q[\thg]}"'] & &
        {(\el_d q(1))\op} \ar[dl, "{q[\thg]}"] \\ & &
        c\set\rlap{ .}
      \end{tikzcd}
    \end{equation*}
  \end{enumerate}
\end{proposition}
The reader may feel concerned that the use of colimits in the third
dot point above runs counter to the idea that, in studying databases,
the only colimits we should allow ourselves are disjoint unions.
However, on further reflection, we see that these colimits are being
taken over the \emph{shapes} which index our conjunctive queries---so
that when we hom it into the data itself, this colimit becomes a
limit. In other words, these colimits serve only to build a new
conjunctive query from a family of existing ones.
\begin{proof}
  Like before, the composite of two connected-limit-preserving
  functors is also connected-limit-preserving, so the first claim
  follows by \cref{cor.TFAE_indexed_duc}. As for the formula, we have
  by~\cref{lem.mapsintopoly} that:
  \begin{equation*}
    p(q(X))(a)
    =
    \sum_{i\in p_a(1)}d\set(p_a[i], q(X)) = \sum_{i \in p_a(1)} \sum_{\alpha \colon p_a[i] \rightarrow q(1)} c\set(\tilde \alpha, X)
  \end{equation*}
  where $\tilde \alpha$ is precisely the colimit presheaf in part (3)
  of the formula. This establishes clauses (1) and (3) of the formula,
  and a straightforward analysis of functoriality now yields (2) and (4).
\end{proof}

\begin{definition}[The $2$-category of multi-variable polynomials]\label{def.polyfun}
  The $2$-category $\polyfunb$ is the $2$-category whose objects are
  small categories $c,d,e$, whose hom-category from $c$ to $d$ is
  $d\set[c]$, and whose composition is composition of multi-variable
  polynomials.
\end{definition}

Generalising \cref{lemma.preservation_of_equalizers}, we have the
following result; its proof is mutatis mutandis the same.
\begin{lemma}[Composition preserves connected
  limits]\label{lemma.multivar_preservation_of_equalizers}
  Each hom-category $\polyfunb(c,d)$ has connected limits, and these
  are preserved by composition in each variable.
\end{lemma}

In the introduction, we explained how multi-variable polynomial
functors between discrete categories may be represented by ``bridge
diagrams'' of the form~\eqref{eqn.bridgediag}. Let us now show that
same thing is possible for prafunctors between categories of
copresheaves. We begin by describing the prafunctors which will interpret
the legs of these bridges.

\begin{notation}[Outfacing maps]\label{notation.outfacing}
  Given a small category $c$ and object $a \in c$, we write $c[a]$ for
  the set of all morphisms in $c$ with domain $a$, and call them
  \emph{$a$-outfacing maps} of $c$. (The reason for
  this notation will become clear in~\cref{thm.cats_comonads}.)
\end{notation}
\begin{definition}[\'Etale maps, a.k.a.\ discrete opfibrations]
  A functor $\varphi\colon c\to d$ is called an \emph{\'etale map}, or a
  \emph{discrete opfibration}, if the induced function
  $ c[a]\to d[\varphi(a)]$, sending $a$-outfacing maps of $c$ to
  $\varphi(a)$-outfacing map of $d$, is a bijection for each
  $a \in \ob(c)$. We prefer the terminology \emph{\'etale}, because the
  word is shorter, prettier, and evokes the correct mental picture:
  there is a ``local homeomorphism'' between the outfacing maps from
  object $a$ and those from its image object $\varphi(a)$.
\end{definition}

Given $F\colon\cat{c}\to\cat{d}$, we denote the precomposition functor
$\cat{d}\set\to\cat{c}\set$ by $\Delta_F$, its right adjoint by
$\Pi_F$, and its left adjoint by $\Sigma_F$, as is standard.

\begin{proposition}[Functors give adjoint prafunctors]\label{prop.functor_adj_bico}
  Let $c$ and $d$ be categories. For any functor
  $F\colon\cat{c}\to\cat{d}$, both functors in the adjunction
  \[
    \begin{tikzcd}[column sep=large]
      d\set\ar[r, shift left=5pt, "\Delta_F"]
      \ar[r, phantom, "\scriptstyle\Rightarrow"]&
      c\set\ar[l, shift left=5pt, "\Pi_F"]
    \end{tikzcd}
  \]
  are parametric right adjoint functors. If $F$ is
  \'etale, then both functors in the adjunction
  \[
    \begin{tikzcd}[column sep=large]
      d\set\ar[r, shift right=5pt, "\Delta_F"']
      \ar[r, phantom, "\scriptstyle\Rightarrow"]&
      c\set\ar[l, shift right=5pt, "\Sigma_F"']
    \end{tikzcd}
  \]
  are also parametric right adjoint functors.
\end{proposition}
\begin{proof}
  Both $\Delta_F$ and $\Pi_F$ are right adjoints, hence parametric right adjoints
  $d\set\To{\Delta_F}c\set$ and $c\set\To{\Pi_F}d\set$, and the first result
  follows from \cref{prop.adjoint_prafunctors}.

  In general $\Sigma_F$ assigns the colimit of a certain comma
  category to each $a\in\ob\cat{c}$. When $F$ is \'etale this diagram
  has a discrete final subcategory, and hence can be computed as a
  coproduct. As such $\Sigma_F$ is a parametric right adjoint and we
  again apply \cref{prop.adjoint_prafunctors}.
\end{proof}
\begin{remark}[$\Delta_F, \Pi_F$ and $\Sigma_F$
  explicitly]\label{rmk.explicitdeltapisigma}
  If $F \colon c \rightarrow d$ is a functor, then the induced functors $\Delta_F$
  and $\Pi_F$ have the explicit formulae
  \begin{equation*}
    \Delta_F(X)(a) = X(Fa) = d\set(\yon^{Fa}, X) \qquad \text{and} \qquad
    \Pi_F(Y)(b) = c\set(\yon^b F, Y)\rlap{ ,}
  \end{equation*}
  where here $\yon^b F$ is the composite functor $c \xrightarrow{F} d
  \xrightarrow{\yon^b} \smset$. If $F$ is \'etale, then $\Sigma_F$ has
  the formula
  \begin{equation*}
    \Sigma_F(Y)(b) = \sum_{a \in F^{-1}(b)} F(a)\rlap{ .}
  \end{equation*}
\end{remark}
\begin{proposition}[\cite{weber2007theory}]\label{prop.bridge_diagram}
  Every prafunctor $p \colon c\set \rightarrow d\set$ can be written as
  $\Sigma_T \circ \Pi_\pi \circ \Delta_S$ for a diagram of categories
  as below wherein $T$ is \'etale.
  \begin{equation*}
  \begin{tikzcd}
    &\cat{e}\ar[dl, "S"']\ar[r, "\pi"]&\cat{b}\ar[dr, "T"]\\[-10pt]
    \cat{c}&&&\cat{d}
  \end{tikzcd}
  \end{equation*}
\end{proposition}
\begin{proof}
  Let $b \coloneqq \el_d p(1)$ and let $T \colon b \rightarrow d$ be
  the usual (\'etale) projection from the category of elements. Now let
  $e$ be the category with:
  \begin{itemize}
  \item \textbf{objects}: tuples $(w,x,a,y)$ where $(w,x) \in \el_d p(1)$, $a \in c$
    and $y \in p_w[x](a)$;
  \item \textbf{morphisms} $(w,x,a,y) \rightarrow (w',x',a',y')$ given by
    a pair of maps $f \colon (w,x) \rightarrow (w',x')$ in $\el_d p(1)$ and $g
    \colon a \rightarrow a'$ in $c$ such that $p_f(g)(y') = y$.
  \end{itemize}
  Let $S$ and $\pi$ be the obvious projection functors. We may now
  calculate directly using \cref{rmk.explicitdeltapisigma} that the
  functor $\Sigma_T \circ \Pi_\pi \circ \Delta_S$ induced by this
  diagram is isomorphic to $p$.
\end{proof}

\subsection{Retrofunctors and $\ccatsharp$}
There is an additional layer of structure that can be added to the
$2$-category of multi-variable polynomials in order to make it into a
\emph{framed bicategory}~\cite{shulman2008framed} As we will see in
the next section, the prafunctors $c\set \rightarrow d\set$ are
exactly \emph{bicomodules} between comonoids in the monoidal category
$\poly$, and the extra layer of structure to be introduced are the
corresponding comonoid \emph{homomorphisms}. One might think these
would be functors between small categories, but in fact, this is not
the case.

\begin{definition}[Retrofunctor]\label{def.retrofunctor}
  Let $c$ and $d$ be categories. A \emph{retrofunctor}
  $\varphi\colon c\coto d$ (also known as a \emph{cofunctor})
  consists of:
  \begin{enumerate}[itemsep=0em]
  \item a \emph{forwards function}  $\varphi_1\colon \ob( c)\to\ob(d)$ on objects and
  \item a \emph{backwards function} $\varphi^\sharp_a\colon d[\varphi_1(a)]\to c[a]$ on morphisms, for each $a\in\ob( c)$,
  \end{enumerate}
  satisfying the following conditions, which say that $\varphi^\sharp$ preserves identities, codomains,
  and composition:
  \begin{enumerate}[label=\roman*.]
  \item $\varphi^\sharp_a(\id_{\varphi_1(a)})=\id_a$ for any $a\in \ob( c)$;
  \item $\varphi_1(\cod \varphi^\sharp_a(f))=\cod f$ for any $a\in \ob( c)$ and $f\in d[\varphi_1(a)]$; and
  \item $\varphi^\sharp_{\cod \varphi^\sharp_a(f)}(g)\circ \varphi^\sharp_a(f)=\varphi^\sharp_a(g\circ f)$ for composable arrows $f,g$ out of $\varphi_1(a)$.
  \end{enumerate}
  With the obvious composition, we obtain a (large) category
  $\catsharp$ of small categories and retrofunctors.
\end{definition}


Here are some examples of retrofunctors in the wild. 

\begin{example}[\'Etale functors as retrofunctors]\label{ex.\'etale_retrofunctor}
  As noted above, an \'etale functor $\varphi\colon c\to d$ is one for
  which the induced map $ c[a]\to d[\varphi(a)]$, sending
  $a$-outfacing maps in $c$ to $\varphi(a)$-outfacing maps in $d$, is
  a bijection for each $a \in \ob(c)$. In this situation, we can
  define $\varphi^\sharp_a\colon d[\varphi(a)]\to c[a]$ to be the
  inverse of this bijection, and one now checks easily that
  $\bar \varphi = (\ob(\varphi),\varphi^\sharp)$ satisfies the
  conditions of being a retrofunctor $c\coto d$. Thus, discrete
  opfibrations $\varphi \colon c \rightarrow d$ can be identified with
  retrofunctors $\bar \varphi \colon c \coto d$ with the special
  property that $\varphi^\sharp_a$ is a bijection for each
  $a\in \ob(c)$.
\end{example}

\begin{example}[Very well-behaved lenses]\label{ex.indiscrete}
  Let $S,T\in\smset$ be sets, and consider the associated codiscrete
  categories $\codisc{S},\codisc{T}$, i.e., those which have a unique
  morphism between any two objects. 
  We read off that a retrofunctor $\codisc{S}\coto\codisc{T}$ consists
  of a function $\varphi_1 \colon S\to T$ and a function
  $\varphi^\sharp\colon S\times T\to S$ satisfying three laws:
  \begin{enumerate}[(1), itemsep=0pt]
  \item $\varphi^\sharp(s,\varphi_1(s))=s$
  \item $\varphi_1(\varphi^\sharp(s,t))=t$
  \item $\varphi^\sharp(\varphi^\sharp(s,t_1),t_2)=\varphi^\sharp(t_2)$.
  \end{enumerate}
  Functional programmers will recognize these as the three \emph{lens
    laws} for what are sometimes called \emph{very well-behaved
    lenses}.
\end{example}


\begin{example}[Retrofunctors of sets]
  Retrofunctors between \emph{discrete} categories $S\yon \coto T\yon$
  are less interesting: they are simply functions $S \to T$.
\end{example}

\begin{example}[Retrofunctors of monoids]
  The full subcategory of $\catsharp$ spanned by the one-object
  categories is isomorphic to the opposite of the category of monoids.
  Indeed, retrofunctors of one-object categories
  $\mathscr{M} \coto \mathscr{N}$ are functions $N \to M$ between the
  unique hom-sets respecting identity and multiplication.
\end{example}

\begin{example}[Bijective on objects functors as retrofunctors]\label{ex.boo_retrofunctor}
  Generalizing the previous example, if $\varphi\colon c\to d$ is a
  functor whose action $\ob(\varphi) \colon \ob(c) \rightarrow \ob(d)$
  is invertible, then defining $\tilde \varphi_1$ to be the inverse of
  $\ob(\varphi)$, and defining
  $\tilde \varphi_a^\sharp \colon c[\tilde \varphi_1(a)] \rightarrow
  d[a]$ to be the action on morphisms of $\varphi$, one can check that
  $\tilde \varphi = (\tilde \varphi_1,\tilde \varphi^\sharp)$ is a
  retrofunctor $d\coto c$. Thus, bijective-on-objects functors
  $\varphi \colon c \rightarrow d$ can be identified with
  bijective-on-objects retrofunctors
  $\tilde \varphi \colon d \coto c$.
\end{example}

We now explain how to link retrofunctors to multi-variable
polynomials. The key is the following decomposition result for retrofunctors:

\begin{lemma}[Spans from retrofunctors]
  \label{lem.retrofunctor_decomp}
  Every retrofunctor $\varphi \colon c \coto d$ has a canonical
  decomposition $\varphi = \widetilde S \then \overline T$ for a span of
  categories and functors
  \begin{equation}\label{eqn.retrofunctor_decomp}
    \begin{tikzcd}
    &\cat{r_\varphi}\ar[dl, "S"']\ar[dr, "T"]\\[-10pt]
    \cat{c}&&\cat{d}
  \end{tikzcd}
\end{equation}
wherein $S$ is bijective on objects and $T$ is \'etale (so that
$\tilde S$ is defined as in \cref{ex.boo_retrofunctor} and
$\overline T$ is defined as in \cref{ex.\'etale_retrofunctor}); we call
this the \emph{span representation} of $\varphi$. Moreover,
composition of retrofunctors corresponds to composition of span
representations by pullback.
\end{lemma}
\begin{proof}
  (cf.~\cite[\S 4.4]{aguiar1997internal}). Define $r_\varphi$ to be
  the category with the same objects as $c$, and as morphisms
  $a \rightarrow a'$, the set of morphisms
  $f \colon \varphi_1(a) \rightarrow \varphi_1(a')$ in $d$ for which
  $\cod\varphi_a^\sharp(f)=a'$, i.e.\ for which $\varphi_a^\sharp(f)$ is a map $a\to a'$.
  Such morphisms are
  easily seen to be composition-closed, so that $r_\varphi$ is a
  category. Let $F$ be the identity-on-objects functor sending
  $f \colon \varphi_1(a) \rightarrow \varphi_1(a')$ to
  $\varphi_a^\sharp(f) \colon a \rightarrow a'$, and let $T$ be the
  functor acting as $\varphi_1$ on objects, and as the identity on
  morphisms. It is now straightforward to verify that $T$ is \'etale and
  that $\widetilde S \then \overline T = \varphi$.
\end{proof}
\begin{definition}[Polynomials from retrofunctors]
  \label{def.poly_from_retro}
  Let $c$ and $d$ be categories, and let
  $\varphi\colon\cat{c}\coto\cat{d}$ be a retrofunctor with span
  representation~\eqref{eqn.retrofunctor_decomp}. We define
  $\varphi\lpush  \colon c\set \rightarrow d\set$ and
  $\varphi^\ast \colon d\set \rightarrow c\set$ by
  \begin{equation*}
    \varphi\lpush  \coloneqq c\set \xrightarrow{\Delta_S} r_\varphi\set \xrightarrow{\Sigma_T} d\set \qquad \text{and} \qquad 
    \varphi^\ast \coloneqq d\set \xrightarrow{\Delta_T} r_\varphi\set \xrightarrow{\Pi_S} c\set\rlap{ .}
  \end{equation*}
\end{definition}

\begin{proposition}
  \label{prop.retrofunctorembedding}
  For each retrofunctor
  $\varphi\colon\cat{c}\coto\cat{d}$, the functors $\varphi\lpush $ and
  $\varphi^\ast$ are prafunctors, and we have 
  \[
    \begin{tikzcd}[column sep=large]
      c\set\ar[r, shift left=5pt, "\varphi\lpush "]
      \ar[r, phantom, "\scriptstyle\Rightarrow"]&
      d\set\rlap{ .}\ar[l, shift left=5pt, "\varphi^\ast"]
    \end{tikzcd}
  \]
\end{proposition}

\begin{proof}
  $\varphi\lpush $ and $\varphi^\ast$ are prafunctors by
  \cref{prop.functor_adj_bico}, and are adjoint by the same result.
\end{proof}

\begin{remark}
  We can describe the polynomials $\varphi\lpush $ and
  $\varphi^\ast$ associated to a retrofunctor
  $\varphi \colon c \coto d$ explicitly.
  \begin{enumerate}[(1)]
  \item 
  The polynomial functor
  $\varphi\lpush  \colon c\set \rightarrow d\set$ takes the $c$-set $X$
  to the $d$-set $\varphi\lpush  X$ given by
  \begin{equation}
    \label{eqn.induced_prafunctor}
    \varphi\lpush  X(b) = \sum_{a \in \varphi_1^{-1}(b)} X(a) \qquad \text{and} \qquad 
    \varphi\lpush  X(f \colon b \rightarrow b') \mapsto
    \Biggl(\begin{aligned}
      \sum_{a \in \varphi_1^{-1}(b)} X(a) &\rightarrow \sum_{a' \in \varphi_1^{-1}(b')} X(a')\\
      (a, x) & \mapsto (a',x')
    \end{aligned}\,\Biggr)
  \end{equation}
  where to the right, $a' \in c$ is the codomain of
  $f' \coloneqq \varphi^\sharp_a(f)$, and $x'$ is $X(f')(x)$.

\item The polynomial functor
  $\varphi^\ast \colon d\set \rightarrow c\set$ takes the $d$-set $Y$
  to the $c$-set $\varphi^\ast Y$ with
  \begin{equation*}
    \varphi^\ast Y(a) = d\set(\varphi\lpush c[a], X) \qquad \text{and} \qquad 
    \varphi^\ast Y(h \colon a \rightarrow a') = (\thg) \circ \varphi\lpush c[h] \colon d\set(\varphi\lpush c[a], X) \rightarrow d\set(\varphi\lpush c[a'], X)
  \end{equation*}
  where explicitly, $\varphi\lpush c[a]$ is the $d$-set with
  \begin{equation*}
    \varphi\lpush c[a](b) = \{f \in c[a] \mid \mathrm{cod}(f) \in \varphi_1^{-1}(b)\} \qquad \text{and} \qquad 
    \varphi\lpush c[a](g \colon b \rightarrow b') \colon f \mapsto \varphi_a^\sharp(g) \circ f
  \end{equation*}
  and $\varphi\lpush c[h] \colon \varphi\lpush c[a'] \rightarrow \varphi\lpush c[a]$ acts by
  precomposing by $h\colon a\to a'$.
  \end{enumerate}
\end{remark}

We would now like to show that the assignments
$\varphi \mapsto \varphi\lpush $ and $\varphi \mapsto \varphi^\ast$ are
(pseudo-)functorial in $\varphi$. The key to doing so will be the fact
that retrofunctors, and their composition, can be completely described
in terms of the span representations.

\begin{definition}
  Let $\catsharpspan$ be the sub-bicategory of
  $\spancat(\smcat)$ obtained by restricting to (bijective on objects,
  \'etale) spans as in~\cref{eqn.retrofunctor_decomp} and all
  span $2$-cells between them.
\end{definition}

We will shortly prove:

\begin{proposition}[Retrofunctors versus spans]\label{prop.catsharpspan_representation}
  Each hom-category $\catsharpspan(c,d)$ is an equivalence
  relation, i.e., a groupoid and a poset,~and 
  \begin{equation}
    \label{eqn.catsharpspan_maps}
    (S,T) \xrightarrow{\ \exists\ } (U,V) \qquad \text{ if and only if } \qquad  \widetilde S\then \overline T = \widetilde U\then \overline V
    \colon c \coto d\rlap{ .}
  \end{equation}
  Thus the assignment $(S,T) \mapsto \widetilde S\then\overline T$ is
  an equivalence of categories
  $\catsharpspan(c,d) \rightarrow \catsharp(c,d)$, viewing
  the codomain as a discrete category. These equivalences are the
  actions on homs of an identity-on-objects biequivalence
  $\catsharpspan \rightarrow \catsharp$, with
  pseudo-inverse $\catsharp \rightarrow \catsharpspan$
  given by taking span representations.
\end{proposition}

However, before doing so, we use this result to establish:

\begin{proposition}
  The assignments
  $\varphi \mapsto \varphi\lpush $ and $\varphi \mapsto \varphi^\ast$
  yield adjoint identity-on-objects pseudofunctors $(\thg)\lpush  \colon \catsharp \rightarrow \polyfunb$
  and $(\thg)^\ast \colon (\catsharp)^\mathrm{op} \rightarrow \polyfunb$.
\end{proposition}
\begin{proof}
  Due to their adjointness (from \cref{prop.retrofunctorembedding}), $(\thg)^\ast$ will be pseudofunctorial so
  long as $(\thg)\lpush $ is, so we need only check the latter fact.
  Consider the assignment $\catsharpspan \rightarrow
  \polyfunb$ which is identity on objects, sends a morphism $(S,T)$ to
  $\Delta_S \then \Sigma_T$, and sends a span $2$-cell $H \colon (S,T)
  \rightarrow (U,V)$ to the canonical $2$-cell
  \begin{equation*}
   \Delta_S \then \Sigma_T \xrightarrow{\cong} \Delta_U \then \Delta_H \then \Sigma_H \then \Sigma_V \xrightarrow{\Delta_U \then \epsilon \then \Sigma_V} \Delta_U \then \Sigma_V
 \end{equation*}
  where the first isomorphism comes from the fact that $S = UH$ and $T
  = VH$.
  We claim this assignment is pseudofunctorial. Indeed, given spans
  $(S,T) \colon c \rightarrow d$ and $(H,K) \colon d \rightarrow e$,
  we can form their pullback:
  \begin{equation*}
\begin{tikzcd}[column sep=1.8em]
  &&\cat{t} \ar[dl, "L"']\ar[dr, "M"] & & \\[-10pt]
  &\cat{r}\ar[dl, "S"']\ar[dr, "T"] & &
  \cat{s}\ar[dl, "H"']\ar[dr, "K"] & &
  \\[-10pt]
    \cat{c}&&\cat{d}&&\cat{e}
  \end{tikzcd}
  \end{equation*}
  and must exhibit an isomorphism
  $\Sigma_K \circ \Delta_H \circ \Sigma_T \circ \Delta_S \cong \Sigma_K \circ \Sigma_M \circ \Delta_L \circ
  \Delta_S$. This will follow if we can
  show that $\Delta_L \circ \Sigma_M \cong \Delta_H \circ \Sigma_T$. This
  \emph{Beck--Chevalley isomorphism}, equating the two ways around a
  pullback square of categories and functors via $\Delta$ and
  $\Sigma$, is known not to hold for arbitrary pullback squares, but
  again we are saved by the fact that $T$ (and so $M$) are \'etale; in
  fact, we can simply read off the desired isomorphism from
  \cref{rmk.explicitdeltapisigma}.
  Thus we have a pseudofunctorial assignment $\Omega \colon \catsharpspan
  \rightarrow \polyfunb$; and to conclude the proof, we need only
  observe that
  \begin{equation*}
    (\thg)\lpush  = \catsharp \xrightarrow{\text{span rep}} \catsharpspan \xrightarrow{\Omega} \polyfunb
  \end{equation*}
  is the composite of two pseudofunctors, and hence pseudofunctorial.
\end{proof}

We now give the proof of~\cref{prop.catsharpspan_representation}.
\begin{proof}
  Suppose that
  \begin{equation}
    \label{eqn.catsharpspan_map_diagram}
    \begin{tikzcd}[row sep=1.2em]
      & r \arrow[ddl, bend right, "S"'] \arrow[bend left, ddr, "T"] \arrow[d, "H"] \\
      & r' \arrow[dl, "U"] \arrow[dr, "V"'] \\
      c & & d
    \end{tikzcd}
  \end{equation}
  is a map of (bijective on objects, \'etale) spans. By standard
  cancellativity properties, $H$ is bijective on objects since $S$ and
  $U$ are, while $H$ is \'etale since $T$ and $V$ are; and a bijective
  on objects, \'etale functor is necessarily an isomorphism.
  Furthermore, since $S$ and $U$ are bijective on objects, the action
  on objects of $H$ is uniquely determined. But any morphism between
  \'etale maps is uniquely determined by its object-component and so $H$
  in its totality is uniquely determined. Thus
  $\catsharpspan(c,d)$ is an equivalence relation.

  To establish the ``only if'' direction
  of~\cref{eqn.catsharpspan_maps}, note that, if $H$ is an invertible
  functor with inverse $K$, then $\widetilde H = \overline{K}$ and
  $\overline H = \widetilde K$ as retrofunctors. Thus,
  given~\cref{eqn.catsharpspan_map_diagram}, we have
  \begin{equation*}
    \widetilde S \then \overline T = \widetilde{U} \then \widetilde H \then \overline T = \widetilde{U} \then \overline K \then \overline T = \widetilde {U} \then \overline{V}
  \end{equation*}
  as desired. 
  For the ``if'' direction
  of~\cref{eqn.catsharpspan_maps}, it suffices to show that, if
  $(U,V) \in \catsharpspan(c,d)$, and
  $\varphi = \widetilde U \then \overline V$, then we have a $2$-cell
  \begin{equation}
    \label{eqn.catsharpspan_map_diagram_2}
    \begin{tikzcd}[row sep=1.2em]
      & r \arrow[ddl, bend right, "U"'] \arrow[bend left, ddr, "V"] \arrow[d, "H"] \\
      & r_\varphi \arrow[dl, "S"] \arrow[dr, "T"'] \\
      c & & d\rlap{ .}
    \end{tikzcd}
  \end{equation}
  Note first that, since $U$ is bijective on objects,
  $\varphi \colon c \coto d$ is completely characterized by the
  clauses:
  \begin{equation*}
    \varphi_1(Ua) = Va \qquad \text{and} \qquad \varphi_{Ua}^\sharp(Va \xrightarrow{f} b') = U(V^\sharp_{a}(f))
  \end{equation*}
  where, since $V$ is \'etale, $V_a^\sharp(f)$ is the unique lift of
  $f \colon Va \rightarrow b'$ to a map $a \rightarrow a'$ in $r$.
  Thus, objects of $r_\varphi$ are just objects of $c$, while maps
  $Ua \rightarrow Ua'$ in $r_\varphi$ are maps
  $f \colon Va \rightarrow Va'$ in $d$ such that $U(V^\sharp_{a}(f))$
  has codomain $Ua'$. Thus, we can define define $H$ to act on objects
  by $a \mapsto Ua$ and on morphisms by $f \mapsto Vf$.
  
  This establishes the existence of equivalences of categories
  $\catsharpspan(c,d) \rightarrow \catsharp(c,d)$.
  Moreover, it is an easy calculation
  (cf.~\cite[Lemma~2.3]{garner2024}) to see that, if we have a
  pullback square
  \begin{equation*}
    \begin{tikzcd}
      a \arrow[r,"F"] \arrow[d,"G"'] &
      b \arrow[d,"H"]\\
      c \arrow[r,"K"'] & d\ar[ul, phantom, very near end, "\lrcorner"]
    \end{tikzcd}
  \end{equation*}
   with $K$ (and hence $F$) bijective on objects
  and with $H$ (and hence $G$) \'etale, then
  $\widetilde F \then \overline G = \overline H \then \widetilde K$.
  It follows that the functors
  $\catsharpspan(c,d) \rightarrow \catsharp(c,d)$ are the
  action on homs of an identity-on-objects biequivalence of
  bicategories $\catsharpspan \rightarrow \catsharp$, with
  pseudoinverse given by taking
  span representations.
\end{proof}

\cref{prop.retrofunctorembedding} tells us that small categories,
retrofunctors, and prafunctors provide us with an instance of the
following notion:
\begin{definition}[Pro-arrow equipment]
  \label{def.proarrow}
  A \emph{pro-arrow equipment} comprises a category $\prov$, a
  bicategory $\proh$, and a bijective-on-objects pseudofunctor
  $(\thg)\lpush  \colon \prov \rightarrow \proh$ with the property that,
  for each $f \in \prov$, the morphism $f\lpush $ has a right adjoint
  $f^\ast$ in the bicategory
  $\proh$.
\end{definition}
An equivalent, and often more convenient formulation of
pro-arrow equipments is in terms of \emph{framed bicategories}. These
are particular kinds of (pseudo) double category: thus, they have
objects $c,d,e,\dots$, \emph{vertical} morphisms
$f \colon c \rightarrow c'$, \emph{horizontal} morphisms
$m \colon c \rightarrow d$, and \emph{squares} as displayed to the left in:
\begin{equation*}
  \begin{tikzcd}
    c\ar[d, "\alpha"']\ar[r, "m", ""' name=M]&d\ar[d, "\beta"] & & 
    c\ar[d, "\alpha\lpush "']\ar[r, "m", ""' name=M'']&d\ar[d, "\beta\lpush "]\\
    c'\ar[r,  "m'"', "" name=M']&d'
    \ar[from=M, to=M', Rightarrow, shorten=2mm, "\varphi"] & & 
    c'\ar[r, "m'"', "" name=M''']&d'\rlap{ ,}
    \ar[from=M'', to=M''', Rightarrow, shorten=2mm, "\varphi"]
  \end{tikzcd}
\end{equation*}
together with composition operations which are vertically strictly
associative, but horizontally associative only up to suitable
invertible cells. Now, the framed bicategory associated to a pro-arrow
equipment $(\thg)\lpush  \colon \prov \rightarrow \proh$ has objects and
vertical morphisms given by objects and morphisms of $\prov$;
horizontal morphisms given by morphisms of $\proh$; and cells as to
the left above given by $2$-cells in $\proh$ as to the right. What
makes this double category into a \emph{framed} bicategory is the
possibility of turning a vertical morphism $f$ into horizontal
morphism in two ways: by taking $f\lpush $ or by taking $f^\ast$. Said
another way, a framed bicategory is a double category with all
\emph{companions} and \emph{conjoints}~\cite[\S
1]{Grandis.Pare:2004a}.

The basic example of a pro-arrow equipment has $\prov$ given by the
category of small categories and $\proh$ the bicategory of
profunctors, with $(\thg)\lpush $ the usual embedding of functors into
profunctors; in this case, applying the construction above yields the
framed bicategory $\mathbb{C}\Cat{at}$ of categories, functors and
profunctors. On the other hand, applying the construction to the
pro-arrow equipment of \cref{prop.retrofunctorembedding} yields:
\begin{definition}[The polynomial ecosystem]\label{def.ccatsharp}
  $\ccatsharp$ is the framed bicategory obtained from the pro-arrow
  equipment $(\thg)\lpush  \colon \catsharp \rightarrow \polyfunb$. Thus,
  its objects are categories, its vertical morphisms are
  retrofunctors, and its horizontal morphisms are prafunctors.
\end{definition}

\section{Building $\ccatsharp$ from $\poly$, abstractly}\label{chap.abstract}
In 2016, Ahman and Uustalu proved an amazing result \cite{ahman2016directed}:%
\begin{center}
  \emph{Polynomial comonads are exactly small categories.}
\end{center}
By a \emph{polynomial comonad}, we mean a comonad on $\smset$ with
polynomial carrier, i.e., a comonoid in the monoidal category
$(\poly,\yon,\tri)$. As already in the previous section, the
corresponding comonoid morphisms are not functors between categories,
but the retrofunctors of~\cref{def.retrofunctor}; so that Ahman and
Uustalu's result provides an explanation of the category $\catsharp$
purely in terms of comonoids in $\poly$.

In this section, we show that we can re-find not just $\catsharp$ but
the whole framed bicategory $\ccatsharp$ by considering polynomial
comonads. The missing ingredient is a description of the prafunctors
between presheaf categories in terms of polynomial comonads; our
contribution is to show that:
\begin{center}
  \emph{Bicomodules between polynomial comonads are prafunctors
    between copresheaf categories.}
\end{center}
Thus, the main result of this section establishes an equivalence
between $\ccatsharp$ and the framed bicategory of polynomial comonads,
comonad morphisms and bicomodules.

One approach to proving this result is very concrete: we simply work
through the definitions on each side and match them up. However, we
also need to check that this matching is well-behaved with respect to
composition, and this leads to a proliferation of messy details.
Instead, we adopt an approach which exploits known results from the
literature characterizing (framed) bicategories of bicomodules; then,
once the result is proved, we can spell out the concrete
correspondence secure in the knowledge that the details are taken care
of.

\subsection{The framed bicategory of bicomodules}\label{sec.bicomod}
If $\mathscr{V}$ is a monoidal category, then we have the standard
notions of \emph{comonoid} and \emph{comonoid morphism} in
$\mathscr{V}$, yielding a category $\Cat{Comon}(\mathscr{V})$; for the
sake of reference, we record the details below. In what follows, we
simplify notation by treating monoidal categories as if they were
strict monoidal categories by suppressing associativity and unitality
constraints. This is justified by the coherence theorem for monoidal
categories.

\begin{definition}[Comonoids and homomorphisms]\label{def.poly_comonad}
  A \emph{comonoid} in a monoidal category $(\mathscr{V}, \alttensor, I)$
  is a triple $(c,\epsilon,\delta)$, where $c\in\mathscr{V}$ is a
  polynomial and $\epsilon\colon c\to I$ and
  $\delta\colon c\to c \alttensor c$ are maps rendering commutative the
  diagrams:
  \begin{equation}\label{eqn.counitality_coassoc}
    \begin{tikzcd}
      &c\ar[dr, equal]\ar[dl, equal]\ar[d, "\delta" description]&&[30pt]
      c\ar[r, "\delta"]\ar[d, "\delta"']&c\alttensor c\ar[d, "c\alttensor\delta"]\\
      c&c\alttensor c\ar[l, "c\alttensor\epsilon"]\ar[r, "\epsilon\alttensor c"']&c&
      c\alttensor c\ar[r, "\delta\alttensor c"']&c\alttensor c\alttensor c
    \end{tikzcd}
  \end{equation}
  or, as string diagrams:
  \addtocounter{equation}{-1}
  \begin{equation}
    \begin{tikzpicture}[stringd,xscale=-1]
      \draw [wire] (0, 1.2) node[obj,above] {$c$} -- (0, .8);
      \draw [wire] (-.45, -1.2) node[obj,below] {$c$} -- (-.45,-.7) to [in=-90,out=90] (-.15, .2);
      \draw [wire] (.3,-.2) to [in=-90,out=90] node[obj,left] {$c$} (.15, .2);
      \draw[hub] (-.3,.2) rectangle (.3, .8) node[pos=.5] {$\delta$};
      \draw[hub] (0,-.8) rectangle (.6, -.2) node[pos=.5] {$\epsilon$};
    \end{tikzpicture}
    \quad=\quad
    \begin{tikzpicture}[stringd,xscale=-1]
      \draw [wire] (0, 1.2) -- node[obj,left] {$c$} (0, -1.2);
    \end{tikzpicture}\qquad=\quad
    \begin{tikzpicture}[stringd]
      \draw [wire] (0, 1.2) node[obj,above] {$c$} -- (0, .8);
      \draw [wire] (-.45, -1.2) node[obj,below] {$c$} -- (-.45,-.7) to [in=-90,out=90] (-.15, .2);
      \draw [wire] (.3,-.2) to [in=-90,out=90] node[obj,right] {$c$} (.15, .2);
      \draw[hub] (-.3,.2) rectangle (.3, .8) node[pos=.5] {$\delta$};
      \draw[hub] (0,-.8) rectangle (.6, -.2) node[pos=.5] {$\epsilon$};
    \end{tikzpicture}
    \qquad\qquad\qquad\qquad
    \begin{tikzpicture}[stringd,xscale=-1]
      \draw [wire] (0, 1.2) node[obj,above] {$c$} -- (0, .8);
      \draw [wire] (-.6, -1.2) node[obj,below] {$c$} -- (-.6,-.7) to [in=-90,out=90] (-.15, .2);
      \draw [wire] (.3,-.2) to [in=-90,out=90] node[obj,left] {$c$} (.15, .2);
      \draw [wire] (0, -1.2) node[obj,below] {$c$} to [in=-90,out=90] (.15, -.8);
      \draw [wire] (.6, -1.2) node[obj,below] {$c$} to [in=-90,out=90] (.45, -.8);
      \draw[hub] (-.3,.2) rectangle (.3, .8) node[pos=.5] {$\delta$};
      \draw[hub] (0,-.8) rectangle (.6, -.2) node[pos=.5] {$\delta$};
    \end{tikzpicture}
    \quad=\quad
    \begin{tikzpicture}[stringd]
      \draw [wire] (0, 1.2) node[obj,above] {$c$} -- (0, .8);
      \draw [wire] (-.6, -1.2) node[obj,below] {$c$} -- (-.6,-.7) to [in=-90,out=90] (-.15, .2);
      \draw [wire] (.3,-.2) to [in=-90,out=90] node[obj,right] {$c$} (.15, .2);
      \draw [wire] (0, -1.2) node[obj,below] {$c$} to [in=-90,out=90] (.15, -.8);
      \draw [wire] (.6, -1.2) node[obj,below] {$c$} to [in=-90,out=90] (.45, -.8);
      \draw[hub] (-.3,.2) rectangle (.3, .8) node[pos=.5] {$\delta$};
      \draw[hub] (0,-.8) rectangle (.6, -.2) node[pos=.5] {$\delta$};
    \end{tikzpicture}
  \end{equation}
  A \emph{comonoid morphism} is a map $\varphi\colon c\to d$ in
  $\mathscr{V}$ rendering commutative the diagrams:
  \begin{equation}\label{eqn.comonoid_hom}
    \begin{tikzcd}
      c\ar[r, "\varphi"]\ar[d, "\epsilon"']&d\ar[d, "\epsilon"]&&
      c\ar[r, "\varphi"]\ar[d, "\delta"']&d\ar[d, "\delta"]\\
      \yon\ar[r, equal]&\yon&&c\alttensor c\ar[r, "\varphi\alttensor\varphi"']&d\alttensor d
    \end{tikzcd}
  \end{equation}
  or, as string diagrams:
  \addtocounter{equation}{-1}
  \begin{equation}
    \begin{tikzpicture}[stringd]
      \draw [wire] (0, 1.2) node[obj,above] {$c$} -- (0, .8);
      \draw[hub] (-.3,.2) rectangle (.3, .8) node[pos=.5] {$\epsilon$};
    \end{tikzpicture}
    \quad=\quad
    \begin{tikzpicture}[stringd]
      \draw [wire] (0, 1.2) node[obj,above] {$c$} -- (0, .8);
      \draw [wire] (0, -.2) -- node[obj,left] {$d$} (0, .2);
      \draw[hub] (-.3,.2) rectangle (.3, .8) node[pos=.5] {$\varphi$};
      \draw[hub] (-.3,-.8) rectangle (.3, -.2) node[pos=.5] {$\epsilon$};
    \end{tikzpicture}
    \qquad\qquad\qquad\qquad 
    \begin{tikzpicture}[stringd]
      \draw [wire] (0, 1.2) node[obj,above] {$c$} -- (0, .8);
      \draw [wire] (.4,-.2) to [in=-90,out=90] node[obj,right] {$c$} (.15, .2);
      \draw [wire] (-.4,-.2) to [in=-90,out=90] node[obj,left] {$c$} (-.15, .2);
      \draw [wire] (.4, -1.2) node[obj,below] {$d$} -- (.4, -.8);
      \draw [wire] (-.4, -1.2) node[obj,below] {$d$} -- (-.4, -.8);
      \draw[hub] (-.3,.2) rectangle (.3, .8) node[pos=.5] {$\delta$};
      \draw[hub] (.1,-.8) rectangle (.7, -.2) node[pos=.5] {$\varphi$};
      \draw[hub] (-.1,-.8) rectangle (-.7, -.2) node[pos=.5] {$\varphi$};
    \end{tikzpicture}\quad=\quad
    \begin{tikzpicture}[stringd]
      \draw [wire] (0, 1.2) node[obj,above] {$c$} -- (0, .8);
      \draw [wire] (0, -.2) -- node[obj,left] {$d$} (0, .2);
      \draw [wire] (.4, -1.2) node[obj,below] {$d$} to [in=-90,out=90] (.15, -.8);
      \draw [wire] (-.4, -1.2) node[obj,below] {$d$} to [in=-90,out=90] (-.15, -.8);
      \draw[hub] (-.3,.2) rectangle (.3, .8) node[pos=.5] {$\varphi$};
      \draw[hub] (-.3,-.8) rectangle (.3, -.2) node[pos=.5] {$\delta$};
    \end{tikzpicture}
  \end{equation}
\end{definition}

As well as homomorphisms between comonoids, we also have the notion of
\emph{bicomodule}:
\begin{definition}[Bicomodules]\label{def.bicomodule}
  Let $(c, \epsilon_c, \delta_c)$ and $(d,\epsilon_d, \delta_d)$ be
  comonoids in $\mathscr{V}$. A \emph{bicomodule from $c$ to $d$} or
  \emph{$(c,d)$-bicomodule} $m \colon c \altbito d$ consists of a tuple $(m,\lambda,\rho)$,
  where $m\in\mathscr{V}$ and $\lambda,\rho$ are maps
  \[
    d \alttensor m\Fromm{\lambda}m\Too{\rho}m \alttensor c
  \]
  making the following diagrams commute in $\mathscr{V}$:
  \begin{gather}\label{eqn.bimod_left}
    \begin{tikzcd}[ampersand replacement=\&]
      d\alttensor m\ar[d, "\epsilon\alttensor m"']\&
      m\ar[l, "\lambda"']\ar[dl, equal, bend left]\\
      m
    \end{tikzcd}
    \hspace{.6in}
    \begin{tikzcd}[ampersand replacement=\&]
      d\alttensor m\ar[d, "\delta\alttensor m"']\&
      m\ar[l, "\lambda"']\ar[d, "\lambda"]\\
      d\alttensor d\alttensor m\&
      d\alttensor m\ar[l, "d\alttensor\lambda"]
    \end{tikzcd}
    \\\label{eqn.bimod_right}
    \begin{tikzcd}[ampersand replacement=\&]
      m\ar[r, "\rho"]\ar[dr, equal, bend right]\&
      m\alttensor c\ar[d, "m\alttensor\epsilon"]\\\&
      m
    \end{tikzcd}
    \hspace{.6in}
    \begin{tikzcd}[ampersand replacement=\&]
      m\ar[r, "\rho"]\ar[d, "\rho"']\&
      m\alttensor c\ar[d, "m\alttensor\delta"]\\
      m\alttensor c\ar[r, "\rho\alttensor c"']\&
      m\alttensor c\alttensor c
    \end{tikzcd}
    \\\label{eqn.bimod_coherence}
    \begin{tikzcd}[ampersand replacement=\&]
      m\ar[r, "\rho"]\ar[d, "\lambda"']\&
      m\alttensor c\ar[d, "\lambda\alttensor c"]\\
      d\alttensor m\ar[r, "d\alttensor\rho"']\&
      d\alttensor m\alttensor c
    \end{tikzcd}
  \end{gather}
  or, as string diagrams:
  \addtocounter{equation}{-3}
  \begin{equation}
    \qquad\begin{tikzpicture}[stringd,xscale=-1]
      \draw [wire] (0, 1.2) node[obj,above] {$m$} -- (0, .8);
      \draw [wire] (-.45, -1.2) node[obj,below] {$m$} -- (-.45,-.7) to [in=-90,out=90] (-.15, .2);
      \draw [wire] (.3,-.2) to [in=-90,out=90] node[obj,left] {$d$} (.15, .2);
      \draw[hub] (-.3,.2) rectangle (.3, .8) node[pos=.5] {$\lambda$};
      \draw[hub] (0,-.8) rectangle (.6, -.2) node[pos=.5] {$\epsilon$};
    \end{tikzpicture}
    \quad=\quad
    \begin{tikzpicture}[stringd,xscale=-1]
      \draw [wire] (0, 1.2) -- node[obj,left] {$m$} (0, -1.2);
    \end{tikzpicture}
    \qquad\qquad\qquad\qquad
    \begin{tikzpicture}[stringd,xscale=-1]
      \draw [wire] (0, 1.2) node[obj,above] {$m$} -- (0, .8);
      \draw [wire] (-.6, -1.2) node[obj,below] {$\mathstrut m$} -- (-.6,-.7) to [in=-90,out=90] (-.15, .2);
      \draw [wire] (.3,-.2) to [in=-90,out=90] node[obj,left] {$d$} (.15, .2);
      \draw [wire] (0, -1.2) node[obj,below] {$\mathstrut d$} to [in=-90,out=90] (.15, -.8);
      \draw [wire] (.6, -1.2) node[obj,below] {$\mathstrut d$} to [in=-90,out=90] (.45, -.8);
      \draw[hub] (-.3,.2) rectangle (.3, .8) node[pos=.5] {$\lambda$};
      \draw[hub] (0,-.8) rectangle (.6, -.2) node[pos=.5] {$\delta$};
    \end{tikzpicture}
    \quad=\quad
    \begin{tikzpicture}[stringd]
      \draw [wire] (0, 1.2) node[obj,above] {$m$} -- (0, .8);
      \draw [wire] (-.6, -1.2) node[obj,below] {$\mathstrut d$} -- (-.6,-.7) to [in=-90,out=90] (-.15, .2);
      \draw [wire] (.3,-.2) to [in=-90,out=90] node[obj,right] {$m$} (.15, .2);
      \draw [wire] (0, -1.2) node[obj,below] {$\mathstrut d$} to [in=-90,out=90] (.15, -.8);
      \draw [wire] (.6, -1.2) node[obj,below] {$\mathstrut m$} to [in=-90,out=90] (.45, -.8);
      \draw[hub] (-.3,.2) rectangle (.3, .8) node[pos=.5] {$\lambda$};
      \draw[hub] (0,-.8) rectangle (.6, -.2) node[pos=.5] {$\lambda$};
    \end{tikzpicture}
  \end{equation}
  \begin{equation}
    \qquad\begin{tikzpicture}[stringd,xscale=-1]
      \draw [wire] (0, 1.2) -- node[obj,left] {$m$} (0, -1.2);
    \end{tikzpicture}\qquad=\quad
    \begin{tikzpicture}[stringd]
      \draw [wire] (0, 1.2) node[obj,above] {$m$} -- (0, .8);
      \draw [wire] (-.45, -1.2) node[obj,below] {$\mathstrut m$} -- (-.45,-.7) to [in=-90,out=90] (-.15, .2);
      \draw [wire] (.3,-.2) to [in=-90,out=90] node[obj,right] {$c$} (.15, .2);
      \draw[hub] (-.3,.2) rectangle (.3, .8) node[pos=.5] {$\rho$};
      \draw[hub] (0,-.8) rectangle (.6, -.2) node[pos=.5] {$\epsilon$};
    \end{tikzpicture}
    \qquad\qquad\qquad\qquad
    \begin{tikzpicture}[stringd,xscale=-1]
      \draw [wire] (0, 1.2) node[obj,above] {$m$} -- (0, .8);
      \draw [wire] (-.6, -1.2) node[obj,below] {$\mathstrut c$} -- (-.6,-.7) to [in=-90,out=90] (-.15, .2);
      \draw [wire] (.3,-.2) to [in=-90,out=90] node[obj,left] {$m$} (.15, .2);
      \draw [wire] (0, -1.2) node[obj,below] {$\mathstrut c$} to [in=-90,out=90] (.15, -.8);
      \draw [wire] (.6, -1.2) node[obj,below] {$\mathstrut m$} to [in=-90,out=90] (.45, -.8);
      \draw[hub] (-.3,.2) rectangle (.3, .8) node[pos=.5] {$\rho$};
      \draw[hub] (0,-.8) rectangle (.6, -.2) node[pos=.5] {$\rho$};
    \end{tikzpicture}
    \quad=\quad
    \begin{tikzpicture}[stringd]
      \draw [wire] (0, 1.2) node[obj,above] {$m$} -- (0, .8);
      \draw [wire] (-.6, -1.2) node[obj,below] {$\mathstrut m$} -- (-.6,-.7) to [in=-90,out=90] (-.15, .2);
      \draw [wire] (.3,-.2) to [in=-90,out=90] node[obj,right] {$c$} (.15, .2);
      \draw [wire] (0, -1.2) node[obj,below] {$\mathstrut c$} to [in=-90,out=90] (.15, -.8);
      \draw [wire] (.6, -1.2) node[obj,below] {$\mathstrut c$} to [in=-90,out=90] (.45, -.8);
      \draw[hub] (-.3,.2) rectangle (.3, .8) node[pos=.5] {$\rho$};
      \draw[hub] (0,-.8) rectangle (.6, -.2) node[pos=.5] {$\delta$};
    \end{tikzpicture}
  \end{equation}
  \begin{equation}
    \begin{tikzpicture}[stringd,xscale=-1]
      \draw [wire] (0, 1.2) node[obj,above] {$m$} -- (0, .8);
      \draw [wire] (-.6, -1.2) node[obj,below] {$\mathstrut c$} -- (-.6,-.7) to [in=-90,out=90] (-.15, .2);
      \draw [wire] (.3,-.2) to [in=-90,out=90] node[obj,left] {$m$} (.15, .2);
      \draw [wire] (0, -1.2) node[obj,below] {$\mathstrut m$} to [in=-90,out=90] (.15, -.8);
      \draw [wire] (.6, -1.2) node[obj,below] {$\mathstrut d$} to [in=-90,out=90] (.45, -.8);
      \draw[hub] (-.3,.2) rectangle (.3, .8) node[pos=.5] {$\rho$};
      \draw[hub] (0,-.8) rectangle (.6, -.2) node[pos=.5] {$\lambda$};
    \end{tikzpicture}
    \quad=\quad
    \begin{tikzpicture}[stringd]
      \draw [wire] (0, 1.2) node[obj,above] {$m$} -- (0, .8);
      \draw [wire] (-.6, -1.2) node[obj,below] {$\mathstrut d$} -- (-.6,-.7) to [in=-90,out=90] (-.15, .2);
      \draw [wire] (.3,-.2) to [in=-90,out=90] node[obj,right] {$m$} (.15, .2);
      \draw [wire] (0, -1.2) node[obj,below] {$\mathstrut m$} to [in=-90,out=90] (.15, -.8);
      \draw [wire] (.6, -1.2) node[obj,below] {$\mathstrut c$} to [in=-90,out=90] (.45, -.8);
      \draw[hub] (-.3,.2) rectangle (.3, .8) node[pos=.5] {$\lambda$};
      \draw[hub] (0,-.8) rectangle (.6, -.2) node[pos=.5] {$\rho$};
    \end{tikzpicture}
  \end{equation}
  
  Just $d\alttensor m\From{\lambda}m$ satisfying \eqref{eqn.bimod_left}
  is called a \emph{left $d$-comodule}; just $m\To{\rho}m\alttensor c$
  satisfying \eqref{eqn.bimod_right} is called a \emph{right
    $c$-comodule}; and \eqref{eqn.bimod_coherence} is a compatibility
  between these structures called the \emph{bicomodule axiom}. 
  A  \emph{morphism}
  $\varphi \colon (m,\lambda,\rho) \rightarrow (m',\lambda',\rho')$
  between $(c,d)$-bicomodules is a
  a map $\varphi \colon m \rightarrow m'$ in $\mathscr{V}$ 
  making
  the following diagram commute:
  \begin{equation}\label{eqn.bico_maps}
    \begin{tikzcd}
      d\alttensor m\ar[d, "d\alttensor\varphi"']&m\ar[d, "\varphi" description]\ar[l, "\lambda"']\ar[r, "\rho"]&m\alttensor c\ar[d, "\varphi\alttensor c"]\\
      d\alttensor m'&m'\ar[l, "\lambda'"]\ar[r, "\rho'"']&m'\alttensor c
    \end{tikzcd}
  \end{equation}
  or, as string diagrams:
  \addtocounter{equation}{-1}
  \begin{equation}
    \begin{tikzpicture}[stringd]
      \draw [wire] (0, 1.2) node[obj,above] {$m$} -- (0, .8);
      \draw [wire] (.4,-.2) to [in=-90,out=90] node[obj,right] {$m$} (.15, .2);
      \draw [wire] (-.4,-.2) to [in=-90,out=90] node[obj,left] {} (-.15, .2);
      \draw [wire] (.4, -1.2) node[obj,below] {$\mathstrut m'$} -- (.4, -.8);
      \draw [wire] (-.4, -1.2) node[obj,below] {$\mathstrut d$} -- (-.4, -.2);
      \draw[hub] (-.3,.2) rectangle (.3, .8) node[pos=.5] {$\lambda$};
      \draw[hub] (.1,-.8) rectangle (.7, -.2) node[pos=.5] {$\varphi$};
    \end{tikzpicture}\quad=\quad
    \begin{tikzpicture}[stringd]
      \draw [wire] (0, 1.2) node[obj,above] {$m$} -- (0, .8);
      \draw [wire] (0, -.2) -- node[obj,left] {$m'$} (0, .2);
      \draw [wire] (.4, -1.2) node[obj,below] {$\mathstrut m'$} to [in=-90,out=90] (.15, -.8);
      \draw [wire] (-.4, -1.2) node[obj,below] {$\mathstrut d$} to [in=-90,out=90] (-.15, -.8);
      \draw[hub] (-.3,.2) rectangle (.3, .8) node[pos=.5] {$\varphi$};
      \draw[hub] (-.3,-.8) rectangle (.3, -.2) node[pos=.5] {$\lambda'$};
    \end{tikzpicture}
    \qquad\qquad\qquad\qquad 
    \begin{tikzpicture}[stringd]
      \draw [wire] (0, 1.2) node[obj,above] {$m$} -- (0, .8);
      \draw [wire] (.4,-.2) to [in=-90,out=90] node[obj,right] {} (.15, .2);
      \draw [wire] (-.4,-.2) to [in=-90,out=90] node[obj,left] {$m$} (-.15, .2);
      \draw [wire] (.4, -1.2) node[obj,below] {$\mathstrut c$} -- (.4, -.2);
      \draw [wire] (-.4, -1.2) node[obj,below] {$\mathstrut m'$} -- (-.4, -.8);
      \draw[hub] (-.3,.2) rectangle (.3, .8) node[pos=.5] {$\rho$};
      \draw[hub] (-.1,-.8) rectangle (-.7, -.2) node[pos=.5] {$\varphi$};
    \end{tikzpicture}\quad=\quad
    \begin{tikzpicture}[stringd]
      \draw [wire] (0, 1.2) node[obj,above] {$m$} -- (0, .8);
      \draw [wire] (0, -.2) -- node[obj,left] {$m'$} (0, .2);
      \draw [wire] (.4, -1.2) node[obj,below] {$\mathstrut c$} to [in=-90,out=90] (.15, -.8);
      \draw [wire] (-.4, -1.2) node[obj,below] {$\mathstrut m'$} to [in=-90,out=90] (-.15, -.8);
      \draw[hub] (-.3,.2) rectangle (.3, .8) node[pos=.5] {$\varphi$};
      \draw[hub] (-.3,-.8) rectangle (.3, -.2) node[pos=.5] {$\rho'$};
    \end{tikzpicture}
  \end{equation}
  We write $\comod(\mathscr{V})(c,d)$ for the category of
  $(c,d)$-bicomodules in $\mathscr{V}$.
\end{definition}
\begin{remark}
  Note that a bicomodule $m \colon c \altbito d$ involves a \emph{right}
  $c$-action and a \emph{left} $d$-action. This is because we view the
  tensor product as a Leibniz-ordered composition, so that a map
  $\lambda \colon m \rightarrow d \alttensor m$ goes from $m$ to ``$d$
  after $m$'', and likewise $\rho \colon m \rightarrow m \alttensor c$
  involves ``$m$ after $c$''; so, naturally, $m$ goes \emph{from}
  $c$ \emph{to} $d$.
\end{remark}

Suppose now that the tensor product $\alttensor$ of $\mathscr{V}$
preserves equalizers of coreflexive pairs in each variable; this is
true, for example, when $\mathscr{V}$ is vector spaces over a field
but also, and more saliently, when $\mathscr{V} = \poly$
(see~\cref{lemma.preservation_of_equalizers}). In this situation, we
have a well-behaved \emph{composition} of bicomodules:
\begin{definition}[Bicomodule composition]\label{def.bicomodulecomposition}
  Let $\mathscr{V}$ be a monoidal category whose tensor product
  preserves coreflexive equalizers in each variable. The
  \emph{composite} of bicomodules $m \colon c \altbito d$ and
  $n \colon d \altbito e$ is the bicomodule
  $n \alttensor_{d} m \colon c \altbito e$ with carrier given by the
  following coreflexive equalizer in $\mathscr{V}$:
  \begin{equation}\label{def.bicomodule_equalizer}
    n\alttensor_d m\to n\alttensor m\tto n\alttensor d\alttensor m.
  \end{equation}
  Here, the two maps
  $n\alttensor m \rightrightarrows n\alttensor d\alttensor m$ are
  $n\alttensor\lambda_m$ and $\rho_n\alttensor m$ respectively, and the
  common retraction making them into a coreflexive pair is
  $n \alttensor \epsilon_d \alttensor m \colon n \alttensor d \alttensor m \rightarrow n
  \alttensor m$. To obtain the $(c,e)$-bicomodule structure on
  $n\alttensor_d m$, note that tensoring the above coreflexive equalizer
  with $e$ and $c$ yields another equalizer as along the bottom of:
  \begin{equation}\label{def.bicomodule_comp}
    \begin{tikzcd}[column sep=large]
      n\alttensor_d m\ar[r]\ar[d, dashed]&n\alttensor m\ar[r, shift left]\ar[r, shift right]\ar[d, "\lambda\alttensor\rho"]&n\alttensor d\alttensor m\ar[d, "\lambda\alttensor d\alttensor\rho"]\\
      e\alttensor n\alttensor_d m\alttensor c\ar[r]&e\alttensor n\alttensor m\alttensor c\ar[r, shift left]\ar[r, shift right]&e\alttensor n\alttensor d\alttensor m\alttensor c\rlap{ .}
    \end{tikzcd}
  \end{equation}
  The bicomodule structure maps are now obtained from the unique
  dashed map making the left square commute, as displayed above; we
  leave it to the reader to check the bicomodule axioms. With the
  obvious extension of this action to bicomodule morphisms, we obtain
  in this way
  a functor
  \begin{equation*}
    \label{equ.bicomodule_comp}
    \alttensor_d \colon \comod(\mathscr{V})(d,e) \times \comod(\mathscr{V})(c,d) \rightarrow \comod(\mathscr{V})(c,e)\rlap{ .}
  \end{equation*}
\end{definition}

\begin{definition}[Bicategory of bicomodules]\label{def.bicat_bicomod}
  Let $\mathscr{V}$ be a monoidal category whose tensor product
  preserves coreflexive equalizers in each variable. The
  \emph{bicategory of bicomodules} $\comod(\mathscr{V})$ has as
  objects, comonoids in $\mathscr{V}$; the hom-category from $c$ to
  $d$ is $\comod(\mathscr{V})(c,d)$ as
  in~\cref{def.bicomodule}, and composition is given as
  in~\cref{def.bicomodule_comp}. The identity morphism on
  $(c, \epsilon, \delta)$ is $c$ itself with left and right coactions
  both equal to $\delta \colon c \rightarrow c \alttensor c$. The
  associativity constraints
  $(o \alttensor_d n) \alttensor_e m \cong o \alttensor_d (n \alttensor_e m)$ come
  from the fact that both triple composites equalize the parallel pair
\begin{equation*}
  \begin{tikzcd}[column sep=10em]
      o\alttensor n \alttensor m\ar[r, shift left, "o \alttensor \lambda_n \alttensor \lambda_m"]\ar[r, shift right,"\rho_o \alttensor \rho_n \alttensor m"'] &o\alttensor d\alttensor n\alttensor e \alttensor m\rlap{ ,}
    \end{tikzcd}
\end{equation*}
and the unit constraints are found similarly.
\end{definition}    

\begin{remark}[Degenerate bicomodules]\label{rem.degenerate}
  Let $c$ be a comonoid in $\mathscr{V}$, and let $I$ be the monoidal unit of
  $\mathscr{V}$, seen as a comonoid in the unique possible way. Then
  for purely formal reasons:
  \begin{itemize}
  \item a $(c,I)$-bicomodule is just a right $c$-comodule (a map
    $\rho \colon m \rightarrow m \alttensor c$ satisfying \eqref{eqn.bimod_left});
  \item an $(I, d)$-bicomodule is just a left $d$-comodule (a  map
    $\lambda \colon d \alttensor m \leftarrow m$ satisfying \eqref{eqn.bimod_right});
  \item an $(I,I)$-bicomodule is just an object $m$.
  \end{itemize}
  This is because there are unique left and right $I$-comodule
  structures on $m$, each of which is automatically compatible with
  any comodule on the other side.
\end{remark}

In much the same way as we did for $\ccatsharp$, we can enhance
$\comod(\mathscr{V})$ to a framed bicategory; and, like before,
it will be convenient to do this by way of a pro-arrow equipment. In
the below, we write $\comon(\mathscr{V})$ for the category of
comonoids and comonoid homomorphisms in $\mathscr{V}$.

\begin{definition}[Bicomodules from homomorphisms]
  \label{def.bicomod_from_homom}
  Let $c$ and $d$ be comonoids in $\mathscr{V}$, and let
  $\varphi\colon c \to d$ be a comonoid homomorphism. We define the
  bicomodules $\varphi\lpush  \colon c \altbito d$ and
  $\varphi^\ast \colon d \altbito c$ to each have underlying object
  $c \in \mathscr{V}$, and respective coactions
  \begin{equation*}
    d \alttensor c\Fromm{\delta \,\then\, (\varphi \alttensor c)} c \Too{\delta} c \alttensor c \qquad \text{and} \qquad 
    c \alttensor c\Fromm{\delta} c \Too{\delta \,\then\, (c \alttensor \varphi)} c \alttensor d\rlap{ .}
  \end{equation*}
\end{definition}
\begin{proposition}
  \label{prop.comonoid_homomorphism_embedding}
  For each comonoid homomorphism 
  $\varphi\colon c \to d$ in $\mathscr{V}$, we have an adjunction
  \[
    \begin{tikzcd}[column sep=large]
      c\ar[r, tick, shift left=5pt, "\varphi\lpush "]
      \ar[r, phantom, "\scriptstyle\Rightarrow"]&
      d\ar[l, tick, shift left=5pt, "\varphi^\ast"]
    \end{tikzcd}
  \]
  in $\comod(\mathscr{V})$. Moreover, the assignments
  $\varphi \mapsto \varphi\lpush $ and $\varphi \mapsto \varphi^\ast$
  yield pseudofunctors $(\thg)\lpush  \colon \comon(\mathscr{V}) \rightarrow \comod(\mathscr{V})$
  and $(\thg)^\ast \colon \comon(\mathscr{V}) \rightarrow \comod(\mathscr{V})$.
\end{proposition}
\begin{proof}
  Standard and hence omitted.
\end{proof}
\begin{definition}[The framed bicategory of bicomodules]\label{def.bicomod_framed}
  $\ccomod(\mathscr{V})$ is the framed bicategory obtained from the pro-arrow
  equipment
  $(\thg)\lpush  \colon \comon(\mathscr{V}) \rightarrow
  \comod(\mathscr{V})$. Thus, its objects are comonoids in
  $\mathscr{V}$, its vertical morphisms are comonoid homomorphisms,
  and its horizontal morphisms are bicomodules.
\end{definition}

\subsection{Characterising $\ccomod(\mathscr{V})$}

In this section, we explain how to characterize framed bicategories of
the form $\ccomod(\mathscr{V})$; in the next section, we will apply
this characterisation to deduce that $\ccatsharp$ is nothing other
than $\ccomod(\poly)$. In fact, it will be easier for us to
characterize the pro-arrow equipment that generates
$\ccomod(\mathscr{V})$ rather than $\ccomod(\mathscr{V})$ itself, as
then we can exploit results from~\cite{Wood:1985a}.

To motivate the development, note first that by \cref{rem.degenerate},
we can re-find the category $\mathscr{V}$ inside the pro-arrow
equipment of comodules as the hom-category $\comod(\mathscr{V})(I,I)$;
moreover, the monoidal structure of $\mathscr{V}$ is re-found as
composition by looking at the composition of $(I,I)$-bicomodules. This
suggests paying special attention to the object $I$. In fact, $I$ has
a universal characterisation: it is terminal in $\comon(\mathscr{V})$.
Indeed, for any other comonoid $c$, the unique comonoid map to $I$ is
given by the counit map $c \rightarrow I$ of $c$.

This leads us to our first assumption; we will gradually add
assumptions like the below, each of which is satisfied by pro-arrow
equipments of bicomodules, until we have a complete characterisation
result.
\begin{assumption}\label{assn.terminal}
  $(\thg)\lpush  \colon \prov \rightarrow \proh$ is a pro-arrow
  equipment for which $\prov$ has a terminal object $I$.
\end{assumption}

Given \cref{assn.terminal}, we can write $\mathscr{V}$ for the
monoidal category $\proh(I,I)$ under composition. Our objective
is to obtain a map of pro-arrow equipments from
$(\prov \rightarrow \proh)$ into the proarrow equipment of
comonoids and bicomodules in $\mathscr{V}$, and then to show that this
is an equivalence. In order for $\comod(\mathscr{V})$ to even exist,
we must assume that $\mathscr{V} = \proh(I,I)$ has coreflexive
equalizers which are preserved by composition in each variable. In
fact, it is natural to immediately assume something stronger:
\begin{assumption}\label{assn.localeqs}
  $(\thg)\lpush  \colon \prov \rightarrow \proh$ has \emph{local
    coreflexive equalizers}: meaning that each hom-category
  $\proh(a,b)$ has coreflexive equalizers and these are
  preserved by pre- and post-composition by $1$-cells of
  $\proh$.
\end{assumption}
Under these assumptions, we now set about constructing a comparison of
pro-arrow equipments:
\begin{equation}\label{eqn.proarrow_comparison}
  \begin{tikzcd}
    \prov
    \ar[r, "(\thg)\lpush "]
    \ar[d, "\vcomp"'] &
    \proh 
    \ar[d, "\hcomp"] \\
    \comon(\mathscr{V}) 
    \ar[r, "(\thg)\lpush "']&
    \comod(\mathscr{V})\rlap{ ,}
  \end{tikzcd}
\end{equation}
that is, a functor $K_v$ and a strong functor of bicategories $K_h$,
fitting into a square which commutes to within an identity-on-objects
pseudonatural equivalence.

We start by constructing the functor down the left of this square. For
any $a \in \prov$, let us write
\begin{equation}
  \label{eqn.sa_ta_adjunction}
  \begin{tikzcd}[column sep=large]
    a\ar[r, shift left=5pt, "s_a"]
    \ar[r, phantom, "\scriptstyle\Rightarrow"]&
    I\ar[l, shift left=5pt, "t_a"]
  \end{tikzcd}
\end{equation}
for the adjunction in $\proh$ obtained by applying $(\thg)\lpush $
and $(\thg)^\ast$ to the unique $\prov$-map $a \rightarrow I$. This
adjunction yields a comonad $s_a\circ t_a$ on $I$, hence a
comonoid $s_a \circ t_a$ in $\mathscr{V} = \proh(I,I)$.

\begin{lemma}
  The assignment $a \mapsto s_a t_a$ is the
  action on objects of a functor
  $\vcomp \colon \prov \rightarrow \comon(\mathscr{V})$.
\end{lemma}

Again, we simplify notation by treating bicategories as if
they were strict $2$-categories by suppressing associativity and
unitality constraints (as we have been doing for monoidal
categories). This is justified by the coherence theorem for
bicategories.
\begin{proof}
  Let $f \colon a \rightarrow b$ be a morphism in $\prov$. Because
  $f$ commutes with the unique maps to the terminal object, and
  because $(\thg)\lpush $ and $(\thg)^\ast$ are pseudofunctorial, we
  have invertible $2$-cells
  \begin{equation}
    \label{eqn.coherence_for_equipment}
    \begin{tikzcd}
      a\ar[dr, "s_a"' name=A]\ar[rr, "f\lpush ", ""' name=M] & \ar[from=M, to=d, Rightarrow, shorten=2mm, "\cong"] &
      b \ar[dl, "s_b"  name=B] & &
      a\ar[dr, <-, "t_a"']\ar[rr, <-, "f^\ast", ""' name=M'] & \ar[from=M', to=d, Rightarrow, shorten=2mm, "\cong"] &
      b \ar[dl, <-,  "t_b"] \\
      & I & & &&  I
    \end{tikzcd}
  \end{equation}
  These $2$-cells give rise to the first map in the following
  composite in $\mathscr{V} = \proh(I,I)$:
  \begin{equation}
    \label{eqn.homomorphism_from_A}
    \vcomp(f) \coloneqq s_a t_a \xrightarrow{\cong} s_b f\lpush  f^\ast t_b \xrightarrow{s_b \epsilon t_b} s_b t_b
  \end{equation}
  where the second map is the whiskered counit of the adjunction
  $f\lpush  \dashv f^\ast$. It is now easy to check that $\vcomp(f)$ is a
  comonoid homomorphism $s_a t_a \rightarrow s_b t_b$, and that this
  assignment is functorial in $f$.
\end{proof}
We next construct the action on homs of the functor down the
right-hand side of \cref{eqn.proarrow_comparison}. Here, and in what
follows, we write ``$K$'' for the action of either $\vcomp$ or
$\hcomp$ where this does not lead to ambiguity.
\begin{definition}\label{def.proarrow_comparison_homs}
  For any $m \colon a \rightarrow b$ in $\proh$, the bicomodule
  $K(m) \colon K(a) \altbito K(b)$
  has underlying object
  \begin{equation*}
    I \xrightarrow{t_a} a \xrightarrow{m} b \xrightarrow{s_b} I
  \end{equation*}
  in $\mathscr{V} = \proh(I,I)$, and left and right actions
  by $K(b) = s_b t_b$ and $K(a) = s_a t_a$ given by
  \begin{equation}
    \label{eqn.bicomodule_actions_from_B}
    s_b t_b s_b m t_a \Fromm{s_b \eta_b m t_a} s_b m t_a \Too{s_b m \eta_a t_a} s_b m t_a s_a t_a
  \end{equation}
  where $\eta_a$ and $\eta_b$ are the units of the adjunctions
  $s_a \dashv t_a$ and $s_b \dashv t_b$. With the
  obvious extension to morphisms, this yields a functor $K_{a,b}
  \colon\proh(a,b) \rightarrow \comod(\mathscr V)\bigl(K(a),K(b)\bigr)$.
\end{definition}
We now show that the functors of the preceding definition provide the
action on homs of a morphism of bicategories
$\hcomp \colon \proh \rightarrow \comod(\mathscr{V})$. However, there
is a subtlety with regard to preservation of composition.

\begin{lemma}\label{lem.functoriality_proarrow_lemma}
  The functors
  $K_{a,b} \colon \proh(a,b) \rightarrow \comod(\mathscr
  V)\bigl(K(a),K(b)\bigr)$ are the action on homs of a map of
  bicategories $\hcomp \colon \proh \rightarrow \comod(\mathscr{V})$
  which preserves identities to within coherent isomorphism, but
  preserves composition of $1$-cells $m \colon a \rightarrow b$ and
  $n \colon b \rightarrow c$ only up to (not necessarily invertible)
  bicomodule maps
  \begin{equation}
    \label{eqn.bicomod_equipment_oplax_constraint}
    \begin{tikzcd}
      K(a) \ar[dr, "K(m)"']\ar[rr, "K(n\circ m)", ""' name=M] & \ar[from=M, to=d, Rightarrow, shorten=2mm, "\phi_{n,m}"] &
      K(c)\rlap{ .} \ar[dl, <-, "K(n)"] \\
      & K(b)
    \end{tikzcd}
  \end{equation}
  Said another way, $\hcomp \colon \proh \rightarrow \comod(\mathscr{V})$ is
  an \emph{oplax normal functor} between bicategories. The maps
  $\phi_{n,m}$ will be invertible---so that $\hcomp$ is a \emph{strong
    functor} of bicategories---so long as, for each $a \in \prov$,
  the following diagram is an equalizer in $\proh(a,a)$:
  \begin{equation}\label{eqn.bicomod_char_equalizer}
    \begin{tikzcd}
      {\id_a} & {t_as_a} && {t_a s_a t_a s_a\rlap{ .}}
      \arrow["\eta_a", from=1-1, to=1-2]
      \arrow["{\eta_a t_as_a}", shift left=1.5, from=1-2, to=1-4]
      \arrow["{t_as_a\eta_a}"', shift right=1.5, from=1-2, to=1-4]
    \end{tikzcd}
  \end{equation}
\end{lemma}
\begin{proof}
  $\hcomp$ sends the identity homomorphism $a \rightarrow a$ to the
  bicomodule $K(a) \altbito K(a)$ carried by the object $s_a t_a$ and
  coactions
  \begin{equation*}
    s_a t_a s_a t_a \Fromm{s_a \eta_a t_a} s_a t_a \Too{s_a \eta_a t_a} s_a t_a s_a t_a\rlap{ ;}
  \end{equation*}
  but these maps are both the comultiplication of the comonoid $K(a)$,
  so that this is the identity bicomodule on $K(a) = s_a t_a$. Thus
  $\hcomp$ preserves identities; as for binary composition, let
  $m \colon a \rightarrow b$ and $n \colon b \rightarrow c$ be maps of
  $\proh$. The composite of $K(n)$ and $K(m)$ in
  $\comod(\mathscr{V})$ is the tensor product
  $K(n) \alttensor_{K(b)} K(m)$, which is the $(K(a), K(c))$-bicomodule
  found as the equalizer along the top row of:
  \begin{equation}
    \label{eqn.bicomod_equipment_functor_diagram}
    \begin{tikzcd}
      {K(n) \alttensor_{K(b)} K(m)} & {s_c n t_bs_bmt_a} && {s_c n t_bs_bt_bs_bmt_a\rlap{ .}} \\
      {s_c nm t_a}
      \arrow[from=1-1, to=1-2]
      \arrow["{s_cn \cdot \eta_b \cdot t_bs_b mt_a}", shift left=2, from=1-2, to=1-4]
      \arrow["{s_cn t_bs_b \cdot \eta_b \cdot mt_a}"', shift right=2, from=1-2, to=1-4]
      \arrow["{\phi_{n,m}}", dashed, from=2-1, to=1-1]
      \arrow["{s_c n \cdot \eta_b \cdot m t_a}"', from=2-1, to=1-2]
    \end{tikzcd}
  \end{equation}
  On the other hand, $K(nm)$ is the bicomodule on the bottom row, and
  so from the commuting fork of bicomodules along the bottom, we
  induce the bicomodule map $\phi_{n,m}$ required for
  \cref{eqn.bicomod_equipment_oplax_constraint}. It is straightforward
  to check functorial coherence of these data.

  Finally, suppose that \cref{eqn.bicomod_char_equalizer} is an
  equalizer. Since it is an equalizer of a coreflexive pair (with
  common retraction $t_a \epsilon_a s_a$), it is preserved by pre-
  and post-composition in $\proh$. Thus the fork along the
  bottom of \cref{eqn.bicomod_equipment_functor_diagram} is, like the
  fork across the top, an equalizer; whence the comparison map
  $\phi_{n,m}$ is invertible.
\end{proof}
The equalizer condition in \cref{eqn.bicomod_char_equalizer} does in
fact hold in equipments of bicomodules, and so could be adopted as an
assumption now; however, it doubtless looks unmotivated and \emph{ad
  hoc}, and so we postpone doing so until we have given a cleaner
reformulation in \cref{prop.wood} below. For now, let us make good on the promise made surrounding \eqref{eqn.proarrow_comparison}.
\begin{lemma}
  The square \cref{eqn.proarrow_comparison} commutes to
  within an identity-on-objects pseudonatural equivalence.
\end{lemma}
\begin{proof}
  This says that we have coherent isomorphisms of bicomodules
  $(Kf)\lpush  \cong K(f\lpush ) \colon K(a) \altbito K(b)$ for each map
  $f \colon a \rightarrow b$ of $\prov$. On the one hand,
  $(Kf)\lpush $ is the bicomodule induced by the homomorphism
  $Kf \colon s_a t_a \rightarrow s_b t_b$ defined as in
  \cref{eqn.homomorphism_from_A}. Thus, it has underlying object $s_a
  t_a$, and actions
  \begin{equation*}
    s_b t_b s_a t_a \Fromm{s_b \epsilon t_b s_a t_a} s_b f\lpush  f^\ast t_b s_a t_a  \Fromm{\cong} s_a t_a s_a t_a \Fromm{s_a \eta t_a}s_a t_a \Too{s_a \eta t_a} s_a t_a s_a t_a\rlap{ .}
  \end{equation*}
  On the other hand, $K(f\lpush )$ has underlying object
  $s_b f\lpush  t_a$, and actions as in
  \cref{eqn.bicomodule_actions_from_B}, thus:
  \begin{equation*}
    s_bt_bs_bf\lpush  t_a  \Fromm{s_b \eta f\lpush  t_a} s_b f\lpush  t_a \Too{s_b f\lpush  \eta t_a} s_b f\lpush  t_a t_b t_a\rlap{ .}
  \end{equation*}
  Now, the isomorphism to the left of
  \cref{eqn.coherence_for_equipment} yields an isomorphism
  $s_a t_a \rightarrow s_b f\lpush  t_a$ in $\mathscr{V}$, which is
  clearly compatible with the $t_b t_a$-actions displayed to the right
  above, and is compatibile with the $s_b s_a$-actions by a short
  calculation. Finally, the isomorphisms
  $(Kf)\lpush  \cong K(f\lpush ) \colon K(a) \altbito K(b)$ are functorial
  in $f$ precisely because $(\thg)\lpush $ is a pseudofunctor.
\end{proof}

We have thus constructed a map of pro-arrow equipments from
$\prov \rightarrow \proh$ to
$\comon(\mathscr{V}) \rightarrow \comod(\mathscr{V})$. We now address
the conditions under which this map will be an equivalence.  The key
notion we require is that of co-Eilenberg--Moore and co-Kleisli objects.

\begin{definition}[Left/right/bi-comodule in a bicategory]\label{def.comod_2}
  Let $\mathscr{K}$ be a bicategory, let $a,b \in \mathscr{K}$ and let
  $c$ be a comonad on $b$. A \emph{left $c$-comodule with domain $a$}
  in $\mathscr{K}$ comprises a morphism $m \colon a \rightarrow b$ and
  a $2$-cell $\lambda \colon m \Rightarrow c m$ satisfying the same
  axioms as in~\eqref{eqn.bimod_left}. (Again, these axioms as stated
  treat $\mathscr{K}$ as if it were strict for convenience, but
  corresponding weakened forms can be used if desired.) These
  comodules are the objects of the category $\Cat{LComod}(a,c)$, whose
  morphisms $(m,\lambda) \rightarrow (m', \lambda')$ are $2$-cells
  $\alpha \colon m \Rightarrow m'$ such that
  $\lambda' \circ \alpha = c \alpha \circ \lambda$.

  Similarly, for $d$ a comonad on $b$, have the dual notion of a
  \emph{right $d$-comodule with codomain $b$} and the corresponding
  category $\Cat{RComod}(c,b)$.

  As before, \emph{$(c, d)$-bicomodule} comprises a right $c$-module
  and a left $d$-module satisfying the bicomodule axiom as
  in~\eqref{eqn.bimod_coherence}. Also note that, under appropriate
  assumptions (i.e., preservation of coreflexive equalizers under
  horizontal composition), these assemble into a (framed) bicategory
  of bicomodules $\ccomod(\mathscr{K})$, just as
  in~\cref{def.bicat_bicomod}.
\end{definition}

\begin{definition}[Co-Eilenberg--Moore and co-Kleisli objects]
  Each morphism $h \colon a \rightarrow a'$ of $\mathscr{K}$ induces a
  functor
  $h^\ast \colon \Cat{LComod}(a',c) \rightarrow \Cat{LComod}(a,c)$ by
  precomposition with $h$; similarly, each $2$-cell
  $\alpha \colon h \Rightarrow h'$ of $\mathscr{K}$ induces a natural
  transformation $\alpha^\ast \colon h^\ast \Rightarrow (h')^\ast$. An
  \emph{co-Eilenberg--Moore object} for the comonad $c$ is a universal
  left $c$-comodule
  $(u^c \colon b^c \rightarrow b, \, \lambda \colon u^c \Rightarrow c
  u^c)$, where universality says that, for any object
  $a \in \mathscr{K}$, the functor
  \begin{align*}
    \mathscr{K}(a,b^c) & \rightarrow \Cat{LComod}(a, c) &
    h & \mapsto h^\ast(u^c, \lambda^c)
  \end{align*}
  is an equivalence of categories.
  Dually, we have the notion of a \emph{right $c$-comodule with
    codomain $a$}, involving a morphism $m \colon b \rightarrow a$ and
  $2$-cell $\rho \colon m \rightarrow mc$; and a \emph{co-Kleisli
    object} for the comonad $c$ is a universal right $c$-comodule
  $(g_c \colon b \rightarrow b_c,\, \rho \colon b_c \Rightarrow b_c
  c)$.
\end{definition}

\begin{example}[Co-Eilenberg--Moore and co-Kleisli objects in $\smcat$]
  In the $2$-category $\smcat$, the co-Eilenberg--Moore object of a
  comonad $C \colon x \rightarrow x$ on a category $x$ is given by the
  category $x^C$ of Eilenberg--Moore coalgebras equipped with its
  forgetful functor $U^C \colon x^C \rightarrow x$. The universal left
  $C$-comodule $\lambda \colon U^C \Rightarrow CU^C$ is the natural
  transformation whose component at a coalgebra
  $(a \in x, \alpha \colon a \rightarrow Ca)$ is given by $\alpha$.

  On the other hand, the co-Kleisli object of $C$ is the co-Kleisli
  category $x_C$, whose objects are those of $x$, and whose maps $a
  \tickar b$ are maps $Ca \rightarrow b$ in $x$, and with composition
  defined using the comonad structure of $C$. The functor $G_C \colon
  x \rightarrow x_C$ is the identity on objects, and sends $f \colon a
  \rightarrow b$ to $f \circ \epsilon_a \colon a \tickar b$; and the
  universal right $C$-comodule $\rho \colon G_C \rightarrow G_C C$ has
  component at $a \in x$ given by $\id_{Ca} \colon a \tickar Ca$.
\end{example}

\begin{example}[Co-Eilenberg--Moore and co-Kleisli objects in $\comod(\mathscr{V})$]\label{ex.cokleisli}
  In the bicategory $\comod(\mathscr{V})$, an object is a comonoid $d$
  in $\mathscr{V}$. A comonad $c$ on that object is a $(d,d)$-bicomodule
  $c$ equipped with bicomodule maps $\iota \colon c \rightarrow d$ and
  $\upsilon \colon c \rightarrow c \alttensor_d c$ satisfying the usual
  axioms. By dualizing the usual argument for algebras, we see that
  these data are equivalent to giving a comonoid $c$ together with a
  comonoid homomorphism $\iota \colon c \rightarrow d$.
  From these data, we obtain the adjunction:
  \begin{equation*}
    \begin{tikzcd}[column sep=large]
      c\ar[r, shift left=5pt, "\iota\lpush "]
      \ar[r, phantom, "\scriptstyle\Rightarrow"]&
      d\ar[l, shift left=5pt, "\iota^\ast"]
    \end{tikzcd}
  \end{equation*}
  in $\comod(\mathscr{V})$ and it is not hard to check that the
  comonad $\iota\lpush  \iota^\ast$ on $d$ generated by this adjunction 
  is the comonad $c$ we started with. It is now the case that:
  \begin{itemize}
  \item The $1$-cell $\iota\lpush  \colon c \rightarrow d$ and $2$-cell $\smash{\iota\lpush  \xRightarrow{\eta \iota\lpush }
    \iota\lpush  \iota^\ast \iota\lpush  = c \iota\lpush }$ form a co-Eilenberg--Moore object for~$c$;
  \item The $1$-cell $\iota^\ast \colon d \rightarrow c$ and $2$-cell
    $\smash{\iota^\ast \xRightarrow{\iota^\ast \eta}
    \iota^\ast \iota\lpush  \iota^\ast = \iota^\ast c}$ form
    a co-Kleisli object for $c$.
  \end{itemize}
\end{example}

Thus, in $\comod(\mathscr{V})$, ``Eilenberg--Moore and Kleisli objects
coincide''. This is analogous to the fact that, in the category of
abelian groups, finite products and finite coproducts coincide. In
fact, this latter coincidence holds not just in abelian groups, but in
any category whose hom-sets have commutative monoid structures
preserved by composition in each variable. Similarly, the coincidence
of co-Eilenberg--Moore and co-Kleisli objects holds in any bicategory
with local reflexive coequalizers:

\begin{proposition}\label{prop.wood}
  Let $\mathscr{K}$ be a bicategory with local reflexive coequalizers.
  Consider an adjunction
  \begin{equation}
    \label{eqn.bicat_adjunction}
    \begin{tikzcd}[column sep=large]
      y\ar[r, shift left=5pt, "s"]
      \ar[r, phantom, "\scriptstyle\Rightarrow"]&
      x\ar[l, shift left=5pt, "t"]
    \end{tikzcd}
  \end{equation}
  in $\mathscr{K}$ with unit $\eta$ and counit $\epsilon$, and let $c$
  be the comonad on $x$ it generates. The following are equivalent:
  \begin{enumerate}[(1)]
  \item $s \eta \colon s \Rightarrow sts$ exhibits
    $s$ as an Eilenberg--Moore object for $c$ (we may say that the
    adjunction is \emph{comonadic});
  \item $\eta t \colon t \Rightarrow tst$ exhibits $t$ as
    a Kleisli object for $c$ (we may say that the adjunction is of
    \emph{co-Kleisli type});
  \item The following diagram is an equalizer in $\mathscr{K}(x,x)$:
    \begin{equation*}
      \begin{tikzcd}
        {\id_a} & {ts} && {tsts\rlap{ .}}
        \arrow["\eta", from=1-1, to=1-2]
        \arrow["{\eta ts}", shift left=1.5, from=1-2, to=1-4]
        \arrow["{ts\eta}"', shift right=1.5, from=1-2, to=1-4]
      \end{tikzcd}
    \end{equation*}
  \end{enumerate}
\end{proposition}
\begin{proof}
  This is \cite[Corollary 23]{Wood:1985a}; we do not reproduce the
  proof here as we will not need the details. 
\end{proof}
\begin{remark}
  The analogy we drew with the fact that, in a category enriched over
  commutative monoids, finite products and finite coproducts coincide,
  is in fact more than just analogy; see
  \cite[Section~15]{Garner.Shulman:2013a}.
\end{remark}

Note that the equalizer diagram in condition (3) of the preceding
proposition is the same diagram as in
\cref{eqn.bicomod_char_equalizer}. Thus, the assumption that these
diagrams are indeed equalizers can be equivalently stated as:
\begin{assumption}\label{assn.eilenberg_moore}
  In the pro-arrow equipment
  $(\thg)\lpush  \colon \prov \rightarrow \proh$, the adjunction
  in $\proh$ induced by each $a \in \prov$, as displayed
  below, is comonadic (and so also of co-Kleisli type):
  \begin{equation}
    \label{eqn.eilenberg_moore}
  \begin{tikzcd}[column sep=large]
    a\ar[r, shift left=5pt, "s_a"]
    \ar[r, phantom, "\scriptstyle\Rightarrow"]&
    I\rlap{ .}\ar[l, shift left=5pt, "t_a"]
  \end{tikzcd}
\end{equation}
\end{assumption}
This assumption not only ensures that $\hcomp \colon \proh
\rightarrow \comod(\mathscr{V})$ is strong, but also goes some of the
way towards verifying it is a biequivalence of bicategories:
\begin{proposition}
  If \crefrange{assn.terminal}{assn.eilenberg_moore} hold, then the
  comparison morphism
  $\hcomp \colon \proh \rightarrow \comod(\mathscr{V})$ of
  \cref{eqn.proarrow_comparison} is strong, and also bi-fully faithful,
  i.e., an equivalence on hom-categories.
\end{proposition}
\begin{proof}
  $\hcomp$ is strong by \cref{lem.functoriality_proarrow_lemma} and
  \cref{prop.wood}. For the second claim, consider objects
  $a,b \in \proh$ and the hom-category $\proh(a,b)$. Since
  the adjunction $s_b \dashv t_b$ is comonadic, we have an equivalence
  of categories
  \begin{equation}
    \label{eqn.bifully_faithful_proof_1}
    \begin{aligned}
      \proh(a,b) & \rightarrow \Cat{LComod}(a, K(b)) & \qquad m &
      \mapsto m^\ast(s_b, s_b \eta_b) = (s_b m, s_b \eta_b m)\rlap{ .}
    \end{aligned}
  \end{equation}
  Now, an object of $\Cat{LComod}(a, K(b))$ involves a $1$-cell
  $n \colon a \rightarrow I$ of $\proh$ together with a $2$-cell
  $\lambda \colon n \Rightarrow s_b t_b n$, plus left comodule axioms
  for $\lambda$. Note that all these data live in the hom-category
  $\proh(a, I)$; and because the adjunction $s_a \dashv t_a$ is
  of co-Kleisli type, there is an equivalence of categories
  \begin{equation}
    \label{eqn.bifully_faithful_proof_2}
    \begin{aligned}
      \proh(a,I) & \rightarrow \Cat{RComod}(K(a), I) & \qquad n &
      \mapsto n\lpush (t_a, \eta_a t_a) = (n t_a, n \eta_a t_a)\rlap{ .}
    \end{aligned}
  \end{equation}
  When we transport the data for an object of $\Cat{LComod}(a,
  K(b))$ under this equivalence, it becomes:
  \begin{itemize}
  \item An object of $\Cat{RComod}(K(a), I)$, so a map
    $p \colon I \rightarrow I$ and $2$-cell
    $\rho \colon p \Rightarrow p s_a t_a$ plus right comodule~axioms;
  \item A map $p \rightarrow s_b t_b p$ of $\Cat{RComod}(K(a), I)$,
    so a $2$-cell $\lambda \colon p \Rightarrow s_b t_b p$ which is
    a map of right comodules;
  \item Plus left comodule axioms for $\lambda$.
  \end{itemize}
  The condition that $\lambda$ is a map of right comodules is exactly
  the \emph{bimodule compatibility axiom} of
  \cref{eqn.bimod_coherence}; and so we have shown that the
  equivalence \cref{eqn.bifully_faithful_proof_2} lifts to an
  equivalence
  \begin{equation*}
    \begin{aligned}
      \Cat{LComod}(a, K(b)) & \rightarrow
      \comod(\mathscr{V})(K(a), K(b)) & (n,\lambda)  &
      \mapsto (nt_a, \lambda t_a, n\eta_a t_a)\rlap{ .}
    \end{aligned}
  \end{equation*}
  Composing this with the equivalence
  of~\cref{eqn.bifully_faithful_proof_1} yields an equivalence
  $\proh(a,b) \rightarrow \comod(\mathscr{V})(K(a), K(b))$ which, by
  inspection, is precisely the action on homs of
  $\hcomp \colon \proh \rightarrow \comod(\mathscr{V})$ as given in
  \cref{def.proarrow_comparison_homs}.
\end{proof}
Although \cref{assn.eilenberg_moore} ensures that
$\hcomp \colon \proh \rightarrow \comod(\mathscr{V})$ is bi-fully
faithful, it does not imply that
$\vcomp \colon \prov \rightarrow \comon(\mathscr{V})$ is fully
faithful. This is because the universal property of a co-Kleisli or
co-Eilenberg--Moore object in $\proh$ does not provide a way to talk
about maps in $\prov$. Thus, we must assume something more.

\begin{assumption}\label{assn.tightness}
  Whenever we have a pseudo-commuting triangle in $\proh$ of the form:
  \begin{equation}\label{eqn.tightness}
    \begin{tikzcd}
      a \arrow[rr, "m"]
      \arrow[dr, "s_a"', ""{name=0, anchor=center, inner sep=0}] &&
      b
      \arrow[dl, "s_b", ""{name=1, anchor=center, inner sep=0}] \\
      & I
      \arrow["\theta", "\raisebox{2pt}{$\cong$}"', shift right, shorten <=10pt, shorten >=10pt, Rightarrow, from=1, to=0]
    \end{tikzcd}
  \end{equation}
  there is a unique map $f \colon a \rightarrow b$ in $\prov$ and
  invertible $2$-cell $\varphi \colon f\lpush  \Rightarrow m$ in $\proh$ such
  that the functoriality constraint of 
  $(\thg)\lpush $ at the pair of composable maps $f \colon a \rightarrow b$ and $! \colon b
  \rightarrow I$ is equal to the composite $2$-cell:
  \setlength{\abovedisplayskip}{6pt} 
  \begin{equation*}
     s_b f\lpush  \xrightarrow{s_b \varphi} s_b m \xrightarrow{\theta} s_a\rlap{ .}
  \end{equation*}
\end{assumption}
\begin{example}
  \cref{assn.tightness} holds in the pro-arrow equipment of
  bicomodules. Indeed, a diagram of the form~\cref{eqn.tightness} in
  $\comod(\mathscr{V})$ amounts to an $(a,b)$-bicomodule
  $(m, \lambda, \rho)$ together with an isomorphism of right
  $a$-comodules
  \begin{equation*}
    \theta \colon (m \xrightarrow{\rho} ma) \rightarrow (a \xrightarrow{\delta} aa)\rlap{ .}
  \end{equation*}
  Now, if  $f \colon a \rightarrow b$ is a map of comonoids, then an
  invertible $2$-cell $\varphi \colon f\lpush  \Rightarrow m$ is a  map of bicomodules
  \begin{equation*}
    \varphi \colon (ba \xleftarrow{fa} aa \xleftarrow{\delta} a \xrightarrow{\delta} aa) \rightarrow
    (bm \xleftarrow{\lambda} m \xrightarrow{\rho} ma)
  \end{equation*}
  and the required compatibility condition is the requirement that
  $\theta$ is inverse to $\varphi$ qua right $a$-comodule map. Thus,
  the result will be established if we can show that every
  $(a,b)$-bicomodule of the form $(a, \lambda, \delta)$ must have
  $\lambda = fa \circ \delta \colon a \rightarrow aa \rightarrow ba$
  for a unique comonoid homomorphism $f \colon a \rightarrow b$. This
  is a form of the Yoneda lemma; indeed, given $\lambda$, we re-find
  the unique $f$ as the composite
  $b\epsilon \circ \lambda \colon a \rightarrow ba \rightarrow b$.
\end{example}

\begin{proposition}
  If \crefrange{assn.terminal}{assn.tightness} hold, then the
  functor
  $\vcomp \colon \prov \rightarrow \comod(\mathscr{V})$ is fully
  faithful.
\end{proposition}
\begin{proof}
  For any bicategory $\mathscr{K}$ and any object $X \in \mathscr{K}$,
  we can form the \emph{slice bicategory} $\mathscr{K}/X$, whose
  objects are arrows into $X$, whose morphisms are pseudo-commuting
  triangles, and whose $2$-cells are defined in the evident way. Any
  strong functor of bicategories
  $F \colon \mathscr{K} \rightarrow \mathscr{L}$ induces a strong
  functor $\tilde F \colon \mathscr{K}/X \rightarrow \mathscr{L}/FX$, and if
  $F$ is bi-fully faithful, then so is $\tilde F$. Applying this to the
  square~\cref{eqn.coherence_for_equipment} and the terminal object $I
  \in \prov$, we obtain a square
  \begin{equation*}
    \begin{tikzcd}
      \llap{$\prov$} = \prov / I
      \ar[r, "\widetilde{(\thg)}\lpush "]
      \ar[d, "\vcomp"'] &
      \proh / I
      \ar[d, "\widetilde{\hcomp}"] \\
      \llap{$\comon(\mathscr{V})$} = \comon(\mathscr{V}) / I
      \ar[r, "\widetilde{(\thg)}\lpush "']&
      \comod(\mathscr{V}) / I \rlap{ .}
    \end{tikzcd}
  \end{equation*}
  which commutes to within an invertible pseudonatural transformation.
  The right-hand vertical map is bi-fully faithful since $\hcomp$ is
  so. On the other hand, \cref{assn.tightness} asserts precisely that
  the horizontal arrows are bi-fully faithful; whence, by the
  cancellativity properties of bi-fully faithful functors, $\vcomp$
  down the left is also bi-fully faithful. Since its domain and
  codomain are $1$-categories, it is thus fully faithful.
\end{proof}

To show that~\cref{eqn.proarrow_comparison} is an
equivalence of pro-arrow equipments, all that remains is to show that
$\vcomp$ is surjective on objects up to isomorphism, and that $\hcomp$ is
surjective on objects up to equivalence. In fact, the former implies
the latter, and is in turn implied by:
\begin{assumption}\label{assn.cauchy_dense}
  In the pro-arrow equipment
  $(\thg)\lpush  \colon \prov \rightarrow \proh$, there is for
  each comonad $c$ on $I \in \proh$ an object
  $\hat c$ and $2$-cell
  $\lambda \colon s_{\hat c} \Rightarrow cs_{\hat c}$ in $\proh$ exhibiting
  $s_{\hat c} \colon \hat c \rightarrow I$ as a co-Eilenberg--Moore object
  for $c$.
\end{assumption}

\begin{proposition}
  If \crefrange{assn.terminal}{assn.cauchy_dense} hold,
  then~\cref{eqn.proarrow_comparison} is an equivalence of pro-arrow equipments.
\end{proposition}
\begin{proof}
  By the above discussion, it remains only to verify that
  $\vcomp \colon \prov \rightarrow \comon(\mathscr{V})$ is surjective
  on objects up to isomorphism. So let $c$ be an object of
  $\comon(\mathscr{V})$, i.e., a comonad on $I$ in $\mathscr{B}$. By
  \cref{assn.cauchy_dense} we have an object $\hat c$ and a universal
  left $c$-comodule structure
  $\lambda \colon s_{\hat c} \Rightarrow cs_{\hat{c}}$ on $\hat{c}$.
  We will exhibit an isomorphism of comonoids between $K(\hat{c})$ and
  $c$.

  To this end, consider the left $c$-comodule with domain $I$ given by
  $(c \colon I \rightarrow I, \delta \colon c \Rightarrow cc)$. By the
  universality of $(s_{\hat c}, \lambda)$, there is a $1$-cell
  $r \colon I \rightarrow \hat c$ fitting into a pseudo-commuting
  triangle as to the left in:
  \begin{equation*}
    \begin{tikzcd}
      I \arrow[rr, "r"]
      \arrow[dr, "c"', ""{name=0, anchor=center, inner sep=0}] &&
      \hat c
      \arrow[dl, "s_{\hat c}", ""{name=1, anchor=center, inner sep=0}] \\
      & I
      \arrow["\theta"', "\raisebox{2pt}{$\cong$}", shift left, shorten <=13pt, shorten >=13pt, Rightarrow, from=0, to=1]
    \end{tikzcd}
    \qquad \qquad 
    \begin{tikzcd}
      c \arrow[r, "\delta"] 
      \arrow[d, "\theta"'] &
      cc
      \arrow[d, "c\theta"] \\
      s_{\hat c} r
      \arrow[r, "\lambda r"'] &
      cs_{\hat c} r 
    \end{tikzcd}
  \end{equation*}
  which renders commutative the square of
  $2$-cells to the right. We now have a map of left
  $c$-comodules with domain $\hat c$ given by
  \begin{equation*}
    (s_{\hat c}, \lambda) \xrightarrow{\lambda} (c s_{\hat c}, \delta s_{\hat c}) \xrightarrow{\theta s_{\hat c}} (s_{\hat c} r s_{\hat c}, \lambda r s_{\hat c})
  \end{equation*}
  and hence again by universality, a unique $2$-cell
  $\alpha \colon \id_{\hat c} \Rightarrow r s_{\hat c}$ for which
  $s_{\hat c} \alpha$ is the above-displayed composite. We also have
  the $2$-cell
  $\beta = \epsilon \circ \theta^{-1} \colon s_{\hat c} r \rightarrow
  c \rightarrow \id_I$ and a short calculation shows that
  $\alpha$ and $\beta$ satisfy the triangle identities exhibiting
  $s_{\hat c} \dashv r$, and that $\theta$ then becomes an isomorphism
  between the comonoid $c$ and the comonoid $s_{\hat c} r$ induced by
  the adjunction $s_{\hat c} \dashv r$; see, for example,
  \cite[Theorem~2]{Street:1972a}. However, since we also know that
  $s_{\hat c} \dashv t_{\hat c}$, we also have an isomorphism between
  this latter comonoid and the comonoid
  $K(\hat c) = s_{\hat c}t_{\hat c}$, as desired.
\end{proof}
This concludes our characterisation of pro-arrow equipments of the
form $\comon(\mathscr{V}) \rightarrow \comod(\mathscr{V})$, and hence
our characterisation of framed bicategories of the form
$\ccomod(\mathscr{V})$.

\subsection{$\ccatsharp$ is $\ccomod(\poly)$}
We are now in a position to show that $\ccatsharp$ is equivalent to
the framed bicategory $\ccomod(\poly)$ of comonoids, comonoid
homomorphisms and bicomodules in $\poly$. We proceed to verify each of
the assumptions of the preceding section for the pro-arrow equipment
$(\thg)\lpush  \colon \catsharp \rightarrow \polyfunb$. To start with, it
is immediate that:

\begin{lemma}
  $\catsharp$ has a terminal object given by the one-object category
  $1$, and $\polyfunb(1,1) = \poly$.
\end{lemma}

Thus, verifying the remaining four assumptions for the pro-arrow
equipment $\catsharp \rightarrow \polyfunb$ will show it to be
equivalent to $\comon(\poly) \rightarrow \comod(\poly)$, as desired.
\cref{assn.localeqs} is precisely
\cref{lemma.multivar_preservation_of_equalizers}, so we move
immediately on to verifying
\cref{assn.eilenberg_moore}.

\begin{proposition}
  For any small category $a$, the following adjunction
  in $\polyfunb$ is comonadic:
  \begin{equation*}
    \begin{tikzcd}[column sep=large]
      a\ar[r, shift left=5pt, "s_a"]
      \ar[r, phantom, "\scriptstyle\Leftarrow"]&
      1\rlap{ .}\ar[l, shift left=5pt, "t_a"]
    \end{tikzcd}
  \end{equation*}
\end{proposition}
\begin{proof}
  $s_a$ is the polynomial induced by the unique retrofunctor
  $a \coto 1$, whose span representation
  \cref{eqn.retrofunctor_decomp} is:
  \begin{equation*}
    \begin{tikzcd}
    &\mathrm{ob}(\cat{a})\ar[dl, "J"']\ar[dr, "!"]\\[-10pt]
    \cat{a}&&\cat{1}\rlap{ .}
  \end{tikzcd}
\end{equation*}
Here $\mathrm{ob}(\cat{a})$ is the set of objects of $a$, seen as a
discrete category, $J$ is the inclusion functor, and $!$ is the unique
functor. Thus, $s_a$ is the composite functor
\begin{equation}\label{eqn.sa_in_polynfun}
  s_a \coloneqq a\set \xrightarrow{\Delta_J} \mathrm{ob}(\cat{a})\set \xrightarrow{\Sigma} \smset
\end{equation}
whose first component is restriction and whose second component is the
coproduct functor. Now $\Delta_J$ creates limits since limits in
$a\set$ are pointwise; while $\Sigma$ creates connected limits because
it is equivalent to the slice category projection
$\smset / \mathrm{ob}(a) \rightarrow \smset$. It follows that the
diagram \cref{eqn.bicomod_char_equalizer} is an equalizer in
$\polyfunb(a,a)$ since its composite with the connected-limit creating
$s_a$ is a split equalizer in $\polyfunb(a,1)$. Thus, by
\cref{prop.wood}, each adjunction as in \cref{eqn.eilenberg_moore} is
comonadic.
\end{proof}
This recaptures one half of Ahman and Uustalu's result, as it shows
that every small category $a$ gives rise to a polynomial comonad on
$\smset$. The other half is showing that every polynomial comonad on
$\smset$ arises in this way, which amounts to verifying
\cref{assn.cauchy_dense}:
\begin{proposition}
  Let $c \colon \smset \rightarrow \smset$ be a polynomial comonad,
  and $U \colon \mathscr{A} \rightarrow \smset$ its category of
  Eilenberg--Moore coalgebras. There exists a small category $\hat c$ for
  which $\mathscr{A} \simeq \hat c\set$ via an equivalence which identifies $U$ with
  the functor~\cref{eqn.sa_in_polynfun}. Consequently, the proarrow
  equipment $\catsharp \rightarrow \polyfunb$ satisfies
  \cref{assn.cauchy_dense}.
\end{proposition}
\begin{proof}
  The category $\mathscr{A}$ has as terminal object the cofree
  coalgebra $c(1) \xrightarrow{\delta_1} cc(1)$, and so by slicing the
  forgetful functor $U$ we get a functor
  \begin{equation*}
    V \coloneqq \mathscr{A} \xrightarrow{\cong} \mathscr{A} / (c1, \delta_1) \xrightarrow{U/(c1, \delta_1)} \smset / c(1) \xrightarrow{\cong}c(1)\set\rlap{ .}
  \end{equation*}
  Since $U$ preserves connected limits, so does $V$; it also preserves
  the terminal object, and so is limit-preserving. Moreover,
  by~\cite[Proposition~3.1]{johnstone2001structure}, $V$ has a right
  adjoint and is comonadic. So $\mathscr{A}$ is equivalent to the category of
  coalgebras for a limit-preserving comonad $Q$ on $c(1)\set$. Since
  the functor part of this comonad preserves limits, it has a left
  adjoint $T$ and now the comonad structure of $Q$ transports to a
  monad structure of $T$ in such a way that the categories of
  $T$-algebras and $Q$-coalgebras are isomorphic. But
  by~\cite[Satz~10.11]{Gabriel1971Lokal}, the category of algebras for
  a cocontinuous monad on $c(1)\set$ must be of the form
  $\Delta_J \colon \hat c\set \rightarrow c(1)\set$ for some small category $\hat c$
  with object-set $c(1)$, and with $J \colon c(1) \rightarrow \hat c$ the
  inclusion-of-objects functor.

  Putting this together, we see that $\mathscr{A} \simeq \hat c\set$
  via an equivalence which identifies
  $V \colon \mathscr{A} \rightarrow c(1)\set$ with
  $\Delta_J \colon \hat{c}\set \rightarrow c(1)\set$. Postcomposing with
  $\Sigma \colon c(1)\set \rightarrow \smset$, we conclude that
  $\mathscr{A} \simeq a\set$ via an equivalence which equates $U$
  with~\cref{eqn.sa_in_polynfun}, as desired.

  It follows that $\catsharp \rightarrow \polyfunb$ satisfies
  \cref{assn.cauchy_dense}: for indeed, since
  $s_{\hat c} \colon \hat c\set \rightarrow \set$ is comonadic in
  $\smcat$, the diagram~\cref{eqn.bicomod_char_equalizer} is an
  equalizer in $\smcat(\hat c\set, \hat c\set)$ and hence also in
  $\polyfunb(\hat c\set, \hat c\set)$; whence this adjunction is
  comonadic in $\polyfunb$ as well, as required.
\end{proof}

All that remains to complete our characterisation is to verify
\cref{assn.tightness}. This seems to be a new result,
for which we have found no better approach than verification by hand:

\begin{proposition}
  Suppose we are given a pseudo-commuting diagram of polynomial functors:
  \begin{equation*}
    \begin{tikzcd}
      c\set \arrow[rr, "p"]
      \arrow[dr, "s_c"', ""{name=0, anchor=center, inner sep=0}] &&
      d\set
      \arrow[dl, "s_d", ""{name=1, anchor=center, inner sep=0}] \\
      & \smset
      \arrow["\theta", "\raisebox{2pt}{$\cong$}", shift right, shorten <=13pt, shorten >=13pt, Rightarrow, from=1, to=0]
    \end{tikzcd}
  \end{equation*}
  where $s_c$ and $s_d$ are as in~\cref{eqn.sa_in_polynfun}. There is
  a unique retrofunctor $\varphi \colon c \coto d$ and isomorphism of
  polynomial functors $\tau \colon \varphi\lpush 
  \cong p$ such that the functoriality constraint $\mu$ of $(\thg)\lpush $ at the composable retrofunctors $\varphi \colon c \coto d$ and $! \colon d
  \coto 1$ is given by:
  \setlength{\abovedisplayskip}{6pt} 
  \begin{equation}
    \label{eqn.cofunctor_representation_diag_1}
    \mu = s_d \varphi\lpush  \xrightarrow{s_d \tau} s_d p \xrightarrow{\theta} s_c\rlap{ .}
  \end{equation}
\end{proposition}
\begin{proof}
  The functor $p \colon c\set \to d\set$ can be written in the form:
  \begin{equation*}
    p(X)(b) = \sum_{i \in p_b(1)} c\set(p_b[i], X)
  \end{equation*}
  for some $p(1) \in d\set$ and $p[\thg] \colon \el_d(p(1))\op
  \rightarrow c\set$.
  Given that functors of the form $s_c$ simply take the elements of a
  presheaf, the functors $s_c \colon c\set \rightarrow \smset$ and
  $s_d \circ p \colon c\set \rightarrow \smset$ are thus given by
  \begin{equation*}
    s_d \circ p \colon X \mapsto \sum_{b \in d, i \in p_b(1)} c\set(p_b[i], X) \qquad \qquad \text{and} \qquad \qquad 
    s_c \colon X \mapsto \sum_{a \in c} X(a)\rlap{ .}
  \end{equation*}
  So the data for an isomorphism $\theta \colon s_d \circ p \cong
  s_c$ involves:
  \begin{enumerate}[(1)]
  \item An isomorphism of sets $\theta_1 \colon \sum_{b \in d} p_b(1)
    \rightarrow \mathrm{ob}(c)$;
  \item For each object $b \in d$ and $i \in p_b(1)$, an isomorphism of
    $a$-sets $\yon^{\theta_1(b,i)} \rightarrow p_{b}[i]$.
  \end{enumerate}
  By transporting along the isomorphisms in (2), we may without loss
  of generality assume that $p_b[i] = \yon^{\theta_1(b,i)}$ for each $b,i$
  and that the isomorphisms in (2) are equalities. So now $p[\thg]$
  takes values in representable $c$-sets, and so, since the Yoneda
  embedding is full and faithful, we have a factorisation
  \begin{equation*}
    p[\thg] = \el_d(p(1))\op \xrightarrow{U^\mathrm{op}} c^\mathrm{op} \xrightarrow{\yon} c\set
  \end{equation*}
  where the action on objects of $U$ is given by the isomorphism
  $\theta_1 \colon \sum_{b \in d} p_b(1) \rightarrow \mathrm{ob}(c)$.
  Thus, writing $V$ for the projection functor from $\el_d(p(1))$, we
  have a span of categories as along the top of:
  \begin{equation*}
    \begin{tikzcd}[row sep=1.2em]
      & \el_d(p(1)) \arrow[ddl, bend right, "U"'] \arrow[bend left, ddr, "V"] \arrow[d, "H"] \\
      & r_\varphi \arrow[dl, "S"] \arrow[dr, "T"'] \\
      c & & d
    \end{tikzcd}
  \end{equation*}
  with $U$ bijective on objects and $V$ \'etale. Let $\varphi$ be the
  retrofunctor $\widetilde U \then \overline V \colon c \coto d$
  corresponding to this span; then $\varphi\lpush  = \Delta_S \then
  \Sigma_T$, where $(S,T)$ is the span representation of $\varphi$ as
  along the bottom of the above diagram. Thus, exactly as
  in~\cref{eqn.catsharpspan_map_diagram_2}, there is a span
  isomorphism $H$ as displayed and so an isomorphism of polynomials
  \begin{equation*}
   p = \Delta_S \then \Sigma_T \xrightarrow{\cong} \Delta_U \then \Delta_H \then \Sigma_H \then \Sigma_V \xrightarrow{\Delta_U \then \epsilon \then \Sigma_V} \Delta_U \then \Sigma_V = \varphi\lpush \rlap{ .}
  \end{equation*}
  A short calculation shows that the inverse
  $\tau \colon \varphi\lpush  \cong p$ of this isomorphism
  satisfies~\cref{eqn.cofunctor_representation_diag_1}. It remains to
  show that $\varphi$ and $\tau$ are \emph{unique}. Since $\theta$ is
  invertible, \cref{eqn.cofunctor_representation_diag_1} implies that
  $s_d \tau$ (and hence $\tau$) is uniquely determined by $\varphi$.
  As for the unicity of $\varphi$, it suffices to show that, if
  $\varphi, \psi \colon c \coto d$ and we have a natural isomorphism
  of polynomials $\varphi\lpush  \cong \psi\lpush $ that is compatible
  with the pseudofunctoriality constraints
  $s_d \circ \varphi\lpush  \cong s_c$ and
  $s_d \circ \psi\lpush  \cong s_c$, then $\varphi = \psi$. This follows
  by a short direct calculation.
\end{proof}

This concludes the verification of \cref{assn.cauchy_dense} and hence
the proof of:
\begin{theorem}\label{thm.garner}
  The pro-arrow equipments $\catsharp \rightarrow \polyfunb$ and
  $\comon(\poly) \rightarrow \comod(\poly)$ are equivalent; whence the
  framed bicategories $\ccatsharp$ and $\ccomod(\poly)$ are equivalent.
\end{theorem}


%
%
%

\section{Building $\ccatsharp$ from $\poly$, concretely}\label{chap.concrete}
Having shown $\ccatsharp \cong \ccomod(\poly)$ abstractly, we take
some time to elaborate on this correspondence more concretely. The
reader who is not interested in seeing these details spelled out may
skip most of this section.

\subsection{Polynomial comonads}\label{sec.comonoids}

A \emph{polynomial comonad} is a comonoid in the monoidal category
$(\poly, \tri, \yon)$; thus, a comonad on $\smset$ whose underlying
endofunctor is a polynomial endofunctor.
The following is a first indication of how a category resembles a polynomial.

\begin{definition}[Outfacing polynomial]\label{def.outfacing}
  Let $c$ be a category. Recall from~\cref{notation.outfacing} that
  for any $a\in\ob(c)$, we write $c[a]$ for the set of
  \emph{$a$-outfacing maps} i.e., the set of morphisms in $c$ whose
  domain is $a$.
  We define the \emph{outfacing polynomial} of $c$ to be
  \[
    \sum_{a\in\ob(c)}\yon^{c[a]}\rlap{ .}
  \]
  In other words, a position is an element of $\ob(c)=c(1)$, and a
  direction at a position is an arrow out of that object.
\end{definition}

A category can be recovered from a comonoid structure on its
outfacing polynomial.%
\footnote{Since domain and codomain play
  symmetrical roles in the definition of category, we have to choose
  one of two opposite conventions when we identify categories with
  polynomial comonoids. The convention opposite ours would instead
  identify a category with a comonoid structure on its \emph{infacing}
  polynomial (the outfacing polynomial of its opposite). Our
  convention induces the conventional definition of retrofunctor
  (\cref{def.retrofunctor}), and it suits the pictures we draw in that
  directions point \emph{out} of positions.}

\begin{proposition}[Polynomial comonads are
  categories]\label{thm.cats_comonads}
  A polynomial comonad with carrier $c$ amounts to a category,
  with outfacing polynomial $c$. This correspondence defines is an isomorphism of groupoids between that of polynomial comonads and their isomorphisms and that of categories and their isomorphisms.
\end{proposition}
\begin{proof}
  The counit $\epsilon\colon c\to\yon$ corresponds to the assignment
  of an identity morphism to each object, and the comultiplication
  $\delta\colon c\to c\tri c$ corresponds to the assignment of a
  codomain and composition formula for morphisms. The comonoid laws
  \eqref{eqn.counitality_coassoc} say that the codomain of identities
  and composites are as they should be, and that composition is unital
  and associative. For the remaining details, see
  \cite{ahman2016directed} or \cite{spivak2022poly},
  or \href{https://www.youtube.com/watch?v=2mWnrgPIrlA}{this video}.
\end{proof}

\begin{figure}[h]
  \centering
  \begin{subfigure}[c]{0.5\textwidth}
    \centering
    \begin{tikzpicture}[corollas]
      \begin{scope}[shift={(-1.5,0)}]
        \node [rcol,vertex,large] (s) at (0, -.8) {};
        \coordinate (t1) at (-.9,.8) {};
        \coordinate (t2) at (-.3,.8) {};
        \coordinate (t3) at (.3,.8) {};
        \coordinate (t4) at (.9,.8) {};
        \draw [rcol,edge] (s) -- (t1);
        \draw [rcol,edge] (s) -- (t2);
        \draw [rcol,edge] (s) -- (t3);
        \draw [rcol,edge] (s) -- (t4);
        \node [annot] at (.3, -.8) {\rlap{$\in c(1)$}};
      \end{scope}
      \node at (0, 0){$\underset{\epsilon}{\mapsto}$};
      \begin{scope}[shift={(1.35,0)}]
        \coordinate (s) at (0,-.95) {};
        \coordinate (t1) at (-.9,.8) {};
        \coordinate (t2) at (-.3,.8) {};
        \coordinate (t3) at (.3,.8) {};
        \coordinate (t4) at (.9,.8) {};
        \draw [dasht] (s) -- (t1);
        \draw [edge] (s) -- (t2);
        \draw [dasht] (s) -- (t3);
        \draw [dasht] (s) -- (t4);
        \node [annot] at (.35, -.8) {\rlap{$\in \yon$}};
      \end{scope}
    \end{tikzpicture}
    \caption{$\epsilon$ picks out an ``identity'' direction at each position.}
  \end{subfigure}%
  \begin{subfigure}[c]{0.5\textwidth}
    \centering
    \begin{tikzpicture}[corollas]
      \begin{scope}[shift={(-1.5,0)}]
        \node [rcol,vertex,large] (s) at (0, -.8) {};
        \coordinate (t1) at (-.9,.8) {};
        \coordinate (t2) at (-.3,.8) {};
        \coordinate (t3) at (.3,.8) {};
        \coordinate (t4) at (.9,.8) {};
        \draw [rcol,edge] (s) -- (t1);
        \draw [rcol,edge] (s) -- (t2);
        \draw [rcol,edge] (s) -- (t3);
        \draw [rcol,edge] (s) -- (t4);
        \node [annot] at (.3, -.8) {\rlap{$\in c(1)$}};
      \end{scope}
      \node at (0, 0){$\underset{\delta}{\mapsto}$};
      \begin{scope}[shift={(2,0)}]
        \node [rcol,vertex] (s) at (0, -.8) {};
        \coordinate (t1) at (-.9,.8) {};
        \coordinate (t2) at (-.3,.8) {};
        \coordinate (t3) at (.3,.8) {};
        \coordinate (t4) at (.9,.8) {};
        \node [rcol,vertex] (m1) at (-1.2,0) {};
        \node [rcol,vertex] (m2) at (-.4,0) {};
        \node [rcol,vertex] (m3) at (.4,0) {};
        \node [rcol,vertex] (m4) at (1.2,0) {};
        \draw [rcol,edge] (s) -- (m1);
        \draw [rcol,edge] (s) -- (m2);
        \draw [rcol,edge] (s) -- (m3);
        \draw [rcol,edge] (s) -- (m4);

        \draw [rcol,edge] (m1) -- (t1);
        \draw [rcol,edge,bend left=22] (m1) to (t1);
        \draw [rcol,edge,bend right=22] (m1) to (t1);

        \draw [rcol,edge] (m2) -- (t1);
        \draw [rcol,edge] (m2) -- (t2);
        \draw [rcol,edge] (m2) -- (t3);
        \draw [rcol,edge] (m2) -- (t4);

        \draw [rcol,edge] (m3) -- (t1);
        \draw [rcol,edge] (m3) -- (t3);
        \draw [rcol,edge,bend left=10] (m3) to (t4);
        \draw [rcol,edge,bend right=10] (m3) to (t4);

        \draw [rcol,edge] (m4) -- (t4);

        \node [annot] at (.2, -.8) {\rlap{$\in c(1)$}};

        \node [annot] at (1.35, 0) {\rlap{$\in c(1)$}};
      \end{scope}
    \end{tikzpicture}
    \caption{$\delta$ specifies codomains and compositions, sending length-2 paths of directions to single directions.}
  \end{subfigure}
  \caption{The morphisms of polynomials $\epsilon\colon c \to \yon$ and
    $\delta\colon c \to c \tri c$, illustrated with corolla replacement (as in \cref{fig.corollamap1}).}
\end{figure}

\begin{figure}[h]
  \centering
  \begin{subfigure}[c]{0.5\textwidth}
    \centering
    \begin{tikzpicture}[corollas]
      \begin{scope}[shift={(-2.7,0)}]
        \coordinate (s) at (0, -.8) {};
        \coordinate (t1) at (-.9,.8) {};
        \coordinate (t2) at (-.3,.8) {};
        \coordinate (t3) at (.3,.8) {};
        \coordinate (t4) at (.9,.8) {};
        \node [rcol,vertex,transparent] (m1) at (-1.2,0) {};
        \node [rcol,vertex] (m2) at (-.4,0) {};
        \node [rcol,vertex,transparent] (m3) at (.4,0) {};
        \node [rcol,vertex,transparent] (m4) at (1.2,0) {};
        \draw [dasht] (s) -- (m1);
        \draw [edge] (s) -- (m2);
        \draw [dasht] (s) -- (m3);
        \draw [dasht] (s) -- (m4);

        \draw [rcol,edge,transparent] (m1) -- (t1);
        \draw [rcol,edge,bend left=22,transparent] (m1) to (t1);
        \draw [rcol,edge,bend right=22,transparent] (m1) to (t1);

        \draw [rcol,edge] (m2) -- (t1);
        \draw [rcol,edge] (m2) -- (t2);
        \draw [rcol,edge] (m2) -- (t3);
        \draw [rcol,edge] (m2) -- (t4);

        \draw [rcol,edge,transparent] (m3) -- (t1);
        \draw [rcol,edge,transparent] (m3) -- (t3);
        \draw [rcol,edge,bend left=10,transparent] (m3) to (t4);
        \draw [rcol,edge,bend right=10,transparent] (m3) to (t4);

        \draw [rcol,edge,transparent] (m4) -- (t4);

        \node at (0, -1.3) {$\delta \then (\epsilon \tri \id_c)$};
      \end{scope}
      \node at (-1, 0){$=$};
      \begin{scope}
        \node [rcol,vertex,large] (s) at (0, -.8) {};
        \coordinate (t1) at (-.9,.8) {};
        \coordinate (t2) at (-.3,.8) {};
        \coordinate (t3) at (.3,.8) {};
        \coordinate (t4) at (.9,.8) {};
        \draw [rcol,edge] (s) -- (t1);
        \draw [rcol,edge] (s) -- (t2);
        \draw [rcol,edge] (s) -- (t3);
        \draw [rcol,edge] (s) -- (t4);
        
        \node at (0, -1.3) {$\id_c$};
      \end{scope}
      \node at (1, 0){$=$};
      \begin{scope}[shift={(2.7,0)}]
        \node [rcol,vertex] (s) at (0, -.8) {};
        \coordinate (t1) at (-.9,.8) {};
        \coordinate (t2) at (-.3,.8) {};
        \coordinate (t3) at (.3,.8) {};
        \coordinate (t4) at (.9,.8) {};
        \coordinate (m1) at (-1.2,0) {};
        \coordinate (m2) at (-.4,0) {};
        \coordinate (m3) at (.4,0) {};
        \coordinate (m4) at (1.2,0) {};
        \draw [rcol,edge] (s) -- (m1);
        \draw [rcol,edge] (s) -- (m2);
        \draw [rcol,edge] (s) -- (m3);
        \draw [rcol,edge] (s) -- (m4);

        \draw [edge] (m1) -- (t1);
        \draw [dasht,bend left=22] (m1) to (t1);
        \draw [dasht,bend right=22] (m1) to (t1);

        \draw [dasht] (m2) -- (t1);
        \draw [edge] (m2) -- (t2);
        \draw [dasht] (m2) -- (t3);
        \draw [dasht] (m2) -- (t4);

        \draw [dasht] (m3) -- (t1);
        \draw [edge] (m3) -- (t3);
        \draw [dasht,bend left=10] (m3) to (t4);
        \draw [dasht,bend right=10] (m3) to (t4);

        \draw [edge] (m4) -- (t4);

        \node at (0, -1.3) {$\delta \then (\id_c \tri \epsilon)$};
      \end{scope}
    \end{tikzpicture}
    \caption{The unit laws ensure $\delta$ identifies every single
      direction with either length two path obtained by appending
      identity (as picked out by $\epsilon$) on either side.}
  \end{subfigure}%
  \begin{subfigure}[c]{0.5\textwidth}
    \centering
    \begin{tikzpicture}[corollas]
      \begin{scope}[shift={(-1.83,0)}]
        
        \node [rcol,vertex,small] (s) at (0, -.8) {};
        \coordinate (t1) at (-.9,.8) {};
        \coordinate (t2) at (-.3,.8) {};
        \coordinate (t3) at (.3,.8) {};
        \coordinate (t4) at (.9,.8) {};

        \node [rcol,vertex,small] (m1_) at (-1.2,-.4) {};
        \node [rcol,vertex,small] (m2_) at (-.4,-.4) {};
        \node [rcol,vertex,small] (m3_) at (.4,-.4) {};
        \node [rcol,vertex,small] (m4_) at (1.2,-.4) {};
        \draw [rcol,edge] (s) -- (m1_);
        \draw [rcol,edge] (s) -- (m2_);
        \draw [rcol,edge] (s) -- (m3_);
        \draw [rcol,edge] (s) -- (m4_);

        \node [rcol,vertex] (m1) at (-1.2,0) {};
        \node [rcol,vertex] (m2) at (-.4,0) {};
        \node [rcol,vertex] (m3) at (.4,0) {};
        \node [rcol,vertex] (m4) at (1.2,0) {};

        \draw [rcol,edge] (m1_) -- (m1);
        \draw [rcol,edge,bend left=32] (m1_) to (m1);
        \draw [rcol,edge,bend right=32] (m1_) to (m1);

        \draw [rcol,edge] (m2_) -- (m1);
        \draw [rcol,edge] (m2_) -- (m2);
        \draw [rcol,edge] (m2_) -- (m3);
        \draw [rcol,edge] (m2_) -- (m4);

        \draw [rcol,edge] (m3_) -- (m1);
        \draw [rcol,edge] (m3_) -- (m3);
        \draw [rcol,edge,bend left=10] (m3_) to (m4);
        \draw [rcol,edge,bend right=10] (m3_) to (m4);

        \draw [rcol,edge] (m4_) -- (m4);

        \draw [rcol,edge] (m1) -- (t1);
        \draw [rcol,edge,bend left=22] (m1) to (t1);
        \draw [rcol,edge,bend right=22] (m1) to (t1);

        \draw [rcol,edge] (m2) -- (t1);
        \draw [rcol,edge] (m2) -- (t2);
        \draw [rcol,edge] (m2) -- (t3);
        \draw [rcol,edge] (m2) -- (t4);

        \draw [rcol,edge] (m3) -- (t1);
        \draw [rcol,edge] (m3) -- (t3);
        \draw [rcol,edge,bend left=10] (m3) to (t4);
        \draw [rcol,edge,bend right=10] (m3) to (t4);

        \draw [rcol,edge] (m4) -- (t4);

        \node at (0, -1.3) {$\delta \then (\delta \tri \id_c)$};
      \end{scope}
      \node at (0, 0){$=$};
      \begin{scope}[shift={(1.87,0)}]
        
        \node [rcol,vertex] (s) at (0, -.8) {};
        \coordinate (t1) at (-.9,.8) {};
        \coordinate (t2) at (-.3,.8) {};
        \coordinate (t3) at (.3,.8) {};
        \coordinate (t4) at (.9,.8) {};

        \node [rcol,vertex,small] (m1) at (-1.2,0) {};
        \node [rcol,vertex,small] (m2) at (-.4,0) {};
        \node [rcol,vertex,small] (m3) at (.4,0) {};
        \node [rcol,vertex,small] (m4) at (1.2,0) {};
        \draw [rcol,edge] (s) -- (m1);
        \draw [rcol,edge] (s) -- (m2);
        \draw [rcol,edge] (s) -- (m3);
        \draw [rcol,edge] (s) -- (m4);

        \node [rcol,vertex,small] (m11) at (-1.45,.4) {};
        \node [rcol,vertex,small] (m12) at (-1.2,.4) {};
        \node [rcol,vertex,small] (m13) at (-.95,.4) {};
        \node [rcol,vertex,small] (m21) at (-.7,.4) {};
        \node [rcol,vertex,small] (m22) at (-.5,.4) {};
        \node [rcol,vertex,small] (m23) at (-.3,.4) {};
        \node [rcol,vertex,small] (m24) at (-.1,.4) {};
        \node [rcol,vertex,small] (m31) at (.1,.4) {};
        \node [rcol,vertex,small] (m32) at (.3,.4) {};
        \node [rcol,vertex,small] (m33) at (.5,.4) {};
        \node [rcol,vertex,small] (m34) at (.7,.4) {};
        \node [rcol,vertex,small] (m41) at (1.2,.4) {};

        \draw [rcol,edge] (m1) -- (m11);
        \draw [rcol,edge] (m1) -- (m12);
        \draw [rcol,edge] (m1) -- (m13);
        \draw [rcol,edge] (m2) -- (m21);
        \draw [rcol,edge] (m2) -- (m22);
        \draw [rcol,edge] (m2) -- (m23);
        \draw [rcol,edge] (m2) -- (m24);
        \draw [rcol,edge] (m3) -- (m31);
        \draw [rcol,edge] (m3) -- (m32);
        \draw [rcol,edge] (m3) -- (m33);
        \draw [rcol,edge] (m3) -- (m34);
        \draw [rcol,edge] (m4) -- (m41);

        \draw [rcol,edge] (m11) -- (t1);
        \draw [rcol,edge,bend left=11] (m11) to (t1);
        \draw [rcol,edge,bend right=11] (m11) to (t1);

        \draw [rcol,edge] (m12) -- (t1);
        \draw [rcol,edge,bend left=11] (m12) to (t1);
        \draw [rcol,edge,bend right=11] (m12) to (t1);

        \draw [rcol,edge] (m13) -- (t1);
        \draw [rcol,edge,bend left=11] (m13) to (t1);
        \draw [rcol,edge,bend right=11] (m13) to (t1);

        \draw [rcol,edge] (m21) -- (t1);
        \draw [rcol,edge,bend left=11] (m21) to (t1);
        \draw [rcol,edge,bend right=11] (m21) to (t1);

        \draw [rcol,edge] (m22) -- (t1);
        \draw [rcol,edge] (m22) -- (t2);
        \draw [rcol,edge] (m22) -- (t3);
        \draw [rcol,edge] (m22) -- (t4);

        \draw [rcol,edge] (m23) -- (t1);
        \draw [rcol,edge] (m23) -- (t3);
        \draw [rcol,edge,bend left=5] (m23) to (t4);
        \draw [rcol,edge,bend right=5] (m23) to (t4);

        \draw [rcol,edge] (m24) -- (t4);

        \draw [rcol,edge] (m31) -- (t1);
        \draw [rcol,edge,bend left=11] (m31) to (t1);
        \draw [rcol,edge,bend right=11] (m31) to (t1);

        \draw [rcol,edge] (m32) -- (t3);

        \draw [rcol,edge] (m33) -- (t4);

        \draw [rcol,edge] (m34) -- (t4);

        \draw [rcol,edge] (m41) -- (t4);

        \node at (0, -1.3) {$\delta \then (\id_c \tri \delta)$};
      \end{scope}
    \end{tikzpicture}
    \caption{The associativity law ensures the two ways that single
      directions may be identified with length three paths via
      $\delta$ agree.}
  \end{subfigure}
  \caption{Comonad laws, as in \cref{eqn.counitality_coassoc}.}
\end{figure}

\begin{example}[Linear polynomial comonads are sets]\label{ex.discrete}
  There is only one possible comonoid structure on $S\yon$,
  corresponding to the discrete category (the unique category with one
  morphism out of each object). The counit $\epsilon$ is the unique
  polynomial map $S\yon\to\yon$, and the comultiplication $\delta$ is
  the map $S\yon\to (S\yon\tri S\yon)\cong S^2\yon$ corresponding to
  the diagonal $S\to S^2$.

  In particular the polynomials
  $0,\yon\in\poly$ each have a unique comonoid structure,
  corresponding to the empty and the one-morphism category,
  respectively. Note that $0$ is the only constant polynomial with a
  comonoid structure.
\end{example}

\begin{example}[Representable polynomial comonads are monoids]\label{ex.monoid_rep}
  If $(M,e,*)$ is a monoid---i.e., a category with one object---its
  outfacing polynomial is $\yon^M$. The counit $\yon^M\to\yon$
  corresponds to the element $e\in M$, the comultiplication map
  $\yon^M\to\yon^M\tri\yon^M$ corresponds to a function $*\colon
  M\times M\to M$, and the diagrams in \eqref{eqn.counitality_coassoc}
  correspond to the unit and associativity laws.
\end{example}

The following we learned from David Jaz Myers.
\begin{example}[Full internal subcategory]\label{prop.DJM}\label{ex.full_internal_finset}
  For any $p\in\poly$, there is a natural comonoid structure on
  $\lens{p}{p}$ (from \cref{prop.JoshMeyers}). This is purely formal;
  when a monoidal operation, e.g., $\tri$, has a (co)closure
  operation, then applying the (co)closure to an object $p$ with
  itself will be a (co)monoid.

  Spelling this out explicitly, the counit map $\epsilon \colon
  \lens{p}{p} \rightarrow \yon$ is the
  transpose of the unit constraint $p \rightarrow \yon \tri p$ under~\cref{eqn.adjunction_coclosure}, while
  the comultiplication $\smash{\lens p p \rightarrow \lens p p \tri
    \lens p p}$ is the transpose under~\cref{eqn.adjunction_coclosure}
  of the composite
  \begin{equation*}
    p \xrightarrow{\mathsf{coev}} \lens p p \tri p \xrightarrow{\lens p p \tri \mathsf{coev}} \lens p p \tri \lens p p \tri p
  \end{equation*}
  where ``$\mathsf{coev}$'' is the unit of the adjunction
  $\lens {\thg} p \dashv (\thg) \tri p$. Hence
  $\lens{p}{p}\cong\sum_{i\in p(1)}\yon^{p(p[i])}$ has a category
  structure. What is it?
  
  Its object set is $p(1)$, while an outfacing morphism from an
  object $i$ consists of another $i' \in p(1)$ and map of direction
  sets $f \colon p[i'] \to p[i]$. So one might guess $\lens{p}{p}$ is
  just the opposite of the category of functions between these
  direction sets $p[i]$. Indeed, one checks from the definitions of
  $\epsilon,\delta$ for this comonoid that the identity on $i$ is
  $(i,\id_{p[i]})$, the codomain of $(i',f)$ is $i'$, and the
  composite of $(i',f)$ and $(i'',f)$ is $(i'',f\then f')$.

  Hence $\lens{p}{p}$ comes with a fully faithful functor to
  $\smset\op$ given by $i\mapsto p[i]$. It may be slightly confusing
  to say $\smash{\lens{p}{p}}$ is a ``subcategory'' of $\smset\op$
  because there may be $i\neq i'$ with $p[i]\cong p[i']$. However, the
  opposite of $\smash{\lens{p}{p}}$ is often called the \emph{full
    internal subcategory} spanned by $p$ (cf.~\cite{jacobs1999categorical}).

  We will be able to construct the full internal subcategory
  itself---rather than its opposite---later on, with the technology of
  \cref{sec.prof}.
\end{example}

\begin{example}\label{ex.costate}
  For any set $S\in\smset$, \cref{prop.DJM} says that
  $S\yon^S=\lens{S}{S}$ has a comonoid structure. It is the so-called
  \emph{costate comonad} from functional programming, i.e., it is the
  comonad one gets from the product--hom adjunction
  $S \times (\thg) \dashv (\thg)^S \colon \smset \rightarrow \smset$.
  The costate comonoid $\lens{S}{S}$ corresponds to
  the indiscrete category with object-set $S$ as in
  \cref{ex.indiscrete}. Indeed, the fully faithful functor from it to
  $\smset\op$ sends every object to $\varnothing$, since
  $S = S\yon^0$.
\end{example}

When we mechanically unpack the definition of comonad morphism
\eqref{eqn.comonoid_hom}, translating through the identification of
polynomial comonads with categories, we obtain the notion of
retrofunctor as in \cref{def.retrofunctor} 

\subsection{Polynomial bicomodules}\label{sec.bicomodules}

Bicomodules (\cref{def.bicomodule}) seem to include a lot of data; what does it all mean? Let
us spend some time closely examining left and right comodules to see
what structures fall out, just like we did for comonoids.

We use the notation $c \bito d$ to denote a bicomodule from the
polynomial comonoid $c$ to the polynomial comonoid $d$.

\begin{proposition}[Left comodules are diagrams in $\poly$]\label{prop_left_comodules}
  Specifying a left comodule $\yon \bito[m] d$ amounts to
  providing a functor $d\to\poly$; here $m$ is
  the sum of assigned polynomials over all objects in $d$.
\end{proposition}
\begin{proof}
  The left comodule map $\lambda\colon m \to d \tri m$ first of all assigns
  an object of $d$ to each $m$-position, thus determining a polynomial
  $m_a$ consisting of the $m$-positions lying over any particular
  object $a\in d(1)$. Second of all, it specifies for each $m$-position $i$
  and arrow $f\colon a \to a'$ out of its assigned object $a$ another
  $m$-position $j$, thus determining a map $f_1\colon m_a(1) \to m(1)$ (yet
  to cohere as a map $f_1\colon m_a(1) \to m_{a'}(1)$). Third of all, it
  maps directions at each such position $f_1(i) \coloneqq j$ backward
  to directions at $i$, thus determining a map
  $f^\sharp_i\colon m[f_1(i)] \to m_a[i]$.
  
  The left comodule laws \eqref{eqn.bimod_left} say that each such $f_1\colon m_a(1) \to m(1)$
  lands in $m_{a'}(1)$ as it should, and that the assignment
  $f \mapsto (f_1, f^\sharp)$ is functorial (preserves identities and
  composition).
\end{proof}
\begin{figure}[h]
  \centering
  \begin{tikzpicture}[corollas]
    \begin{scope}[shift={(-1.85,0)}]
      \node [ycol,vertex,large] (s) at (0, -.8) {};
      \coordinate (t1) at (-1.2,.8) {};
      \coordinate (t2) at (-.6,.8) {};
      \coordinate (t3) at (.0,.8) {};
      \coordinate (t4) at (.6,.8) {};
      \coordinate (t5) at (1.2,.8) {};
      \draw [ycol,edge] (s) -- (t1);
      \draw [ycol,edge] (s) -- (t2);
      \draw [ycol,edge] (s) -- (t3);
      \draw [ycol,edge] (s) -- (t4);
      \draw [ycol,edge] (s) -- (t5);
      \node [annot] at (.3, -.8) {\rlap{$\in m(1)$}};
    \end{scope}
    \node at (0, 0){$\underset{\lambda}{\mapsto}$};
    \begin{scope}[shift={(2.5,0)}]
      \node [bcol,vertex] (s) at (0, -.8) {};
      \coordinate (t1) at (-1.2,.8) {};
      \coordinate (t2) at (-.6,.8) {};
      \coordinate (t3) at (0,.8) {};
      \coordinate (t4) at (.6,.8) {};
      \coordinate (t5) at (1.2,.8) {};
      \node [ycol,vertex] (m1) at (-1.5,0) {};
      \node [ycol,vertex] (m2) at (-.5,0) {};
      \node [ycol,vertex] (m3) at (.5,0) {};
      \node [ycol,vertex] (m4) at (1.5,0) {};
      \draw [bcol,edge] (s) -- (m1);
      \draw [bcol,edge] (s) -- (m2);
      \draw [bcol,edge] (s) -- (m3);
      \draw [bcol,edge] (s) -- (m4);

      \draw [ycol,edge] (m2) -- (t1);
      \draw [ycol,edge] (m2) -- (t2);
      \draw [ycol,edge] (m2) -- (t3);
      \draw [ycol,edge] (m2) -- (t4);
      \draw [ycol,edge] (m2) -- (t5);

      \draw [ycol,edge,bend left=11] (m3) to (t3);
      \draw [ycol,edge,bend right=11] (m3) to (t3);
      \draw [ycol,edge,bend left=11] (m3) to (t4);
      \draw [ycol,edge,bend right=11] (m3) to (t4);
      \draw [ycol,edge,bend left=11] (m3) to (t5);
      \draw [ycol,edge,bend right=11] (m3) to (t5);

      \draw [ycol,edge] (m4) -- (t4);
      \draw [ycol,edge] (m4) -- (t5);

      \node [annot] at (.2, -.8) {\rlap{$\in d(1)$}};
      \node [annot] at (1.7, 0) {\rlap{$\in m(1)$}};
    \end{scope}
  \end{tikzpicture}
  \caption{Polynomial morphism $\lambda\colon m \to d \tri m$ as in a left
    comodule.
  }
\end{figure}

\begin{figure}[h]
  \centering
  \begin{subfigure}[c]{0.5\textwidth}
    \centering
    \begin{tikzpicture}[corollas]
      \begin{scope}[shift={(-2.3,0)}]
        \coordinate (s) at (0, -.8) {};
        \coordinate (t1) at (-1.2,.8) {};
        \coordinate (t2) at (-.6,.8) {};
        \coordinate (t3) at (0,.8) {};
        \coordinate (t4) at (.6,.8) {};
        \coordinate (t5) at (1.2,.8) {};

        \node [ycol,vertex,transparent] (m1) at (-1.5,.0) {};
        \node [ycol,vertex] (m2) at (-.5,0) {};
        \node [ycol,vertex,transparent] (m3) at (.5,0) {};
        \node [ycol,vertex,transparent] (m4) at (1.5,0) {};
        \draw [dasht] (s) -- (m1);
        \draw [edge] (s) -- (m2);
        \draw [dasht] (s) -- (m3);
        \draw [dasht] (s) -- (m4);

        \draw [ycol,edge] (m2) -- (t1);
        \draw [ycol,edge] (m2) -- (t2);
        \draw [ycol,edge] (m2) -- (t3);
        \draw [ycol,edge] (m2) -- (t4);
        \draw [ycol,edge] (m2) -- (t5);

        \draw [ycol,edge,transparent,bend left=11] (m3) to (t3);
        \draw [ycol,edge,transparent,bend right=11] (m3) to (t3);
        \draw [ycol,edge,transparent,bend left=11] (m3) to (t4);
        \draw [ycol,edge,transparent,bend right=11] (m3) to (t4);
        \draw [ycol,edge,transparent,bend left=11] (m3) to (t5);
        \draw [ycol,edge,transparent,bend right=11] (m3) to (t5);

        \draw [ycol,transparent,edge] (m4) -- (t4);
        \draw [ycol,transparent,edge] (m4) -- (t5);

        \node at (0, -1.3) {$\lambda \then (\epsilon \tri \id_m)$};
      \end{scope}
      \node at (0, 0){$=$};
      \begin{scope}[shift={(1.5,0)}]
        \node [ycol,vertex,large] (s) at (0, -.8) {};
        \coordinate (t1) at (-1.2,.8) {};
        \coordinate (t2) at (-.6,.8) {};
        \coordinate (t3) at (0,.8) {};
        \coordinate (t4) at (.6,.8) {};
        \coordinate (t5) at (1.2,.8) {};
        \draw [ycol,edge] (s) -- (t1);
        \draw [ycol,edge] (s) -- (t2);
        \draw [ycol,edge] (s) -- (t3);
        \draw [ycol,edge] (s) -- (t4);
        \draw [ycol,edge] (s) -- (t5);

        \node at (0, -1.3) {$\id_m$};
      \end{scope}
    \end{tikzpicture}
    \caption{The unit law ensures identity arrows of $d$ send any
      position or direction to itself.}
  \end{subfigure}%
  \begin{subfigure}[c]{0.5\textwidth}
    \centering
    \begin{tikzpicture}[corollas]
      \begin{scope}[shift={(-2.35,0)}]
        \node [bcol,vertex,small] (s) at (0, -.8) {};
        \coordinate (t1) at (-1.2,.8) {};
        \coordinate (t2) at (-.6,.8) {};
        \coordinate (t3) at (0,.8) {};
        \coordinate (t4) at (.6,.8) {};
        \coordinate (t5) at (1.2,.8) {};

        \node [bcol,vertex,small] (m1_) at (-1.5,-.4) {};
        \node [bcol,vertex,small] (m2_) at (-.5,-.4) {};
        \node [bcol,vertex,small] (m3_) at (.5,-.4) {};
        \node [bcol,vertex,small] (m4_) at (1.5,-.4) {};

        \node [ycol,vertex] (m1) at (-1.5,0) {};
        \node [ycol,vertex] (m2) at (-.5,0) {};
        \node [ycol,vertex] (m3) at (.5,0) {};
        \node [ycol,vertex] (m4) at (1.5,0) {};

        \draw [bcol,edge] (s) -- (m1_);
        \draw [bcol,edge] (s) -- (m2_);
        \draw [bcol,edge] (s) -- (m3_);
        \draw [bcol,edge] (s) -- (m4_);

        \draw [bcol,edge] (m1_) -- (m1);
        \draw [bcol,edge,bend left=32] (m1_) to (m1);
        \draw [bcol,edge,bend right=32] (m1_) to (m1);

        \draw [bcol,edge] (m2_) -- (m1);
        \draw [bcol,edge] (m2_) -- (m2);
        \draw [bcol,edge] (m2_) -- (m3);
        \draw [bcol,edge] (m2_) -- (m4);

        \draw [bcol,edge] (m3_) -- (m1);
        \draw [bcol,edge] (m3_) -- (m3);
        \draw [bcol,edge,bend left=10] (m3_) to (m4);
        \draw [bcol,edge,bend right=10] (m3_) to (m4);

        \draw [bcol,edge] (m4_) -- (m4);

        \draw [ycol,edge] (m2) -- (t1);
        \draw [ycol,edge] (m2) -- (t2);
        \draw [ycol,edge] (m2) -- (t3);
        \draw [ycol,edge] (m2) -- (t4);
        \draw [ycol,edge] (m2) -- (t5);

        \draw [ycol,edge,bend left=11] (m3) to (t3);
        \draw [ycol,edge,bend right=11] (m3) to (t3);
        \draw [ycol,edge,bend left=11] (m3) to (t4);
        \draw [ycol,edge,bend right=11] (m3) to (t4);
        \draw [ycol,edge,bend left=11] (m3) to (t5);
        \draw [ycol,edge,bend right=11] (m3) to (t5);

        \draw [ycol,edge] (m4) -- (t4);
        \draw [ycol,edge] (m4) -- (t5);

        \node at (0, -1.3) {$\lambda \then (\delta \tri \id_m)$};
      \end{scope}
      \node at (0, 0){$=$};
      \begin{scope}[shift={(2.5,0)}]
        \node [bcol,vertex] (s) at (0, -.8) {};
        \coordinate (t1) at (-1.2,.8) {};
        \coordinate (t2) at (-.6,.8) {};
        \coordinate (t3) at (0,.8) {};
        \coordinate (t4) at (.6,.8) {};
        \coordinate (t5) at (1.2,.8) {};

        \node [ycol,vertex,small] (m11) at (-1.8,.4) {};
        \node [ycol,vertex,small] (m12) at (-1.5,.4) {};
        \node [ycol,vertex,small] (m13) at (-1.2,.4) {};
        \node [ycol,vertex,small] (m21) at (-.89,.4) {};
        \node [ycol,vertex,small] (m22) at (-.63,.4) {};
        \node [ycol,vertex,small] (m23) at (-.37,.4) {};
        \node [ycol,vertex,small] (m24) at (-.12,.4) {};
        \node [ycol,vertex,small] (m31) at (.12,.4) {};
        \node [ycol,vertex,small] (m32) at (.37,.4) {};
        \node [ycol,vertex,small] (m33) at (.63,.4) {};
        \node [ycol,vertex,small] (m34) at (.89,.4) {};
        \node [ycol,vertex,small] (m41) at (1.5,.4) {};

        \node [bcol,vertex,small] (m1) at (-1.5,0) {};
        \node [bcol,vertex,small] (m2) at (-.5,0) {};
        \node [bcol,vertex,small] (m3) at (.5,0) {};
        \node [bcol,vertex,small] (m4) at (1.5,0) {};

        \draw [bcol,edge] (m1) -- (m11);
        \draw [bcol,edge] (m1) -- (m12);
        \draw [bcol,edge] (m1) -- (m13);
        \draw [bcol,edge] (m2) -- (m21);
        \draw [bcol,edge] (m2) -- (m22);
        \draw [bcol,edge] (m2) -- (m23);
        \draw [bcol,edge] (m2) -- (m24);
        \draw [bcol,edge] (m3) -- (m31);
        \draw [bcol,edge] (m3) -- (m32);
        \draw [bcol,edge] (m3) -- (m33);
        \draw [bcol,edge] (m3) -- (m34);
        \draw [bcol,edge] (m4) -- (m41);

        \draw [bcol,edge] (s) -- (m1);
        \draw [bcol,edge] (s) -- (m2);
        \draw [bcol,edge] (s) -- (m3);
        \draw [bcol,edge] (s) -- (m4);

        \draw [ycol,edge] (m22) -- (t1);
        \draw [ycol,edge] (m22) -- (t2);
        \draw [ycol,edge] (m22) -- (t3);
        \draw [ycol,edge] (m22) -- (t4);
        \draw [ycol,edge] (m22) -- (t5);

        \draw [ycol,edge,bend left=5] (m23) to (t3);
        \draw [ycol,edge,bend right=5] (m23) to (t3);
        \draw [ycol,edge,bend left=3] (m23) to (t4);
        \draw [ycol,edge,bend right=3] (m23) to (t4);
        \draw [ycol,edge,bend left=3] (m23) to (t5);
        \draw [ycol,edge,bend right=3] (m23) to (t5);

        \draw [ycol,edge] (m24) -- (t4);
        \draw [ycol,edge] (m24) -- (t5);

        \draw [ycol,edge,bend left=5] (m32) to (t3);
        \draw [ycol,edge,bend right=5] (m32) to (t3);
        \draw [ycol,edge,bend left=3] (m32) to (t4);
        \draw [ycol,edge,bend right=3] (m32) to (t4);
        \draw [ycol,edge,bend left=3] (m32) to (t5);
        \draw [ycol,edge,bend right=3] (m32) to (t5);
        \draw [ycol,edge] (m33) -- (t4);
        \draw [ycol,edge] (m33) -- (t5);
        \draw [ycol,edge] (m34) -- (t4);
        \draw [ycol,edge] (m34) -- (t5);

        \draw [ycol,edge] (m41) -- (t4);
        \draw [ycol,edge] (m41) -- (t5);

        \node at (0, -1.3) {$\lambda \then (\id_d \tri \lambda)$};
      \end{scope}
    \end{tikzpicture}
    \caption{The associativity law ensures any length two
      path of arrows in $d$ acts on on positions and directions of $m$
      in the same way as the composite arrow in $d$.}
  \end{subfigure}
  \caption{Left comodule laws, as in \cref{eqn.bimod_left}.}
\end{figure}

\begin{example}[Degenerate left comodule]
  As per \cref{rem.degenerate}, we expect the degenerate case of a left
  $\yon$-module to be just a polynomial. Indeed, a
  diagram in $\poly$ indexed by $y$, the one-morphism category, is just
  an object.
\end{example}

\begin{notation}[Indexed component $m_a$ of a left
  comodule]\label{not.indexed_component_left}
  Let $m$ be a left $d$-comodule and $a$ an object of $d$. We continue
  to write $m_a \in \poly$ for the polynomial that $m$ (construed as a
  functor $d \to \poly$) assigns to $a$. Explicitly, $m_a$ consists of
  the $m$-positions lying via $\lambda_1$ over $(d \tri m)$-positions
  whose $d$-position component is $a$. More precisely, consider the composite
  \[
    m\To{\lambda}d\tri m\To{d\tri\; !}d\tri 1.
  \]
  Then $m_a$ is the pullback of this map along $a\colon 1\to d(1)$:
  \[
    \begin{tikzcd}
      m_a\ar[r]\ar[d]\ar[dr,phantom, very near start, "\lrcorner"]&m\ar[d]\\
      1\ar[r, "a"']&d(1)
    \end{tikzcd}
  \]
  We have an isomorphism of polynomials
  \begin{equation}\label{eqn.family_map}
    m\cong\sum_{a\in d(1)}m_a.
  \end{equation}
  
  We will use this notation $m_a$ often and refer to it as the
  \emph{$a$-indexed component of $m$}.
\end{notation}

\begin{proposition}[Right comodules are indexed sets of
  $c$-sets]\label{prop_right_comodules} Specifying a right comodule $c
  \bito[m] \yon$ amounts to providing for each $i \in m(1)$ a
  $c$-set with elements $m[i]$.
\end{proposition}
\begin{proof}
  Restricted to a particular $m$-position $i$, the right comodule map
  $\rho\colon m \to m \tri d$ gives an $m$-position $i'$, assigns objects
  of $c$ to $m[i']$---thus determining a set $P_i(a)$ consisting of
  elements lying over $a$ for any particular object $a$---and assigns
  to each arrow $f\colon a \to a'$ a map $P_i(f)\colon P_i(a) \to m[i]$ (yet to
  cohere as a map from $P_i(f)\colon P_i(a) \to P_i(a')$).

  The right comodule laws \eqref{eqn.bimod_right} say first of all
  that $i' = i$---so the sets $P_i(a)$ partition $m[i]$
  itself---second of all that each such $P_i(f)$ lands in $P_i(a')$ as
  it should, and third of all that the assignment $f \mapsto P_i(f)$
  is functorial (preserves identities and composition).
\end{proof}

\begin{figure}[h]
  \centering
  \begin{tikzpicture}[corollas]
    \begin{scope}[shift={(-1.85,0)}]
      \node [ycol,vertex,large] (s) at (0, -.8) {};
      \coordinate (t1) at (-1.2,.8) {};
      \coordinate (t2) at (-.6,.8) {};
      \coordinate (t3) at (.0,.8) {};
      \coordinate (t4) at (.6,.8) {};
      \coordinate (t5) at (1.2,.8) {};
      \draw [ycol,edge] (s) -- (t1);
      \draw [ycol,edge] (s) -- (t2);
      \draw [ycol,edge] (s) -- (t3);
      \draw [ycol,edge] (s) -- (t4);
      \draw [ycol,edge] (s) -- (t5);
      \node [annot] at (.3, -.8) {\rlap{$\in m(1)$}};
    \end{scope}
    \node at (0, 0){$\underset{\rho}{\mapsto}$};
    \begin{scope}[shift={(2.5,0)}]
      \node [ycol,vertex] (s) at (0, -.8) {};
      \coordinate (t1) at (-1.2,.8) {};
      \coordinate (t2) at (-.6,.8) {};
      \coordinate (t3) at (-0,.8) {};
      \coordinate (t4) at (.6,.8) {};
      \coordinate (t5) at (1.2,.8) {};
      \node [rcol,vertex] (m1) at (-1.6,0) {};
      \node [rcol,vertex] (m2) at (-.8,0) {};
      \node [rcol,vertex] (m3) at (0,0) {};
      \node [rcol,vertex] (m4) at (.8,0) {};
      \node [rcol,vertex] (m5) at (1.6,0) {};
      \draw [ycol,edge] (s) -- (m1);
      \draw [ycol,edge] (s) -- (m2);
      \draw [ycol,edge] (s) -- (m3);
      \draw [ycol,edge] (s) -- (m4);
      \draw [ycol,edge] (s) -- (m5);

      \draw [rcol,edge] (m1) -- (t1);
      \draw [rcol,edge] (m1) -- (t2);

      \draw [rcol,edge] (m2) -- (t1);
      \draw [rcol,edge] (m2) -- (t2);

      \draw [rcol,edge] (m3) -- (t3);
      \draw [rcol,edge] (m3) -- (t4);
      \draw [rcol,edge] (m3) -- (t5);

      \draw [rcol,edge] (m4) -- (t4);
      \draw [rcol,edge] (m4) -- (t5);

      \draw [rcol,edge] (m5) -- (t4);
      \draw [rcol,edge] (m5) -- (t5);

      \node [annot] at (.2, -.8) {\rlap{$\in m(1)$}};
      \node [annot] at (1.75, 0) {\rlap{$\in c(1)$}};
    \end{scope}
  \end{tikzpicture}
  \caption{Polynomial morphism $\rho\colon m \to m \tri c$ as in a right comodule.}
\end{figure}

\begin{figure}[h]
  \centering
  \begin{subfigure}[c]{0.5\textwidth}
    \centering
    \begin{tikzpicture}[corollas]
      \begin{scope}[shift={(-1.5,0)}]
        \node [ycol,vertex,large] (s) at (0, -.8) {};
        \coordinate (t1) at (-1.2,.8) {};
        \coordinate (t2) at (-.6,.8) {};
        \coordinate (t3) at (0,.8) {};
        \coordinate (t4) at (.6,.8) {};
        \coordinate (t5) at (1.2,.8) {};
        \draw [ycol,edge] (s) -- (t1);
        \draw [ycol,edge] (s) -- (t2);
        \draw [ycol,edge] (s) -- (t3);
        \draw [ycol,edge] (s) -- (t4);
        \draw [ycol,edge] (s) -- (t5);

        \node at (0, -1.3) {$\id_m$};
      \end{scope}
      \node at (0, 0){$=$};
      \begin{scope}[shift={(2.3,0)}]
        \node [ycol,vertex] (s) at (0, -.8) {};
        \coordinate (t1) at (-1.2,.8) {};
        \coordinate (t2) at (-.6,.8) {};
        \coordinate (t3) at (0,.8) {};
        \coordinate (t4) at (.6,.8) {};
        \coordinate (t5) at (1.2,.8) {};

        \coordinate (m1) at (-1.6,0) {};
        \coordinate (m2) at (-.8,0) {};
        \coordinate (m3) at (0,0) {};
        \coordinate (m4) at (.8,0) {};
        \coordinate (m5) at (1.6,0) {};

        \draw [ycol,edge] (s) -- (m1);
        \draw [ycol,edge] (s) -- (m2);
        \draw [ycol,edge] (s) -- (m3);
        \draw [ycol,edge] (s) -- (m4);
        \draw [ycol,edge] (s) -- (m5);

        \draw [edge] (m1) -- (t1);
        \draw [dasht] (m1) -- (t2);

        \draw [dasht] (m2) -- (t1);
        \draw [edge] (m2) -- (t2);

        \draw [edge] (m3) -- (t3);
        \draw [dasht] (m3) -- (t4);
        \draw [dasht] (m3) -- (t5);

        \draw [edge] (m4) -- (t4);
        \draw [dasht] (m4) -- (t5);

        \draw [dasht] (m5) -- (t4);
        \draw [edge] (m5) -- (t5);

        \node at (0, -1.3) {$\rho \then (\id_m \tri \epsilon)$};
      \end{scope}
    \end{tikzpicture}
    \caption{The unit law ensures each direction is sent to itself by
      the relevant identity arrow in $d$.}
  \end{subfigure}%
  \begin{subfigure}[c]{0.5\textwidth}
    \centering
    \begin{tikzpicture}[corollas]
      \begin{scope}[shift={(-2.35,0)}] 
        \node [ycol,vertex,small] (s) at (0, -.8) {};
        \coordinate (t1) at (-1.2,.8) {};
        \coordinate (t2) at (-.6,.8) {};
        \coordinate (t3) at (-0,.8) {};
        \coordinate (t4) at (.6,.8) {};
        \coordinate (t5) at (1.2,.8) {};

        \node [rcol,vertex,small] (m1) at (-1.6,-.4) {};
        \node [rcol,vertex,small] (m2) at (-.8,-.4) {};
        \node [rcol,vertex,small] (m3) at (0,-.4) {};
        \node [rcol,vertex,small] (m4) at (.8,-.4) {};
        \node [rcol,vertex,small] (m5) at (1.6,-.4) {};

        \node [rcol,vertex] (m1_) at (-1.6,0) {};
        \node [rcol,vertex] (m2_) at (-.8,0) {};
        \node [rcol,vertex] (m3_) at (0,0) {};
        \node [rcol,vertex] (m4_) at (.8,0) {};
        \node [rcol,vertex] (m5_) at (1.6,0) {};

        \draw [ycol,edge] (s) -- (m1);
        \draw [ycol,edge] (s) -- (m2);
        \draw [ycol,edge] (s) -- (m3);
        \draw [ycol,edge] (s) -- (m4);
        \draw [ycol,edge] (s) -- (m5);

        \draw [rcol,edge] (m1) -- (m1_);
        \draw [rcol,edge] (m1) -- (m2_);

        \draw [rcol,edge] (m2) -- (m1_);
        \draw [rcol,edge] (m2) -- (m2_);

        \draw [rcol,edge] (m3) -- (m3_);
        \draw [rcol,edge] (m3) -- (m4_);
        \draw [rcol,edge] (m3) -- (m5_);

        \draw [rcol,edge] (m4) -- (m4_);
        \draw [rcol,edge] (m4) -- (m5_);

        \draw [rcol,edge] (m5) -- (m4_);
        \draw [rcol,edge] (m5) -- (m5_);

        \draw [rcol,edge] (m1_) -- (t1);
        \draw [rcol,edge] (m1_) -- (t2);

        \draw [rcol,edge] (m2_) -- (t1);
        \draw [rcol,edge] (m2_) -- (t2);

        \draw [rcol,edge] (m3_) -- (t3);
        \draw [rcol,edge] (m3_) -- (t4);
        \draw [rcol,edge] (m3_) -- (t5);

        \draw [rcol,edge] (m4_) -- (t4);
        \draw [rcol,edge] (m4_) -- (t5);

        \draw [rcol,edge] (m5_) -- (t4);
        \draw [rcol,edge] (m5_) -- (t5);

        \node at (0, -1.3) {$\rho \then (\rho \tri \id_c)$};
      \end{scope}
      \node at (0, 0){$=$};
      \begin{scope}[shift={(2.5,0)}]
        \node [ycol,vertex] (s) at (0, -.8) {};
        \coordinate (t1) at (-1.2,.8) {};
        \coordinate (t2) at (-.6,.8) {};
        \coordinate (t3) at (-0,.8) {};
        \coordinate (t4) at (.6,.8) {};
        \coordinate (t5) at (1.2,.8) {};

        \node [rcol,vertex,small] (m1) at (-1.6,0) {};
        \node [rcol,vertex,small] (m2) at (-.8,0) {};
        \node [rcol,vertex,small] (m3) at (0,0) {};
        \node [rcol,vertex,small] (m4) at (.8,0) {};
        \node [rcol,vertex,small] (m5) at (1.6,0) {};

        \node [rcol,vertex,small] (m11) at (-1.8,.4) {};
        \node [rcol,vertex,small] (m12) at (-1.4,.4) {};
        \node [rcol,vertex,small] (m21) at (-1.0,.4) {};
        \node [rcol,vertex,small] (m22) at (-.6,.4) {};
        \node [rcol,vertex,small] (m31) at (-.3,.4) {};
        \node [rcol,vertex,small] (m32) at (.0,.4) {};
        \node [rcol,vertex,small] (m33) at (.3,.4) {};
        \node [rcol,vertex,small] (m41) at (.6,.4) {};
        \node [rcol,vertex,small] (m42) at (1.0,.4) {};
        \node [rcol,vertex,small] (m51) at (1.4,.4) {};
        \node [rcol,vertex,small] (m52) at (1.8,.4) {};

        \draw [ycol,edge] (s) -- (m1);
        \draw [ycol,edge] (s) -- (m2);
        \draw [ycol,edge] (s) -- (m3);
        \draw [ycol,edge] (s) -- (m4);
        \draw [ycol,edge] (s) -- (m5);

        \draw [rcol,edge] (m1) -- (m11);
        \draw [rcol,edge] (m1) -- (m12);

        \draw [rcol,edge] (m2) -- (m21);
        \draw [rcol,edge] (m2) -- (m22);

        \draw [rcol,edge] (m3) -- (m31);
        \draw [rcol,edge] (m3) -- (m32);
        \draw [rcol,edge] (m3) -- (m33);

        \draw [rcol,edge] (m4) -- (m41);
        \draw [rcol,edge] (m4) -- (m42);

        \draw [rcol,edge] (m5) -- (m51);
        \draw [rcol,edge] (m5) -- (m52);

        \draw [rcol,edge] (m11) -- (t1);
        \draw [rcol,edge] (m11) -- (t2);

        \draw [rcol,edge] (m12) -- (t1);
        \draw [rcol,edge] (m12) -- (t2);

        \draw [rcol,edge] (m21) -- (t1);
        \draw [rcol,edge] (m21) -- (t2);

        \draw [rcol,edge] (m22) -- (t1);
        \draw [rcol,edge] (m22) -- (t2);

        \draw [rcol,edge] (m31) -- (t3);
        \draw [rcol,edge] (m31) -- (t4);
        \draw [rcol,edge] (m31) -- (t5);

        \draw [rcol,edge] (m32) -- (t4);
        \draw [rcol,edge] (m32) -- (t5);

        \draw [rcol,edge] (m33) -- (t4);
        \draw [rcol,edge] (m33) -- (t5);

        \draw [rcol,edge] (m41) -- (t4);
        \draw [rcol,edge] (m41) -- (t5);

        \draw [rcol,edge] (m42) -- (t4);
        \draw [rcol,edge] (m42) -- (t5);

        \draw [rcol,edge] (m51) -- (t4);
        \draw [rcol,edge] (m51) -- (t5);

        \draw [rcol,edge] (m52) -- (t4);
        \draw [rcol,edge] (m52) -- (t5);

        \node at (0, -1.3) {$\rho \then (\id_m \tri \delta)$};
      \end{scope}
    \end{tikzpicture}
    \caption{The associativity law ensures the action by a length two
      path of arrows in $d$ on directions of $m$, induced by applying
      $\rho$ twice, is the action by the composite arrow in $d$}
  \end{subfigure}
  \caption{Right comodule laws, as in \cref{eqn.bimod_right}.}
\end{figure}

\begin{example}[Degenerate right comodule]
Again, we expect the degenerate case of a right $\yon$-module to be
just a polynomial. Indeed, a polynomial is precisely an indexed set of
sets.
\end{example}

\begin{proposition}[($c$,$d$)-bicomodules are $d$-shaped diagrams of
  $c$-variable polynomials]\label{prop_bicomodules}
  Specifying a bicomodule $c\bito[m]d$ amounts to providing a
  functor $d \to \smset[c]$; here $m$ is the sum of assigned
  underlying polynomials over all objects in $d$.
\end{proposition}
\begin{proof}
  In \cref{prop_left_comodules} we saw that left module structure on
  $m$ gives a functor $d \to \poly$, $a \mapsto m_a$, and in
  \cref{prop_right_comodules} we saw that right module structure on
  $m$ gives a $c$-set structure on the direction set at each
  position of $m$, or in other words, exhibits $m$ as the underlying
  polynomial of a $c$-variable polynomial. A fortiori this induces a
  $c$-set structure on the directions at each position of any
  $m_a$, exhibiting it as the underlying polynomial of a
  $c$-variable polynomial.

  Note that a morphism in $\smset[c]$ amounts to a morphism of
  underlying polynomials in $\poly$ such that, at each position, the
  map backwards on directions preserves the $c$-set
  structure.

  The bicomodule compatibility law
  \eqref{eqn.bimod_coherence} says that, for every arrow $f\colon a \to a'$
  in $d$ and position $i$ of $m_a$, the assigned map
  $f^\sharp_i\colon m_{a'}[j] \to m_a[i]$ does indeed preserve the
  $c$-set structure on $m_{a'}[j]$, and so the functor
  $d \to \poly$ factors through $\smset[c]$.
\end{proof}
\begin{figure}[h]
  \centering
  \begin{tikzpicture}[corollas]
    \begin{scope}[shift={(-2.3,0)}]
      \node [bcol,vertex,small] (s) at (0, -.8) {};
      \coordinate (t1) at (-1.2,.8) {};
      \coordinate (t2) at (-.6,.8) {};
      \coordinate (t3) at (-0,.8) {};
      \coordinate (t4) at (.6,.8) {};
      \coordinate (t5) at (1.2,.8) {};

      \node [rcol,vertex] (m1_) at (-1.6,0) {};
      \node [rcol,vertex] (m2_) at (-.8,0) {};
      \node [rcol,vertex] (m3_) at (0,0) {};
      \node [rcol,vertex] (m4_) at (.8,0) {};
      \node [rcol,vertex] (m5_) at (1.6,0) {};

      \node [ycol,vertex,small] (m1) at (-1.5,-.4) {};
      \node [ycol,vertex,small] (m2) at (-.5,-.4) {};
      \node [ycol,vertex,small] (m3) at (.5,-.4) {};
      \node [ycol,vertex,small] (m4) at (1.5,-.4) {};

      \draw [ycol,edge] (m2) -- (m1_);
      \draw [ycol,edge] (m2) -- (m2_);
      \draw [ycol,edge] (m2) -- (m3_);
      \draw [ycol,edge] (m2) -- (m4_);
      \draw [ycol,edge] (m2) -- (m5_);

      \draw [ycol,edge,bend left=13] (m3) to (m3_);
      \draw [ycol,edge,bend right=13] (m3) to (m3_);
      \draw [ycol,edge,bend left=13] (m3) to (m4_);
      \draw [ycol,edge,bend right=13] (m3) to (m4_);
      \draw [ycol,edge,bend left=8] (m3) to (m5_);
      \draw [ycol,edge,bend right=8] (m3) to (m5_);

      \draw [ycol,edge] (m4) -- (m4_);
      \draw [ycol,edge] (m4) -- (m5_);

      \draw [bcol,edge] (s) -- (m1);
      \draw [bcol,edge] (s) -- (m2);
      \draw [bcol,edge] (s) -- (m3);
      \draw [bcol,edge] (s) -- (m4);

      \draw [rcol,edge] (m1_) -- (t1);
      \draw [rcol,edge] (m1_) -- (t2);

      \draw [rcol,edge] (m2_) -- (t1);
      \draw [rcol,edge] (m2_) -- (t2);

      \draw [rcol,edge] (m3_) -- (t3);
      \draw [rcol,edge] (m3_) -- (t4);
      \draw [rcol,edge] (m3_) -- (t5);

      \draw [rcol,edge] (m4_) -- (t4);
      \draw [rcol,edge] (m4_) -- (t5);

      \draw [rcol,edge] (m5_) -- (t4);
      \draw [rcol,edge] (m5_) -- (t5);

      \node at (0, -1.3) {$\rho \then (\lambda \tri \id_c)$};
    \end{scope}
    \node at (0, 0){$=$};
    \begin{scope}[shift={(2.3,0)}]
      \node [bcol,vertex] (s) at (0, -.8) {};
      \coordinate (t1) at (-1.2,.8) {};
      \coordinate (t2) at (-.6,.8) {};
      \coordinate (t3) at (0,.8) {};
      \coordinate (t4) at (.6,.8) {};
      \coordinate (t5) at (1.2,.8) {};

      \node [ycol,vertex,small] (m1) at (-1.5,0) {};
      \node [ycol,vertex,small] (m2) at (-.5,0) {};
      \node [ycol,vertex,small] (m3) at (.5,0) {};
      \node [ycol,vertex,small] (m4) at (1.5,0) {};

      \node [rcol,vertex,small] (m21) at (-.875,.4) {};
      \node [rcol,vertex,small] (m22) at (-.675,.4) {};
      \node [rcol,vertex,small] (m23) at (-.5,.4) {};
      \node [rcol,vertex,small] (m24) at (-.325,.4) {};
      \node [rcol,vertex,small] (m25) at (-.125,.4) {};
      \node [rcol,vertex,small] (m31) at (.125,.4) {};
      \node [rcol,vertex,small] (m32) at (.275,.4) {};
      \node [rcol,vertex,small] (m33) at (.425,.4) {};
      \node [rcol,vertex,small] (m34) at (.575,.4) {};
      \node [rcol,vertex,small] (m35) at (.725,.4) {};
      \node [rcol,vertex,small] (m36) at (.875,.4) {};
      \node [rcol,vertex,small] (m41) at (1.3,.4) {};
      \node [rcol,vertex,small] (m42) at (1.7,.4) {};

      \draw [bcol,edge] (s) -- (m1);
      \draw [bcol,edge] (s) -- (m2);
      \draw [bcol,edge] (s) -- (m3);
      \draw [bcol,edge] (s) -- (m4);

      \draw [ycol,edge] (m2) -- (m21);
      \draw [ycol,edge] (m2) -- (m22);
      \draw [ycol,edge] (m2) -- (m23);
      \draw [ycol,edge] (m2) -- (m24);
      \draw [ycol,edge] (m2) -- (m25);
      \draw [ycol,edge] (m3) -- (m31);
      \draw [ycol,edge] (m3) -- (m32);
      \draw [ycol,edge] (m3) -- (m33);
      \draw [ycol,edge] (m3) -- (m34);
      \draw [ycol,edge] (m3) -- (m35);
      \draw [ycol,edge] (m3) -- (m36);
      \draw [ycol,edge] (m4) -- (m41);
      \draw [ycol,edge] (m4) -- (m42);

      \draw [rcol,edge] (m21) -- (t1);
      \draw [rcol,edge] (m21) -- (t2);

      \draw [rcol,edge] (m22) -- (t1);
      \draw [rcol,edge] (m22) -- (t2);

      \draw [rcol,edge] (m23) -- (t3);
      \draw [rcol,edge] (m23) -- (t4);
      \draw [rcol,edge] (m23) -- (t5);

      \draw [rcol,edge] (m24) -- (t4);
      \draw [rcol,edge] (m24) -- (t5);

      \draw [rcol,edge] (m25) -- (t4);
      \draw [rcol,edge] (m25) -- (t5);

      \draw [rcol,edge] (m31) -- (t3);
      \draw [rcol,edge] (m31) -- (t4);
      \draw [rcol,edge] (m31) -- (t5);

      \draw [rcol,edge] (m32) -- (t3);
      \draw [rcol,edge] (m32) -- (t4);
      \draw [rcol,edge] (m32) -- (t5);

      \draw [rcol,edge] (m33) -- (t4);
      \draw [rcol,edge] (m33) -- (t5);

      \draw [rcol,edge] (m34) -- (t4);
      \draw [rcol,edge] (m34) -- (t5);

      \draw [rcol,edge] (m35) -- (t4);
      \draw [rcol,edge] (m35) -- (t5);

      \draw [rcol,edge] (m36) -- (t4);
      \draw [rcol,edge] (m36) -- (t5);

      \draw [rcol,edge] (m41) -- (t4);
      \draw [rcol,edge] (m41) -- (t5);
      \draw [rcol,edge] (m42) -- (t4);
      \draw [rcol,edge] (m42) -- (t5);

      \node at (0, -1.3) {$\lambda \then (\id_d \tri \rho)$};
    \end{scope}
  \end{tikzpicture}
  \caption{Bicomodule law, as in \cref{eqn.bimod_coherence}.}
\end{figure}

\begin{corollary}[$d$-sets as $(0,d)$-bicomodules]\label{prop_0_comodules}
  \cref{prop_bicomodules} tells us in particular that bicomodules
  $0\bito d$ are $d$-sets, since $d\set[0] \cong
  d\set$.

  Moreover, the carrier of such a bicomodule $0\bito[m]d$ is
  constant: $m=M$ where $M\in\smset$ is the set of elements of the
  corresponding $d$-set.
\end{corollary}

\begin{remark}[Indexed component $M_a$ of a $d$-set]
  Suppose given a $d$-set $0 \bito[M] d$. Then for any object $a\in d(1)$, its $a$-indexed
  component $M_a$, as defined in \cref{not.indexed_component_left}, is
  simply the set $M(a)$.
\end{remark}

\begin{example}[Hom-sets $c(a,a')$ as components {$c[a]_{a'}$}]\label{ex.hom_sets}
  Suppose given a category $c$. For any object $a\in c(1)$, we have a
  representable $c$-set $0\bito[{c[a]}]c$. For any $a'\in c(1)$,
  its $a'$-indexed component is the hom-set
  \[
    c[a]_{a'}\cong c(a,a').
  \]
\end{example}

\begin{example}[$d$-sets as left $d$-comodules with linear
  carrier]\label{prop_yon_comodules}
  In general, for arbitrary $c$ we may identify $d$-sets with
  the bicomodules $c\bito[M]d$ with constant carrier; these correspond
  to maps $d \to \smset[c] \cong \coco((c\set)\op)$ that factor
  through the full subcategory of coproducts of the initial $c$-set's
  representable functor, which is a category isomorphic to $\smset$
  (as in \cref{rem.set_copies}).

  Similarly, we may identify $d$-sets with the bicomodules $c\bito d$
  that correspond to maps $d \to \coco((c\set)\op)$ that factor
  through the full subcategory of coproducts of the \emph{terminal}
  $c$-set's representable functor, which is also a category isomorphic
  to $\smset$ (as in \cref{rem.set_copies}). In the case
  $c \coloneqq \yon$, these are the bicomodules with linear carrier,
  i.e., of the form $\yon\bito[A\yon]d$ for $A$ a set.
\end{example}

\begin{remark}[Unreasonably effective notation]\label{rem.fascinating_representable}
  Let $c\bito[m]d$ be a bicomodule. The notation for the underlying polynomial
  \begin{equation}\label{eqn.gorgeous_notation}
    m=\sum_{i\in m(1)}\yon^{m[i]}
  \end{equation}
  is surprisingly well-suited for regarding $m$ as an object of $d\set[c]$.

  First, the notation $m(1)$ in \eqref{eqn.gorgeous_notation} could refer to either
  \begin{itemize}
  \item the set one obtains by applying the polynomial functor
    $m\colon\smset\to\smset$ to the terminal object $1\in\smset$, or
  \item the $d\set$ one obtains by applying the prafunctor $m\colon
    c\set\to d\set$ to the terminal object $1\in c\set$.
  \end{itemize}
  It turns out that the first is always equal to the set of elements of the
  second, fitting our notation all along; see
  \cref{notation.main,prop_0_comodules}. Thus both readings are correct: the ambiguity
  disappears.

  Second, we saw in \cref{prop_right_comodules} that for each
  $i\in m(1)$, the set $m[i]$ is endowed by $\rho$ with the additional
  structure of (the elements of) a $c\set$.  Hence for each object $a$
  of $d$, the $a$-indexed component of $m$
  \[m_a = \sum_{i\in m_a(1)}\yon^{m[i]}\] could refer to either
  \begin{itemize}
  \item the polynomial $m_a\colon \smset \to \smset$ where $\yon^{m[i]}$
    denotes the representable functor $\smset(m[i], \thg)$, or
  \item the $c$-variable polynomial $m_a\colon c\set\to \smset$, where
    $\yon^{m[i]}$ denotes the representable functor $c\set(m[i], \thg)$,
    outputting the component at $a$ of the functor
    $m \colon c\set\to d\set$.
  \end{itemize}

  In other words, the notation for the polynomial $m$ as in
  \eqref{eqn.gorgeous_notation} simultaneously serves as a
  single-variable polynomial and a $c$-variable polynomial, which
  applied to a $c$-set yields the elements of its assigned
  $d$-set. Thus again, both readings are correct: the ambiguity
  disappears. (But be aware the entire ``sum'', ranging over the elements
  of a $d$-set, does not denote a coproduct in $d\set[c]$.)
\end{remark}

To finish concretely verifying the equivalence
$\ccomod(\poly) \cong \ccatsharp$, it would remain to show that
bicomodule maps correspond to natural transformations of prafunctors,
bicomodule composition corresponds to prafunctor composition, and that
horizontal composition of bicomodule maps corresponds to horizontal
composition of natural transformations. However, we do not find
spelling out these details to be particularly enlightening, and since
the equivalence was already established abstractly in the previous
section, we leave such calculations to the interested reader.

\begin{notation}
  The equivalence
  \[\comod(\poly)(c,d)\simeq d\set[c],\] 
  adds a sixth equivalent category to the list in
  \cref{cor.TFAE_indexed_duc}.  We will generally default to the
  notation $d\set[c]$ for all six. For example given $m\in d\set[c]$
  we can think of it as a functor $m\colon d\to\smset[c]$, as a
  functor $m\colon d\to\coco((c\set)\op)$, as a prafunctor
  $m\colon c\set\to d\set$, as a connected limit preserving functor
  $m\colon c\set\to d\set$, or as a bicomodule $c\bito[m]d$. In
  particular, we have
  \[
    d\set\coloneqq d\set[0]\cong\catfun{d}{\smset}
    \qqand
    \smset[c]\coloneqq \yon\set[c]\cong\coco((c\set)\op).
  \]
\end{notation}

\begin{remark}[Carrier of bicomodule
  composite]\label{rem.bimodule_comp_combinatorial}
  In general, the composite $m\tri_dn\in e\set[c]$ of bicomodules
  $c\bito[n]d\bito[m]e$
  has carrier polynomial
  \begin{equation}\label{eqn.composite_bico}
    m\tri_dn\coloneqq\sum_{i\in m(1)}\sum_{j\in
      d\set(m[i],n(1))}\yon^{\left(\colim\limits_{x\in
          m[i]}n[j(x)]\right)}
  \end{equation}
  where $\colim_{x\in m[i]}n[j(x)]$ denotes the colimit of the
  $c$-sets $n[j(x)]$, taken over elements $x\in(\el_dm[i])\op$.
\end{remark}

\begin{remark}[Indexed component $m_a$ of a
  bicomodule]\label{not.indexed_component}
  Recall \cref{not.indexed_component_left}: for $m$ a left $d$-comodule
  and $a$ an object in $d$, we obtain the $a$-indexed component
  $m_a$. If $m$ moreover has the structure of a right $c$-comodule, $m_a$
  inherits the structure of a right $c$-comodule from $m$. (Indeed, we saw
  in \cref{prop_right_comodules} that a right $c$-comodule amounts to
  $c$-set structure on the direction set of each position
  independently.) When $m$ is a $(c,d)$-bicomodule, the
  right $c$-comodule structure on $m_a$ can be written as the composite
  \[
    c\bito[m]d\bito[\yon^{d[a]}]\yon
  \]
  where the right comodule $\yon^{d[a]}$ corresponds to the
  representable $d$-set with representing object $a$.   
\end{remark}

\section{Structures in $\ccatsharp$}\label{chap.structures}
We have now have a good enough working knowledge of $\ccatsharp$ that
we can begin both to locate familiar friends from category
theory---functors, profunctors, etc.---and to explore
universal constructions--adjoints, Kan extensions, etc.--within it.


\begin{proposition}[Adjoints in $\ccatsharp$]\label{prop.adjoint_prafunctors}
  Let $F\colon c\set\to d\set$ be a parametric right adjoint regarded
  as a $(c,d)$-bicomodule.
  \begin{itemize}
  \item $F$ has a right adjoint (i.e., is a left adjoint) within
    $\ccatsharp$ if and only if $F$ has a right adjoint
    $d\set\to c\set$ within the category of large categories.
  \item $F$ has a left adjoint (i.e., is a right adjoint) within
    $\ccatsharp$ if and only if $F$ has a left adjoint
    $d\set\to c\set$ within the category of large categories
    and that it too is a parametric right adjoint.
  \end{itemize}
\end{proposition}
\begin{proof}
  Right adjoints are in particular parametric right adjoints. The
  result follows from
  $d\set[c] \cong \Praf(c\set,d\set) \hookrightarrow
  \catfun{c\set}{d\set}$, specifically that this sub bicategory is
  locally full.
\end{proof}

\begin{remark}\label{prop.right_adjoint_prafunctors}
  If $F$ has a left adjoint (i.e., is a right adjoint) within the
  category of large categories, then in general one \emph{cannot}
  conclude that $F$ has a left adjoint in $\ccatsharp$.
\end{remark}

We now recall a basic fact that will help us orient
ourselves.

\begin{proposition}\label{prop.profunctor_left}
  The following bicategories are equivalent (objects are categories
  and arrows between objects $c$ and $d$ are as follows):
  \begin{enumerate}
  \item profunctors $c\op\times d\to\smset$, and
  \item left adjoint functors $c\set\to d\set$.
  \end{enumerate}
\end{proposition}
\begin{proof}
By currying, a profunctor $c\op\times d\to\smset$ can be viewed as a
functor $c\op \to d\set$. Since $c\set$ is the free cocompletion of
$c\op$, functors $c\op \to d\set$ are in correspondence with
colimit-preserving functors $c\set \to d\set$, which are the same as
left adjoints. One can check that the horizontal composites agree as
well.
\end{proof}

Recall the notation $m_a$ from \cref{not.indexed_component}.

\begin{definition}[Conjunctive bicomodule]\label{def.conjuctive}
  We say that a bicomodule $c\bito[m] d$ is \emph{conjunctive} if
  for each $a\in d(1)$ the $c$-variable polynomial $m_a$ is
  representable (conjunctive in the sense of \cref{ex.variable_is_duc}),
  i.e., if $m_a(1)=1$ for each object $a\in d(1)$.
\end{definition}

\begin{proposition}\label{prop.profunctor}
  The following bicategories are equivalent (objects are categories
  and arrows between objects $c$ and $d$ are as follows):
  \begin{enumerate}
  \item conjunctive $(c,d)$-bicomodules,
    \item profunctors $d\op\times c\to\smset$, \emph{with 2-cells
        between profunctors reversed}, and
    \item right adjoint functors $c\set\to d\set$.
  \end{enumerate}
\end{proposition}
\begin{proof}
  For $1 \iff 3$, a functor $c\set\to d\set$ is right adjoint if and
  only if it preserves limits, if and only if its component
  $c\set \to \smset$ at each object of $d$ preserves limits (since
  limits in $d\set$ are taken pointwise), if and only if each such
  component is representable.

  Now $2\iff 3$ follows from \cref{prop.profunctor_left}, but we may
  also see $1 \iff 2$ directly. By currying, a profunctor
  $d\op\times c\to\smset$ can be viewed as a functor $d\op \to
  c\set$. Taking the opposite functor (this step responsible for
  reversal of 2-cells) and composing with the Yoneda embedding, we
  obtain a functor
  $d\to(c\set)\op\to\coco((c\set)\op)\cong \smset[c]$, i.e., a
  bicomodule. We see profunctors correspond to those bicomodules
  $d\to \coco((c\set)\op)$ which factor through the Yoneda
  embedding, i.e., having the form
  \[
    p\cong\sum\limits_{a\in d(1)}\yon^{P_a}
  \]
  for $P_a\colon c\to\smset$ arbitrary. These are exactly the
  conjunctive bicomodules. Note this bicomodule as a functor
  $c\set \to d\set$ is indeed the right adjoint of the left adjoint
  functor $d\set \to c\set$ that extends the profunctor
  $d\op \to c\set$ as in \cref{prop.profunctor_left}.
\end{proof}

In later sections we will denote by $\ccatsharpcon$ the (not full)
locally full sub framed bicategory of $\ccatsharp$ consisting of
conjunctive bicomodules.

\begin{example}[Not all conjunctive bicomodules are right adjoints in $\ccatsharp$]
  Let $d\coloneqq\fbox{$\bullet^1\tto\bullet^2$}$ be the
  parallel-arrows category. The diagonal functor $\smset\to d\set$
  corresponds to a profunctor $d\profunctor 1$ and hence to a
  bicomodule of the form $\yon\bito[2\yon]d$.%
  \footnote{The carrier is $2\yon^1$ because each of the two objects
    in $d$ is assigned $\smset(1,\thg)$. As a profunctor
    $d\op\times 1\to\smset$, it is constantly $1$.}  It has a left
  adjoint functor $d\set\to\smset$, but this left adjoint (given by
  taking the coequalizer of sets $X_1\tto X_2$) is itself not a
  prafunctor. One can see that by the fact that coequalizers do not
  preserve connected limits.
\end{example}

The following generalizes \cref{prop.JoshMeyers}.

\begin{proposition}[Left Kan extensions in $\ccatsharp$]\label{prop.generalcoclosure}
  $\ccatsharp$ has left Kan extensions (in the sense of Kan
  extensions in a general 2-category). That is, for any bicomodules $c\bito[p]
  e$ and $c \bito[q] d$ we
  can define a bicomodule $\lenskan{p}{q}{d}{e} \colon d\bito e$ via adjunction
  \begin{equation}\label{eqn.adjunction_generalcocolosure}
    e\set[d]\left(\lenskan{p}{q}{d}{e},p'\right) \cong e\set[c](p,p'\tri_d q).
  \end{equation}
  
  In other words, for any diagram as shown left
  
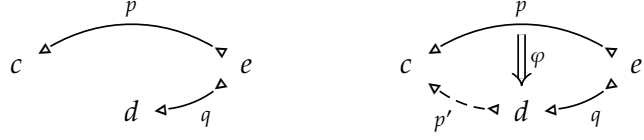
\[
    \begin{tikzcd}[row sep=0]
      c\ar[rr, bend left, biml-bimr, "p", ""' name=P]\ar[dr, bend right=15pt, biml-bimr, "q"']&&
      e\\ &
      d
    \end{tikzcd}
    \hspace{.6in}
    \begin{tikzcd}[row sep=0]
      c\ar[rr, bend left, biml-bimr, "p", ""' name=P]\ar[dr, bend right=15pt, biml-bimr, "q"']&&
      e\\ &
      |[alias=D]|d\ar[ur, bend right=15pt, dashed, biml-bimr, "p'"']
      \ar[from=P-|D, to=D, Rightarrow, "\varphi"]
    \end{tikzcd}
  \]
  the category consisting of bicomodules $d\bito[p']e$ and 2-cells
  $\varphi\colon p\to p'\tri_dq$, as shown right, has an initial
  object, which we denote $\lenskan{p}{q}{d}{e}$, carried by the polynomial:
  \begin{equation}\label{eqn.generalcoclosure}
    \lenskan{p}{q}{d}{e}\coloneqq\sum_{i \in p(1)}\yon^{q\tri_c p[i]}.
  \end{equation}
\end{proposition}
\begin{proof}
  We begin by explaining how to read the formula, and then show that
  it indeed has the structure of a $(d,e)$-bicomodule.
  First note that the right $c$-comodule structure on $p$ endows each
  $p[i]$ with the structure of a $c\set$. Thus we can consider it as
  a bicomodule $p[i] \colon 0\bito c$ at which point the composite
  $q\tri_c p[i]$ in \eqref{eqn.generalcoclosure} makes
  sense.
  Hence we find a canonical $d$-set structure on
  the composite
  \[q\tri_cp[i]\cong\sum_{j\in q(1)}c\set(q[j],p[i])\] 
  for each $i\in p(1)$. This composite varies contravariantly over
  $i\in\el_ep(1)$. Indeed, given a map $f\colon a\to a'$ in $e$, we
  get a map we'll temporarily call $f\lpush \colon p_a(1)\to
  p_{a'}(1)$. This gives a map of $c$-sets $p[f\lpush (i)]\to p[i]$, and
  hence a map $q\tri_cp[f\lpush (i)]\to q\tri_cp[i]$. Thus we have
  established a $(d,e)$-bicomodule structure on $\lenskan{p}{q}{d}{e}$.

  We next give the map $p\to \lenskan{p}{q}{d}{e}\tri_dq$ in $e\set[c]$. Its type is
  \[
    \int\limits_{a{\in}e(1)}\prod_{i\in p_a(1)}\sum_{i'\in p_a(1)}\sum_{\varphi\in d\set(q\tri_cp[i'],q(1))}\lim_{(b,j,w)\in\el_d(q\tri_cp[i'])}c\set(q[j],p[i]).
  \]
  As complicated as it looks, giving an element of this type is straightforward. Indeed, we simply use
  \[
    a\mapsto i\mapsto \Big(i, \big(((j,w)\mapsto j), (b,j,w)\mapsto x\mapsto w(x)\big)\Big)
  \]
  where $a\in e(1)$, $i'\coloneqq i\in p_a(1)$, $j\in q(1)$, $w\in c\set(q[j],p[i'])$, $\varphi(j,w)\coloneqq j$, $b\in d(1)$, $x\in q[j]$, and $w(x)\in p[i']=p[i]$.

  It remains to show that the above is universal, so suppose given $d\bito[p']e$ and a map $\psi\colon p\to p'\tri_d q$ in $e\set[c]$. It has type
  \[
    \psi\in\int\limits_{a\in e(1)}\prod_{i\in p_a(1)}\sum_{i'\in p'_a(1)}\sum_{\varphi\in d\set(p'[i'],q(1))}\lim_{(b,w')\in\el_dp'[i']}c\set(q[\varphi w'],p[i]).
  \]
  In order to give a map $\lenskan{p}{q}{d}{e}\to p'$ in $e\set[d]$ we need something of type
  \[
    \int\limits_{a{\in}e(1)}\prod_{i\in p_a(1)}\sum_{i'\in p'_a(1)}d\set(p'[i'],q\tri_cp[i]),
  \]
  and one can check that this type is isomorphic to that defining $\psi$. We leave the remaining details to the reader.
\end{proof}

\subsection{The full sub framed bicategory $\ccatsharpdisc$}\label{sec.catsharpdisc}

Of special interest to us is the full sub framed bicategory
$\ccatsharpdisc\ss\ccatsharp$ spanned by the discrete objects. This is
the framed bicategory of \emph{set}-variable polynomials as mentioned in the
introduction.

\begin{remark}[Multi-variable multi-output polynomials]
  We have seen that bicomodules $m \in d\set[c]$ are $c$-variable
  $d$-valued polynomials. In particular if $c \coloneqq C\yon$ and
  $d \coloneqq D\yon$ are discrete categories (\cref{ex.discrete_multivariable}),
  then
\[m \coloneqq \sum_{i \in m(1)}\yon^{P_i}\] is seen quite literally as
a polynomial in multiple input and output variables, where for each
term $i$ the component $P_i(a)$ of the $c$-set $P_i$ at $a \in C$ is
the exponent of the $a$-indexed input variable, and the component
$m_b$ of the functor $c\set \to d\set$ at $b\in D$ produces the value
of the $b$-indexed output variable:
\[m_b \cong \sum_{i \in m_b}\prod_{a \in C}\yon^{P_i(a)}.\]

\end{remark}

\begin{remark}[Colored corollas]\label{thm.GK_myform}
  Let $C,D\in\smset$ be sets, regarded as discrete categories
  $C\yon, D\yon$ in $\ccatsharp$. By abuse of notation, we will denote
  the category of bicomodules between them by
  \[
    D\set[C]\coloneqq (D\yon)\set[C\yon].
  \]
  
  A bicomodule $m \in D\set[C]$ is simply a map
  from $m$-positions to elements of $D$ and a map from $m$-directions
  to elements of $C$.
  
  \begin{figure}[h]
    \centering
    \begin{tikzpicture}[corollas]
      \begin{scope}[shift={(-2,0)}]
        \node [ycol,vertex,large] (s) at (0, -.8) {};
        \coordinate (t1) at (-1.2,.8) {};
        \coordinate (t2) at (-.6,.8) {};
        \coordinate (t3) at (.0,.8) {};
        \coordinate (t4) at (.6,.8) {};
        \coordinate (t5) at (1.2,.8) {};
        \draw [ycol,edge] (s) -- (t1);
        \draw [ycol,edge] (s) -- (t2);
        \draw [ycol,edge] (s) -- (t3);
        \draw [ycol,edge] (s) -- (t4);
        \draw [ycol,edge] (s) -- (t5);
        \node [annot] at (.3, -.8) {\rlap{$\in m(1)$}};
      \end{scope}
      \node at (0, 0){$\mapsto$};
      \begin{scope}[shift={(2,0)}]
        \node [bcol,vertex] (s) at (0, -.8) {};
        \coordinate (t1) at (-1.2,.8) {};
        \coordinate (t2) at (-.6,.8) {};
        \coordinate (t3) at (-0,.8) {};
        \coordinate (t4) at (.6,.8) {};
        \coordinate (t5) at (1.2,.8) {};
        
        \node [ycol,vertex] (s1) at (0,-.3) {};
        
        \node [rcol,vertex] (m1) at (-1.2,.4) {};
        \node [rcol,vertex] (m2) at (-.6,.4) {};
        \node [rcol,vertex] (m3) at (0,.4) {};
        \node [rcol,vertex] (m4) at (.6,.4) {};
        \node [rcol,vertex] (m5) at (1.2,.4) {};
        
        \draw [bcol,edge] (s) -- (s1);
        
        \draw [ycol,edge] (s1) -- (m1);
        \draw [ycol,edge] (s1) -- (m2);
        \draw [ycol,edge] (s1) -- (m3);
        \draw [ycol,edge] (s1) -- (m4);
        \draw [ycol,edge] (s1) -- (m5);

        \draw [rcol,edge] (m1) -- (t1);
        \draw [rcol,edge] (m2) -- (t2);
        \draw [rcol,edge] (m3) -- (t3);
        \draw [rcol,edge] (m4) -- (t4);
        \draw [rcol,edge] (m5) -- (t5);

        \node [annot] at (.2, -.8) {\rlap{$\in D$}};
        \node [annot] at (.2, -.3) {\rlap{$\in m(1)$}};
        \node [annot] at (1.35, .4) {\rlap{$\in C$}};
      \end{scope}
    \end{tikzpicture}
    \caption{The morphism $m \to D\yon \tri m \tri C\yon$
      (built from $\lambda$ and $\rho$) for a bicomodule between
      discrete categories. The corolla shown represents a term in a
      polynomial of $C$ input variables and $D$ output variables; such
      a term is a product of variables of various types $c \in C$ and
      outputs a type $d \in D$.}
  \end{figure}

  Indeed, we have seen that an object in $D\set[C]$ consists of
  $D$-many polynomials in $C$-many variables; this is just the
  assignment of output types to terms and input types to factors.
\end{remark}

Thus $C$-variable $D$-valued polynomials are in correspondence up to
isomorphism with diagrams as in \cref{prop.bridge_diagram} such that
all the categories involved are discrete, i.e., sets $C$, $E$, $B$,
and $D$ with maps:
\begin{equation}\label{eqn.bridge}
  \begin{tikzcd}
    &E\ar[dl, "s"']\ar[r, "\pi"]&B\ar[dr, "t"]\\[-10pt]
    C&&&D
  \end{tikzcd}
\end{equation}

Here $E \To{\pi} B$ is a polynomial as a bundle
(\cref{rem.poly_bundle}), so $t$ assigns to its positions elements
of $D$ and $s$ assigns to its directions elements of $C$.

As explained earlier, such a \emph{bridge diagram} neatly decomposes
the associated functor $C\set \to D\set$ as
$\Sigma_t\circ\Pi_\pi\circ\Delta_s$ (where as usual $\Delta_f$ for
$f\colon X \to Y$ denotes the reindexing functor $Y\set \to X\set$ via
$f$, and $\Sigma_f$ and $\Pi_f$ respectively denote its left and right
adjoint).

\begin{proposition}\label{thm.GK_migration}
  Let $m \in D\set[C]$. Its bridge diagram is
  \[C\From{s} \pnt{m}(1) \To{\pi} m(1) \To{t} D \] where $s$ is the
  map induced by the right $C$-comodule structure, $\pi$ is the
  polynomial carrier of $m$ viewed as a bundle, and $t$ is the map
  induced by the left $D$-comodule structure. The functor
  $C\set\to D\set$ described by this bridge diagram,
  $\Sigma_t\circ\Pi_\pi\circ\Delta_s$, is the same as that described
  by $m$ as in \eqref{thm.garner}.
\end{proposition}
\begin{proof}
  The parametric right adjoint described by $m$ sends $X\in D\set$ and
  $b\in D$ to
  \[
    \sum_{i\in m_b}C\set(m[i],X).
  \]
  Indeed, $\Delta_s$ reindexes according to $s$; then, at each
  position $i \in m(1)$, $\Pi_\pi$ gives the product of sets indexed
  by elements of $\pnt{m}$ over $i$ (directions from $i$); then, at
  each object $b \in D$, $\Sigma_t$ gives the sum of sets indexed by
  elements of $m(1)$ over $b$ (positions of $m_b$). So

  \[\Sigma_t\circ\Pi_\pi\circ\Delta_s(X)(b) \cong \sum_{i\in
      m_b}\prod_{j\in m[i]}X(s(j)) \cong \sum_{i\in m_b}C\set(m[i],X).\]
\end{proof}

\begin{remark}\label{rem.bridge_lin_con}
  The linear and conjunctive $(C,D)$-bicomodules are those with bridge
  diagrams of the form
  \[
    C\From{s}E=\!=E\To{t}D \qqand C\From{s}E\To{\pi}D=\!=D
  \]
  respectively, and as functors may respectively be written
  $\Sigma_t\circ\Delta_s$ and $\Pi_\pi\circ\Delta_s$.
\end{remark}

\begin{remark}[Combinatorial descriptions in
  $\ccatsharpdisc$]\label{rem.comb_maps_catsharpdisc}
  The combinatorial description of bicomodule maps
  from \cref{rem.comb_bicomodule_maps} simplifies greatly when the
  categories involved are
  discrete. We have the following combinatorial
  description of morphisms between $m,n\in D\set[C]$:
  \begin{equation}\label{eqn.bicomod_map_combinatorics_disc}
    D\set[C](m,n)\cong\sum_{\varphi_1\in
      D\set(m(1),n(1))}\;\prod_{i\in
      m(1)}C\set\big(n[\varphi_1(i)],m[i]\big).
  \end{equation}
  In other words, a $(C,D)$-bicomodule map is a map $\varphi$ of underlying
  polynomials such that $\varphi_1$ preserves assigned elements of $D$
  and each $\varphi^\sharp_i$ preserves assigned elements of $C$.

  The combinatorial description \eqref{eqn.composite_bico} of
  composition $C\yon\bito[n]D\yon\bito[m]E\yon$ also simplifies
  from a colimit to a coproduct in the discrete categories case,
  becoming:
  \begin{equation}\label{eqn.composite_bico_linear}
    m\tri_{D\yon}n\coloneqq\sum_{i\in m(1)}\;\sum_{j\in
      d\set(m[i],n(1))}\yon^{\;\sum\limits_{x\in m[i]}n[j(x)]}.
  \end{equation}

  In other words, the positions of $m\tri_{D\yon}n$ are positions of
  $m \tri n$ (two-layer trees) such that the $n$-position at each
  $m$-direction matches its assigned element of $D$.
\end{remark}

\begin{remark}\label{rem.polyfun}
  In \cite{kock2012polynomial}, Gambino and Kock define a framed
  bicategory $\ppolyfun_\call{E}$ given an arbitrary locally cartesian
  closed category $\call{E}$. The horizontal morphisms are polynomial
  functors (i.e., functors that can be expressed in the form
  $\Sigma_t\circ\Pi_\pi\circ\Delta_s$) between slice categories of
  $\call{E}$, which they identify with bridge diagrams
  as in \eqref{eqn.bridge} (up to isomorphism). The bicategory's 2-cells are
  natural transformations satisfying a certain ``strength'' condition
  (disappearing in the case $\call{E} = \smset$), which they
  identify with diagrams of the form
  \begin{equation}\label{eqn.bridge_map}
    \begin{tikzcd}
      C\ar[dd,equal]&E\ar[l, "s"']\ar[r,"\pi"]&B\ar[r,"t"]\ar[d, equal]&D\ar[dd,equal]\\
      &\bullet\ar[dr, phantom, very near start,
      "\lrcorner"]\ar[r]\ar[d]\ar[u, "\varphi^\sharp"]&B\ar[d,
      "\varphi_1"]\\
      C'&E'\ar[l,"s'"]\ar[r,"\pi'"']&B'\ar[r,"t'"']&D'
    \end{tikzcd}
  \end{equation}
  (up to isomorphism); note this is just \cref{eqn.map_bundle_view}
  with the additional constraint that the assigned elements of $C$ and
  $D$ are preserved, as expected. The square 2-cells are similar to
  \eqref{eqn.bridge_map}, but with arbitrary maps instead of
  identities along the sides.

  In particular, $\ccatsharpdisc$ is identified with
  $\ppolyfun_{\smset}$.
\end{remark}

\begin{lemma}\label{lemma.right_adj_GK}
  For sets $C,D\in\smset$, every right (resp.\ left) adjoint
  $C\set\to D\set$ corresponds to a right (resp.\ left) adjoint in
  $\ccatsharp$.
\end{lemma}
\begin{proof}
  \cref{prop.profunctor} tells us the bicomodules describing right
  adjoint functors are precisely the conjunctive ones, so by
  \cref{rem.bridge_lin_con}, a right adjoint $C\set\to D\set$ is
  represented as $\Pi_g\circ\Delta_f$ for some
  $C\From{f}E\To{g}D$. Its left adjoint is
  $\Sigma_f\circ\Delta_g$. Since functors between discrete categories,
  such as $f\colon E\to C$, are \'etale, we have by
  \cref{prop.functor_adj_bico} that $\Sigma_f\circ\Delta_g$ is a
  parametric right adjoint and thus a left adjoint in $\ccatsharp$.
\end{proof}

\begin{corollary}[Adjoint 1-cells in $\ccatsharpdisc$]\label{prop.adjoint_bicoms}
  Let $C,D\in\smset$, and let $C\yon\bito[m] D\yon$ be a bicomodule. Then
  \begin{enumerate}
  \item $m$ is a left adjoint in $\ccatsharp$ iff $m$ is linear.
  \item $m$ is a right adjoint in $\ccatsharp$ iff $m$ is conjunctive.
  \end{enumerate}
  Specifically, given a linear bicomodule with carrier
  $\sum_{a\in C}\sum_{b\in D}E_{a, b}\yon$ and bridge diagram
  \[
    C\From{f}E=\!=E\To{g}D,
  \]
  its right adjoint is the conjunctive bicomodule with carrier
  $\sum_{a\in C}\yon^{\sum_{b\in D}E_{a,b}}$ and bridge diagram
  \[
    D\From{g}E\To{f}C=\!=C.
  \]
\end{corollary}

Following \cref{notation.main}, we will denote the left adjoint of $F$
by $F\ldag$ and the right adjoint of $F$ by $F\rdag$.

\begin{example}\label{ex.polys_and_bundles}
  Given a polynomial $p\in\poly$, we note two ways of capturing its
  bundle form $\pnt{p}(1)\to p(1)$ as in \eqref{eqn.poly_bundleform}
  inside of $\ccatsharpdisc$:
  \[
    \yon\bito[p]p(1)\yon
    \qqand
    p(1)\yon\bito[p\ldag]\yon
  \]
  with respective bridge diagrams
  \[
    1\from \pnt{p}(1)\To{\pi} p(1)=\!=p(1)
    \qqand
    p(1)\From{\pi} \pnt{p}(1)=\!=\pnt{p}(1)\to 1.
  \]
  If $p=\sum_{i\in p(1)}\yon^{p[i]}$ then $p\ldag=\sum_{i\in
    p(1)}p[i]\yon$. As a functor $\smset\to p(1)\set$, the former
  sends $X$ to the $p[i]$-fold product $X^{p[i]}$ for $i\in p(1)$. As
  a functor $p(1)\set\to\smset$, the latter sends $X\to p(1)$ to its
  pullback along $\pnt{p}(1)\to p(1)$.

  These two representations correspond to two sorts of maps one can
  imagine between bundles. Indeed, consider 2-cells of the form
  \[
    \begin{tikzcd}
      \yon\ar[d, equal]\ar[r, biml-bimr, "p"]\ar[dr, phantom,
      "\phantom{f^\sharp}\Downarrow f^\sharp"]&p(1)\yon\ar[d, "f"]&&
      p(1)\yon\ar[d, "f"']\ar[r, biml-bimr, "p\ldag"]&\yon\ar[d,
      equal]\ar[dl, phantom, "\phantom{f^\sharp}\Downarrow f^\flat"]\\
      \yon\ar[r, biml-bimr, "q"']& q(1)\yon&& q(1)\yon\ar[r,
      biml-bimr, "q\ldag"']&\yon
    \end{tikzcd}
  \]
  The data $f\colon p(1)\to q(1)$ is the same in both cases. One
  checks that $f^\sharp$ is the data of a map $q[f(i)]\to p[i]$ for
  each $i\in p(1)$; in other words the left-hand side is just another
  representation of a map of polynomials $p\to q$. On the other hand,
  $f^\flat$ is the data of a map $p[i]\to q[f(i)]$ for each $i\in p(1)$;
  in other words it is just a commutative square. Here are
  set-theoretic representations of the 2-cells above:
  \[
    \begin{tikzcd}
      \pnt{p}(1)\ar[r]&p(1)\ar[d, equal]\\
      \bullet\ar[u, "f^\sharp"]\ar[d]\ar[dr,
      phantom, very near start, "\lrcorner"]\ar[r]&p(1)\ar[d,
      "f"]\\
      \pnt{q}(1)\ar[r]&q(1)
    \end{tikzcd}
    \hspace{.5in}
    \begin{tikzcd}
      p(1)\ar[d, "f"']&\pnt{p}(1)\ar[d,"f^\flat"]\ar[l]\\
      q(1)&\pnt{q}(1)\ar[l]
    \end{tikzcd}
  \]

  The above recovers the operation referred to in \cite[Section
  4]{spivak2020dirichlet} as the \emph{Dirichlet transform} between
  polynomials and Dirichlet series. Whereas in that reference the
  operation was regarded as purely syntactic, we are now able to see
  that it has a universal property: it is the adjoint
  $p\mapsto p\ldag$.
\end{example}

\subsubsection{The sub framed bicategory $\sspan\cong\ccatsharplin$}\label{sec.sspan}

Recall that we refer to a polynomial $p$ as linear if it is of the
form $p\cong P\yon$ for some set $P$. Let
$\ccatsharplin\ss\ccatsharpdisc\ss\ccatsharp$ denote the (not full)
locally full sub framed bicategory consisting of polynomial comonads
with linear carrier (i.e., discrete categories) and bicomodules with
linear carrier as horizontal maps.

\begin{figure}[h]
  \centering
  \begin{tikzpicture}[corollas]
    \begin{scope}[shift={(-1.4,0)}]
      \node [ycol,vertex,large] (s) at (0, -.8) {};
      \coordinate (t1) at (.0,.8) {};
      \draw [ycol,edge] (s) -- (t1);
      \node [annot] at (.3, -.8) {\rlap{$\in M$}};
    \end{scope}
    \node at (0, 0){$\mapsto$};
    \begin{scope}[shift={(1.4,0)}]
      \node [bcol,vertex] (s) at (0, -.8) {};
      \coordinate (t1) at (-0,.8) {};
      
      \node [ycol,vertex] (s1) at (0,-.3) {};
      
      \node [rcol,vertex] (m1) at (0,.4) {};
      
      \draw [bcol,edge] (s) -- (s1);
      
      \draw [ycol,edge] (s1) -- (m1);

      \draw [rcol,edge] (m1) -- (t1);

      \node [annot] at (.2, -.8) {\rlap{$\in D$}};
      \node [annot] at (.2, -.3) {\rlap{$\in M$}};
      \node [annot] at (.2, .4) {\rlap{$\in C$}};
    \end{scope}
  \end{tikzpicture}
  \caption{The morphism
    $M\yon \to D\yon \tri M\yon \tri C\yon$ (built from $\lambda$ and
    $\rho$) for a linear bicomodule $C \yon \bito D \yon$ between
    discrete categories.}
\end{figure}

Recall that the framed bicategory $\sspan$ of spans has sets as
objects, functions as vertical morphisms, spans as horizontal
morphisms, and diagrams (left) as 2-cells (right):
\[
  \begin{tikzcd}
    C\ar[d]&S\ar[l]\ar[r]\ar[d]&D\ar[d]\\
    C'&S'\ar[l]\ar[r]&D'
  \end{tikzcd}
  \qquad\text{as}\qquad
  \begin{tikzcd}
    C\ar[r, tick, "S"]\ar[d]\ar[dr, phantom, "\Downarrow"]&D\ar[d]\\
    C'\ar[r, tick, "S'"']&D'
  \end{tikzcd}
\]

\begin{proposition}[$\sspan$ as linears]\label{prop.span_iso_lin}
  There is an isomorphism of framed bicategories
  \[\sspan\cong\ccatsharplin.\]
\end{proposition}
\begin{proof}
  The isomorphism identifies squares in $\sspan$ as to the left with
  squares in $\ccatsharplin$ as to the right:
  \[
    \begin{tikzcd}
      C\ar[r, tick, "S"]\ar[d]\ar[dr, phantom, "\Downarrow"]&D\ar[d]&&
      C\yon\ar[r, biml-bimr, "S\yon"]\ar[d]\ar[dr, phantom, "\Downarrow"]&D\yon\ar[d]\\
      C'\ar[r, tick, "S'"']&D'&&
      C'\yon\ar[r, biml-bimr, "S'\yon"']&D'\yon
    \end{tikzcd}
  \]
  Indeed, we can view spans as bridge diagrams \eqref{eqn.bridge} for
  which $S\coloneqq E=B$ and $\pi=\id$, and by
  \cref{rem.bridge_lin_con} these correspond to linear
  polynomials. As for morphisms between them, the middle column of
  \eqref{eqn.bridge_map} reduces in this case to a map $E=B\To{\varphi_1} B'=E'$
  making the outer rectangles commute, which is in agreement with maps
  of spans.
\end{proof}

\begin{corollary}\label{cor.span_left_adj}
  The framed bicategory $\sspan$ can be regarded as that of the left adjoints in $\ccatsharpdisc$. 
\end{corollary}
\begin{proof}
  This follows from \cref{prop.span_iso_lin,prop.adjoint_bicoms}.
\end{proof}


Recall from \cref{prop.adjoint_bicoms} that the corresponding right
adjoints are the conjunctive bicomodules $D\yon\bito[m]C\yon$, i.e.,
those for which $m_a(1)=1$ for all $a\in C$. We denote the bicategory of
these right adjoint bicomodules between linear comonoids by
$\ccatsharpdisccon$. Taking adjoints gives a contravariant, in both
1-cells and 2-cells, equivalence of bicategories between
$\ccatsharpdisccon$ and $\ccatsharplin$:
\begin{equation}\label{eqn.disc_con}
  \ccatsharplin(C\yon,D\yon)\cong\ccatsharpdisccon(D\yon,C\yon)\op.
\end{equation}
We will see $\ccatsharpdisccon$ again in
\cref{thm.linear_conjunctive_dual} where we will obtain an equation
much like \eqref{eqn.disc_con} but with $\ccatsharpdisccon(D\yon,C\yon)\op$
replaced by $\ccatsharpdisccon(C\yon,D\yon)\op$.

\begin{remark}\label{rem.linearconj}
  Note that the $(C\yon,D\yon)$-bicomodules underlying monoid
  homomorphisms $D\yon \to C\yon$ (i.e., functions $D \to C$) are
  precisely the linear, conjunctive ones.

  We do not call $\ccatsharpdisccon$ a sub \emph{framed} bicategory of
  $\ccatsharp$, since the left adjoints of such maps (the
  $(D\yon,C\yon)$-bicomodules underlying monoid homomorphisms
  $D\yon \to C\yon$) are often not within $\ccatsharpdisccon$.
\end{remark}

\section{The Dirichlet product}\label{chap.dirichlet}

The remainder of the paper requires an additional ingredient, namely a
monoidal structure on polynomials which interacts well with
composition. We begin by describing this first in the single-variable
case---so giving a second monoidal structure on the category
$\poly$---and then upgrade this to monoidal structures on
multi-variable polynomials. These monoidal structures will come in two
forms: on the one hand, a ``global'' monoidal structure on a category
of \emph{all} multivariable polynomials, regardless of their input-
and output-types; and on the other hand, ``local'' monoidal structures
on the categories $\polyfunb(c,d)$ for fixed $c$ and $d$. We will construct
the global structure directly from the single-variable structure on
$\poly$, and then deduce the local structure---which is of primary
interest for our applications---via the theory of monoidal fibrations
developed in~\cite{shulman2008framed}.

\subsection{The single-variable case}
The category $\poly$ is cartesian closed, but we have found
surprisingly little use for this fact. Much more useful to us is the
following monoidal structure.

\begin{proposition}[The Dirichlet monoidal structure on $\poly$]\label{prop.mon_closed}
  The category $\poly$ has a symmetric monoidal closed structure with unit $\yon$, with multiplication denoted $\otimes$, called \emph{Dirichlet product}, and given by the formula
  \begin{equation}\label{eqn.dirichlet_product}
    p\otimes q\coloneqq\sum_{(i,j)\in p(1) \times q(1)}\yon^{p[i]\times q[j]},
  \end{equation}
  and with closure given by the formula
  \[
    [p,q]\coloneqq\sum_{\varphi\colon p\to q}\yon^{\;\sum\limits_{i\in p(1)}q[\varphi_1(i)]}.
  \]
\end{proposition}
\begin{proof}
  It is easy to see directly from the formula that $\otimes$ is
  associative, unital, and symmetric. The proof that $(\thg) \otimes p$ is
  left adjoint to $[p,\thg]$ is an elementary calculation using
  \cref{prop.comb_description}. We refer the interested reader to
  \cite{spivak2022poly} (or our generalization in
  \cref{prop.local_closure}) for the details, but we give the
  evaluation map $p\otimes[p,q]\to q$ here for the reader's
  convenience. Given positions $i\in p(1)$ and $\varphi\colon p\to q$,
  evaluation returns $j\coloneqq\varphi_1(i)\in q(1)$ together with
  the function
  \[q[j]\to p[i]\times\sum_{i'\in p(1)}q[\varphi_1(i')]\] 
  given by sending $e\in q[j]$ to $(\varphi^\sharp(i,e), i, e)$.
\end{proof}

\begin{example}[Duality in $\poly$]\label{ex.duality_in_poly}
  In any symmetric monoidal closed category $\mathscr{V}$, we write
  $\dual{X} = [X,I]$ for the result of homming $X \in \mathscr{V}$
  into the monoidal unit $I$; in the terminology
  of~\cite{Dold1983Duality}, $\dual{X}$ is the \emph{weak dual} of
  $X$. Weak duality induces a contravariant adjunction
  \begin{equation}
    \label{eqn.weak_duality}
    \begin{tikzcd}[column sep=large, ampersand replacement=\&]
      \mathscr{V}\ar[r, shift left=5pt, "{\dual{(\thg)}}"] \ar[r,
      phantom, "\scriptstyle\Rightarrow"]\& \mathscr{V}\op\ar[l, shift
      left=5pt, "{\dual{(\thg)}}"]
    \end{tikzcd}
  \end{equation}
  which, like any adjunction, restricts to an equivalence on the
  \emph{fixpoints}---in this case, the objects
  $X$ for which the unit map $X \rightarrow
  \ddual{X}$ is invertible. In~\cite{Dold1983Duality}, such objects
  were called \emph{reflexive}, and so we can say that: in any
  monoidal category,
  $\dual{(\thg)}$ gives a contravariant autoequivalence of the
  subcategory of reflexive objects.

  In the case of $\poly$, if $p = \sum_{i \in p(1)} \yon^{p[i]}$, then
  \begin{equation*}
    \dual{p} = [p, \yon] = \sum_{f \in \Pi_{i \in p(1)} p[i]} \yon^{p(1)} \qquad \text{and} \qquad 
    \ddual{p} = [[p, \yon], \yon] = \sum_{f \colon \Pi_{i \in p(1)} p[i] \rightarrow p(1)} \yon^{\Pi_{i \in p(1)} p[i]}\rlap{ ,}
  \end{equation*}
  and the unit map $\eta \colon p \rightarrow \ddual{p}$ acts on
  positions by sending $i \in p(1)$ to the function
  $\Pi_{i \in p(1)} p[i] \rightarrow p(1)$ which is constant at $i$,
  and on directions by the functions
  $\eta^\sharp(i, \thg) \colon \Pi_{j \in p(1)} p[j] \rightarrow p[i]$
  which evaluate at $i$. Thus, $\eta$ will be invertible precisely
  when \emph{either} $p(1)$ is a singleton \emph{or} each $p[i]$ is a
  singleton. In other words, the reflexive objects are those of the
  form $A\yon$ and $\yon^A$. In fact, we have the isomorphisms
  \begin{equation}\label{eqn.duality_in_poly}
    \dual{(A\yon)}\cong\yon^A
    \qqand
    \dual{(\yon^A)}\cong A\yon\rlap{ ,}
  \end{equation}
  and so the adjunction~\cref{eqn.weak_duality} restricts to a
  contravariant equivalence between the subcategory of objects $A\yon$
  (which is equivalent to $\smset$) and the subcategory of objects
  $\yon^A$ (which is equivalent to $\smset\op$).
\end{example}

We will generalize the above duality story from single-variable to
multi-variable polynomials in~\cref{thm.linear_conjunctive_dual}
below. In order to do this, we need to generalize the monoidal closed
Dirichlet structure from single-variable to multi-variable
polynomials. One might hope for this to follow straightforwardly from
the equivalence $\ccatsharp \simeq \ccomod(\poly)$, but it seems it
does not. As such, we will describe a different abstract machinery
which captures the single-variable Dirichlet product, and which
generalizes directly to the multi-variable setting.

The starting point for this treatment is the observation that the
functor $P \colon \poly \rightarrow \smset$ sending $p$ to $p(1)$
is a Grothendieck fibration. Indeed:
\begin{itemize}
\item The reindexing of
$\sum_{i \in p(1)}\yon^{p[i]}$ along $f \colon Q \rightarrow p(1)$ is the
polynomial $\sum_{j \in Q} \yon^{p[f(j)]}$; 
\item The $P$-cartesian map
  $\sum_{j \in Q} \yon^{p[f(j)]} \rightarrow \sum_{i \in
    p(1)}\yon^{p[i]}$ witnessing the reindexing acts as $f$ does on
  positions, and is the identity on directions;
\item  The general $P$-cartesian maps are those
$\varphi \colon p \rightarrow q$ which act \emph{invertibly} on directions.
\end{itemize}
As observed in~\cite{spivak2020dirichlet}, the fibration $P$ is
closely related to the fibration of \emph{Dirichlet polynomials}:

\begin{definition}[Fibration of Dirichlet polynomials]\label{def.dir}
  A \emph{Dirichlet polynomial} is a functor $d \colon \smset\op
  \rightarrow \smset$ that is isomorphic to a sum of representables:
  \begin{equation*}
    d \cong \sum_{i \in d(0)} d[i]^{\yon}
  \end{equation*}
  where here we write $C^{\yon}$ for the contravariant representable
  $\smset(\thg, C) \colon \smset\op \rightarrow \smset$. We write
  $\Cat{Dir}$ for the category of Dirichlet polynomials and natural
  transformations.
\end{definition}
The morphisms $d \rightarrow e$ of $\Cat{Dir}$ can be described as
pairs comprising a map $f_0 \colon d(0) \rightarrow e(0)$ on
``positions'', and maps $f^\sharp_i \colon d[i] \rightarrow e[f(i)]$
on ``directions''. It follows that the functor
$Q \colon \Cat{Dir} \rightarrow \smset$ sending $d$ to $d(0)$ is again
a Grothendieck fibration under an identical reindexing structure as
for $\poly$. (Viewing Dirichlet polynomials as bundles of
sets, this amounts to the codomain fibration of $\smset$.)

Now the two fibrations $P \colon \poly \rightarrow \smset$ and
$Q \colon \Cat{Dir} \rightarrow \smset$ are related by
what~\cite{spivak2020dirichlet} calls the \emph{Dirichlet transform}
(cf.~\cref{ex.polys_and_bundles} above). In abstract categorical
terms, this amounts to the fact that $P$ and $Q$ are \emph{opposite
  fibrations}. In general, if $P \colon \Cat{E} \rightarrow \Cat{B}$
is a Grothendieck fibration, we can form the associated
(pseudo)functor $E \colon \Cat{B}\op \rightarrow \Cat{CAT}$; take
``fiberwise opposites'' to get a pseudofunctor
$E(\thg)\op \colon \Cat{B}\op \rightarrow \Cat{CAT}$; and apply the
Grothendieck construction. This yields a new fibration
$\tilde P \colon \tilde{\Cat{E}} \rightarrow \Cat{B}$ termed the
\emph{opposite fibration} of $P$. More concretely:
\begin{definition}[Opposite fibration]
  \label{defn.opposite_fibration}
  Let $P \colon \Cat{E} \rightarrow \Cat{B}$ be a Grothendieck
  fibration. The opposite fibration $\tilde P \colon \tilde{\Cat{E}}
  \rightarrow \Cat{B}$ has:
  \begin{itemize}
  \item Objects of $\tilde{\Cat{E}}$ the same as objects of $\Cat{E}$;
  \item Morphisms $e \rightarrow e'$ in $\tilde{\Cat{E}}$ being
    equivalence classes of $\Cat{E}$-spans:
    \begin{equation*}
      \begin{tikzcd}[column sep=1em, row sep=0.8em]
        & f \arrow[dl,"s"'] \arrow[dr,"t"] \\
        e & & e'
      \end{tikzcd}
    \end{equation*}
    where $s$ is $P$-vertical (i.e., $P(s)$ is an identity) and $t$ is
    $P$-cartesian, and where two spans $(s,t)$ and $(s',t')$ are
    identified if they are isomorphic via a vertical isomorphism
    $f \rightarrow f'$.
  \item Composition in $\tilde{\Cat{E}}$ is by pullback;
  \item $\tilde P \colon \tilde{\Cat{E}} \rightarrow \Cat{B}$
    acts as $P$ does on objects, and on morphisms sends $(s,t)$ to $Pt$.
  \end{itemize}
\end{definition}
In particular, applying this construction to
$Q \colon \Cat{Dir} \rightarrow \smset$ yields
$P \colon \poly \rightarrow \smset$. The relevance of this for us is
that $Q$ has a \emph{cartesian} monoidal structure, which transports
to a monoidal structure on the corresponding opposite fibration $P$,
which is the Dirichlet monoidal structure. Let us first recall what it
means for a fibration to be monoidal:

\begin{definition}[Monoidal fibration]\cite[Section~12]{shulman2008framed}
  A functor $P \colon \Cat{E} \rightarrow \Cat{B}$ between monoidal
  categories is called a \emph{monoidal fibration} if:
  \begin{itemize}
  \item $P$ is both a strict monoidal functor and a Grothendieck
    fibration; and
  \item The binary tensor product $\otimes$ of $\Cat{E}$ preserves
    $P$-cartesian maps.
  \end{itemize}
\end{definition}
Now, $Q$ is a cartesian monoidal fibration, i.e., a monoidal fibration
when \emph{both} $\Cat{Dir}$ and $\smset$ are equipped with the
cartesian product. Indeed, the cartesian product in $\Cat{Dir}$ can be
taken to be
\begin{equation*}
  d \times e = \sum_{(i,j) \in d(0) \times e(0)} (d[i] \times e[j])^{\yon}\rlap{ ,}
\end{equation*}
i.e., the dual formula to~\cref{eqn.dirichlet_product}. Now clearly
$Q$ preserves the (cartesian) monoidal structure strictly, and the
fact that products of $Q$-cartesian maps are $Q$-cartesian is simply
the fact that the cartesian product of isomorphisms (on directions) is
an isomorphism.

The existence of the Dirichlet monoidal structure on $\poly$ now
follows from:

\begin{lemma}[Opposite monoidal fibrations]\label{lem.opposite_monoidal_fib}
  If $P \colon \Cat{E} \rightarrow \Cat{B}$ is a monoidal 
  fibration, then the opposite fibration $\tilde P \colon \tilde{\Cat{E}}
  \rightarrow \Cat{B}$ is also monoidal, where:
  \begin{itemize}
  \item The monoidal structure of $\Cat{B}$ is the same one;
  \item The tensor product of $\tilde{\Cat{E}}$ is that of $\Cat{E}$
    on objects, and on morphisms is given by applying the tensor product componentwise:
    \begin{equation*}
      \vcenter{\hbox{
          \begin{tikzcd}[column sep=1em, row sep=0.8em]
            & f \arrow[dl,"s"'] \arrow[dr,"t"] \\
            d & & d'
          \end{tikzcd}}} \quad \otimes \quad 
      \vcenter{\hbox{
          \begin{tikzcd}[column sep=1em, row sep=0.8em]
            & g \arrow[dl,"u"'] \arrow[dr,"v"] \\
            e & & e'
          \end{tikzcd}}} \quad = \quad 
      \vcenter{\hbox{
          \begin{tikzcd}[column sep=1em, row sep=0.8em]
            & f \otimes g \arrow[dl,"s \otimes u"'] \arrow[dr,"t \otimes v"] \\
            d \otimes e & & d' \otimes e'
          \end{tikzcd}}}
    \end{equation*}
    noting here that the tensor of vertical maps is (trivially)
    vertical, and the tensor of cartesian maps is cartesian by
    assumption.
  \item The associativity coherences $\tilde \alpha$ in $\tilde{\Cat{E}}$ are defined
    in terms of the corresponding coherences $\alpha$ for $\Cat{E}$ as:
    \begin{equation*}
      \begin{tikzcd}[column sep=1em, row sep=0.8em]
        & (c \otimes d) \otimes e \arrow[dl,"1"'] \arrow[dr,"\alpha_{cde}"] \\
        (c \otimes d) \otimes e & & c \otimes (d \otimes e)
      \end{tikzcd}
    \end{equation*}
    noting that the maps $\alpha$ in $\Cat{E}$ are invertible, hence
    cartesian.
  \end{itemize}
  In particular, if $P$ is cartesian monoidal, then $\tilde{\Cat{E}}$
  bears a monoidal structure which we term the \emph{Dirichlet
    monoidal structure}.
\end{lemma}

In fact, the preceding framework also allows us to explain the
closedness of the Dirichlet monoidal structure,
by~\cref{prop.closed_dirichlet} below. This is a special case of the
results of~\cite{Nunes2024Monoidal}---but one which is special enough
that it seems easiest to give our own self-contained account. First we
recall some preliminary notions from the theory of fibrations; we cite
the textbook~\cite{jacobs1999categorical} as a convenient reference.

\begin{definition}[Simple coproducts] (\cite[Definition~1.9.1]{jacobs1999categorical})
  Let $\Cat{B}$ be a category with finite products. We say that a
  fibration $P \colon \Cat{E} \rightarrow \Cat{B}$ has \emph{simple
    coproducts} if, for all $a,b \in \Cat{B}$, the reindexing functor
  between fiber categories induced by the product projection
  $\pi_1 \colon a \times b \rightarrow a$ has a left adjoint $\Sigma_b$:
  \begin{equation*}
  \begin{tikzcd}[column sep=large]
    \Cat{E}_{a \times b} \ar[r, shift left=5pt, "\Sigma_b"]
    \ar[r, phantom, "\scriptstyle\Rightarrow"]&
    \Cat{E}_{a}\ar[l, shift left=5pt, "\pi_1^\ast"]
  \end{tikzcd}
  \end{equation*}
  and these adjoints satisfy the \emph{Beck--Chevalley condition},
  meaning that for any map $f \colon a' \rightarrow a$ in $\Cat{B}$, the
  canonical $2$-cell in the following square is invertible:
  \begin{equation*}
    \begin{tikzcd}
      \Cat{E}_{a \times b} \arrow[r,"\Sigma_b"] \arrow[d,"(f \times 1)^\ast"']  &
      \Cat{E}_{a} \arrow[d,"f^\ast"] \\
      \Cat{E}_{a' \times b} \arrow[r,"\Sigma_b"'] &
      \Cat{E}_{a'}\rlap{ .}
    \end{tikzcd}
  \end{equation*}  
\end{definition}

\begin{definition}[Locally small fibration]
  (\cite[Definition~1.9.1]{jacobs1999categorical}) A fibration
  $P \colon \Cat{E} \rightarrow \Cat{B}$ is said to be \emph{locally small}
  if, for all $b \in \Cat{B}$ and all $x,y$ in the fiber category
  $\Cat{E}_b$, the functor
  \begin{equation}
    \label{eqn:locally_small_repn}
    (\Cat{B}/b)\op \rightarrow \smset \qquad \qquad
    (f \colon a \rightarrow b) \quad \mapsto \quad
    \Cat{E}_a(f^\ast x, f^\ast y)\rlap{ ;}
  \end{equation}
  with action on morphisms given by reindexing, is representable.
  Concretely, this means there is an object $E(x,y)$ (the
  ``$\Cat{B}$-valued hom'' from $x$ to $y$) and map
  $\upsilon \colon E(x,y) \rightarrow b$ in $\Cat{B}$, together with a
  ``universal map''
  $\xi \colon \upsilon^\ast x \rightarrow \upsilon^\ast y$ in
  $\Cat{E}_{E(x,y)}$; its universality expresses that, whenever we are
  given $g \colon a \rightarrow b$ in $\Cat{B}$ and a map
  $k \colon g^\ast x \rightarrow g^\ast y$ in $\Cat{E}_a$, there is a
  unique $\present{g,k} \colon a \rightarrow E(x,y)$ with
  $\upsilon \circ \present{g,k} = g$ and which renders the following
  diagram commutative:
  \begin{equation*}
    \begin{tikzcd}[column sep=large]
      {g^\ast x} \arrow[d, "\cong"'] \arrow[r,"k"] &
      {g^\ast y} \arrow[d, "\cong"]\\
      {\present{g,k}^\ast \upsilon^\ast x} \arrow[r, "\present{g,k}^\ast \xi"'] & 
      {\present{g,k}^\ast \upsilon^\ast y}
    \end{tikzcd}
  \end{equation*}
\end{definition}

Our result can now be stated as:
\begin{proposition}\label{prop.closed_dirichlet}
  Let $P \colon \Cat{E} \rightarrow \Cat{B}$ be a cartesian monoidal
  fibration. The monoidal structure $\otimes$ this induces on the
  total category of the opposite fibration
  $\tilde P \colon \tilde{\Cat{E}} \rightarrow \Cat{B}$ is closed so
  long as:
  \begin{itemize}
  \item $\Cat{B}$ is cartesian closed with finite limits;
  \item $P \colon \Cat{E} \rightarrow \Cat{B}$ is locally small with
    simple coproducts.
  \end{itemize}
\end{proposition}
Note that the fibration $Q \colon \Cat{Dir} \rightarrow \smset$
satisfies these two assumptions. The first one is obvious. As for the
second, the simple coproduct functors $\Sigma_B \colon \Cat{Dir}_{A
  \times B} \rightarrow \Cat{Dir}_A$ are given by
\begin{equation*}
    \sum_{(i,j) \in A \times B}d[i,j]^{\yon} \quad \mapsto \quad \sum_{i \in A} \Bigl( \sum_{j \in B} d[i,j]\Bigr)^\yon \text{,} 
\end{equation*}
and the locally small hom between objects $\sum_{i \in B} d[i]^\yon$
and $\sum_{i \in B} e[i]^\yon$ is given by
\begin{equation*}
  \sum_{i \in B} e[i]^{d[i]} \xrightarrow{\pi_1} B\rlap{ .}
\end{equation*}
Thus, this result recovers the closedness of the Dirichlet monoidal
structure on $\poly$.
\begin{proof}
  We first note two consequences of our assumptions. First,
  each pullback functor
  $\pi_1^\ast \colon \Cat{B}/a \rightarrow \Cat{B}/(a \times b)$ has a
  right adjoint
  $\Pi_b \colon \Cat{B}/(a \times b) \rightarrow \Cat{B}/a$. Indeed,
  given $f \colon c \rightarrow a \times b$, we construct $\Pi_b(f)$
  as the arrow to the left of the following pullback:
  \begin{equation*}
    \begin{tikzcd}
      \bullet \arrow[r] \arrow[d, "\Pi_b(f)"'] &
      c^b \arrow[d, "f^c"]  \\
      a \arrow[r, "\eta"']& (a \times b)^b\ar[ul, phantom, very near end, "\lrcorner"]
    \end{tikzcd}
  \end{equation*}
  where $c^b$ and $(a \times b)^b$ are internal homs for the
  cartesian closed structure of $\Cat{B}$, and $\eta$ is the transpose
  of the identity $a \times b \rightarrow a \times b$;
  see~\cite[Proposition~1.9.8]{jacobs1999categorical} for the
  remaining details.

  Second, if we are given objects $x \in \Cat{E}_a$ and
  $y \in \Cat{E}_b$, then their cartesian product $x \times_{\Cat{E}}
  y$ in $\Cat{E}$ also
  has the universal property of a product in $\Cat{E}_{a \times b}$;
  indeed, we have bijections, natural in $c \in \Cat{E}_{a \times b}$,
  of the form
  \begin{equation}
    \label{eqn.dirichlet_closed}
    \Cat{E}_{a \times b}(c, x \times_{\Cat{E}} y) \cong \Cat{E}_{a \times b}(c, \pi_1^\ast x) \times
    \Cat{E}_{a \times b}(c, \pi_2^\ast y)
  \end{equation}
  where
  $\pi_1^\ast \colon \Cat{E}_{a} \rightarrow \Cat{E}_{a \times b}$ and
  $\pi_2^\ast \colon \Cat{E}_{b} \rightarrow \Cat{E}_{a \times b}$
  reindex along the product projections in $\Cat{B}$.

  Now, suppose we are given $y \in \Cat{E}_b$ and $z \in \Cat{E}_c$.
  We can reindex $y$ along $\pi_2 \colon c^b \times b \rightarrow b$
  and $z$ along the evaluation map $\mathsf{ev} \colon c^b \times b \rightarrow c$
  to obtain objects $\pi_2^\ast(y)$ and
  $\mathsf{ev}^\ast(z) \in \Cat{E}_{c^b \times b}$. By local
  smallness, we can thus form the $\Cat{B}$-valued hom, as to the left in:
  \begin{equation*}
    \upsilon \colon E(\mathsf{ev}^\ast z, \pi_2^\ast y) \rightarrow c^b \times b \qquad \qquad \nu = \Pi_b(\upsilon) \colon [y,z]_1 \rightarrow c^b\rlap{ ,}
  \end{equation*}
  and can now define $[y,z]_1 \in \Cat{B}$ and
  $\nu \colon [y,z]_1 \rightarrow c^b$ by applying
  $\Pi_b \colon \Cat{B} / {c^b \times b} \rightarrow \Cat{B} / c^b$ to
  $\upsilon$, as to the right. Now for any $a \in \Cat{B}$, we have a
  chain of bijections, natural in $a$, of the form:
  \begin{align*}
    \Cat{B}(a, [y,z]_1) &\cong \sum_{\bar f \in \Cat{B}(a, c^b)} (\Cat{B} / c^b) (\bar f, \nu) & \text{(property of slice categories)}\\
    &\cong \sum_{\bar f \in \Cat{B}(a, c^b)} (\Cat{B} / c^b \times b) (\pi_1^\ast \bar f, \upsilon) & \text{(since $\nu = \Pi_b \upsilon$)}\\
    &\cong \sum_{\bar f \in \Cat{B}(a, c^b)} \Cat{E}_{a \times b}(({\bar f} \times a)^\ast \mathsf{ev}^\ast z, ({\bar f} \times a)^\ast \pi_2^\ast y) & \text{(universal property of $\upsilon$)}\\
    &\cong \sum_{f \in \Cat{B}(a \times b, c)} \Cat{E}_{a \times b}(f^\ast z, \pi_2^\ast y)\rlap{ .} & \text{(definition of exponential transpose)}
  \end{align*}
  Now, consider the exponential transpose
  $\bar \nu \colon [y,z]_1 \times b \rightarrow c$ of $\nu$, and form the
  object
  \begin{equation*}
    [y,z] \coloneqq \Sigma_b ({\bar \nu}^\ast z) \in \Cat{E}_{[y,z]_1}\rlap{ .}
  \end{equation*}
  We now have the following chain of natural bijections (where $x \in
  \Cat{E}_a$):
  \begin{align*}
    \tilde{\Cat{E}}(x, [y,z]) &\cong \sum_{h \in \Cat{B}(a, [y,z]_1)}\Cat{E}_{a}(h^\ast[y,z],x) &\text{(definition of opposite fibration)} \\
    &\cong \sum_{f \in \Cat{B}(a, [y,z]_1)}\Cat{E}_{a}(h^\ast \Sigma_b({\bar \nu}^\ast z),x) & \text{(definition of $[y,z]$)}\\
    &\cong \sum_{h \in \Cat{B}(a, [y,z]_1)}\Cat{E}_{a}(\Sigma_b (h \times 1)^\ast{\bar \nu}^\ast z, x) & \text{(Beck--Chevalley)}\\
    &\cong \sum_{h \in \Cat{B}(a, [y,z]_1)}\Cat{E}_{a\times b}((h \times 1)^\ast{\bar \nu}^\ast z, \pi_1^\ast x) & \text{($\Sigma_b \dashv \pi_1^\ast$)}
\intertext{
  Now, from above we have bijections $\Cat{B}(a, [y,z]_1) \cong \sum_{f \in
    \Cat{B}(a \times b, c)} \Cat{E}_{a \times b}(f^\ast z, \pi_2^\ast y)$ and
  under this bijection, the map $f \colon a \times b \rightarrow c$
  corresponding to $h \colon a \rightarrow [y,z]_1$ is given by $\bar
  \nu \circ (h \times 1)$. Thus, we continue the chain of isomorphisms:
}
    &\cong \sum_{f \in \Cat{B}(a \times b, c)} \Cat{E}_{a \times b}(f^\ast z, \pi_2^\ast y) \times \Cat{E}_{a\times b}(f^\ast z, \pi_1^\ast x)\\ &\cong
    \sum_{f \in \Cat{B}(a \times b, c)} \Cat{E}_{a \times b}(f^\ast z, x \times_\Cat{E} y)  & \text{(by~\cref{eqn.dirichlet_closed})} \\ &\cong \tilde{\Cat{E}}(x \otimes y, z) & \text{(definition of opposite monoidal fibration)} & \qedhere
  \end{align*}
\end{proof}

Finally in this section, we consider the interaction of the Dirichlet
product on $\poly$ with the composition product of polynomials. This
does not admit an explanation in terms of the preceding abstractions,
since the composition product of polynomials does not transport any
monoidal structure on Dirichlet polynomials; nonetheless, we have:

\begin{proposition}[Duoidality $(\tri\otimes\tri)\to(\otimes\tri\otimes)$]\label{prop.duoidality}
  Dirichlet product is naturally colax monoidal with respect to composition product. That is, for any polynomials $p,q,p',q'\in\poly$ there is a natural map
  \begin{equation}\label{eqn.duoidality}
    (p\tri q)\otimes(p'\tri q')\too(p\otimes p')\tri(q\otimes q')
  \end{equation}
  which, along with the unique map $\yon\to\yon$ between their units, forms a duoidal structure on $\poly$.
\end{proposition}
\begin{proof}[Proof sketch]
  We supply the displayed map \eqref{eqn.duoidality} and leave the rest to the reader. We need to give an element of the following type:
  \begin{gather*}
    \sum_{i\in p(1)}\;
    \sum_{j\colon p[i]\to q(1)}\;
    \sum_{i'\in p'(1)}\;
    \sum_{j'\colon p'[i']\to q'(1)}\;
    \yon^{
      \sum\limits_{d\in p[i]}\;
      \sum\limits_{d'\in p'[i']}
      q[j(d)]\times q'[j'(d')]}\\
    \begin{tikzcd}
      \ar[d,"\varphi"]\\{}
    \end{tikzcd}\\
    \sum_{(i,i')\in (p\otimes p')(1)}\;
    \sum_{(j,j')\colon (p\otimes p')[(i,i')]\to (q\otimes q')(1)}\;
    \yon^{
      \sum\limits_{(d,d')\in(p\otimes p')[(i,i')]}(q\otimes q')[(j (d,d'), j'(d,d'))]}
  \end{gather*}
  Define $\varphi$ on positions by
  \[
    (i,j,i',j')\mapsto \big(i,i',(d,d')\mapsto (j(d),j'(d'))\big)
  \]
  and on directions by the obvious bijection
  \[
    \sum_{(d,d')\in(p\otimes p')[(i,i')]}(q\otimes q')[(j(d), j'(d'))]
    \cong
    \sum_{d\in p[i]}\;
    \sum_{d'\in p'[i']}
    q[j(d)]\times q'[j'(d')].
    \qedhere
  \]
\end{proof}

\subsection{The multi-variable case: global tensor product}\label{sec.multivar_tensor}

Our goal in this section is extend the Dirichlet monoidal structure
from single-variable polynomials to multi-variable polynomials. Given
$p, q \in d\set[c]$, their \emph{Dirichlet product} will be given by
\begin{equation}
  \label{eqn.dirichlet_monoidal}
    (p\ot{c}{d}q)(1) = p(1) \times q(1) 
    \qqand
    (p\ot{c}{d}q)_{a}[(i,j)] = p_a[i] \times q_a[j]\rlap{ .}
  \end{equation}
  It is not hard to see directly that this is a monoidal structure,
  with as unit object the polynomial $I = \yoncool{c}{d}$ satisfying
  $I_a(1) = \{\ast\}$ and $I_{a}[\ast] = 1$. However, we can also
  proceed via~\cref{lem.opposite_monoidal_fib} applied to a suitable
  fibration:
\begin{definition}[Dirichlet maps of polynomials]\label{def.dir_map}
  Let $p, q \in d\set[c]$ be $c$-variable $d$-valued polynomials. A
  \emph{Dirichlet map} $\varphi \colon p \rightarrow q$ comprises a
  map of $d$-sets $\varphi_1 \colon p(1) \rightarrow q(1)$ together
  with maps of $c$-sets
  \begin{equation*}
    \varphi^\flat(i, \thg) \colon p_b[i] \rightarrow q_b[\varphi_1(i)]
  \end{equation*}
  natural with respect to maps of $\el_d p(1)$. We write
  $d\text-\Cat{Dir}[c]$ for the category of $c$-variable $d$-valued
  polynomials and Dirichlet maps.
\end{definition}
\begin{remark}
  The reader might expect us to consider the category of
  ``$c$-variable $d$-valued Dirichlet polynomials'', i.e., functors
  $(c\set)\op \to d\set$ which are pointwise coproducts of
  representables. However, the objects of this category do \emph{not}
  correspond to $c$-variable $d$-valued polynomials. Indeed, if
  $p \colon (c\set)\op \to d\set$ satisfies $p_b = \sum_{i \in
    p_b(0)} p_b[i]^{\yon}$ for each $b \in d$,  then the reindexing
  maps $p_b \rightarrow p_{b'}$ for $f \colon b \rightarrow b'$ in $d$
  involve maps $p_b[i] \rightarrow p_{b'}[p_f(0)(i)]$, rather than
  maps $p_{b'}[p_f(0)(i)] \rightarrow p_b[i]$ as in an object of
  $d\set[c]$. Thus, we are led to define directly
\end{remark}
\begin{lemma}
  The functor $Q \colon d\text-\Cat{Dir}[c] \rightarrow d\set$ sending
  $p$ to $p(1)$ is a cartesian monoidal fibration. Its opposite is the
  evaluation-at-1 functor $P \colon d\set[c] \rightarrow d\set$, and
  the monoidal structure on $d\set[c]$ is the Dirichlet monoidal
  structure.
\end{lemma}
\begin{proof}
  The reindexing of $p$ along $f \colon q \rightarrow p(1)$ is the
  polynomial with $(f^\ast p)_a = \sum_{i \in q_a} \yon^{p_a[f(i)]}$.
  The cartesian lift $f^\ast p \rightarrow p$ acts on positions via
  $f$, and on directions as the identity. A general cartesian map is
  one whose action on directions is invertible. Like before, the
  product in $d\text-\Cat{Dir}[c]$ can be taken as
  \begin{equation*}
    (p \times q)_a = \sum_{(i,j) \in p_a(1) \times q_a(1)} (p_a[i] \times q_a[j])^{\yon}
  \end{equation*}
  whence $Q \colon d\text-\Cat{Dir}[c] \rightarrow d\set$ is a
  cartesian monoidal fibration, whose opposite is clearly
  $P \colon d\set[c] \rightarrow d\set$, endowed with the Dirichlet
  monoidal structure of~\cref{eqn.dirichlet_monoidal}.
\end{proof}

However, the real reason for adopting the above approach is that it
allows us to deduce the existence of closures $\ih{c}{-}{-}{d}$ for
the monoidal product $\ot{c}{d}$. For this, we must verify that $Q
\colon d\text-\Cat{Dir}[c] \rightarrow d\set$ satisfies the additional
hypotheses of~\cref{prop.closed_dirichlet}. Clearly, the base category
$d\set$, like any presheaf category, is cartesian closed with finite
limits. What remains to show is that:
\begin{proposition}
  The fibration $Q \colon d\text-\Cat{Dir}[c] \rightarrow d\set$ is
  locally small with simple coproducts.
\end{proposition}
\begin{proof}
  The fiber category of $Q$ over $X \in d\set$ is given
  by the presheaf category $[(\el_d X)\op, c\set]$, and reindexing
  along $f \colon X \rightarrow Y$ is given by precomposition with
  $(\el_d f)\op \colon (\el_d X)\op \rightarrow (\el_d Y)\op$.
  We first show local smallness.  Given $p \in d\set[c]$
  with $p(1) = X$, we first define, for every $b \in d$ and
  $i \in X_b$, a functor
  $\res p i \colon (b / d)^\mathrm{op} \rightarrow c\set$ as the
  composite
  \begin{equation*}
    (b/d)^\mathrm{op} \xrightarrow{=} \mathrm{El}_d(\yon^b)^\mathrm{op} \xrightarrow{\mathrm{El}_d(\iota)^\mathrm{op}} \mathrm{El}_d(X)^\mathrm{op} \xrightarrow{p[\thg]} c\set
  \end{equation*}
  where $\iota \colon \yon^b \rightarrow X$ classifies $i \in X_b$.
  Explicitly, $\res p i$ acts on objects and morphisms by
  \begin{equation*}
    (f \colon b \rightarrow b') \qquad \mapsto \qquad p_{b'}[X_f(i)]
  \end{equation*}
  and
  \begin{equation*}
    \begin{tikzcd}
      & {b} \arrow[dl,"f"'] \arrow[dr,"g"] \\
      {b'} \arrow[rr,"h"'] & &
      {b''}
    \end{tikzcd}
    \qquad \mapsto \qquad
    p_{b''}[X_g(i)] \xrightarrow{p^\sharp_{h}(X_f(i), \thg)} 
    p_{b'}[X_f(i)]\rlap{ .}
  \end{equation*}
  Now, given $p,q \in d\set[c]$ with $p(1) = q(1) = X$,
  we define $E(p,q) \in d\set$ on objects by
  \begin{equation*}
    E(p,q)_b = \{(i, \alpha) \mid i \in X_b, \, \alpha \colon \res p i \Rightarrow \res q i\}\rlap{ .}
  \end{equation*}
  On morphisms, given $f \colon b \rightarrow b'$ in $d$, we define 
  \begin{equation*}
    E(p, q)_b \rightarrow E(p, q)_{b'} \qquad \qquad (i,
    \alpha) \mapsto (X_f i, f\lpush  \alpha)
  \end{equation*}
  where
  $f\lpush  \alpha \colon \res p {X_f i} \Rightarrow \res q {X_f i}$ is
  the whiskering of
  $\alpha \colon \res p i \Rightarrow \res q i \colon (b/d)\op
  \rightarrow c\set$ by
  $(\thg) \circ f \colon (b'/d)\op \rightarrow (b/d)\op$. Now
  $E(p,q)$, with the obvious projection to $X$, provides the desired
  representation for~\cref{eqn:locally_small_repn}.

  Indeed, given a map $g \colon Y \rightarrow X$ in $d\set$, to give a
  lifting of $\alpha$ through the projection $E(p,q) \rightarrow X$ is
  to give, for each $b \in d$ and each $j \in Y_b$, a natural
  transformation
  $\alpha_j \colon \res p {g j} \Rightarrow \res q {g j}$, in such a
  way that, for each $f \colon b \rightarrow b'$ in $d$, we have
  $f\lpush (\alpha_j) = \alpha_{Y_f j}$. This last condition means that
  the $\alpha_j$'s are completely determined by giving the component
  of each $\alpha_j$ at $\id_b \in b/d$, thus, by a family of natural transformations
  \begin{equation}
    \label{eqn.d_dir_c_local_small}
    (\alpha_j)_{\id_b} \colon p_b[g(j)] \Rightarrow q_b[g(j)]\rlap{ ;}
  \end{equation}
  and the condition that $f\lpush (\alpha_j) = \alpha_{Y_f j}$ now
  reduces to the condition that the maps
  in~\cref{eqn.d_dir_c_local_small} are natural in $j \in \el_d(Y)$.
  Thus, liftings of $g$ through $E(p,q) \rightarrow X$ are in
  bijection with maps $g^\ast p \rightarrow g^\ast q$ in the fiber
  $[\el_d(Y)^\mathrm{op}, c\set]$, as desired.

  We now show that $Q \colon d\text-\Cat{Dir}[c] \rightarrow d\set$
  has simple coproducts. Given $X,Y \in d\set$, the reindexing functor
  from the $X$-fiber to the $X \times Y$-fiber is given by
  precomposition with the projection $\pi_1 \colon \el_d(X \times Y) \rightarrow
  \el_d X$, and this certainly has a left adjoint, given by left Kan extension:
  \begin{equation*}
    \begin{tikzcd}[column sep=large]
      {[(\el_d (X \times Y))\op, c\set]} \ar[r, shift left=5pt, "\mathrm{Lan}_{\pi_1}"]
      \ar[r, phantom, "\scriptstyle\Rightarrow"]&
      {[(\el_d X)\op, c\set]}\ar[l, shift left=5pt, "(\thg) \circ \pi_1"]\rlap{ .}
    \end{tikzcd}
  \end{equation*}
  However, we must also verify the Beck--Chevalley condition. To this
  end, let us first note that, given
  $p \colon (\el_d (X \times Y))\op \rightarrow c\set$, the value of
  $\mathrm{Lan}_{\pi_1}(p)$ at some $(b, i) \in \el_{d} X$ is given by
  the colimit
  \begin{equation*}
    \colim\Bigl(((b,i) \downarrow \pi_1)\op \xrightarrow{\text{proj}} (\el_d (X \times Y))\op \xrightarrow{p} c\set\Bigr)
  \end{equation*}
  where the comma category $(b,i) \downarrow \pi_1$ has:
  \begin{itemize}
  \item Objects being pairs of $(b', i', j') \in \el_d(X \times Y)$
    and $f \colon (b,i) \rightarrow (b',i')$ in $\el_d X$;
  \item Maps $(b',i',j',f) \rightarrow (b'',i'',j'',g)$ being maps
    $h \colon b' \rightarrow b''$ so that $g = hf$, $X_h(i') =
    i''$ and $Y_h(j') = j''$.
  \end{itemize}
  Now, if $f \colon (b,i) \rightarrow (b',i')$ in $\el_d X$, then
  necessarily $i' = X_f(i)$. Thus, objects of $(b,i) \downarrow \pi_1$
  are equally pairs $(f \colon b \rightarrow b', j' \in Y_{b'})$, and,
  in fact, we have an isomorphism of categories
  $(b,i) \downarrow \pi_1 \cong \el_{d}(\yon^b \times Y)$, so that
  $\mathrm{Lan}_{\pi_1}(p)$ is equally given at $(b,i)$ by the
  following colimit, wherein
  $\tilde \imath \colon \yon^b \rightarrow X$ classifies $i \in X_b$:
  \begin{equation*}
    \colim\Bigl(\el_d(\yon^b \times Y)\op \xrightarrow{\el_d(\tilde \imath \times Y)\op} (\el_d (X \times Y))\op \xrightarrow{p} c\set\Bigr)\rlap{ .}
  \end{equation*}

  Let us now use this to verify the Beck--Chevalley condition for the
  simple coproducts of $Q \colon d\text-\Cat{Dir}[c] \rightarrow
  d\set$. Given $f \colon X' \rightarrow X$ and $Y$ in $d\set$, we
  must verify that the following square commutes to within
  isomorphism:
  \begin{equation}
    \label{eqn.d_dir_c_simple_coproducts}
    \begin{tikzcd}
      {[\el_d(X \times Y)\op, c\set]} \arrow[r,"\mathrm{Lan}_{\pi_1}"] \arrow[d,"(\thg) \circ \el_d(f \times 1)"']  &
      {[\el_d(X)\op, c\set]} \arrow[d,"(\thg) \circ \el_d(f)"] \\
      {[\el_d(X' \times Y)\op, c\set]} \arrow[r,"\mathrm{Lan}_{\pi_1}"'] &
      {[\el_d(X')\op, c\set]}\rlap{ .}
    \end{tikzcd}
  \end{equation}  
  Now, given $p \in [\el_d(X \times Y)\op, c\set]$, its image under
  the upper composite is the functor $\el_d(X')\op \rightarrow c\set$
  whose value at $(b,i)$ is the colimit of
  \begin{equation*}
    \el_d(\yon^b \times Y)\op \xrightarrow{\el_d(\widetilde {f(i)} \times Y)\op} (\el_d (X \times Y))\op \xrightarrow{p} c\set\rlap{ .}
  \end{equation*}
  On the other hand, the lower composite is the functor with value at
  $(b,i)$ given by the colimit of
  \begin{equation*}
    \el_d(\yon^b \times Y)\op \xrightarrow{\el_d(\tilde {\imath} \times Y)\op} (\el_d (X' \times Y))\op \xrightarrow{\el_d(f \times Y)\op} (\el_d (X \times Y)) \xrightarrow{p} c\set\rlap{ .}
  \end{equation*}
  But since $\widetilde{f(i)} = f \circ \tilde\imath$, these functors
  are the same, and so~\cref{eqn.d_dir_c_simple_coproducts} commutes
  to within isomorphism.
\end{proof}
By combining the preceding proof with that
of~\cref{prop.closed_dirichlet}, we obtain closure of the Dirichlet
monoidal structures, along with explicit formulae for the closure:

\begin{corollary}[Local closure]\label{prop.local_closure_general}
  For any categories $c, d\in\ccatsharp$, the Dirichlet product on
  $d\set[c]$ is monoidal closed, with the internal hom
  $[q,r] = \ih{c}{q}{r}{d}$ of polynomials $q,r \in d\set[c]$ given as
  follows, wherein $Y^b \in d\set[c]$ denotes the polynomial functor
  which is constant at the representable $\yon^b \in d\set$:
  \begin{equation}\label{eqn.local_closure_general}
    \begin{aligned}
      [q,r]_b(1) &\coloneqq d\set[c](Y^b \times q, r)
      \\ [q,r]_b[\alpha] &\coloneqq \colim\Bigl(\el_d(\yon^b \times q(1))\op \xrightarrow{\el_d(\alpha_1)\op} \el_d(r(1))\op \xrightarrow{r[\thg]}c\set\Bigr)\rlap{ .}
    \end{aligned}
  \end{equation}
  For any $f \colon b \rightarrow b'$ in $d$, the map
  $[q,r]_b(1) \rightarrow [q,r]_{b'}(1)$ is precomposition with
  $Y^f \times q \colon Y^{b'} \times q \rightarrow Y^{b} \times q$;
  and if $\alpha \in [q,r]_b(1)$ has image $\beta \in [q,r]_{b'}(1)$
  under this map, then the induced map of $c$-sets
  $[q,r]_{b'}[\beta] \rightarrow [q,r]_{b}[\alpha]$ is the unique map
  on colimits induced by the commuting triangle:
    \begin{equation*}
      \begin{tikzcd}[column sep={3.7em,between origins},row sep={1em}]
        {(\el_d(\yon^{b'} \times q(1))\op} \ar{rrrr}{\el_d (\yon^f \times q(1))} \ar[dr, "\el_d \beta"'] & & & & 
        {(\el_d (\yon^b \times q(1)))\op} \ar[dl, "\el_d \alpha"]  \\
        & {(\el_d r(1))\op} \ar[dr, "{r[\thg]}"'] & &
        {(\el_d r(1))\op} \ar[dl, "{r[\thg]}"] \\ & &
        c\set\rlap{ .}
      \end{tikzcd}
    \end{equation*}
\end{corollary}

For our applications in~\cref{chap.aggregation}, we will be
particularly interested in the special case of the above where the
category $d$ is the discrete category $D\yon$ on a set $D$. The
formulae above then simplify:

\begin{corollary}[Discrete local closure]\label{prop.local_closure}
Given a category $c\in\ccatsharp$, a set $D\in\smset$, and
  bicomodules $q, r$ from $c$ to $D\yon$, the internal hom $\ih{c}{q}{r}{D\yon}$ for the
  Dirichlet monoidal structure on $D\yon\set[c]$ is given by
  \begin{equation}\label{eqn.simple_closure}
    (\ih{c}{q}{r}{D\yon})_b = \sum_{\varphi\in\smset[c](q_b,r_b)}\yon^{\;\sum\limits_{j\in q_b(1)}r_b[\varphi(j)]}\;.
  \end{equation}
\end{corollary}

Finally in this section, we consider the multi-variable analogue of
the duoidal interaction between the Dirichlet and composition monoidal
structures on $\poly$. Like before, it seems easiest to proceed
directly.
\begin{proposition}[Duoidality]\label{prop.catsharp_duoidality}
  The local $\otimes$-monoidal structures interact ``duoidally'' under
  composition. That is, given $p, p' \colon c\bito d$ and $q,q' \colon
  d\bito e$ there is
  a natural map 
  \begin{equation}\label{eqn.big_duoidal}
    (q\tri_dp)\ot{c}{e}(q'\tri_dp')\to
    (q\ot{d}{e}q')\tri_d(p\ot{c}{d}p'),
  \end{equation}
  as well as natural maps of the form:
  \begin{equation}\label{eqn.big_duoidal_units}
    \yoncool{c}{e}\to \yoncool{d}{e} \tri_d \yoncool{c}{d}\qqand
    J_e\ot{e}{e}J_e\to J_e\qqand
    \yoncool{e}{e}\to J_e
  \end{equation}
  where $\yoncool{c}{d}\coloneqq d(1)\yon^{c(1)}$ is the $\ot{c}{d}$
  unit and $J_e\coloneqq\sum_{a\in e(1)}\yon^{e[a]}$ is the $\tri_e$
  unit. These satisfy the usual axioms.%
  \footnote{By the usual axioms, we mean those expressing that
    $\ccatsharp$ is a pseudo-category object in the 2-category
    $\Cat{MonCat}_l$ of monoidal categories and lax monoidal functors,
    just as a duoidal category is a pseudo-monoid object in
    $\Cat{MonCat}_l$; see \cite{nlab:duoidal_category}.}
\end{proposition}
\begin{proof}[Proof sketch]
  Again, we supply only the map \eqref{eqn.big_duoidal} and leave the
  rest to the reader. For each $b \in d$, we need to give a map in
  $\smset[c]$ of the following type:
  \begin{gather*}
    \sum_{i\in p_b(1)}\;
    \sum_{j\colon p_b[i]\to q(1)}\;
    \sum_{i'\in p'_b(1)}\;
    \sum_{j'\colon p'_b[i']\to q'(1)}\;
    \yon^{
      \int^{x\in \el_c(p_b[i])}\;
      \int^{x'\in \el_c(p'_{b}[i'])}
      q[j(x)]\times q'[j'(x')]}\\
    \begin{tikzcd}
      \ar[d,"\varphi"]\\{}
    \end{tikzcd}\\
    \sum_{(i,i')\in (p\otimes p')_b(1)}\;
    \sum_{(j,j')\colon (p\otimes p')_b[(i,i')]\to (q\otimes q')(1)}\;
    \yon^{
      \int^{(x,x')\in \el_c((p\otimes p')_b[(i,i')])}(q\otimes q')[(j (x,x'), j'(x,x'))]}
  \end{gather*}
  Define $\varphi$ on positions by
  \[
    (i,j,i',j')\mapsto \big(i,i',(d,d')\mapsto (j(d),j'(d'))\big)
  \]
  and on directions by the obvious map
  \[
    \int^{(x,x')\in\el_c((p\otimes p')_b[(i,i')])}(q\otimes q')_b[(j(x), j'(x'))]
    \rightleftarrows
    \int^{x\in \el_c(p_b[i])}\;
    \int^{x'\in \el_c(p'_b[i'])}
    q_b[j(x)]\times q'_b[j'(x')].
    \qedhere
  \]
\end{proof}

\subsection{Duality}\label{sec.duals}

In this section, we that in the monoidal category $D\set[C]$, weak
duality in the sense of \cref{ex.duality_in_poly} allows us to
delineate the relationships between the four ways the bicategory
$\sspan$ appears in $\ccatsharp$ (as observed in \cref{sec.sspan}):
$\ccatsharplin$ and its opposite (indeed, $\sspan \cong \sspan\op$)
and $(\ccatsharpdisccon)\co$ and its opposite.

\begin{notation}[Dualizing object]\label{def.dualizing_ob}
  Let $C,D\in\smset$ be sets. The unit of the Dirichlet monoidal
  structure on $D\set[C]$ is the terminal span $C\from
  (C\times D)\to D$ seen as a linear bicomodule as in \cref{sec.sspan}:
  \[
    C\yon\bito[CD\yon]D\yon.
  \]
  We denote it $_C\bot_D\in D\set[C]$, or simply by $\bot$ if $C,D$
  are clear from context, and call it the \emph{dualizing
    object}. In accordance with \cref{ex.duality_in_poly}, we write
  $\dual{p}$ for the internal hom $\ih{C\yon}{p}{\bot}{D\yon}$ and
  call  it the \emph{(weak) dual} of $p$ in $D\set[C]$.
\end{notation}

In $\poly=\yon\set[\yon]$, the dualizing object is $\yon$. Note that
for any $C,D,E$, we have $({_D\bot_E})\tri_D ({_C\bot_D})\cong
({_C\bot_E})$. The following theorem now generalizes
\cref{ex.duality_in_poly}.


\begin{theorem}[Linear-conjunctive
  duality]\label{thm.linear_conjunctive_dual}\label{prop.comp_dual_spans}
  For all sets $C,D$, weak duality $p \mapsto \dual{p}$ restricts to a contravariant
  equivalence of categories
  \begin{equation}
    \label{eqn.weak_duality_c_d}
    \begin{tikzcd}[column sep=60pt]
      \ccatsharplin(C\yon,D\yon)\op\ar[r, shift left, "\dual{(-)}"]&
      \ccatsharpdisccon(C\yon,D\yon)\ar[l, shift left, "\dual{(-)}"]
    \end{tikzcd}
  \end{equation}
  which are the action on homs of an identity-on-objects
  biequivalence $\dual{(\thg)}$ between $\ccatsharpdisccon$ and
  $(\ccatsharplin)\co$.
\end{theorem}
\begin{proof}
  To establish the given contravariant equivalences, we show that for
  linear and conjunctive bicomodules between discrete categories,
  weak dualizing just moves the coefficients to exponents and vice versa:
  \[
    \dual{\bigg(\sum_{b\in D}M_b\yon\bigg)}\cong\sum_{b\in D}\yon^{M_b}
    \qqand
    \dual{\bigg(\sum_{b\in D}\yon^{M_b}\bigg)}\cong\sum_{b\in D}M_b\yon.
  \]
  Indeed, let $m,n\in D\set[C]$, and suppose that $m$ is linear and that $n$
  is conjunctive. Then by definition we have isomorphisms
  \[
    m\cong\sum_{b\in D}\sum_{i\in m_b(1)}\yon^{\{a(i)\}}
    \qqand
    n\cong\sum_{b\in D}\yon^{n[b]}
  \]
  where $a\colon m(1)\to C$ is a function and $n[b]\in C\set$ (i.e.,
  there is a function $n[b]\to C$) for each $b\in D$. Starting with
  $m$, one checks using \cref{eqn.simple_closure} that
  \begin{equation}\label{eqn.dual_on_linear}
    \dual{m}\cong\sum_{b\in D}\sum_{\varphi\in\smset[C](m,C\yon)}\yon^{\sum_{j\in m_b(1)}\{\varphi(j)\}}.
  \end{equation}
  To see that this is conjunctive when $m$ is linear, it suffices to
  show that $\smset[C](m,C\yon)=1$. This is straightforward, either
  using \cref{eqn.bicomod_map_combinatorics_disc}, or, if one prefers,
  using bridge diagrams \cref{eqn.bridge_map} with $E'=B'=C=C'$,
  $D=D'=1$ and $E=B$.

  Moving on to $n$, one checks again using \cref{eqn.simple_closure} that
  \begin{equation}\label{eqn.dual_on_conjunctive}
    \dual{n}\cong\sum_{b\in D}\sum_{\varphi\in\smset[C](\yon^{n[b]},C\yon)}\yon^{\{\varphi_1(!)\}}
  \end{equation}
  which is clearly linear, as all the exponents have one element
  (namely $\varphi_1$ applied to the unique position $!$ in the
  polynomial $\yon^{n[b]}$).

  Finally, to see that the two constructions are mutually inverse,
  notice that $\{\varphi(j)\}$ is a one-element set for any $j\in
  m_b(1)$, so the exponent of \eqref{eqn.dual_on_linear} is
  $m_b(1)$. The second sum-index in \eqref{eqn.dual_on_conjunctive} is
  $\smset[C](\yon^{n[b]},C\yon)$ and it remains only to show that it
  is isomorphic to $n[b]$. Again by
  \eqref{eqn.bicomod_map_combinatorics_disc} we calculate the
  following for any $N\in C\set$:
  \[
    \smset[C](\yon^{N},C\yon)\cong\sum_{a\in C}C\set(\{a\},N)\cong\sum_{a\in C}N_a\cong N.
  \]

  This establishes the adjoint
  equivalences~\cref{eqn.weak_duality_c_d}, and it remains to show
  that these assemble into an identity-on-objects biequivalence
  $\ccatsharpdisccon \simeq (\ccatsharplin)\co$. Preservation of
  identities is straightforward; as for composition, we show that,
  given a pair of linear bimodules
  $D\yon\bito[N\yon]A\yon\bito[M\yon]B\yon$, there is an isomorphism
  of bicomodules
  \[
    \dual{(M\yon)}\tri_{A\yon}\dual{(N\yon)}\cong\dual{(M\yon\tri_{A\yon}N\yon)}.
  \]
  Indeed, by explicit calculation from
  \cref{prop.local_closure,eqn.composite_bico_linear}, we have
  \begin{align*}
    \dual{(M\yon)}\tri_{A\yon}\dual{(N\yon)}&
                                              \cong\ih{A\yon}{M\yon}{\bot}{B\yon}\tri_{A\yon}\ih{D\yon}{N\yon}{\bot}{A\yon}\\&
    \cong\ih{D\yon}{M\yon\tri_{A\yon}N\yon}{\bot}{B\yon}\\&
    \cong\dual{(M\yon\tri_{A\yon}N\yon)}.\qedhere
  \end{align*}
\end{proof}

\begin{remark}[Weak duals in terms of $\Delta,\Pi,\Sigma$]\label{rem.duals_delta_sigma_pi}
  One can understand the weak duality from
  \cref{thm.linear_conjunctive_dual} in terms of the perhaps more
  familiar $\Delta,\Pi,\Sigma$ operations. Namely, weak dualizing switches
  $\Pi\circ\Delta$ to $\Sigma\circ\Delta$ and vice versa. That is,
  given a span $C\From{f}M\To{g}D$, the prafunctor corresponding to
  the bicomodule $M\yon \colon C\yon\bito D\yon$ is
  $\Sigma_f\circ\Delta_g\colon d\set\to c\set$, while its weak dual
  $\dual{(M\yon)} \colon C\yon\bito D\yon$ corresponds to
  $\Pi_f\circ\Delta_g$.
\end{remark}

\begin{lemma}\label{lemma_Ay_selfdual}
  For any span of the form $A\from B =\!= B$, the corresponding bicomodule
  $A\yon\bito[B\yon]B\yon$ is self weak dual: $\dual{(B\yon)}\cong
  B\yon$.
\end{lemma}
\begin{proof}
  The bicomodule $A\yon\bito[B\yon]B\yon$ is both conjunctive and
  linear, and so we can use either \eqref{eqn.dual_on_linear} or
  \eqref{eqn.dual_on_conjunctive} to prove the result.
\end{proof}


Transposing a span---switching $C\From{f} M\To{g} D$ to $D\From{g}
M\To{f}C$, which is usually considered as purely syntactic---in fact
splits up into the composite of two more primitive operations, each
with a universal property.

\begin{corollary}[Transpose spans]\label{cor.transpose_spans}
  Given a span of sets $C\from M\to D$ considered as a bicomodule
  $C\yon\bito[M\yon]D\yon$, the following are equivalent
  \begin{itemize}
  \item its transpose $D\yon\bito[(M\yon)^\top]C\yon$,
  \item its right adjoint's weak dual $\dual{\big((M\yon)\rdag\big)}$, and
  \item its weak dual's left adjoint $\big(\dual{(M\yon)}\big)\ldag$.
  \end{itemize}

  These operations moreover yield equivalent equivalences
  $\sspan \cong \sspan\op$.
\end{corollary}
\begin{proof}
  This follows from \cref{prop.adjoint_bicoms,thm.linear_conjunctive_dual}.
\end{proof}

\begin{figure}[h]
  \centering
  \begin{tikzpicture}[corollas]
    \begin{scope}[shift={(-3.15,2.15)}]
      \draw[rounded corners] (-2.325,-1.375) rectangle (2.325,1.375);
      \begin{scope}[shift={(0,.125)}]
        \node [bcol,vertex,large] (s1) at (-.6, -1) {\tiny$b_1$};
        \coordinate (t1) at (-1.8,1) {};
        \coordinate (t2) at (-1.2,1) {};
        \coordinate (t3) at (-.6,1) {};
        \coordinate (t4) at (0,1) {};
        \coordinate (t5) at (.6,1) {};
        
        \node [ycol,vertex,large] (ss1) at (-.6,-.3) {};
        
        \node [rcol,vertex,large] (m1) at (-1.8,.4) {\tiny$a_1$};
        \node [rcol,vertex,large] (m2) at (-1.2,.4) {\tiny$a_1$};
        \node [rcol,vertex,large] (m3) at (-.6,.4) {\tiny$a_1$};
        \node [rcol,vertex,large] (m4) at (0,.4) {\tiny$a_2$};
        \node [rcol,vertex,large] (m5) at (.6,.4) {\tiny$a_2$};
        
        \draw [bcol,edge] (s1) -- (ss1);
        
        \draw [ycol,edge] (ss1) -- (m1);
        \draw [ycol,edge] (ss1) -- (m2);
        \draw [ycol,edge] (ss1) -- (m3);
        \draw [ycol,edge] (ss1) -- (m4);
        \draw [ycol,edge] (ss1) -- (m5);

        \draw [rcol,edge] (m1) -- (t1);
        \draw [rcol,edge] (m2) -- (t2);
        \draw [rcol,edge] (m3) -- (t3);
        \draw [rcol,edge] (m4) -- (t4);
        \draw [rcol,edge] (m5) -- (t5);
        
        \node [bcol,vertex,large] (s2) at (1.5, -1) {\tiny$b_2$};
        \coordinate (t6) at (1.2,1) {};
        \coordinate (t7) at (1.8,1) {};
        
        \node [ycol,vertex,large] (ss2) at (1.5,-.3) {};
        
        \node [rcol,vertex,large] (m6) at (1.2,.4) {\tiny$a_2$};
        \node [rcol,vertex,large] (m7) at (1.8,.4) {\tiny$a_3$};
        
        \draw [bcol,edge] (s2) -- (ss2);
        
        \draw [ycol,edge] (ss2) -- (m6);
        \draw [ycol,edge] (ss2) -- (m7);

        \draw [rcol,edge] (m6) -- (t6);
        \draw [rcol,edge] (m7) -- (t7);

      \end{scope}

      \node at (0, 1.75) {$\yon^5 + \yon^2$};
    \end{scope}
    \begin{scope}[shift={(3.15,2.15)}]
      \draw[rounded corners] (-2.325,-1.375) rectangle (2.325,1.375);
      \begin{scope}[shift={(0,.125)}]
        \node [bcol,vertex,large] (s1) at (-1.8,-1) {\tiny$b_1$};
        \node [bcol,vertex,large] (s2) at (-1.2,-1) {\tiny$b_1$};
        \node [bcol,vertex,large] (s3) at (-.6,-1) {\tiny$b_1$};
        \node [bcol,vertex,large] (s4) at (0,-1) {\tiny$b_1$};
        \node [bcol,vertex,large] (s5) at (.6,-1) {\tiny$b_1$};
        
        \node [bcol,vertex,large] (s6) at (1.2,-1) {\tiny$b_2$};
        \node [bcol,vertex,large] (s7) at (1.8,-1) {\tiny$b_2$};
        
        \coordinate (t1) at (-1.8,1) {};
        \coordinate (t2) at (-1.2,1) {};
        \coordinate (t3) at (-.6,1) {};
        \coordinate (t4) at (0,1) {};
        \coordinate (t5) at (.6,1) {};
        
        \coordinate (t6) at (1.2,1) {};
        \coordinate (t7) at (1.8,1) {};
        
        \node [ycol,vertex,large] (ss1) at (-1.8,-.3) {};
        \node [ycol,vertex,large] (ss2) at (-1.2,-.3) {};
        \node [ycol,vertex,large] (ss3) at (-.6,-.3) {};
        \node [ycol,vertex,large] (ss4) at (0,-.3) {};
        \node [ycol,vertex,large] (ss5) at (.6,-.3) {};
        
        \node [ycol,vertex,large] (ss6) at (1.2,-.3) {};
        \node [ycol,vertex,large] (ss7) at (1.8,-.3) {};
        
        \node [rcol,vertex,large] (m1) at (-1.8,.4) {\tiny$a_1$};
        \node [rcol,vertex,large] (m2) at (-1.2,.4) {\tiny$a_1$};
        \node [rcol,vertex,large] (m3) at (-.6,.4) {\tiny$a_1$};
        \node [rcol,vertex,large] (m4) at (0,.4) {\tiny$a_2$};
        \node [rcol,vertex,large] (m5) at (.6,.4) {\tiny$a_2$};
        
        \node [rcol,vertex,large] (m6) at (1.2,.4) {\tiny$a_2$};
        \node [rcol,vertex,large] (m7) at (1.8,.4) {\tiny$a_3$};
        
        \draw [bcol,edge] (s1) -- (ss1);
        \draw [bcol,edge] (s2) -- (ss2);
        \draw [bcol,edge] (s3) -- (ss3);
        \draw [bcol,edge] (s4) -- (ss4);
        \draw [bcol,edge] (s5) -- (ss5);
        
        \draw [bcol,edge] (s6) -- (ss6);
        \draw [bcol,edge] (s7) -- (ss7);
        
        \draw [ycol,edge] (ss1) -- (m1);
        \draw [ycol,edge] (ss2) -- (m2);
        \draw [ycol,edge] (ss3) -- (m3);
        \draw [ycol,edge] (ss4) -- (m4);
        \draw [ycol,edge] (ss5) -- (m5);
        
        \draw [ycol,edge] (ss6) -- (m6);
        \draw [ycol,edge] (ss7) -- (m7);

        \draw [rcol,edge] (m1) -- (t1);
        \draw [rcol,edge] (m2) -- (t2);
        \draw [rcol,edge] (m3) -- (t3);
        \draw [rcol,edge] (m4) -- (t4);
        \draw [rcol,edge] (m5) -- (t5);

        \draw [rcol,edge] (m6) -- (t6);
        \draw [rcol,edge] (m7) -- (t7);

      \end{scope}

      \node at (0, 1.75) {$5\yon + 2\yon = 7\yon$};
    \end{scope}
    \begin{scope}[shift={(-3.15,-2.15)}]
      \draw[rounded corners] (-2.325,-1.375) rectangle (2.325,1.375);
      \begin{scope}[shift={(0,.125)}]
        \node [rcol,vertex,large] (s1) at (-1.8,-1) {\tiny$a_1$};
        \node [rcol,vertex,large] (s2) at (-1.2,-1) {\tiny$a_1$};
        \node [rcol,vertex,large] (s3) at (-.6,-1) {\tiny$a_1$};
        \node [rcol,vertex,large] (s4) at (0,-1) {\tiny$a_2$};
        \node [rcol,vertex,large] (s5) at (.6,-1) {\tiny$a_2$};
        \node [rcol,vertex,large] (s6) at (1.2,-1) {\tiny$a_2$};
        \node [rcol,vertex,large] (s7) at (1.8,-1) {\tiny$a_3$};
        
        \coordinate (t1) at (-1.8,1) {};
        \coordinate (t2) at (-1.2,1) {};
        \coordinate (t3) at (-.6,1) {};
        \coordinate (t4) at (0,1) {};
        \coordinate (t5) at (.6,1) {};
        \coordinate (t6) at (1.2,1) {};
        \coordinate (t7) at (1.8,1) {};
        
        \node [ycol,vertex,large] (ss1) at (-1.8,-.3) {};
        \node [ycol,vertex,large] (ss2) at (-1.2,-.3) {};
        \node [ycol,vertex,large] (ss3) at (-.6,-.3) {};
        \node [ycol,vertex,large] (ss4) at (0,-.3) {};
        \node [ycol,vertex,large] (ss5) at (.6,-.3) {};
        \node [ycol,vertex,large] (ss6) at (1.2,-.3) {};
        \node [ycol,vertex,large] (ss7) at (1.8,-.3) {};
        
        \node [bcol,vertex,large] (m1) at (-1.8,.4) {\tiny$b_1$};
        \node [bcol,vertex,large] (m2) at (-1.2,.4) {\tiny$b_1$};
        \node [bcol,vertex,large] (m3) at (-.6,.4) {\tiny$b_1$};
        \node [bcol,vertex,large] (m4) at (0,.4) {\tiny$b_1$};
        \node [bcol,vertex,large] (m5) at (.6,.4) {\tiny$b_1$};
        \node [bcol,vertex,large] (m6) at (1.2,.4) {\tiny$b_2$};
        \node [bcol,vertex,large] (m7) at (1.8,.4) {\tiny$b_2$};
        
        \draw [rcol,edge] (s1) -- (ss1);
        \draw [rcol,edge] (s2) -- (ss2);
        \draw [rcol,edge] (s3) -- (ss3);
        \draw [rcol,edge] (s4) -- (ss4);
        \draw [rcol,edge] (s5) -- (ss5);
        \draw [rcol,edge] (s6) -- (ss6);
        \draw [rcol,edge] (s7) -- (ss7);
        
        \draw [ycol,edge] (ss1) -- (m1);
        \draw [ycol,edge] (ss2) -- (m2);
        \draw [ycol,edge] (ss3) -- (m3);
        \draw [ycol,edge] (ss4) -- (m4);
        \draw [ycol,edge] (ss5) -- (m5);
        \draw [ycol,edge] (ss6) -- (m6);
        \draw [ycol,edge] (ss7) -- (m7);

        \draw [bcol,edge] (m1) -- (t1);
        \draw [bcol,edge] (m2) -- (t2);
        \draw [bcol,edge] (m3) -- (t3);
        \draw [bcol,edge] (m4) -- (t4);
        \draw [bcol,edge] (m5) -- (t5);
        \draw [bcol,edge] (m6) -- (t6);
        \draw [bcol,edge] (m7) -- (t7);

      \end{scope}
      
      \node at (0, -1.75) {$3\yon+3\yon + \yon = 7\yon$};
    \end{scope}
    \begin{scope}[shift={(3.15,-2.15)}]
      \draw[rounded corners] (-2.325,-1.375) rectangle (2.325,1.375);
      \begin{scope}[shift={(0,.125)}]
        \node [rcol,vertex,large] (s1) at (-1.2,-1) {\tiny$a_1$};
        \node [rcol,vertex,large] (s2) at (.6,-1) {\tiny$a_2$};
        \node [rcol,vertex,large] (s3) at (1.8,-1) {\tiny$a_3$};
        
        \coordinate (t1) at (-1.8,1) {};
        \coordinate (t2) at (-1.2,1) {};
        \coordinate (t3) at (-.6,1) {};
        \coordinate (t4) at (0,1) {};
        \coordinate (t5) at (.6,1) {};
        \coordinate (t6) at (1.2,1) {};
        \coordinate (t7) at (1.8,1) {};
        
        \node [ycol,vertex,large] (ss1) at (-1.2,-.3) {};
        \node [ycol,vertex,large] (ss2) at (.6,-.3) {};
        \node [ycol,vertex,large] (ss3) at (1.8,-.3) {};
        
        \node [bcol,vertex,large] (m1) at (-1.8,.4) {\tiny$b_1$};
        \node [bcol,vertex,large] (m2) at (-1.2,.4) {\tiny$b_1$};
        \node [bcol,vertex,large] (m3) at (-.6,.4) {\tiny$b_1$};
        \node [bcol,vertex,large] (m4) at (0,.4) {\tiny$b_1$};
        \node [bcol,vertex,large] (m5) at (.6,.4) {\tiny$b_1$};
        \node [bcol,vertex,large] (m6) at (1.2,.4) {\tiny$b_2$};
        \node [bcol,vertex,large] (m7) at (1.8,.4) {\tiny$b_2$};
        
        \draw [rcol,edge] (s1) -- (ss1);
        \draw [rcol,edge] (s2) -- (ss2);
        \draw [rcol,edge] (s3) -- (ss3);
        
        \draw [ycol,edge] (ss1) -- (m1);
        \draw [ycol,edge] (ss1) -- (m2);
        \draw [ycol,edge] (ss1) -- (m3);
        \draw [ycol,edge] (ss2) -- (m4);
        \draw [ycol,edge] (ss2) -- (m5);
        \draw [ycol,edge] (ss2) -- (m6);
        \draw [ycol,edge] (ss3) -- (m7);

        \draw [bcol,edge] (m1) -- (t1);
        \draw [bcol,edge] (m2) -- (t2);
        \draw [bcol,edge] (m3) -- (t3);
        \draw [bcol,edge] (m4) -- (t4);
        \draw [bcol,edge] (m5) -- (t5);
        \draw [bcol,edge] (m6) -- (t6);
        \draw [bcol,edge] (m7) -- (t7);

      \end{scope}
      
      \node at (0, -1.75) {$2\yon^3 + \yon$};
    \end{scope}
    \draw[<->] (-.4,2.15) -- node[above=-.1]{$\dual{}$} (.4, 2.15);
    \draw[<->] (-.4,-2.15) -- node[above=-.1]{$\dual{}$} (.4, -2.15);
    \draw[<->] (-3.15,.4) -- node[left]{$\ldag$} node[right]{$\rdag$} (-3.15, -.4);
    \draw[<->] (3.15,.4) -- node[left]{$\rdag$} node[right]{$\ldag$} (3.15, -.4);
  \end{tikzpicture}
  \caption{Weak duals and adjoints of conjunctive and linear bicomodules
    between $C\yon$ and $D\yon$.}
\end{figure}

\begin{remark}[Transpose as adjoint weak dual, in terms of
  $\Delta,\Sigma,\Pi$]
  As in \cref{rem.duals_delta_sigma_pi}, consider the span
  $C\From{f}M\To{g}D$ and corresponding bicomodule
  $M\yon \colon C\yon\bito D\yon$, which corresponds to
  $\Sigma_g\circ\Delta_f\colon C\set\to D\set.$ Its right adjoint
  $(M\yon)\rdag \colon D\yon\bito C\yon$ corresponds to
  $\Pi_f\circ\Delta_g$, and its right adjoint's weak dual corresponds to
  $\Sigma_f\circ\Delta_g$. Its weak dual
  $\dual{(M\yon)} \colon C\yon\bito D\yon$ corresponds to $\Pi_f\circ\Delta_g$,
  and hence its weak dual's left adjoint is again
  $\Sigma_f\circ\Delta_g$. In both cases the result,
  $\Sigma_f\circ\Delta_g$ corresponds to the transpose of the original
  span.
\end{remark}

\subsubsection{Categories and profunctors in $\ccatsharp$}\label{sec.prof}

In this section we use the machinery built up in previous sections to
examine two ways the bicategory $\ccat$ of categories and profunctors
arises in $\ccatsharp$. We will look at (co)monads and bi(co)modules
\emph{within} $\ccatsharp \coloneqq \ccomod(\poly)$, i.e.,
$\ccomod(\ccomod(\poly))$ and $\mmod(\ccomod(\poly))$.

First, we will see a way in which $\ccatsharp$ arises within itself.



\begin{proposition}\label{prop.iteratedcomonoids}
  Let $\mathscr{V}$ be a monoidal category whose tensor product
  preserves coreflexive equalizers in each variable (so that the
  framed bicategory $\ccomod(\mathscr{V})$ exists). Then we may
  describe comonads, comonad homomorphisms, and bicomodules in
  $\ccomod(\mathscr{V})$ (i.e., $\ccomod(\ccomod(\mathscr{V}))$) as
  follows.
  \begin{itemize}
  \item A comonad in $\ccomod(\mathscr{V})$ on the object (i.e.,
    comonoid) $d$ amounts to a comonoid $c$ in $\mathscr{V}$ over $d$
    (that is, equipped with a homomorphism of comonoids $c \to d$).
  \item A homomorphism of comonads in $\ccomod(\mathscr{V})$ amounts
    to a homomorphism of comonoids over $d$ (that is, a commutative
    triangle of homomorphisms) in $\mathscr{V}$.
  \item A bicomodule $c \altbito c'$ in $\ccomod(\mathscr{V})$
    just amounts to a bicomodule in $\mathscr{V}$ between the
    comonoids $c$ and $c'$.
\end{itemize}
\end{proposition}

In fact, this is true in greater generality: dropping the assumption
that the tensor product preserves coreflexive equalizers,
$\ccomod(\mathscr{V})$ is merely a covirtual double category (rather
than a double category). However, (co)virtual double categories are
outside the scope of this paper.
\begin{proof}
  As in \cref{ex.cokleisli}, this is obtained by dualizing the usual
  arguments for algebras. See also \cite[Examples 11.3 and
  11.6]{shulman2008framed} and \cite[Theorem~2.3.18]{Spivak.Schultz.Rupel:2016a}.
\end{proof}


Applied to $\ccatsharp \coloneqq \ccomod(\poly)$, we obtain that a
comonoid in $\ccatsharp$ on the category $d$ amounts to a category
$c$ with a retrofunctor into $d$. In particular, we note that every
category comes with a canonical retrofunctor into the discrete
category consisting of its objects.

\begin{proposition}\label{prop.racs_lams}\label{cor.omnibus}
  The framed bicategory $\ccatsharp$ is identified with the full sub
  framed bicategory of $\ccomod(\ccatsharpdisc)$ spanned by the
  comonads with conjunctive carrier.
\end{proposition}
\begin{proof}
  Follows from \cref{prop.iteratedcomonoids}, noting that the comonads in
  $\ccatsharpdisc$ with conjunctive carrier are precisely those
  corresponding to retrofunctors into discrete categories such that
  the map on objects is bijective, i.e., canonical retrofunctors from
  categories into their underlying sets of objects.
\end{proof}

\begin{corollary}\label{cor.profdisc}
  The bicategory of conjunctive polynomial bicomodules $\ccatsharpcon$
  is identified with $\ccomod(\ccatsharpdisccon)$.
\end{corollary}

Recall from \cref{prop.span_iso_lin} that
$\ccatsharplin \cong \sspan$, and so
$\mmod(\ccatsharplin) \cong \mmod(\sspan) \cong \ccat\op$, the
(opposite of the) framed bicategory of categories and
profunctors.\footnote{Now in the context of this paper, it seems a
  sensible convention to identify $\mmod(\ccatsharplin)$ with
  $\ccat\op$ rather than $\ccat$, so that the arrows in a category as
  well as the heteromorphisms of a profunctor point as directions of
  polynomials. Indeed, based on our previous conventions, a span from
  $A$ to $B$ as a bicomodule $A \bito B$ has directions pointing in
  the reverse: positions labeled by $B$ and directions labeled by
  $A$. So to us a category is seen more naturally as a monoid in
  $(\ccatsharplin)\op$ than a monoid in $\ccatsharplin$.} Recall also
from \cref{prop.profunctor} that $\ccatsharpcon$ is canonically
identified with $\ccat\op$.

\begin{proposition}\label{prop.new_omnibus}
  Weak dualization (\cref{thm.linear_conjunctive_dual}) provides an
  equivalence of bicategories
  \[
    \begin{tikzcd}[column sep=60pt]
      \phantom{ \cong(\ccatsharpcon)\co} \mmod(\ccatsharplin) \ar[r,
      shift left, "\dual{-}"]&
      \ccomod(\ccatsharpdisccon)\co\cong(\ccatsharpcon)\co \ar[l,
      shift left, "\dual{-}"]
    \end{tikzcd}
  \]
  that is also obtained by identifying of both sides with the
  (opposite of the) bicategory $\ccat$ of categories and profunctors.
\end{proposition}
\begin{proof}
  The equivalence of bicategories induced by weak dualization described in
  \cref{thm.linear_conjunctive_dual} identifies respective (co)monads
  and bi(co)modules. Moreover, it is easy to check that a profunctor
  construed as a bimodule in $\ccatsharplin$ weak dualizes to the same
  profunctor construed as a conjunctive bicomodule in $\poly$.
\end{proof}

\section{Aggregation}\label{chap.aggregation}

The previous section showed how $\ccat$, the framed bicategory of
categories, fits within $\ccatsharp$, the framed bicategory of
parametric right adjoint functors between copresheaf categories. The
latter are data migration functors, as described in the introduction.

Thus we can import concepts from category theory, placing
them in the same framework that supports data migration. In the
followng sections, we present aggregation as an example of what
can be expressed in this language (whereas we had previously found it
difficult to discuss within a unified formal framework that also
includes querying).

First, recall from the introduction what aggregation is: given a
commutative monoid $M$ and a function $s$ from a finite set $E$ to a
finite set $D$, there is a canonical map from functions
$E \to M$ to functions $D \to M$ given by \emph{aggregating} along $s$:

\begin{equation}\label{eqn.aggregate_general}
  \begin{tikzcd}
    E\ar[r, "s"]\ar[d, "\pi"']&M\\
    D\ar[ur, dashed, "(\mathtt{sum}\ s)_\pi"']
  \end{tikzcd}
\end{equation}
where \eqref{eqn.aggregate_general} is not intended to commute.

We briefly give an abstract account of what is happening here. To
rephrase the above, every function between finite sets $s: E \to D$
determines a function $\mathtt{sum}\ s: M^E \to M^D$. We will see, as
this suggests, that aggregation amounts to a functor
$\mathtt{sum}\colon \finset \to \smset$ mapping $E$ to $M^E$.

The key is that $\finset$ under $+$ is (up to equivalence) the free
symmetric monoidal category on a commutative monoid object: there is a
distinguished commutative monoid object in $(\finset, 0, +)$ carried by
$1$, with its $N$-ary multiplication given by the unique arrow
$N \cong 1 + \ldots + 1 \to 1$. Moreover, the arrows in $\finset$ are
freely generated under $+$ by these arrows into $1$. (Every function
decomposes uniquely as a coproduct of functions into $1$: the
codomain decomposes as elements, and the domain decomposes as
corresponding fibers.) Indeed, any commutative monoid object in a
symmetric monoidal category extends uniquely (up to isomorphism) to a
symmetric monoidal functor from $(\finset, 0, +)$.

In particular, a commutative monoid $M$ (in $\smset$) induces an
aggregation functor $\finset \to \smset$ where $N \mapsto M^N$; the
arrow $N \to 1$ is sent to the monoid multiplication $M^N \to M$, and,
since we send $+$ to $\times$, an arbitrary arrow in $\finset$ is sent
to the product of multiplication maps its fibers are sent to. Thus, as
in \eqref{eqn.aggregate_general}, a map $E \to D$ between finite sets
induces a map $M^E$ to $M^D$.

That is aggregation in the abstract. Now on to our goal in this
section: to place aggregation within $\ccatsharp$ (i.e., in the
language of data migration). As ordinary in categorical database
theory (\cite{spivak2012functorial}), we will model database
instances as $c$-sets; however, in order to handle aggregation, we
furthermore need to assign values in commutative monoids to the
elements in the $c$-sets. We will be interested in modeling situations
in which each object potentially takes values in a different
commutative monoid, as occurs in practice.


Here is the full story in brief: if $X$ is a finite-valued
$c$-set, and we additionally have a monoid $M_a$ and an assignment
$\alpha_a: X(a) \to M_a$ for each object $a$, then we can aggregate along
$s: a \to b$ in $c$ by sending it through the composite
$c \To{X} \finset \To{\agg_a} \smset$, where $\agg_a$ denotes the
aggregation functor $\finset \to \smset$ corresponding to the
commutative monoid $M_a$.
We will express this as a bicomodule map \eqref{eqn.key_diagram},
which at once sends every arrow $f: a \to b$ in $c$ to the appropriate
aggregated $X(b)$-tuple of elements of $M_a$.

\subsection{Schemas and instances}

Our first step is to define database schemas and instances on
them. Our definition (\cref{def.database_schema}) is not the most
general nor the most commonly considered one (as we will explain
in \cref{rem.simplistic_attributes}). It was chosen for our purpose in
this paper, which is simply to show that the most basic aspect of
aggregation can be defined in terms of the universal structures
(monoidal structures, adjoints, etc.) in $\ccatsharp$, and to leave
the development of the theory for later work; the paper is already
long enough.

\begin{definition}[Schema, instance]\label{def.database_schema}
  A \emph{schema} $(\cat{c}, \agg)$ consists of
  \begin{itemize}
  \item a small category $\cat{c}$ (the \emph{specification of the
      database}, as in \cite{spivak2012functorial}) and
  \item an \emph{aggregation} functor $\agg_a: \finset \to \smset$ for
    each object $a$ of $\cat{c}$ (specifying for each finite set its
    possible assignments of aggregable values, and specifying how
    these value-assignments aggregate along functions, as in the
    introduction to this section).
  \end{itemize}

  An \emph{instance} $(X, \alpha)$ on the schema $(\cat{c},\agg)$
  consists of
  \begin{itemize}
  \item a finite-valued $\cat{c}$-set $X\colon\cat{c}\to\finset$ (the
    \emph{data}, as in \cite{spivak2012functorial}) and
  \item an element $\alpha_a$ of $\agg_a(X(a))$ for each object
    $a\in\ob(\cat{c})$ (the \emph{assignment of aggregable values} for
    the set of elements at $a$).
  \end{itemize}
\end{definition}

\begin{example}[Concrete database schema]
  If $\cat{c}$ is a category and
  $M\colon \ob(\cat{c})\to\Cat{CommMon}$ is a function assigning to
  each object $a$ of $\cat{c}$ a commutative monoid $M_a$, we obtain a
  schema where $\agg_a$ is the aggregation functor $N\mapsto M_a^N$, as described in the introduction to this section. The
  functors $\finset \to \smset$ corresponding to commutative monoids
  in this way are precisely those that send finite coproducts to
  products.

  As for an instance, the $c$-set $X$ describes data as usual (see
  \cref{ex.variable_is_duc} or \cite{spivak2012functorial}) and
  $\alpha_a$ simply amounts to a map $X(a) \to M_a$, assigning each
  element of data a value in the corresponding commutative monoid.

  We call this a \emph{concrete database schema}, since this is the
  situation arising in practice we wish to model. The commutative
  monoid $M_a$ models \emph{attributes} for $a$, e.g., possible
  salaries for an employee.
\end{example}

\begin{example}[Other aggregation functor] 
  Nothing prevents us from using more general ``aggregation'' functors
  $\finset \to \smset$. For example, we may use the coproduct of
  aggregation functors corresponding to several commutative monoids,
  which would let us assign to all elements values in just one of
  these monoids at a time. (However, this would seem to be of limited
  practical use.)
\end{example}

\begin{example}[Product aggregation functor]
  More practically useful, the product of aggregation functors
  corresponding to several commutative monoids allows us assign to
  every element a value in each of these monoids. But note that
  this is simply the aggregation functor corresponding to the product
  of the monoids.

  Hence concrete database schemas are able to model multiple
  attributes per object.
\end{example}

\begin{remark}\label{rem.simplistic_attributes}
  Our notion of attributes in a concrete database schema is a bit
  different than the notion of attributes one finds in other
  category-theoretic work on databases. In \cite{johnson2000entity},%
  \footnote{This article was prescient in the sense that EA
    (entity-attribute) sketches are coproduct-limit sketches, which
    seem quite related to prafunctors (indexed disjoint unions of
    conjunctive queries). One could investigate relationships between
    EA sketches and algebras for monads $c\bito[m]c$.}  the authors
  take an attribute on $a\in\cat{c}$ to just be a map from $a$ to a
  coproduct of copies of 1. This puts the attributes directly into the
  schema, whereas for implementations it often seems preferable to
  have each attribute point to an external data type. That is what's
  done in \cite{patterson2021categorical}, in which the types are
  taken to be a discrete category $T$ (which is then equipped with a
  map $V\colon T\to\smset$, assigning types their elements), and the
  attributes on each object $a\in\cat{c}$ are taken to be a $T$-set
  --- i.e., a set $P(a, t)$ for every $t \in T$ --- varying
  functorially in $a$. Functoriality in $a$ means that attributes on
  $a$ play a role similar to arrows out of $a$, in that each
  $a' \to a$ and attribute on $a$ induces an attribute on $a'$. This
  amounts to a profunctor $P\colon \cat{c}\profunctor T$. Also
  relevant is \cite{schultz2017algebraic} (which has been implemented
  by a coauthor of that paper, Ryan Wisnesky) where the authors use
  not a discrete category $T$ but a multi-sorted theory $\Cat{Type}$
  (representing a programming language), allowing attributes to form
  algebraic structures: the attributes on each object $a\in\cat{c}$
  are taken to be an algebra $P_a$ (again varying functorially in
  $a$).

  Let us attempt to connect our notion of concrete database schema
  with \cite{patterson2021categorical} and
  \cite{schultz2017algebraic}. To us, each object has one attribute
  type with just one attribute (so it's as though we multiplied all
  the attributes together).
  An analogous schema in \cite{patterson2021categorical} would simply
  identify $T$ with $\ob(\cat{c})$ and take each $P(a, a)$ to be free
  on a single element. Our $\agg_a(1)$ is the analogue of $V(a)$
  (whereas the latter has no ability to aggregate\footnote{To go from
    an attribute with values in a set $V$ to a commutative monoid
    attribute, the free thing to do is replace $V$ by $M(V)$, the
    monoid of multisets in $V$. Thus the approaches in
    \cite{johnson2000entity} or \cite{patterson2021categorical} can be
    replaced by one in our formulation that enables the ``group-by''
    operation in database theory.}). On the other hand, an analogous
  schema in \cite{schultz2017algebraic} would take $\Cat{Type}$ to be
  the free finite product theory on $\ob(\cat{c})$ many monoid objects
  and take $P_a$ to be freely generated by an element in the monoid
  for $a$. Our entire $\agg$ corresponds to fixing a particular
  algebra $V$ of $\Cat{Type}$.\footnote{Our description of concrete
    database schema casts monoids as algebras of the \emph{symmetric
      monoidal} theory of a monoid $(\finset, 0, +)$, rather than
    algebras of the \emph{finite product} theory of a monoid.}

  Hence we see our notion of concrete database schema differs from
  \cite{schultz2017algebraic} as follows: we require $\Cat{Type}$ to
  be a particular form, supporting just one sort of summable attribute
  for each object, and we use just one attribute per object $a$ where
  these do not functorially in $a$ (or to put it another way, we
  require they vary freely).
\end{remark}

\subsection{Aggregation in $\ccatsharp$}

Now we express the notions of schema and instance from the previous
section within the language of $\ccatsharp$. In \cref{sec.prof} we saw
several ways to view the bicategory $\ccat$ of categories and
profunctors in $\ccatsharp$: as $(\ccatsharpcon)\co$ or its opposite,%
\footnote{Recall that the assignment $c\mapsto c\op$ identifies the bicategory of categories and profunctors with its opposite.}
as $\ccomod(\ccatsharpdisccon)\co$ or its opposite, or as
$\mmod(\ccatsharplin)$ or its opposite (moreover, any can be
translated into any other by way of universal constructions in
$\ccatsharp$). As noted in \cref{rem.linearconj}, out of these only
$\mmod(\ccatsharplin)$ and its opposite furthermore feature the
\emph{squares} of $\ccat$ as a \emph{framed} bicategory.

In \cref{ex.full_internal_finset} we saw how subcategories of
$\smset\op$ arise naturally in $\ccatsharp$, and in this way we
construct $\finset\op$ (up to equivalence).

\begin{definition}
  The \emph{list polynomial}\footnote{We use the notation $u$ because the HoTT version of $u$ can be regarded as the universe of finite sets; see \cite{aberle2024polynomial}.} $u \in \poly$ is the polynomial given by the formula
\[
  u\coloneqq\sum_{N\in\nn}\yon^{N}.
\]
\end{definition}
\begin{lemma}\label{cor.skeleton_finset}
  The induced comonad structure on the left Kan extension of $u$ along
  itself $\lens{u}{u}$ (see \cref{prop.JoshMeyers}) as a category is a
  skeleton of $\finset\op$.
\end{lemma}
\begin{proof}
  As per \cref{ex.full_internal_finset}, $\lens{u}{u}$ is the opposite
  of the full subcategory of $\smset$ spanned by the sets $N$, and
  every finite set is in bijection with some such $N$.
\end{proof}
To rephrase the definition of schema (\cref{def.database_schema}) in a
slightly more unified way, it is a category $\cat{c}$ and a profunctor
$\agg: \ob(\cat{c}) \profunctor \finset$. In particular we have that
$\agg$ (really $\agg\op: \finset\op \profunctor \ob(\cat{c})$) amounts
to a conjunctive bicomodule as per \cref{prop.profunctor}
\[c(1)\yon \bito[\agg] \lens{u}{u}\]
and we may subsequently choose to translate this into various
equivalent forms. Indeed, applying
$(\ccatsharpcon)\co \cong \ccomod(\ccatsharpdisccon)\co$ as in
\cref{cor.profdisc} and
$\ccomod(\ccatsharpdisccon)\co \cong \mmod(\ccatsharplin)$ as in
\cref{prop.new_omnibus} this becomes a left module (or a bimodule, but
the right module structure is superfluous) in $\ccatsharplin$:
\begin{equation}\label{eqn.aggleftmod}
  \begin{tikzcd}[row sep=0]
    &|[alias=u1]|u(1)\yon\ar[dr, bend left=15pt, biml-bimr, "\dual{\lens{u}{u}}"]\\
    c(1)\yon\ar[ur, bend left=15pt, biml-bimr,
    "\dual{\agg}"]\ar[rr, bend right, biml-bimr, "\dual{\agg}"'
    name=olu]&&u(1)\yon
    \ar[from=u1, to=olu-|u1, shorten >=5mm,shorten <=1mm, Rightarrow, "\lambda_\agg"]
  \end{tikzcd}
\end{equation}

As for expressing an \emph{instance}, we also have flexibility. The
finite-valued presheaf $X$ may either be viewed directly as a functor
$\cat{c} \to \finset$ or just as a $c$-set $\cat{c} \to \smset$
satisfying a finiteness property. The former will be more
straightforward to describe in $\ccatsharp$.
Thus we express
$X\colon\cat{c} \to \finset$ (really
$X\op\colon \cat{c}\op \to \finset\op$) as a homomorphism of monads
in $\ccatsharplin$:
\begin{equation}\label{eqn.monoidhom}
  \begin{tikzcd}[column sep=50pt]
    c(1)\ar[r,biml-bimr,"c\ldag" name=M]\ar[d,"\ob(X)"']&c(1)\ar[d,"\ob(X)"]\\
    u(1)\ar[r,biml-bimr,"\dual{\lens{u}{u}}"' name=M']&u(1)
    \ar[from=M, to=M', Rightarrow, shorten=2mm, "X"]
  \end{tikzcd}
\end{equation}

Finally, the assignment $\alpha$ is a square as shown left in
$\ccat$, corresponding to a square as shown right in $\ccatsharplin$:
\begin{equation}\label{eqn.alphasquare}
  \begin{tikzcd}
    \ob(\cat{c})\ar[r,equal,""name=M]\ar[d,equal]&\ob(\cat{c})\ar[d,"\ob(X)"]\\
    \ob(\cat{c})\ar[r,tick,"\agg"']\ar[r,phantom, name=M']&\finset
    \ar[from=M, to=M', Rightarrow, shorten=2mm, "\alpha"]
  \end{tikzcd}
  \hspace{.6in}
  \begin{tikzcd}
    c(1)\yon\ar[r,equal, "" name=M]\ar[d,equal,"c(1)\yon"']&c(1)\yon\ar[d,"\ob(X)"]\\
    c(1)\yon\ar[r,biml-bimr,"\dual{\agg}"' name=M']&u(1)\yon
    \ar[from=M, to=M', Rightarrow, shorten=2mm, "\alpha"]
  \end{tikzcd}
\end{equation}

\begin{theorem}[Aggregation]\label{thm.main}
  Let $(c,\agg)$ be a database schema and $(X,\alpha)$ a finite-valued
  instance on it. Then there is an induced homomorphism of left modules (shown
  left) relative to the monad homomorphism
  $X\colon c\ldag\to\dual{\lens{u}{u}}$ (shown right):
  \[
    \begin{tikzcd}[column sep=50pt]
      c(1)\yon\ar[r, biml-bimr, "c\ldag", ""' name=c1]\ar[d, equal]&c(1)\yon\ar[d, "\ob(X)"]\\
      c(1)\yon\ar[r, biml-bimr, "\dual{\agg}"', "" name=u1]&u(1)\yon
      \ar[from=c1, to=u1, shorten=2mm, Rightarrow, "\agg\alpha"]
    \end{tikzcd}
    \hspace{1in}
    \begin{tikzcd}[column sep=50pt]
      c(1)\yon\ar[r, biml-bimr, "c\ldag", ""' name=c1]\ar[d, "\ob(X)"']&c(1)\yon\ar[d, "\ob(X)"]\\
      u(1)\yon\ar[r, biml-bimr, "\dual{\lens{u}{u}}"', "" name=u1]&u(1)\yon
      \ar[from=c1, to=u1, shorten=2mm, Rightarrow, "X"]
    \end{tikzcd}
  \]
\end{theorem}
\begin{proof}
  The desired map $\agg\alpha$ is defined to be the composite
  \begin{equation}\label{eqn.key_diagram}
    \begin{tikzcd}[column sep=large]
      c(1)\yon\ar[d, equal]\ar[r, equal, "c(1)\yon", ""' name=c2]\ar[rr, bend left=40, biml-bimr, "c\ldag", ""' name=c1]&
      c(1)\yon\ar[d, "\ob(X)"]\ar[r, biml-bimr, "c\ldag", ""' name=c3]&
      c(1)\yon\ar[d, "\ob(X)"]\\
      c(1)\yon\ar[r, biml-bimr, "\dual{\agg}"' name=u2]\ar[rr, bend right=40, biml-bimr, "\dual{\agg}"' name=u4]&
      u(1)\yon\ar[r, biml-bimr, "\dual{\lens{u}{u}}"' name=u3]&
      u(1)\yon
      \ar[from=c1, to=c2-|c1, shorten >=4mm, equal]
      \ar[from=c2, to=u2, shorten=3mm, Rightarrow, "\alpha"]
      \ar[from=c3, to=u3, shorten=3mm, Rightarrow, "X"]
      \ar[from=u2-|u4, to=u4, shorten=1mm, Rightarrow, "\lambda_\agg"]
    \end{tikzcd}
  \end{equation}
  Here, the 2-cell $\agg$ defines aggregation as in
  \eqref{eqn.aggleftmod}, the 2-cell labeled $X$ defines $X$ as a
  finite-valued $c$-set as in \eqref{eqn.monoidhom}, and the 2-cell
  labeled $\alpha$ completes the instance $(X,\alpha)$ as in
  \eqref{eqn.alphasquare}.
  These diagrams live entirely within $\ccatsharplin \cong \sspan$;
  the module homomorphism corresponds to a map between the profunctors
  $c\op(\thg,\thg)\colon\cat{c}\op \profunctor \ob(\cat{c})$ and
  $\agg\op\colon\finset\op \profunctor \ob(\cat{c})$.
\end{proof}

\subsection{Conclusion}
We hope this paper will encourage further investigation into the
framed bicategory $\ccatsharp$. The recently established
correspondence between categories and polynomial comonads is
delightfully striking, where retrofunctors arise as a natural object
of study. In this paper we filled in more of this story by providing a
proof that polynomial bicomodules from $c$ to $d$ are parametric right
adjoints from $c\set$ to $d\set$. We also constructed a closed
monoidal structure, the Dirichlet product, on each such category.

We saw that within $\ccatsharp$, the operation of transposing a span
splits into two more primitive operations: taking adjoints and taking
weak duals. There is room for further investigation of weak duals in
$\ccatsharp$; we only considered the case where the categories
involved are discrete, focusing in particular on linear bicomodules
(spans) and conjunctive bicomodules.

It would be interesting to see what is the most general context in
which various parts of the theory go through. For simplicity we have
always stayed internal to $\smset$, but presumably much of what we
have done generalizes.

We have kept our treatment of database aggregation brief. Still, it
should serve to illustrate the expressive power of the language of
$\ccatsharp$. Since all of $\ccat$ arises within $\ccatsharp$, we can
translate a whole cast of structures into suitable arrangements of
polynomial bicomodules. In this way it should be possible to realize
arbitrary ways of manipulating data, for example more sophisticated
variants of aggregation than the basic example we showcased. Indeed,
we hope placing $\ccat$ within the context of $\ccatsharp$ provides a
link between the data migration functors of existing categorical
database theory and seemingly more general forms of data manipulation.

\bibliographystyle{elsarticle-num}
\bibliography{Library20211112}

\end{document}